\documentclass[11pt,reqno]{amsart}
\usepackage{times}
\usepackage{amsmath,amsfonts,amstext,amssymb,amsbsy,amsopn,amsthm,eucal}
\usepackage{txfonts}
\usepackage{dsfont}
\usepackage{graphicx}   
\usepackage{hyperref}
\usepackage{color}
\usepackage{verbatim}  

\numberwithin{equation}{section}
\hypersetup{
  pdftitle={Rectifiability of Singular Sets in Lower and Bounded Ricci Curvature},
  pdfauthor={Jeff Cheeger, Wenshuai Jiang and Aaron Naber },
  pdfsubject={},
  pdfkeywords={},
  pdfpagelayout=SinglePage,
  pdfpagemode=UseOutlines,
  colorlinks,
  linkcolor=[rgb]{0,0,0.7},
  urlcolor=[rgb]{0,0,0.4},
  citecolor=[rgb]{0.4,0.1,0}
}

\renewcommand{\rmdefault}{ptm}
 
\renewcommand{\baselinestretch}{1.15}

\newcommand{\togh}{\stackrel{d_{GH}}{\longrightarrow}}
\newcommand{\hs}{\text{ }\text{ }}
\newcommand{\hhs}{\text{ }\text{ }\text{ }\text{ }}
\newcommand{\hhhs}{\text{ }\text{ }\text{ }\text{ }\text{ }\text{ }}
\newcommand{\rv}{{\rm v}}
\newcommand{\KK}{\mathds{K}}
\newcommand{\NN}{\mathds{N}}
\newcommand{\RR}{\mathds{R}}
\newcommand{\Sn}{\mathds{S}}
\newcommand{\Diff}{\text{Diff}}

\newcommand{\Tw}{\text{Tw}}
\newcommand{\Cyl}{\text{Cyl}}
\renewcommand{\ker}{\text{ker}}
\newcommand{\Ker}{\text{Ker}}
\newcommand{\im}{\text{im}}
\renewcommand{\Im}{\text{Im}}
\newcommand{\degree}{\text{deg}}
\newcommand{\Ran}{\text{Ran}}
\newcommand{\Span}{\text{span}}
\newcommand{\tr}{\text{tr}}
\newcommand{\WM}{\text{WM}}

\newcommand{\Spec}{\text{Spec}}

\newcommand{\Rm}{{\rm Rm}}
\newcommand{\Ric}{{\rm Ric}}
\newcommand{\Vol}{{\rm Vol}}
\newcommand{\diam}{{\rm diam}}
\newcommand{\divergence}{\text{div}}
\newcommand{\inj}{\text{inj}}
\newcommand{\Sing}{{\rm Sing}}
\newcommand{\tv}{{\rm v}}
\newcommand{\Lip}{{\rm Lip}}
\newcommand{\dC}{\mathds{C}}
\newcommand{\dG}{\mathds{G}}
\newcommand{\dk}{\mathds{k}}
\newcommand{\dN}{\mathds{N}}
\newcommand{\dP}{\mathds{P}}
\newcommand{\dQ}{\mathds{Q}}

\newcommand{\dR}{\mathbb{R}}

\newcommand{\dS}{\mathds{S}}
\newcommand{\dZ}{\mathds{Z}}
\newcommand{\Hess}{\text{Hess}}
\newcommand{\dist}{\text{dist}}
\newcommand{\Exp}{\text{exp}}
\newcommand{\Cl}{\text{Cl}}
\newcommand{\bt}{\text{\bf{t}}}
\newcommand{\cA}{\mathcal{A}}
\newcommand{\cB}{\mathcal{B}}
\newcommand{\cBA}{\mathcal{BA}}
\newcommand{\cC}{\mathcal{C}}
\newcommand{\cD}{\mathcal{D}}
\newcommand{\cE}{\mathcal{E}}

\newcommand{\cF}{\mathcal{P}}

\newcommand{\cG}{\mathcal{G}}
\newcommand{\cH}{\mathcal{H}}
\newcommand{\cI}{\mathcal{I}}
\newcommand{\cL}{\mathcal{L}}
\newcommand{\cLM}{\mathcal{LM}}
\newcommand{\cM}{\mathcal{M}}
\newcommand{\cN}{\mathcal{N}}
\newcommand{\cOC}{\Omega\mathcal{C}}
\newcommand{\cP}{\mathcal{P}}
\newcommand{\cPM}{\mathcal{PM}}
\newcommand{\cR}{\mathcal{R}}
\newcommand{\cS}{\mathcal{S}}
\newcommand{\wcS}{\widehat{\mathcal{S}}}
\newcommand{\cU}{\mathcal{U}}
\newcommand{\cV}{\mathcal{V}}
\newcommand{\cW}{\mathcal{W}}
\newcommand{\fg}{\mathfrak{g}}
\newcommand{\cVP}{\mathcal{VP}}

\newcommand{\codim}{\text{codim}}
\newcommand{\Hom}{\text{Hom}}
\newcommand{\Prim}{\text{Prim}}
\newcommand{\simp}{\text{simp}}
\newcommand{\op}{\text{op}}

\newcommand{\norm}[1]{\left\|#1\right\|}
\newcommand{\ps}[2]{\left\langle#1,#2\right\rangle}
\newcommand{\ton}[1]{\left(#1\right)}
\newcommand{\qua}[1]{\left[#1\right]}
\newcommand{\cur}[1]{\left\{#1\right\}}
\newcommand{\abs}[1]{\left|#1\right|}
\newcommand{\wto}{\rightharpoonup}

\newcommand{\Sr}{S_{\bar r}}
\newcommand{\tSr}{\tilde S _{\bar r}}
\newcommand{\be}{Besicovitch }

\newcommand{\B}[2]{B_{#1}\ton{#2}}
\newcommand{\tB}[2]{\tilde B_{#1}\ton{#2}}
\newcommand{\supp}[1]{\operatorname{supp}\ton{#1}}

\newcommand{\N}{\mathbb{N}}
\newcommand{\Z}{\mathbb{Z}}
\newcommand{\R}{\mathbb{R}}

\renewcommand{\paragraph}[1]{\ \newline \ \textbf{#1\ }}

\newcommand{\hol}{H\"older }
\newcommand{\al}{Ahlfors }

\newtheorem{theorem}{Theorem}[section]

\newtheorem{proposition}[theorem]{Proposition}
\newtheorem{lemma}[theorem]{Lemma}

\newtheorem{sublemma}[theorem]{Sublemma}
\newtheorem{definition}[theorem]{Definition}
\theoremstyle{definition}
\newtheorem{remark}[theorem]{Remark}
\theoremstyle{remark}
\newtheorem{example}[theorem]{Example}

\begin{document}

\title{Rectifiability of Singular Sets in Noncollapsed Spaces\\
 with Ricci Curvature bounded below}

\author{Jeff Cheeger}
\address{J. Cheeger, ~Courant Institute, 251 Mercer St., New York, NY 10011}
\email{cheeger@cims.nyu.edu}
\author{Wenshuai Jiang}
\address{W. Jiang, ~School of mathematical Sciences,  Zhejiang University}
\email{wsjiang@zju.edu.cn}
\author{Aaron Naber}\thanks{The first author was supported by NSF Grant DMS-1406407 and a Simons Foundation Fellowship. The second author was partially supported by the EPSRC on a Programme Grant entitled "Singularities of Geometric Partial Differential Equations" reference number EP/K00865X/1, the Fundamental Research Funds for the Central Universities (No. 2017QNA3001) and NSFC (No. 11701507). The third author was supported by NSF grant DMS-1809011}

\address{A. Naber, ~Department of mathematics, 2033 Sheridan Rd., Evanston, IL 60208-2370}
\email{anaber@math.northwestern.edu}


\date{\today}

\begin{abstract}

This paper is concerned with the structure of Gromov-Hausdorff limit 
spaces $(M^n_i,g_i,p_i)\stackrel{d_{GH}}{\longrightarrow} (X^n,d,p)$ of 
Riemannian manifolds satisfying a uniform lower Ricci curvature bound $\Ric_{M^n_i}\geq -(n-1)$ 
as well as the noncollapsing assumption $\Vol(B_1(p_i))>\rv>0$.  
In such cases, there is a filtration of the singular set, $S_0\subset S_1\cdots S_{n-1}:= S$,
where $S^k:= \{x\in X:\text{ no tangent cone at $x$ is }(k+1)\text{-symmetric}\}$;
equivalently no tangent cone splits off a Euclidean factor $\R^{k+1}$ isometrically.  Moreover, by \cite{ChCoI},
 $\dim S^k\leq k$.  
However, little else has been understood about the structure of the singular set $S$.

Our first result for such limit spaces $X^n$ states that $S^k$ is $k$-rectifiable. 
 In fact, we will show that for $k$-a.e. $x\in S^k$, {\it every} tangent cone $X_x$ at $x$ is $k$-symmetric i.e.
that $X_x= \dR^k\times C(Y)$ where $C(Y)$ might depend on the particular $X_x$.
 We use this to show that there exists $\epsilon=\epsilon(n,\tv)$,
and a $(n-2)$-rectifible set $S^{n-2}_\epsilon$, with finite 
$(n-2)$-dimensional Hausdorff measure $\cH^{n-2}(S_\epsilon^{n-2})<C(n,\rv)$, such that
 $X^n\setminus S^{n-2}_\epsilon$ is bi-H\"older equivalent to a smooth riemannian manifold.  
This improves the regularity results of \cite{ChCoI}.  Additionally, we will see that tangent cones are unique
 of a subset of Hausdorff $(n-2)$ dimensional measure zero. 

In the case of limit spaces $X^n$ satisfying a $2$-sided Ricci curvature bound $|\Ric_{M^n_i}|\leq n-1$, we 
can use these structural results to give a new proof of a conjecture from \cite{ChCoI}
stating that that $\Sing(X)$ is $n-4$ rectifiable with uniformly bounded measure.  We will can also
 conclude from this structure that tangent cones are unique away from a set of $(n-4)$-dimensional
Hausdorff measure zero.

Our results on the stratification $S^k$ are obtained from stronger results on the quantitative 
stratification introduced in \cite{CheegerNaber_Ricci}.
These results are in turn a consequence of our analysis on 
the structure of neck regions in $X^n$, a concept first introduced in \cite{NaVa_YM} and \cite{JiNa_L2} 
in the context of bounded Ricci curvature.  Our analysis is based on several 
new ideas, including a sharp cone-splitting theorem and a geometric transformation theorem, 
which will allow us to control the degeneration of harmonic functions on these neck regions.

\end{abstract}

\maketitle

\newpage
\tableofcontents

\newpage

\section{Introduction and statement of results}
\label{s:introduction}

This paper is concerned with the structure of noncollapsed limit spaces with a lower bound
on Ricci curvature:
\begin{align}
\label{e:lrb}
\Ric_{M^n_i}\geq -(n-1)\, ,
\end{align}
\vskip-6mm

\begin{align}
\label{e:lvb}
\Vol(B_1(p_i))>\rv>0\, .
\end{align}

\noindent
Our  results represent both a 
qualitative and quantitative improvement of 
over what was previously known about noncollapsed Gromov Hausdorff limits spaces with Ricci curvature bounded below.  For two sided Ricci curvature bounds
\begin{align}
\label{e:2rb}
|\Ric_{M^n_i}|\leq n-1\, .
\end{align}
we are able to combine our techniques with the Codimension 4 Conjecture, proved in \cite{CheegerNaber_Codimensionfour}, in order to give a new proof that the singular is rectifiable with 
with a definite bound on its $(n-4)$ dimensional Hausdorff measure, a result originally proved by the second two authors in \cite{JiNa_L2}.

\vskip2mm

\subsection{The classical stratification.}
Let $C(Y)$ denote the metric cone on the metric space $Y$.
We begin by recalling the following definition. 
\begin{definition}
\label{d:ksymm}
The metric space $X$ is called {\it $k$-symmetric} if 
$X$ is isometric to $\dR^k\times C(Z)$ for some $Z$.
\end{definition}
\begin{remark}
We say $X$ is $k$-symmetric at $x\in X$ if  there is an isometry of $X$ with $ \dR^k\times C(Z)$ 
which carries $x$ to a vertex of the cone $ \dR^k\times C(Z)$.	
\end{remark}

In \cite{ChCoI} a
 {\it filtration} on the singular set $S$ was defined. Namely,
\begin{align}
\label{e:stratification}
\emptyset\subset S^0\subseteq\cdots\subseteq  S^{n-1}:= S\subseteq X^n\, , 
\end{align}
where
\begin{align}
\label{e:sing_set}
S^k:= \{x\in X\, :\, \text{no tangent cone at } x \,\, {\rm is\,\,} (k+1)\text{-symmetric}\}\, .
\end{align}
The set $S^k\setminus S^{k-1}$ 
is called the $k$th {\it stratum} of the singular set.  A key result of \cite{ChCoI} is the Hausdorff dimension bound
\begin{align}\label{e:sing_set_haus_est}
\dim\, S^k\leq k\, ,\;\;\;\;\; \text{for all}\,\, k\, .
\end{align}
In  \cite{ChCoI},
\cite{CheegerNaber_Codimensionfour}, by showing that $S^{n-1}\setminus S^{n-2}=\emptyset$, respectively $S^{n-1}\setminus
S^{n-4}=\emptyset$, the following sharper estimates were proved:
\begin{align}
\label{e:de}
&\dim \, S \leq n-2 ,\, \text{ if \,\,}\Ric_{M^n_i}\geq -(n-1)\, . \\
&\dim \, S \leq n-4, \,  \text{ if \,\,}|\Ric_{M^n_i}|\leq (n-1)\, .
\end{align}
\vskip1mm

Note that for noncollapsed limit spaces satisfying the lower Ricci bound  \eqref{e:lrb},
{\it the singular set can be dense and one can have 
$\cH^{n-2}(S\cap B_1(p))=\infty$}; see Example \ref{example:finiteness_sharp}.  For general strata, essentially 
nothing else beyond the dimension estimate in
  \eqref{e:sing_set_haus_est} was previously known about the structure of the sets $S^k$.  
In the present paper,
we will show that $S^k$ 
is $k$-rectifiable {\it for all} $k$ and in addition, that for $\cH^k$-a.e. $x\in S^k$,  {\it every} tangent cone at $x$ is $k$-symmetric; see Theorem \ref{t:main_eps_stationary}, Theorem \ref{t:main_stratification}.
\footnote{At the above mentioned points, uniqueness of
tangent cones can actually  fail to hold for $k<n-2$; see Example \ref{example:symmetry_uniqueness_sharpness}.  
Namely, the non-Euclidean factor need not be unique.
 However, as a consequence of Theorem \ref{t:main_stratification}, it will follow that the tangent cones are unique 
 \hbox{$\cH^{n-2}$-a.e.}.
 This should be seen as a first step toward a conjecture of \cite{CoNa2,Na_14ICM}, stating that tangent cones are unique away from a set of codimension three.  Theorems \ref{t:main_eps_stationary} and \ref{t:main_stratification} give the precise results in this context.}

For the case in which the lower Ricci bound
 \eqref{e:lrb} is strengthened to the $2$-sided  Ricci bound \eqref{e:2rb},
 the singular set is closed. In this case,
we will give new proofs of conjectures stated in \cite{ChCoI}. Specifically,  the singular set
$S=S^{n-4}$ is \hbox{$(n-4)$-rectifiable} and has an {\it a priori} bound on its $(n-4)$-dimensional
Hausdorff measure:  
$$
\cH^{n-4}(S\cap B_1(p))\leq C(n,\rv)\, .
$$
The first proofs of those conjectures were given by the second two authors in \cite{JiNa_L2}, who even proved {\it a priori} $L^2$ curvature estimates on $M^n$; for earlier results in which integral bounds on curvature 
were {\it assumed}; see \cite{Cheeger}, \cite{CCTi_eps_reg}.   The proofs in the present paper
 are based on new estimates, which assume
only a lower bound on Ricci. In that  case,
the stronger estimates proved \cite{JiNa_L2}, which require assuming a
$2$-sided bound, can fail to hold.

\subsection{The quantitative stratification}

The quantitative stratification involves sets $S^k_{\epsilon,r}$, whose definition will be recalled below.  The quantitative stratification was  introduced in \cite{CheegerNaber_Ricci}
in the context of Ricci curvature,
in order to state and prove new {\it effective} 
estimates on noncollapsed manifolds with Ricci curvature bounded below, 
and in particular Einstein manifolds. 
These quantitative stratification ideas have been since used in  
a variety contexts, see 
\cite{ChNa1,ChNaHa1,ChNaHa2,ChNaVa,
NaVa_CriticalSets,lammbre,Chu16,Wang16,NaVa_YM}, \cite{EdEng}, to prove similar results in other areas including minimal submanifolds,
 harmonic maps, mean curvature flow, harmonic map flow,
crtical sets of linear elliptic pde's, bi-harmonic maps,
 stationary Yang-Mills and free boundary problems.
\vskip2mm

Next, we recall
some relevant definitions; compare \eqref{e:stratification}.
Let $X$ denote a metric  space.

\begin{definition}\label{d:epssymm}
Given $\epsilon>0$ we say a ball $B_r(x)\subset X$ is {\it $(k,\epsilon)$-symmetric} 
if there exists a $k$-symmetric metric cone $X'=\dR^k\times C(Z)$, with $x'$ a vertex of $\dR^k\times C(Z)$,
such that $d_{GH}(B_r(x),B_r(x'))<\epsilon r$.
\end{definition}
\begin{remark} \label{rm:epssym}
If $\iota: B_r(x')\to B_r(x)$ is the $\epsilon r$-GH map and $\cL_{x,r}:= \iota(\dR^k\times\{x'\})$, 
then we may say $B_r(x)$ is $(k,\epsilon)$-symmetric with respect to $\cL_{x,r}$.
\end{remark}

\begin{definition}
\label{d:stratification}
${}$
\vskip1mm

\begin{enumerate}
\item For $\epsilon,r>0$ we define the $k^{th}$ $(\epsilon,r)$-stratum to be $S^k_{\epsilon,r}\setminus S^k_{\epsilon,r}$ 
where $S^{-1}:=\emptyset$ and for $k\geq 0$,
\begin{align}
S^k_{\epsilon,r}:=\{x\in B_1(p)\, :\, \text{ for no }r\leq s<1\text{ is $B_{s}(x)$ a }(k+1,\epsilon)\text{-symmetric ball}\}.
\end{align}
\item For $\epsilon>0$ we define the $k^{th}$ $\epsilon$-stratum 
to be $S^k_\epsilon\setminus S^{k-1}_\epsilon$, where  $S^{-1}:= \emptyset$ and for $k\geq 0$,
\begin{align}
\label{e:qs1}
S^k_{\epsilon}:= \bigcap_{r>0} S^k_{\epsilon,r}(X):= \{x\in B_1(p):\text{ for no }0< r<1\text{ is $B_{r}(x)$ a }(k+1,\epsilon)\text{-symmetric ball}\}.
\end{align}
\end{enumerate}
\end{definition}
\begin{remark}
The the standard and quantitative stratification are related as follows:
\begin{align}
\label{e:equal}
S^k= \bigcup_{\epsilon>0} S^k_{\epsilon}\, .
\end{align}
One can see this through a simple and instructive (though not a priori obvious) contradiction argument.
\end{remark}
\vskip1mm

To summarize, the sets $S^k$ 
are defined by grouping together all points  $x\in X$ {\it all of whose tangent cones fail to have 
$k+1$ independent translational symmetries}.
 The sets $S^k_\epsilon$
 are defined by grouping together all points $x\in X$ such that all balls  {it fail by a definite amount to have
 at most $k+1$ independent translational symmetries}. The sets
 $S^k_{\epsilon,r}$ are defined by grouping together points of $x\in X$ such that {\it all balls $B_s(x)$ of radius at least $r$ fail by a definite amount have at most $k+1$ translational symmetries}.
\vskip2mm

\subsection{Significance of the quantitative stratification}
According to \eqref{e:qs1}, \eqref{e:equal},
 the quantitative stratification carries more information than the standard stratification. Thus, estimates
 proved for the quantitative stratification have immediate consequences for
  the standard stratification. The latter, however, are significantly weaker.  In order 
to illustrate this, we 
introduce the following notation.\\

{\bf Notation:} Let $B_r(A)=\bigcup_{a\in A} B_r(a)$ denote tubular neighborhood of $A\subset X$ with radius $r$.

 In \cite{CheegerNaber_Ricci}, the Hausdorff dimension estimates \eqref{e:sing_set_haus_est} on $S^k$ were improved to the Minkowski type estimate,
\begin{align}
\label{e:sing_set_mink_delta_est}
\Vol(B_r(S^k_{\epsilon,r}\cap B_1(p)))\leq c(n,\rv,\epsilon,\eta) \cdot r^{n-k-\eta}\, ,\text{  }\qquad ({\rm for \,\,all}\,\, \eta>0)\, .
\end{align}
This is further sharpened in the present paper; see 
Theorem   \ref{t:main_quant_strat_stationary}, where the $\eta$ in 
\eqref{e:sing_set_mink_delta_est} is removed.  
\vskip1mm

A complementary point to \eqref{e:sing_set_mink_delta_est}, which is {\it crucial for various applications}, 
accounts for much of the significance of the 
quantitative stratification.  Namely, for solutions of various geometric equations we have on the complement
of the tubular neighborhood \eqref{e:sing_set_mink_delta_est} (see also 
Theorem   \ref{t:main_quant_strat_stationary} for the improved version) that the solution has 
a definite amount of regularity, as measured by the so called {\it regularity scale}.  Essentially, this means
that if $x$ lies in the complement of $B_r(S^k_{\epsilon,r}\cap B_1(p))$, then on $B_{r/2}(x)$ the solution satisfies {\it uniform} scale invariant estimates on its derivatives.  A key element of this is the existence of an $\epsilon$-regularity theorem, stated in scale 
invariant form.  For balls of radius $2$ the $\epsilon$-regularity theorem typically states: 
There exists $k$ (whose value 
depends on the particular equation being considered)
such that.
\vskip2mm

{\it If $B_2(x)$ is $(k,\epsilon)$-symmetric then $B_1(x)$ has bounded regularity.}
\vskip2mm

In the context of the present paper see Theorem \ref{t:eps_reg} for the appropriate $\epsilon$-regularity theorem for spaces with $2$-sided Ricci curvature bounds.  
Such results let us turn estimates on the quantitative stratification into classical regularity estimates 
on the solution itself.  See Theorem \ref{t:main_bounded_finite_measure}, as well as the 
$L^p$ estimates proved in \cite{CheegerNaber_Ricci}, \cite{CheegerNaber_Codimensionfour}.

\vskip2mm

\subsection{Main results on the quantitative stratification}
 In this subsection, we give our main results on the quantitative stratification for limit spaces
 satisfying the lower Ricci bound \eqref{e:lrb} and the noncollapsing condition \eqref{e:lvb}.
Our first result gives us $k$-dimensional Minkowski estimates on the quantitative stratification.  That is, we can remove the constant $\eta>0$ in \eqref{e:sing_set_mink_delta_est}.
\vskip2mm

\begin{theorem}[Measure bound for $S_{\epsilon,r}^k$]
\label{t:main_quant_strat_stationary}
For each $\epsilon>0$ there exists $C_\epsilon=C_\epsilon(n,\rv,
\epsilon)$ such that the following holds.
Let $(M^n_i,g_i,p_i)\stackrel{d_{GH}}{\longrightarrow}(X,d,p)$ satisfy $\Vol(B_1(p_i))\geq \rv>0$ and $\Ric_{M^n_i}\geq -(n-1)$.  Then 
\begin{align}
\label{e:n-k}
\Vol\Big(B_r\big(S^k_{\epsilon,r}\big)\cap B_1(p)\Big)\leq c(n,\rv,\epsilon)\cdot
 r^{n-k}\, .
\end{align}
\end{theorem}
\vskip2mm

Showing that one can replace  $(n-k-\eta)$  in \eqref{e:sing_set_mink_delta_est}
 by $n-k$ in \eqref{e:n-k} requires techniques which are fundamentally  different from those used to establish 
 \eqref{e:sing_set_mink_delta_est} and arguments which are
significantly harder.  This is because such estimates are tied 
in with the underlying structure of the singular set itself.  On the other hand, the new techniques enable 
us to prove much more.  Our next result states that the set $S^k_{\epsilon}$ is rectifiable.  Let us recall the definition of rectifiablity for our context.

\begin{definition}
\label{d:rectifiable} 
A metric space $Z$ is {\it $k$-rectifiable} if there exists a countable collection of 
$\cH^k$-measurable subsets $Z_i\subset Z$,
and bi-Lipschitz maps $\phi_i:Z_i\to \R^k$ such that $\cH^k(Z\setminus \bigcup_i\, Z_i)=0$.
\end{definition}

For further details on rectifiability, especially for subsets of Euclidean space, 
see  \cite{Fed}.  Our main theorem on the structure of the quantitative stratification $\cS^k_\epsilon$ is now the following:

\begin{theorem}[$\epsilon$-Stratification]
\label{t:main_eps_stationary}
Let $(M^n_i,g_i,p_i)\stackrel{d_{GH}}{\longrightarrow}(X,d,p)$ satisfy $\Vol(B_1(p_i))\geq \rv>0$ and 
$\Ric_{M^n_i}\geq -(n-1)$.
Then for each $\epsilon>0$ there exists $C_\epsilon=C_\epsilon(n,\rv,
\epsilon)$ such that
\begin{align}\label{e:main_eps_stationary:mink}
\Vol\Big(B_r\big(S^k_{\epsilon}(X)\big)\cap B_1(p)\Big)\leq C_\epsilon \cdot
r^{n-k}\, .
\end{align}
In particular, 
\begin{align}
\label{e:hme}
\cH^{k}(S^k_\epsilon \cap B_1(p))\leq C_\epsilon\, .
\end{align}
Moreover, the set $S^k_\epsilon$ is $k$-rectifiable, and 
for $\cH^k$-a.e. $x\in S^k_{\epsilon}$ {\it every} tangent cone at $x$ is 
$k$-symmetric.
\end{theorem}
\vskip2mm

\begin{remark}
The techniques used in proving the above results provide an even stronger estimate than the Minkowski estimate of \eqref{e:main_eps_stationary:mink}. Namely, they lead to a uniform {\it $k$-dimensional packing content estimate}:   Let 
$\{B_{r_i}(x_i)\}$ denote {\it any} collection of {\it disjoint} balls such that $x_i\in S^k_{\epsilon,\,r}$.  Then 
\begin{align}
\label{e:minkowski}
\sum r_i^k \leq C_\epsilon\, .
\end{align}
\end{remark}
\vskip2mm

\begin{remark}
\label{r:sharp}
The structural results above are actually sharp.  In Example \ref{example:k_rect_sharp} we will explain a
construction from \cite{LiNa17_example} of a noncollapsed limit space $X^n$ 
such that
\begin{align}
S=S^k=S^k_\epsilon \notag\\
0<\cH^k(S)<\infty\, ,\notag
\end{align}
for which $ S^k_\epsilon$ is both $k$-rectifiable and bi-Lipschitz to a $k$-dimensional (fat) Cantor set.
In particular the singular set has no manifold points. However, it is still an open question to show that in the presence
 of a $2$-sided bound on Ricci curvature, the singular set must contain manifold points.\footnote{If $M^n$ is K\"ahler with a polarization then it has been shown in \cite{DS1}, \cite{Tian} that the singular set is {\it topologically} a variety.  However, the smoothness or even the bi-Lipschitz structure of the singular set is still unknown even in this case. }
\end{remark}
\vskip1mm

\subsection{Results for the classical stratification}
We now state our main results for the classical stratification $S^k$.  
They follow as special cases of the preceding 
results on the quantitative stratification.  
\vskip2mm

Since  $S^k=\bigcup_\epsilon S^k_\epsilon$,
the following theorem is essentially an immediate consequence of Theorem \ref{t:main_eps_stationary}.
\begin{theorem}[Stratification]
\label{t:main_stratification}
Let $(M^n_i,g_i,p_i) \stackrel{d_{GH}}{\longrightarrow}(X^n,d,p)$ satisfy $\Vol(B_1(p_i))\geq \rv>0$ and \newline $\Ric_{M^n_i}\geq -(n-1)$.  Then $S^k$ is $k$-rectifiable and for $\cH^k$-a.e. $x\in S^k$ every tangent cone at $x$ is $k$-symmetric.
\end{theorem}

\begin{remark}
\label{r:nomb}
Note that unlike in the Hausdorff measure bound on $S^k_\epsilon\subset S^k$ given in \eqref{e:hme},
we are not asserting a finite measure bound on all of $S^k$.   Example \ref{example:finiteness_sharp} shows that
such a bound need not hold.  However, as will become clear in the proof of Theorem
\ref{t:main_stratification},  to prove results which concern
 the structure of the sets $S^k$, 
it is  crucial to be able to break   the stratification into the well behaved finite measure subsets 
$S^k_\epsilon$.
\end{remark}
\vskip1mm

We end this subsection with two results which are essentially direct applications of Theorems \ref{t:main_eps_stationary}, \ref{t:main_stratification} (see also \cite{ChCoI} for the bi-H\"older estimate).

\begin{theorem}[Manifold Structure]
\label{t:main_homeo}
Let $(M^n_i,g_i,p_i)\stackrel{d_{GH}}{\longrightarrow}(X^n,d,p)$ satisfy 
$\Vol(B_1(p_i))\geq \rv>0$ and \newline $\Ric_{M^n_i}\geq -(n-1)$.  Then there exists a subset $S_\epsilon\subseteq X^n$ which is 
$(n-2)$-rectifiable with 
$\cH^{n-2}\Big(S_\epsilon\cap B_1(p)\Big)\leq C(n,\rv,\epsilon)$
and such that $X^n\setminus S_\epsilon$ is bi-H\"older homeomorphic to a smooth riemannian manifold.
\end{theorem}
\vskip1mm


\begin{theorem}[Tangent Uniqueness]
\label{t:tangentcone_uniqueness}
Let $(M^n_i,g_i,p_i) \stackrel{d_{GH}}{\longrightarrow}(X^n,d,p)$ satisfy $\Vol(B_1(p_i))\geq \rv>0$ and \newline $\Ric_{M^n_i}\geq -(n-1)$.  Then there exists a subset $\tilde S\subseteq X$ with $\cH^{n-2}(\tilde S)=0$ such that for each $x\in X\setminus \tilde S$ the tangent cones are unique and isometric to $\dR^{n-2}\times C(S^1_r)$ for some $0<r\leq 1$.
\end{theorem}
\vspace{1mm}

\subsection{$2$-sided Bounds on Ricci Curvature}
\label{ss:boundedricci_results}

In this subsection, we state  a result for noncollapsed limit spaces with a $2$-sided bound on  Ricci curvature, Theorem \ref{t:main_bounded_finite_measure}.  
Recall that under the assumption of a $2$-sided bound the singular set $S$ is closed and can be described as the set of points no neighborhood of which 
is diffeomorphic to an open subset of $\dR^n$.
Our result follows quickly by combining the quantitative stratification results of for limit spaces satisfying \eqref{e:lrb}
with the $\epsilon$-regularity theorem
of \cite{CheegerNaber_Codimensionfour}; see Subsection \ref{ss:2sided} for a review of this material.  
A stronger version of Theorem \ref{t:main_bounded_finite_measure} was first proved in \cite{JiNa_L2}, where additionally $L^2$ bounds on the curvature were produced, but by using estimates and techniques which definitely require a $2$-sided bound on Ricci curvature:

\begin{theorem}[Two Sided Ricci]
\label{t:main_bounded_finite_measure}
Let $(M^n_i,g_i,p_i)\stackrel{d_{GH}}{\longrightarrow}(X^n,d,p)$ satisfy $\Vol(B_1(p_i))\geq \rv>0$ and $|\Ric_{M^n_i}|\leq  (n-1)$.  Then $S$ is $(n-4)$-rectifiable and there exists $C=C(n,\rv)$\,\,
such that
\begin{align}
\Vol\Big(B_r(\big(S\big)\cap B_1(p)\Big)\leq Cr^4\, .
\end{align}
In particular, $\cH^{n-4}(S\cap B_1)\leq C$.  Furthermore, for $\cH^{n-4}$-a.e. $x\in X$, the tangent cone at $x$ is unique and isometric to $\dR^{n-4}\times C(S^3/\Gamma)$, where $\Gamma\subseteq O(4)$ acts freely.
\end{theorem}
\vskip2mm

\subsection{The remainder of the paper}
\label{ss:outline}
The  paper can viewed as having four parts.
The first part consists of the present section and Section 
\ref{s:outline_proof}; the second consists of Sections \ref{s:ouline_proof_neckstructure}--\ref{s:examples};
 the third part 
consists of
Sections 
 \ref{s:Sharp_splitting}--\ref{s:nondegeneration}; 
the fourth part 
 consists of Sections \ref{s:neck_region} and \ref{s:decomposition}.   In detail:
\vskip2mm

Section \ref{s:outline_proof} contains the definition and concept of "neck region", including an explanation of the role played by each of the 
conditions in the definition, the statements of the Neck Structure Theorem \ref{t:neck_region2} and
 the Neck Decomposition Theorem \ref{t:decomposition2} and some basic examples.
 In addition, this section contains the proofs
of our main results on the quantitative stratification, under the assumption that the Neck Structure Theorem \ref{t:neck_region2} and the Neck Decomposition Theorem \ref{t:decomposition2} hold.  
Part three of the paper is devoted to  developing the new tools which are
needed for the proofs the neck theorems, while the proofs themselves are given in part four.
\vskip2mm

The second part of the paper begins with Section \ref{s:examples}, in which
we give some examples beyond those given in Section \ref{s:outline_proof}.
One of these 
 concerns neck regions. The remaining  examples illustrate the sharpness of
our results on the quantitative stratification.

In Section \ref{s:Prelim}, 
 we collect background results
 which are needed parts three and four. 
  Some of these results are by now rather standard in the smooth riemannian geometry context (as opposed to the
 context of synthetic lower Ricci bounds). In such cases,  we will just
 give references. 
 For  the more technical results which are less well known, 
 we will give the proofs or at least outlines. 
\vskip2mm

In Section \ref{s:ouline_proof_neckstructure}, 
we give a brief outline of part three
(Sections  \ref{s:Sharp_splitting}--\ref{s:nondegeneration}) and of 
part four (Sections \ref{s:neck_region} and \ref{s:decomposition}). 
 \vskip2mm

In Sections \ref{s:Sharp_splitting}--\ref{s:nondegeneration}, which form the third part of the paper, 
we prove {\it sharp estimates} on quantitative 
cone-splitting. The statements of these theorems involve the local pointed entropy.
Like harmonic splitting maps and heat kernel estimates, the entropy can be 
viewed as analytical tool which, once it has been controlled by the geometry, 
enables one to draw additional (and in this case sharp)
  geometrical conclusions from the original geometric hypotheses.
The results on necks, 
especially the Neck Structure Theorem \ref{t:neck_region2}, depend 
on the new sharp estimates. 
The estimates enable us to take full advantage of the 
behavior of the geometry over an {\it arbitrary number of consecutive scales}.
This is crucial for the proofs of the neck theorems.
\vskip2mm

Sections \ref{s:neck_region} and Section \ref{s:decomposition} constitute the fourth part of
the paper.
In Section \ref{s:neck_region} we prove the Neck Structure Theorem \ref{t:neck_region2}.  
The proof depends on the results of Sections
\ref{s:Sharp_splitting}--\ref{s:nondegeneration}.   In
Section \ref{s:decomposition}, via an induction argument, we prove the Neck  Decomposition 
Theorem \ref{t:decomposition2}.  
Remarkably, for the most part the proof only involves (highly nontrivial) covering arguments,
and only at a certain point is an appeal to Theorem \ref{t:neck_region2} made.

\vskip-3mm

\begin{remark}
\label{r:incomplete} (Future directions.) 
Although in the present paper we have stated our results for fixed 
$k$,
the complete description of the geometry should include {\it simultaneously}
 all $k=0,1, \ldots, n-1$.  
In the general case, it should also involve behavior on multiple scales,
thereby generalizing
 the bubble tree decompositions in 
 \cite{Anderson-Cheeger2}, 
\cite{B90}, and Section
 $4$ of \cite{CheegerNaber_Codimensionfour}.
\end{remark}



\section{Proofs of the Stratification Theorems  modulo Results on Neck Regions}
\label{s:outline_proof}

In this section we will begin by introducing the notion of neck regions and state our main theorems for them: namely the Structure Theorem \ref{t:neck_region2} and Decomposition Theorem \ref{t:decomposition2}.  Proving these results will constitute the bulk of this paper, and we will outline their proofs in the next section.  After introducing them in this section, we will end in the last subsection by assuming them and using them to prove all of the results on quantitative and classical stratifications which were stated in Section \ref{s:introduction}.  In few places, we will appeal to results which are reviewed in 
Section \ref{s:Prelim}.

 \subsection{Background and motivation}
\label{ss:background}
Let $\Vol^{-\kappa}(B_r)$ denote the volume of an $r$-ball in a simply connected space of constant curvature $M^n_{-\kappa}$.  Define the {\it volume ratio} by
 \begin{align}
 \label{e:volratio}
 \cV^\kappa_r(x) := \frac{\Vol(B_r(x))}{\Vol^{-\kappa}(B_r)}\, ,
 \end{align}
 The Bishop-Gromov inequality states that If $\Ric_{M^n}\geq -(n-1)\kappa$, then
 $\cV^\kappa_r(x)$ is {\it monotone nonincreasing} in $r$:
 \begin{align}
 \label{e:monotone}
 \frac{d}{dr}\cV^\kappa_r(x)\leq 0\, .
 \end{align}

 In addition to being monotone, the quantity $ \cV_r(x)$  {\it coercive} in the following sense.
 Given $\epsilon>0$, there exists $0<\delta=\delta(\epsilon)$, such that  if
$r^2\kappa<\delta $  and
 \begin{align}
 \label{e:vc}
 \big| \cV^\kappa_{r}(x)-\cV^\kappa_{r/2}(x)\big|<\delta\, ,
 \end{align}
 then $B_r(x)$ is 
$\epsilon$-Gromov Hausdorff close to a ball $B_r(y^*)\subset C(Y)$, for some metric cone with cross-section $Y$ and vertex $y^*$.  This statement 
is the "almost volume cone implies almost metric cone theorem" of [ChCo2];
 see Section \ref{ss:prelim:volume_cone} for a more complete review. 
\vskip2mm

\begin{remark}
\label{r:notation}
Whenever we have specified a definite lower bound, say $\Ric_{M^n}\geq -(n-1)\xi$,
we will write  $\cV_r(x)$ for $\cV^\kappa_r(x)$. Similarly, for a sequence $M^n_i\stackrel{d_{GH}}{\longrightarrow} X^n$ if
$\liminf_{i\to\infty} \Ric_{M^n_i}\geq 0$, we will write $\cV_r(x)$ for $\cV^0_r(x)$.
\end{remark}
\vskip1mm

 
 The noncollapsing assumption \eqref{e:lvb} and the monotonicity
 \eqref{e:monotone} of $\cV_r(x)$ directly imply 
 \begin{align}
 \label{e:sum1}
 \sum_i \, | \cV_{\delta 2^{-i}}(x)-\cV_{\delta^{-1} 2^{-(i+1)}}(x)|\leq C(n,\rv,\delta)\, .
 \end{align}
 As an immediate consequence, for any $\delta>0$ 
 \begin{align}
 \label{e:rtozero}
 \lim_{r\to 0} | \cV_{\delta r}(x)-\cV_{\delta^{-1}r}(x) |=0\, .
 \end{align}
 This, together with the "almost volume cone implies almost metric cone" theorem, was used in \cite{ChCoI} to prove that for noncollapsed limit spaces satisfying \eqref{e:lrb}, \eqref{e:lvb}, {\it every tangent cone is a metric cone}.  
 
For applications which concern $S^k$, the 
  "cone-splitting principle" is also crucial. In abstract form, where we are using Definition \ref{d:ksymm}, it can be stated as follows:
 \vskip3mm
 
 \noindent
{\bf The cone-splitting principle.} {\it  Let $X$ be a metric space which is $0$-symmetric with respect to two distinct points $x_0,x_1\in X$.  Then $X$ is $1$-symmetric with respect to these points.}
\vskip3mm

The estimate \eqref{e:sum1} together with the
cone-splitting principle was used in \cite{CheegerNaber_Ricci} to prove the weak Minkowski estimate \eqref{e:sing_set_mink_delta_est}. 
\vskip3mm


\noindent {\bf Notation:}  A {\it scale} is just a  number of the form $r_j = 2^{-j}$. 
\vskip1mm
\noindent Note then that \eqref{e:sum1} actually yields the following:
\vskip3mm

\noindent
{\bf Effective version of \eqref{e:rtozero}.}  {\it Given $\epsilon>0$, on all but a definite
 number $N_\epsilon$ of scales,  relation \eqref{e:vc} will hold and  $B_r(x)$ will be $(0,\epsilon)$-symmetric. }
 \vskip3mm
 \begin{remark} (Lack of Sharpness)
  \label{r:suspect}
 The effective version of \eqref{e:rtozero}, together with a quantitative version of cone-splitting, 
was used in \cite{CheegerNaber_Ricci} to obtain effective estimates on the sets $S^k_\epsilon$, 
notably \eqref{e:sing_set_mink_delta_est}.  Clearly this makes use of more information
 than the classical dimension reduction arguments of \cite{ChCoI}, 
which require only \eqref{e:rtozero}.  Nonetheless, a lot of information is being disregarded
 when passing from \eqref{e:sum1} to the above effective version of \eqref{e:rtozero}.  
The ability to take full advantage of
  \eqref{e:sum1}  eventually leads to the main volume and rectifiability estimates of this paper. However,
this requires a number of new ideas in order to not lose any information, all of which is essential.
  \end{remark}
 \vskip1mm

 
\subsection{Neck regions}
\label{ss:neck_regions_intro}

As explained in Section \ref{s:introduction}, 
our results on the classical stratification $S^k$ follow from
structural results for the quantitative stratification $S^k_\epsilon$, and these results follow from results on neck regions and neck decompositions.  Neck decompositions of the type employed here were first introduced in \cite{JiNa_L2} and \cite{NaVa_YM}, where they played a key role in the proofs of the a priori $L^2$ curvature bound for spaces with a $2$-sided bound on Ricci 
curvature and the energy identity, respectively;\footnote{In those papers, only the top stratum of the neck regions could be controlled, and only under much more restrictive hypotheses.} compare
 also  \cite{NaVa_Rect_harmonicmap}.  As these papers illustrate, neck
decompositions
 are  of interest their own right.
In particular, their uses 
go beyond applications to  structural results on singular sets, which are the main focus of the present paper. 
\vskip2mm

We will need the following notion of a tubular neighborhood 
of variable radius.
\noindent
\begin{definition}[Tube of  variable radius]
\label{d:tube2}
If $\mathcal D\subseteq X$ is a closed subset and  $x \to r_x$ (the radius function)
is  a nonnegative continuous  function defined 
on $\mathcal D$, then  the corresponding tube of variable radius is:
$$
\overline B_{r_x}(\mathcal D) :=\bigcup_{x\in\mathcal D} \overline B_{r_x}(x)\, .
$$
\end{definition}
\vskip2mm

Recall from Definition \ref{d:epssymm} and Remark \ref{rm:epssym},
 the notion of $(k,\epsilon)$-symmetry with respect to a subspace.  We now give our definition of a neck region:

\begin{definition}[Neck Regions]
\label{d:neck2}
Let $(M^n_i,g_i,p_i)\togh (X,d,p)$ satisfy $\Ric_{M^n_i}
\geq -(n-1)\delta^2$, $\Vol(B_1(p_i))>\rv>0$ and let $\eta>0$.  Let $\cC=\cC_0\cup \cC_+\subseteq B_2(p)$ denote a closed subset with $p\in \cC$, and let $r_x:\cC\to \dR^+$ be continuous such that $r_x:=  0$ on $\cC_0$ and $r_x>0$ on $\cC_+$.  The
set $\cN=B_2(p)\setminus \overline B_{r_x}(\cC)$ is a $(k,\delta,\eta)$-neck region if 
for all $x\in \mathcal C$, the following hold:
\vskip2mm
	
\begin{enumerate}
\item[(n1)] 	$\{B_{\tau^2_n r_x}(x)\}\subseteq B_2(p)$ are pairwise disjoint, where $\tau_n = 10^{-10n}\omega_n$.
\vskip1mm

\item[(n2)] $|\cV_{\delta^{-1}}(x)-\cV_{\delta r_x}(x)|<\delta^2$. 
\vskip1mm

\item[(n3)] For each $r_x\leq r\leq \delta^{-1}$ the ball $B_r(x)$ is 
$(k,\delta^2)$-symmetric, wrt $\cL_{x,r}$, but not $(k+1,\eta)$-symmetric. 
\vskip1mm

\item[(n4)] For each $r\geq r_x$ with $B_{2r}(x)\subseteq B_2(p)$ we have $\cL_{x,r}\subseteq B_{\tau_n r}(\cC)$ and $\cC\cap B_r(x)\subseteq B_{\tau_n r}(\cL_{x,r})$.
\vskip1mm

\item[(n5)] $|\Lip\, r_x|\leq \delta$.
\end{enumerate}
\end{definition}


\begin{remark}[Vitali covering terminology] Throughout the paper a covering as in (n1), but possibly
with some other constant $\gamma<1/6$ in place of $\tau_n$, will be referred to as a {\it Vitali covering}.
\end{remark}

\begin{remark}
\label{r:properties}
The set $\cC$ will be referred to as {\it the set of centers}  of $\mathcal N$. Below we
provide some explanation for the various conditions, (n1)--(n5), in Definition \ref{d:neck2}.
\vskip2mm

\begin{enumerate}
\item[1)] The effective disjointness property of $(n1)$ guarantees that we do not overly cover, 
which would prevent property (3) of the Neck Structure Theorem \ref{t:decomposition2} from holding.
The center set $\cC$ is used not soley as approximation to the singular set but also as an approximation
to the relevant Hausdorff  measure; see the packing measure definied in Definition \ref{d:mu_measure2}.  
Without $(n1)$ would have no hope of controlling this packing measure, see Theorem \ref{t:neck_region2}.  
Another simple consequence is that the set $\cC_+$ consists of a discrete set of points.
\vskip2mm

\item[2)] Condition $(n2)$ has the consequence that even if the neck region involves
infinitely many scales, there is a summable energy condition over the whole region.  
This summable energy is key for both the rectifiability and measure estimates of 
Theorem \ref{t:neck_region2}.\vskip2mm

\item[3)]  One consequence of $(n3)$ is that if $x\in \cC$ then  $x\in S^k_{\eta,r_x}$; in particular $\cC_0\subseteq S^k_{\eta}$.  Both the assumed  $k$-symmetry and the assumed
 lack of $(k+1)$-symmetry play a key role in the Geometric Transformation Theorem \ref{t:transformation2}.  
These conditions act as a form of rigidity which stops harmonic splitting
 maps from degenerating in uncontrollable ways.
\vskip2mm

\item[4)] Condition $(n4)$ plays the role of a Reifenberg condition on the singular set.  It is strong enough to 
prove bi-H\"older control on $\cC$, but not bi-Lipschitz, which  requires in addition, $(n2)$,
 and is the main goal of this paper.
\vskip2mm

\item[5)] Condition (n5) says that if $x\in \cC$ then $r_x$ looks roughly constant on $B_{10^4 r_x}(x)$.
  It turns out that constructing neck regions with this condition is quite painful, but it is especially important 
for the Nondegeration Theorem \ref{t:nondegeneration}.  It allows us to take integral estimates on neck 
regions and use them to control the behavior of the center points themselves.
\vskip2mm

\item[6)]
If $\cN$ is a neck region in a smooth riemannian manifold $M^n$, then $\cC_0=\emptyset$.
\vskip2mm

\item[7)]
If $\cN\subseteq B_2(p)$ is a $(k,\delta,\eta)$-neck region and $B_{2s}(q)\subseteq B_2(p)$ 
with $q\in \cC$, then $\cN\cap B_{2s}(q)\subseteq B_{2s}(q)$ defines a $(k,\delta,\eta)$-neck region.

\end{enumerate}
\end{remark}
\vskip1mm

 \begin{remark}(Important convention)
 \label{r:br1}
 Often throughout the paper, we wil state a result for balls of radius $1$ and use it (often without
 comment) for balls of radius $r<1$, where the more general follows immediately from the special case by scaling.
 \end{remark}
 \vskip1mm
 
We will want to view $\cC$ as a discrete approximation of a $k$-dimensional set.  Similarly, we want to associate to it a measure which is a discrete approximation of the $k$-dimensional hausdorff measure on $\cC$:

 \begin{definition}
 \label{d:mu_measure2}
 	Let $\cN :=  B_2(p)\setminus \overline B_{r_x}(\cC)$ denote a $k$-neck region.
  The associated packing measure is the measure
	\begin{align}
  \mu:=  \mu_\cN := \sum_{x\in \cC_+} r_x^{k}\delta_{x} + \cH^{k}|_{\cC_0}\, ,
\end{align}
where $\cH^{k}|_{\cC_0}$ denotes the $k$-dimensional Hausdorff measure restricted to $\cC_0$.
 \end{definition}

Our main result on the structure of $k$-neck regions is the following.  The proof, which
 will be outlined in the Section \ref{s:ouline_proof_neckstructure},
depends on several new ideas. It constitutes the bulk of the paper:

\begin{theorem}[Neck Structure Theorem]
\label{t:neck_region2}
Fix $\eta>0$ and  $\delta\leq \delta(n,\rv,\eta)$.
Then if $\cN = B_2(p)\setminus \overline B_{r_x}(\cC)$ is a $(k,\delta,\eta)$-neck region, the following hold:

\begin{enumerate}
	\item For each $x\in \cC$ and $B_{2r}(x)\subset B_2(p)$ the induced packing measure $\mu$ is Ahlfors regular:
\begin{align}	
\label{e:ar}
	 A(n)^{-1}r^{k}<\mu(B_r(x)) <A(n)r^{k} \, .
\end{align}
	
	\item $\cC_0$ is $k$-rectifiable.
\end{enumerate}
\end{theorem}
\vskip2mm

\begin{remark}
\label{r:ahlforsreg}
One can view the Ahlfors regularity condition \eqref{e:ar} as an effective consequence of rectifiability.
  Indeed, for simplicity imagine  $u(\cC_0)\cup\{B_{r_x}(u(x))\}$ is a covering of $B_2(0^k)$,
it is a simple but highly instructive exercise to see that \eqref{e:ar} would follow immediately.
 Conversely, much of the work of this paper will be devoted to showing  that if \eqref{e:ar} holds, 
then such a mapping $u$ exists.  More precisely, the mapping $u$ will be taken to be a
 harmonic splitting function. If \eqref{e:ar} holds then we will see that $u$ is automatically
 bi-Lipschitz, at least on most of $\cC$.  One must do this carefully in order to close the loop.
Thus, we will show  essentially simultaneously through an inductive argument,
that \eqref{e:ar} holds and that $u$ is bi-Lipschitz.  The proof of this, which is quite involved, takes up
most of the paper; see Section \ref{s:ouline_proof_neckstructure} for a detailed outline.
\end{remark}
Before continuing let us mention the simplest example of a $k$-neck region:

\begin{example}[Simplest]
\label{ex:basic_neck}
Consider the metric cone space $X=\dR^k\times C(S^1_r)$, where $S^1_r$ is a circle of radius $r< 1$.  Denote by $0\in C(S^1_r)$ the cone point, so that $\cL:=  \dR^k\times\{0\}$ is the singular set of $X$.  Choose any function $r_x:B_2(0^k)\subseteq \cL\to \dR^+$ such that $|\nabla r_x|\leq \delta$ and let $\cC\subseteq B_2(0^k)\times\{0\}$ be any closed subset such that $\{B_{\tau^2_n r_x}(x)\}$ is a maximal disjoint set.  Then for $r<1-C(n)\eta$ it is an easy but instructive exercise to check that $B_2\setminus \overline B_{r_x}(\cC)$ is a $(k,\delta,\eta)$-neck region.  Note that it is trivial that $\cC_0$ is $k$-rectifiable, as $\cC_0\subseteq \dR^k$ canonically.  Similarly, the Ahlfors regularity condition \eqref{e:ar} may be verified as $\{\overline B_{r_x}(x)\}$ forms a Vitali covering of $B_2(0^k)$.
\end{example}


\vskip2mm

For additional and more complicated examples, see Section \ref{s:examples}.

\subsection{Neck decompositions}
\label{sss:outline:neck_decomposition}
In order to prove our theorems on stratifications, we also need
to suitably control the part of $X^n$ which does not consist of neck regions. This is provided by the following result. 
\begin{theorem}[Neck Decomposition]
\label{t:decomposition2}
		Let $(M^n_i,g_i,p_i)\togh (X^n,d,p)$ satisfy $\Vol(B_1(p_i))>\rv>0$ 
		and $\Ric_{M^n_i}\geq -(n-1)$.  Then for each $\eta>0$ and $\delta\le \delta(n,\rv,\eta)$ we can write:
	\begin{align}
		&B_1(p)\subseteq \bigcup_a \big(\cN_a\cap B_{r_a}\big) \cup \bigcup_b B_{r_b}(x_b) \cup \cS^{k,\delta,\eta}\, ,\\
		&\cS^{k,\delta,\eta}\subseteq \bigcup_a \big(\cC_{0,a}\cap B_{r_a}\big)\cup \tilde \cS^{k,\delta,\eta}\, ,
	\end{align}
	where:
	\begin{enumerate}
		\item 
		For all $a$, the set, $\cN_a=B_{2r_a}(x_a)\setminus \overline B_{r_x}(\cC)$,
		is a $(k,\delta,\eta)$-neck region.
		\vskip2mm
		
		\item The balls $B_{2r_b}(x_b)$ are $(k+1,2\eta)$-symmetric; hence
		$x_b\not\in S^k_{2\eta,r_b}$.		\vskip2mm
		
		\item $\sum_a r_a^{k} + \sum_b r_b^{k} + \cH^{k}
		\big(\cS^{k,\delta,\eta}
		\big) \leq C(n,\rv,\delta,\eta)$.		
		\vskip2mm
		
		\item $\cC_{0,a}\subseteq B_{2r_a}(x_a)$ is the $k$-singular set associated to $\cN_a$.	
	\vskip2mm
		
		\item $\tilde{\cS}^{k,\delta,\eta}
		$ satisfies $\cH^{k}\big(\tilde \cS^{k,\delta,\eta}
		\big)=0$.		\vskip2mm
		
		\item $\cS^{k,\delta,\eta}$
		is $k$-rectifiable.		
		\vskip2mm
		
		\item For any $\epsilon$ if $\eta\le \eta(n,\rv,\epsilon)$ and $\delta\le \delta(n,\rv,\eta,\epsilon)$ we have 
		$S_\epsilon^k
		\subset \cS^{k,\delta,\eta}
		$.
	\end{enumerate}
\end{theorem}
\vskip2mm

\begin{remark}
\label{r:neck_smooth}
In the case of a smooth manifold $M^n$, 
we have $\cS^{k,\delta,\eta}=\emptyset$; compare (6) of Remark \ref{r:properties}.
 In that case,  $M^n$  decomposes into only two types of regions, $k$-neck regions and the $k+1$-symmetric balls $B_{r_b}$.
 \end{remark}


Let us use the following two examples to explain our decomposition theorem. 

\begin{example}[$k$-Symmetric Space Example]
\label{ex:cone_decomposition}
Assume $X^n=C(S^{n-s-1}/\Gamma)\times \mathbb{R}^s$ is $s$-symmetric with $0\le s\le n-2$ and $B_1(p)\subset X^n$, where $p=(y_c,0^s)$ is a cone vertex.  Assume further that $C(S^{n-s-1}/\Gamma)$ is not $(1,\eta_0)$-symmetric. For each integer $0\le k\le n-2$ and $\eta\le \eta_0$ and $\delta=0$, we are able to choose a decomposition as in Theorem \ref{t:decomposition2}. To see this, we will divide it into three cases:\\

Case 1: $0\le k\le s-1$.  We can choose our decomposition to be the single ball $B_{r_b}(x_b)=B_{2}(p)$, which 
is $k+1$-symmetric.\\ 

Case 2: $k=s$.  We can choose $B_{r_a}(x_a)=B_2(p)$ with $\cN_a=B_{r_a}(x_a)\setminus \{y_c\}\times \mathbb{R}^s$ and $\cS^{k,0,\eta}=\{y_c\}\times \mathbb{R}^s\cap B_2(p)$. Then 
\begin{align}
B_1(p)\subseteq \cN_a\cup \cS^{k,0,\eta}\, ,
\end{align}
and $\cS^{k,0,\eta}$ is $k$-rectifiable. In this case, $\cN_a$ is a $(k,0,\eta)$-neck region. \\

Case 3: $k\ge s+1$.  For each $r>0$, let us consider a Vitali covering $\{B_{\epsilon_0 r}(x_{r,j}),j=1,\cdots, N_r\}$ 
of $ B_2(p)\cap B_{2r}(\{y_c\}\times \mathbb{R}^s)\setminus B_{r}(\{y_c\}\times \mathbb{R}^s)$, where $\epsilon_0\leq \epsilon_0(n,\Gamma)$ so that $B_{2\epsilon_0 r}(x_{r,j})$ is $n$-symmetric.  Then the cardinality $N_r\le C(n,\Gamma)r^{-s}$. Each $B_{2\epsilon_0 r}(x_{r,j})$ is $n$-symmetric and we will belong to the $b$-ball in the decomposition. Let us define $\cS^{k,0,\eta}=\tilde{{\cS}}^{k,0,\eta}=\{y_c\}\times \mathbb{R}^s$. Then we have 
\begin{align}\label{ex:decomposition}
B_1(p)\subseteq \bigcup_{1\ge r_b=2^{-b}>0}\bigcup_{i=1}^{N_{r_b}}B_{\epsilon_0 r_b}(x_{r_b,i}) \cup {\cS}^{k,0,\eta}\, .
\end{align}
We have $\cH^k(\tilde{\cS}^{k,0,\eta})=0$, and the $k$-content of $b$-balls satisfies 
\begin{align}
\sum_{1\ge r_b>0}\sum_{i=1}^{N_{r_b}}r_{b}^k\le \sum_{1\ge r_b>0}C(n,\Gamma)r_b^{k-s}\le C(n,\Gamma).
\end{align}
Hence \eqref{ex:decomposition} is the desired decomposition. $\square$

\end{example}
 
\begin{example}[The boundary of a simplex]
 \label{e:1}

Let $X^n=\partial \sigma^{n+1}$ denote the boundary of the standard 
 $(n+1)$-simplex in $\dR^{n+1}$ normalized 
 so that all edges have length $1$.
 Let $\Sigma^k$  denote the closed $k$-skeleton of $X^n$. By appropriately smoothing the sequence of boundaries
 $\partial B_{r_i}(\sigma^{n+1})$, of the tubular neighborhoods,  $B_{r_i}(\sigma^{n+1})$,
 and letting $r_i\to 0$, one see that $X^n$ is a limit space with $\Ric_{M^n_i}\geq 0$, indeed $sec_{M^n_i}\geq 0$.  Note that $S^k=\Sigma^k$ is   $k$-rectifiable and $\cH^k(S^k\cap B_1(p))< c(n)$, for all $0\le k\le n-2$. 
 
  For each $0\le k\le n-2$ and $0<\delta, \eta\le \eta(n)$, we will
 build a decomposition for $X^n$ as Theorem \ref{t:decomposition2}. 
The idea is similar with Case 3 of Example \ref{ex:cone_decomposition}. The decomposition consists of two parts,
corresponding to the $a$-balls and $b$-balls of \ref{t:decomposition2}, respectively.


(1) Neck regions: We will construct neck regions with center in $S^k\setminus S^{k-1}$. For each $0<r\le 1$,
consider a Vitali covering $\{B_{\delta^2 r}(x_{a,r}),x_{a,r}\in S^k\setminus S^{k-1}\}$ 
of the annuli $A_{r,2r}(S^{k-1})\cap B_{\delta^3r}(S^k)$.  One checks that 
$\cN_{a,r}=B_{\delta^2 r}(x_{a,r})\setminus S^k$ is a $(k,\delta,\eta)$-neck
 region for $\eta\le \eta(n)$. The neck regions $\{\cN_{a,r_i},r_i=2^{-i},i=1,\cdots \}$ 
are the desired neck regions of the decomposition. Moreover, by noting that 
$\cH^k(S^k)\le C(n)$, we obtain the $k$-content estimate
\begin{align}
\sum_{a}\sum_i r_{a,i}^k\le C(n,\delta).
\end{align}
\vskip2mm

(2) $(k+1)$-symmetric balls: Consider a Vitali covering 
$\{B_{\delta^4r}(x_{d,r}),x_{d,r}\in A_{r,2r}(S^{k-1})
\setminus B_{\delta^3r}(S^k)\}$ of $A_{r,2r}(S^{k-1})\setminus B_{\delta^3r}(S^k)$. 
The cardinality of this covering is less than $C(n,\delta) r^{-k+1}$. 
From the construction we have that $B_{\delta^4r}(x_{d,r})\cap S^k=\emptyset$, 
which implies that $B_{\delta^4 r}(x_{d,r})$ is $(k+1)$-splitting. For each $\eta>0$, 
by the almost volume cone and almost metric cone theorem, we have that for each 
$y\in B_{\delta^4 r}(x_{d,r})$ the ball $B_{\gamma \delta^4 r}(y)$ would be
 $(0,\eta^2)$-symmetric for some $\gamma(n,\delta,\eta)\le 1$. Therefore,
$B_{\gamma \delta^4 r}(y)$ is $(k+1,\eta/2)$-symmetric which implies that 
each $B_{\delta^4 r}(x_{d,r})$ can be covered by finite many
 $(k+1,\eta/2)$-symmetric balls.  Hence, we can choose at most $N_r=C(n,\delta,\eta)r^{-k+1}$
 $(k+1,\eta/2)$-symmetric balls $B_{\gamma\delta^4 r}$, whose union  covers $A_{r,2r}(S^{k-1})\setminus B_{\delta^3r}(S^k)$. 
By combining them all for $r=r_i=2^{-i}\le 1$, 
we get the desired $b$-balls in our decomposition such that the content estimate 
\begin{align}
\sum_{0<r_i=2^{-i}\le 1}\sum_{j=1}^{N_{r_i}} (\gamma \delta^4r_i)^k\le 
\sum_{0<r_i=2^{-i}\le 1} C(n,\eta,\delta)r_i\le C(n,\eta,\delta). 
\end{align}
\vskip2mm

Define $\tilde{\cS}^{k,\delta,\eta}:=S^{k-1}$, then 

\begin{align}
		&X^n\subseteq \bigcup_a \big(\cN_a\cap B_{r_a}\big) \cup \bigcup_b B_{r_b}(x_b) \cup \cS^{k,\delta,\eta}\, ,\\
		&\cS^{k,\delta,\eta}\subseteq \bigcup_a \big(\cC_{0,a}\cap B_{r_a}\big)\cup \tilde \cS^{k,\delta,\eta}\, ,
	\end{align}
with $\cC_{0,a}=B_{r_a}\cap S^k$. This completes the description
of the decomposition for $X^n=\partial \sigma^{n+1}$. $\square$

\end{example}

 \begin{remark}[Role of the $\sum_b r^k_b$ bound] 
\label{r:r_b}
In light of the fact that the $b$-balls are approximately $(k+1)$-symmetric,
the crucial role of
the a priori bound on $\sum_b r_b^k$ in the Neck Decomposition Theorem
\ref{t:decomposition2} might not be immediately obvious if one
thinks only of the application to 
 $\mathcal H^{k}(S^k_\epsilon\cap B_1(x))$. Recall, however that
 that our volume bounds for the quantative stratification pertain to tubes of
fixed radius $r$, while the function $r_x$ from which figures in the definition of
the $a$-balls, goes to zero as $x\to S^{k-1}$. This suggests that it should not
suffice to consider only $a$-balls in obtaining the applications to the volumes of tubes
around the quantitative strata from the theorems on neck, particularly, 
 the Neck Decomposition Theorem
\ref{t:decomposition2}.  This should be kept in mind when reading the details of 
the proofs which are given in the next subsection.
\end{remark}

\subsection{Proofs of the  stratification theorems assuming the neck theorems}

In this subsection we will prove the main stratification Theorems \ref{t:main_quant_strat_stationary}, \ref{t:main_eps_stationary} 
and the classical stratification Theorems \ref{t:main_stratification}, \ref{t:main_homeo}, \ref{t:main_bounded_finite_measure}, under the assumption that the Neck Structure and Decompositions of 
Theorems \ref{t:neck_region2}, \ref{t:decomposition2} hold.  We will outline the proof of the Neck Structure theorem in the next section.
\vskip2mm

The main result  concerns the $(\epsilon,r)$-stratification of Theorem \ref{t:main_quant_strat_stationary}. The other theorems follow fairly quickly from it and the Decomposition Theorem \ref{t:decomposition2}.


\begin{proof}(of Theorem \ref{t:main_quant_strat_stationary})
From (6), (7) of the Neck Decomposition Theorem it follows that 
$S_\epsilon$ is rectifiable.
Thus, it remains to prove estimate \eqref{e:n-k} in Theorem \ref{t:main_quant_strat_stationary} which
states:
\begin{align}
\label{e:stratification_M_i}
\Vol\Big(B_r\Big(S^k_{\epsilon,r}\cap B_1(p)\Big)\Big)\leq C_\epsilon r^{n-k}\, .
\end{align}
\vskip2mm

By the Volume Convergence Theorem of  \cite{Co1,Cheeger01}
and the definition of the sets $S^k_{\epsilon,r}$, to obtain the
estimate in \eqref{e:stratification_M_i} for the case of limit spaces,
it easily suffices to prove   \eqref{e:stratification_M_i}
for the case of manifolds $M^n$.
We will now give the proof in that case.
\vskip2mm

Given $\epsilon>0$ let $\eta\le \eta(n,\rv,\epsilon)$ and 
$\delta\le \delta(n,\rv,\epsilon,\eta)$ be chosen sufficiently small, to be fixed later. Recall that for the case of manifolds, the Decomposition Theorem \ref{t:decomposition2} states:
\vskip-2mm

\begin{align}
\label{e:decomp_stratification}
	B_1(p)\subset \bigcup_a \Big(\cN_a\cap B_{r_a}(x_a)\Big)\cup\bigcup_b B_{r_b}(x_b)\, ,
\end{align}
where $\cN_a\subset B_{2r_a}(x_a)$ is a $(k,\delta,\eta)$-neck and
 $B_{2r_b}(x_b)$ is $(k+1,2\eta)$-symmetric. 
In addition, Theorem  \ref{t:decomposition2} provides the $k$-content estimate:
\vskip-4mm

\begin{align}
 \label{e:content}
 \sum_a r_a^k+\sum_b r_b^k\le C(n,\rv,\delta,\eta)\, .
\end{align}
\vskip1mm

The proof of Theorem \ref{t:main_quant_strat_stationary} 
amounts to combining  the estimates of Lemma, \ref{l:1} and Lemma  \ref{l:3}
below.
The proof of Lemma \ref{l:1} relies on the Ahlfors regularity of the packing measures $\mu_a$ on the balls $B_{r_a}(x_a)$, see \eqref{e:ar} of
Theorem \ref{t:neck_region2}. 

\begin{lemma}
\label{l:1}
Let $\eta\le \eta(n,\rv,\epsilon)$, $\delta\le \delta(n,\rv,\epsilon)$ and $\chi\le \chi(\epsilon,n,\rv)$.  If the neck region $\cN_a$ satisfies $r_a\geq \chi^{-1} r$, then 
\begin{align}
\label{e:1}
\Vol\Big(B_{r}\left(S_{\epsilon,r}^k\cap \cN_a\right)\Big)\le C(n,\rv,\chi)r_a^k\cdot r^{n-k}
\, .
\end{align}
\begin{align}
\label{e:volume_Na_r_ar1}
	\Vol\left(B_{r}\left(S_{\epsilon,r}^k
	\cap \bigcup_{r_a\ge \chi^{-1}r}
	\cN_a\right)\right) 
	\leq C(n,\rv,\delta,\eta,\chi) r^{n-k}\, .
\end{align}
\end{lemma}
\begin{proof}
First we  will prove \eqref{e:1}.  Let $\cC_a\subset B_{2r_a}(x_a)$ be the associated center points of the neck region $\cN_a$ and let $\mu_a$ be the associated packing measure. \\ 

\noindent
\textbf{Claim:} If $y\in S_{\epsilon,r}^k\cap \cN_a$ then  $d(y,\cC_a)\le \chi^{-1} r\, .$\\

Let us prove the claim. We will show that if $y\in \cN_a$ with $d(y,\cC_a)\ge \chi^{-1}r$, then there exists $B_{s}(y)$ with $s\ge 2r$ such that $B_s(y)$ is $(k+1,\epsilon/2)$-symmetric, which implies that $y\notin S_{\epsilon,r}^k$. Hence it will prove the claim.

For $y\in \cN_a$ with $d(y,\cC_a)\ge \chi^{-1}r$, by the almost volume cone implies  almost metric cone Theorem \ref{t:almostmetriccone}, if $\chi\le \chi(n,\epsilon,\rv)$ we have $B_{s}(y)$ is $(0,\epsilon^2)$-symmetric for some $s>2r$. On the other hand, by the almost splitting Theorem \ref{t:splitting_function} we have for $\delta\le \delta(n,\rv,\epsilon)$ that $B_{2s}(y)$ is $\epsilon^2s$-close to a product space $\mathbb{R}^{k+1}\times Z$. These are good enough to imply that $B_s(y)$ is $(k+1,\epsilon/2)$-symmetric. Hence $y\notin S_{\epsilon,r}^k$. The proof of the claim is completed. 
\vskip2mm

 Now choose a maximal disjoint collection of balls $\{B_{r}(x_j),~~x_j\in \cC_a, j=1,\cdots, K_a\}$ with centers in $\cC_a$.  By the Ahlfors regularity for $\mu_a$,
 \eqref{e:ar} of Theorem \ref{t:neck_region2}, we have 
\begin{align}
	K_aC(n,\chi)r^k\le \sum_{j=1}^{K_a}\mu_a(B_{2\chi^{-1} r}(x_j))\le C(n,\chi)\sum_{j=1}^{K_a}\mu_a(B_r(x_j))\le C(n,\chi)\mu_a(B_{2r_a}(x_a))\le C(n,\chi)r_a^k\, .
\end{align}
Thus, $K_a\le C(n,\chi)r^{-k}r_a^k$, which clearly implies \eqref{e:1} by using the claim.
\vskip2mm

Relation  \eqref{e:volume_Na_r_ar1} follows by
summing \eqref{e:1}  over all neck regions 
and using \eqref{e:content}. Namely,
\begin{align}
\label{e:volume_Na_r_ar2}
	\Vol\left(B_{r}\left(S_{\epsilon,r}^k\cap \bigcup_a
	\cN_a\right)\right) \leq \sum_a
	\Vol\Big(B_{r}\left(S_{\epsilon,r}^k\cap \cN_a\right)\Big) \leq C(n,\rv,\chi)\sum r_a^kr^{n-k}\leq C(n,\rv,\delta,\eta,\chi) r^{n-k}\, .
\end{align}
This completes the proof of of Lemma \ref{l:1}.
\end{proof}
\vskip4mm

\begin{lemma}
\label{l:2}
 Let $\gamma\le \gamma(n,\rv,\epsilon)$, $\eta\le \eta(n,\rv,\epsilon)$.
 If the $b$-ball $B_{r_b}(x_b)$ satisfies $r\le \gamma \cdot r_b$, then 
 \begin{align}
 \label{e:empty}
 S_{\epsilon,r}^k\cap B_{3r_b/2}(x_b)=\emptyset\, .
 \end{align}
 \end{lemma}
 \begin{proof}
It suffices to show that for $y\in B_{3r_b/2}(x_b)$, the ball
 $B_s(y)$ is $(k+1,\epsilon/2)$-symmetric for some $s\ge \gamma r_b$. 
\vskip2mm

To see this fix $\eta'=\eta'(\eta,\rv,\epsilon)>0$ to be chosen below.  If $\eta\le \eta(\eta',n,\rv)$, then since $B_{2r_b}(x_b)$ is $(k+1,2\eta)$-symmetric, it follows that
 $B_{r_b/4}(y)$ is $(k+1,\eta')$-splitting. Also, by the almost metric cone and
\eqref{e:vc}, \eqref{e:sum1}, \eqref{e:rtozero}, it follows that 
for some $\gamma=\gamma(n,\rv,\epsilon)$,
 the ball $B_{\gamma r_b}(y)$ is 
 $(0,\epsilon^2)$-symmetric. For  $\eta'(n,\rv,\epsilon)$ sufficiently small, this implies that $B_{\gamma r_b}(y)$ is 
 $(k+1,\epsilon/2)$-symmetric. This completes the proof
 of \eqref{e:empty}, and thus, of Lemma \ref{l:2}.
 \end{proof}
 \vskip1mm
 
 \begin{lemma}
 \label{l:3}
 Let
 $\Omega:= \{x_1,\cdots, x_N\}$ denote
 a minimal $r/4$-dense subset 
 of $S_{\epsilon,r}^k\setminus \cup_{r_a\ge \chi^{-1}r} B_{r}\left(S_{\epsilon,r}^k\cap \cN_a\right)$ for $\chi$ the constant in Lemma \ref{l:1}.  Then for $\gamma$  the constant in Lemma \ref{l:2}, the following hold:
 \begin{enumerate}

 \item[(1)] Any ball  $B_{r/4}(x_j)$ satisfies:
\begin{align}
\label{e:pakcing_Br_j}
	\sum_{B_{r_a}\subset B_{4\chi^{-1} r}(x_j)}r_a^k+\sum_{B_{r_b}\subset B_{4\gamma^{-1}r}(x_j)}r_b^k\ge C(n,\rv,\gamma,\chi)r^k\, .
\end{align}
\item[(2)] The cardinality of $\Omega$ satisfies $N\le r^{-k}C(n,\rv,\delta,\eta,\gamma,\chi)$.
 \vskip2mm
\item[(3)] The measure estimate: $
\Vol\Big(B_r\left(S_{\epsilon,r}^k\setminus \cup_{r_a\ge 
\chi^{-1}r} B_{r}\left(S_{\epsilon,r}^k\cap \cN_a\right)\right)\Big)\le  C(n,\rv,\delta,\eta,\gamma,
	\chi)r^{n-k}\, .
$

\end{enumerate}
\end{lemma}
\begin{proof}
First we will prove (1).
\vskip2mm

 Since $x_j \in S_{\epsilon,r}^k\setminus \cup_{r_a\ge \chi^{-1}r} B_{r}\left(S_{\epsilon,r}^k\cap \cN_a\right)$, we have $B_{r}(x_j)\cap \cN_a=\emptyset$ for any $r_a\ge \chi^{-1}r$.  In addition, for any $r_a<r\chi^{-1}$ if $B_{r/4}(x_j)\cap \cN_a\ne \emptyset$ then we have $B_{r_a}(x_a)\subset B_{4\chi^{-1}r}(x_j)$. If $B_{r/4}(x_j)\cap B_{r_b}\ne\emptyset$, by Lemma \ref{l:2} we have $r_b\le \gamma^{-1}r$ which implies $B_{r_b}(x_b)\subset B_{4\gamma^{-1} r}(x_j)$. Therefore, by   \eqref{e:decomp_stratification} of the Decomposition Theorem, we have 
\begin{align}
	B_{r/4}(x_j)\subset \Big(\bigcup_{B_{r_a}\subset B_{4\chi^{-1} r}(x_j)} B_{r_a}(x_a)\Big)\cup 
	\Big(\bigcup_{B_{r_b}\subset B_{4\gamma^{-1} r}(x_j)} B_{r_b}(x_b)\Big)\, .
\end{align}
\vskip1mm

\noindent
Thus, 
\begin{align}
\label{}
	C(n,\rv)r^n&\le \Vol(B_{r/4}(x_j))\le \sum_{B_{r_a}\subset B_{4\chi^{-1}r}(x_j)} \Vol(B_{r_a}(x_a))\,\, +\,\, \sum_{B_{r_b}\subset B_{4\gamma^{-1}r}(x_j)} \Vol(B_{r_b}(x_b))\\ \nonumber
	&\le C(n,\rv)\left(\sum_{B_{r_a}\subset B_{4\chi^{-1}r}(x_j)} r_a^n+\sum_{B_{r_b}\subset B_{4\gamma^{-1}r}(x_j)} r_b^n\right)\notag\\
	&{}\notag\\
	&\le C(n,\rv,\gamma,\chi)r^{n-k}\left(\sum_{B_{r_a}\subset B_{4\chi^{-1}r}(x_j)} r_a^k+\sum_{B_{r_b}\subset B_{4\gamma^{-1}r}(x_j)} r_b^k\right)\, .
\end{align} 
\vskip3mm

\noindent
This will imply \eqref{e:pakcing_Br_j} i.e. (1).  Furthermore, from \eqref{e:content} and the fact that the balls
$B_{r/10}(x_j)$ are disjoint, we have 
\begin{align}
N C(n,\rv) r^n &\le \sum_{j=1}^N\Vol(B_{r/4}(x_j))\le C(n,\rv,\gamma,\chi) r^{n-k}\sum_{j=1}^N\left(\sum_{B_{r_a}\subset B_{4\chi^{-1}r}(x_j)} r_a^k+\sum_{B_{r_b}\subset B_{4\gamma^{-1}r}(x_j)} r_b^k\right)\,\\
&\le C(n,\rv,\gamma,\chi) r^{n-k}\left(\sum_a r_a^k+\sum_b r_b^k\right)\le C(n,\rv,\delta,\gamma,\chi)r^{n-k}\, ,
\end{align}
which implies (2).

\vskip2mm
For (3),  let us consider the covering $\{B_{2r}(x_j),j=1,\cdots, N\}$ of 
$$
S_{\epsilon,r}^k\setminus \bigcup_{r_a\ge \chi^{-1}r} B_{r}\left(S_{\epsilon,r}^k\cap \cN_a\right)\, .
$$
By the definition of $\Omega$ this is  also a covering of
 $$
 B_r\left(S_{\epsilon,r}^k\setminus \bigcup_{r_a\ge r} B_{r}\left(S_{\epsilon,r}^k\cap \cN_a\right)\right)\, .
$$
Thus, we have 
\begin{align}\label{e:volume_Na_raa}
	\Vol\Big(B_r\left(S_{\epsilon,r}^k\setminus \cup_{r_a\ge \chi^{-1}r} B_{r}\left(S_{\epsilon,r}^k\cap \cN_a\right)\right)\Big)\le \sum_{j=1}^N \Vol\Big(B_{2r}(x_j)\Big)\le C(n,\rv)Nr^n\le C(n,\rv,\delta,\eta,\gamma,
	\chi)r^{n-k}\, .
\end{align}
This completes the proof 
Lemma \ref{l:3}.
\end{proof}
\vskip2mm

Now we  can complete the proof of Theorem 
\ref{t:main_quant_strat_stationary} as follows.
Fix $\gamma=\gamma(n,\rv,\epsilon),\eta=\eta(n,\rv,\epsilon)$ and $\delta=\delta(n,\rv,\epsilon,\eta), \chi=\chi(n,\rv,\epsilon)$ as in the previous Lemmas. 
Combining the estimates in \eqref{e:volume_Na_r_ar1} and \eqref{e:volume_Na_raa}
gives volume estimate \eqref{e:stratification_M_i},
which completes the proof of Theorem \ref{t:main_quant_strat_stationary}.
\end{proof}
\vskip4mm


\begin{proof}(Proof of $\epsilon$-Stratification Theorem)
Since $S_\epsilon^k\subset S_{\epsilon,r}^k$, the estimate for $S_\epsilon^k$ follows directly from Theorem \ref{t:main_quant_strat_stationary}. On the other hand, by the Decomposition Theorem \ref{t:decomposition2}, for $\eta\le \eta(n,\rv,\epsilon)$ and $\delta\le \delta(n,\rv,\epsilon,\eta)$, we have  $S_\epsilon^k\subset \cS^{k,\delta,\eta}$, where by
Theorem \ref{t:decomposition2} the set $\cS^{k,\delta,\eta}$
is $k$-rectifiable. 
\vskip2mm

For $\cH^k$-a.e. $x\in S_\epsilon^k$ let us show that every tangent cone at $x$ is $k$-symmetric. In fact we will show that for any $\delta$ there exists a subset $\tilde{S}_\delta\subset S_\epsilon^k$ with $\cH^k(\tilde{S}_\delta)=0$ such that every tangent cone of $x\in S_\epsilon^k\setminus \tilde{S}_\delta$ is $(k,\delta)$-symmetric. Indeed, we can choose $\tilde{S}_\delta=\tilde{\cS}^{k,\delta,\eta}$ as in Theorem \ref{t:decomposition2}, which satisfies the desired estimate by definition due to the neck structure. Now we consider 
$\tilde{S}=\cup_{i=1}^{\infty} \tilde{S}_{2^{-i}}$,
where $\cH^k(\tilde{S})=0$. For any $x\in S_\epsilon^k\setminus \tilde{S}$ we have that every tangent cone of $x$ is $(k,\delta)$-symmetric for any $\delta$, which in particular implies that every tangent cone of $x$ is $k$-symmetric. This completes the proof of Theorem \ref{t:main_eps_stationary}.
\end{proof}
\vskip4mm

\begin{proof}(of Theorem \ref{t:main_stratification})
Theorem \ref{t:main_stratification} follows directly from Theorem \ref{t:main_eps_stationary} and the fact that $S^k(X)=\cup_{j\ge 1}S^k_{2^{-j}}(X)$, which is a countable union of rectifiable sets.
\end{proof}
\vskip4mm

\begin{proof} (of Theorem \ref{t:main_homeo})
Let us choose $S_\epsilon=S_\epsilon^{n-2}$. Then for any $x\in X^n\setminus S_\epsilon$, we have for some $r_x>0$ that $B_{2r_x}(x)$ is $(n-1,\epsilon)$-symmetric and hence $B_{r_x}(x)$ is $(n,\epsilon')$-symmetric for $\epsilon\le \epsilon(\epsilon',n,\rv)$. By Reifenberg Theorem \ref{t:reifenbergcanonical} (see also \cite{ChCoI}) that $B_{r_x/2}(x)$ is bi-H\"older to $B_{r_2/2}(0^n)\subset\mathbb{R}^n$ for $\epsilon'$ small. This is enough to conclude the theorem.
\end{proof}
\vskip4mm

\begin{proof} (of  Theorem \ref{t:tangentcone_uniqueness})
As was shown in \cite{ChCoI}, $S=S^{n-2}$.  From Theorem \ref{t:main_stratification} we now know that for $\cH^{n-2}$-a.e. $x\in S^{n-2}$ every tangent cone is $(n-2)$-symmetric. For such $x$ any tangent cone is isometric to $\mathbb{R}^{n-2}\times C(S_\beta^1)$, where $S^1_\beta$ denotes the circle of length $\beta<2\pi$.  By Theorem \ref{t:cross_tangent_cone},  $\beta$ is determined by the limiting volume ratio, $\lim_{r\to 0}\cV_r(x)$.  This suffices to complete the proof.
\end{proof} 
\vskip4mm

\begin{proof}(of Theorem \ref{t:main_bounded_finite_measure})
The theorem follows from the $\epsilon$-regularity theorem,
Theorem  \ref{t:eps_reg} and the stratification of Theorem \ref{t:main_eps_stationary}.
\vskip2mm	
	
	  To see this, note that if $y\notin S^{n-4}_\epsilon(X)$ then there exists some $r_y>0$ such that $B_{r_y}(y)$ is $(n-3,\epsilon)$-symmetric. 
	  According to Theorem \ref{t:eps_reg} we then have the harmonic radius bound $r_h(y)\ge c(n)r_y>0$, which in particular implies $y\notin S(X)$.  Thus, we have shown that $S\subseteq S^{n-4}\subset S^{n-4}_\epsilon$. The volume estimates of 
	  $B_r\big(S\big)\cap B_1(p)$ now follows from Theorem \ref{t:main_eps_stationary}. 
\vskip2mm
	
The proof of the tangent cone uniqueness result is similar  to that of Theorem \ref{t:tangentcone_uniqueness}. By Theorem \ref{t:main_stratification}, there exists a $(n-4)$-Hausdorff measure zero set $\tilde{S}\subset S^{n-4}=S(X)$ such that every tangent cone at $x\in X\setminus \tilde{S}$ is $(n-4)$-symmetric. In particular, this means the every tangent cone is isometric to $\mathbb{R}^{n-4}\times C(Y^3)$ for some metric space $Y^3$. By the main result of\cite{CheegerNaber_Codimensionfour}, which states that the singular set of a
noncollapsed limit space with a $2$-sided bound on Ricci curvature has codimension $4$,  it follows that $Y^3$ is a $3$-dimensional  smooth manifold with 
$\Ric_Y=2g_Y$. This  implies that $Y^3$ is a space form $S^3/\Gamma$ for some discrete subgroup $\Gamma$ of $O(4)$ acting freely. 
By Theorem \ref{t:cross_tangent_cone},
the order of subgroup $\Gamma$ is completely determined by the volume ratio at $x$.  Since the space of cross sections of tangent cones at one point is path connected (see Theorem \ref{t:cross_tangent_cone}) it follows that $\Gamma$ is unique. Thus the tangent cone at $x$ is unique. This finishes the proof of Theorem \ref{t:main_bounded_finite_measure}.
	\end{proof}

\vspace{.2cm}

\section{Additional examples}
\label{s:examples}
This is the first of three sections which constitute the second part of the paper.

Basic examples of neck regions and the neck decomposition was given
in Example \ref{ex:basic_neck}, \ref{ex:cone_decomposition}, \ref{e:1}.
In the present brief section, we will provide
 some additional examples. They show the
sharpness of our results and illustrate how more naive versions of the statements can fail to hold.
\vskip-2mm

\subsection{Example 1:  Conical Neck Region}
\label{example:cone_neck}
A key result in this paper states that the packing measure of a neck region 
$\cN=B_2(p)\setminus \overline B_{r_x}(\cC)$ is uniformly Ahlfors regular;
see Example \ref{ex:basic_neck}
 and Theorem \ref{t:neck_region2}.  The key technical result needed to prove
 this the statement that if $u:B_2(p)\to \dR^k$ is a harmonic splitting function, then for {\it $\cH^k$}-most
  points of $\cC_\epsilon\subseteq \cC$, $u$ 
  is a 
 $(1+\epsilon)$-bi-Lipschitz map onto its image; see Proposition \ref{p:bilipschitz_structure}. 
 In the simplest example of a neck region, 
Example \ref{ex:basic_neck}, we could take $\cC_\epsilon=\cC$.  The present example shows that in general,
 this is not the case.
 \vskip2mm
 
 In fact, the map $u$ can degenerate on parts of a neck region. This explains
 the statement of the structural result given in  Proposition \ref{p:bilipschitz_structure}. 
Although it deals with a what at first glance might seem like a
relatively minor technical point, this example  is useful to remember
when one is faced with traversing the maze of technical results which come later in the paper.
 In particular, it demonstrates why simpler sounding statements just do not hold.
\vskip2mm

Let $Y_r := \text{Susp}(S^1_r)$ denote the suspension of a circle
of radius $r$.  Note that if $r=1$ then $Y_1=S^2$.  For $r<1$, the space
$Y_r$ will have two singular points $p,q\in Y$ at antipodal points. It will
look like an american football.  By using a warped product construction,  one can easily check  that
$Y_r$  can be smoothed to obtain $Y_{\epsilon,r}$, which is diffeomorphic to $S^2$, satisfies $Y_{\epsilon}=Y_{\epsilon,r}$ outside of $B_{\epsilon}(p)\cup B_{\epsilon}(q)$ and  and has sectional curvature $\geq 1$.
\vskip2mm

Let $X^3=C(Y_{\epsilon,r})$ denote the cone over $Y_{\epsilon,r}$.  Note that $X^3$ has a unique singular point at $0\in X^3$.  Using the techniques of \cite{CoNa2},
 one can check that $X^3$ itself arises as a Ricci limit space.  Let $\gamma_p$, $\gamma_q$ denote the rays in $X^3$ through the cross-section points $p,q\in Y_{\epsilon,r}$.  Though $X^3$ is smooth along these rays, for $\epsilon$ very small $X^3$ is looking increasingly singular.  For each $x\in \gamma_p \cup \gamma_q$ let $r_x = r_0\cdot d(x,0)$, where $r_0>>\epsilon$ is fixed and small.  Finally let $\cC = \{0\}\cup \{x_i\}\subseteq (\gamma_p \cup \gamma_q)\cap B_2(p)$ be a maximal subset such that $B_{\tau^2 r_i}(x_i)$ are disjoint.  Note then that for any $\delta,\eta>0$ one can check for $\epsilon<<\delta$ that $\cN :=  B_2(p)\setminus \overline B_{r_x}(\cC)$ defines a $(1,\delta,\eta)$-neck region.
\vskip2mm

Now let $u:B_2(p)\to \dR$ denote a harmonic $\epsilon$-splitting map.  By using separation of variables one can check that  $|\nabla u(x)|\to 0$ as $x\to 0$ approaches the vertex of the cone.
Indeed, this holds for {\it any} harmonic function on $B_2(p)$.
  In particular, it is certainly not possible that $u$ defines a 
  $(1+\delta)$-bi-Lipschitz map on all of $\cC$.  One can check that as $\epsilon\to 0$, $u$ remains
   bi-Lipschitz on $\cC$ away from an increasingly small ball around $0$.  This shows the sharpness of the 
   bi-Lipschitz structure of Proposition \ref{p:bilipschitz_structure}.
\vskip-2mm

\subsection{Example 2:  Sharpness of $k$-rectifiable Structure}\label{example:k_rect_sharp}

One of the primary results of this paper is to show that the $k$th-stratum 
$S^k\setminus S^{k-1}$ of the singular set is $k$-rectifiable.   The following example from  \cite{LiNa17_example} shows that this statement is sharp in the sense that their need
not exist any points in the singular set $S$ in a neighborhood of  which $S$ is a manifold.
\vskip2mm

In \cite{LiNa17_example} the following examples are produced:  For each real number $s\in [0,n-2]$ there exists $(M^n_j,g_j)\stackrel{d_{GH}}{\longrightarrow} (X,d)$, with $\diam \,(M^n_j) \leq 1$,  $\Vol \,(M^n_j) > v>0 $,
 such that the singular set, $S$ satisfies 
\begin{align}
&\dim\, S=s\, ,\notag\\
&S \text{ is a $s$-Cantor set}\, .
\end{align}
If $s=k\in \dN$ is an integer one can further arrange it so that $S = S^k = 
S^k_\epsilon$ satisfies $0<\cH^k(S)<\infty$ is both $k$-rectifiable and a $k$-Cantor set.  
In particular, we see from these examples that the structure theory of Theorem \ref{t:main_stratification} is sharp, and one cannot hope for better. \\
\vskip2mm

We will briefly explain the example above from \cite{LiNa17_example} for the case, $0\leq s\leq 1$. Higher dimensional examples are built in an analogous manner. 
\vskip2mm

 Let $Z=\overline B_1(0^2)\times [0,1]\subseteq \dR^3$ denote the closed 3-cylinder.  Observe that $Z$ is an Alexandrov space with boundary, and that its singular set is the codimension 2 circles $S(Z) = \partial B_1(0^2)\times \{0,1\}$.  
\vskip2mm

Double $Z$  to obtain an Alexandrov space without boundary $\tilde Z$ with codimension 2 singular sets $S(\tilde Z) = S^1\times\dZ_2 \subseteq \tilde Z$.  It is not difficult to see that $\tilde Z$ may be smoothed to obtain a manifolds by rounding off the doubled boundary points.  Intuitively, the key point of this example is that these circle singular sets are sets of infinite positive sectional curvature in every direction, as opposed to the easier construction of a codimension 2 singular set built by looking at 
$\dR\times C(S^1_r)$ as in the neck example.  Note that because of this,
the singular set $S(\tilde Z)$ is not totally geodesic. However,   the regular set of 
$\tilde Z$ is convex. In fact, by \cite{CoNa1}, this must be the case.
\vskip2mm

Now choose an arbitrary open set 
$U=\bigcup (a_i,b_i)\subseteq S(Z)$.  By using the fact that $S(Z)$ consists of 
completely convex points of $Z\subseteq \dR^3$, one can construct a subset $Y\subseteq \dR^3$ by (informally speaking) "sanding off" each interval $(a_i,b_i)$ to obtain smooth boundary points such that $Y$ is still convex. Hence, after this procedure
has been carried out, $Z$ is still an Alexandrov space.
\vskip2mm

 At this point, we have $S(Y) = S(Z)\setminus U$, and we can again double $\tilde Y$ to obatain Alexandrov space without boundary such that $S(\tilde Y)$ is isometric to $S(Y)$.  By choosing $U$ as in the standard Cantor constructions, we can make $S(\tilde Y)$ a $s$-Cantor set for $0\leq s\leq 1$, as claimed.  Since $S(Y)$ is contained in a circle,  
 this set is $1$-rectifiable.*
\vskip-5mm

\subsection{Example 3:  Sharpness of $k$-symmetries of Tangent Cones}\label{example:symmetry_uniqueness_sharpness}

One of the main statements in Theorem \ref{t:main_stratification} is that for $\cH^k$-a.e. $x\in S^k$, all  tangent cones are $k$-symmetric. Recall
however that we do not assert that tangent cones are unique.   In this example, we show that indeed tangent cones is need {\it not} be unique
 for $k$-a.e. $x\in S^k$.
\vskip2mm

The examples are rather straight forward.  For instance, let $Y$ be a
 noncollapsed limit space such that there is an isolated singularity at $p\in Y$.  
 Assume $p\in S^0_\epsilon(Y)$ is such that the tangent cone at $p$ is not unique;
  see for instance \cite{ChCoI}, \cite{CoNa2},  for such examples.  Nonetheless, every tangent cone is $0$-symmetric as this is a noncollapsed limit.  Put $X=\dR^k\times Y$.
 Then the singular set of $X$ satisfies $S = S^k(X) = S^k_\epsilon =  \dR^k\times\{p\}$.  In this case, as claimed, every tangent cone in $S^k(X)$ is 
 $k$-symmetric, but none are unique.
 \vskip2mm

\subsection{Example 4:  Sharpness of $S^k_\epsilon$-finiteness}\label{example:finiteness_sharp}
 Theorem \ref{t:main_stratification} states that  $S^k$ is $k$-rectifiable. Theorem \ref{t:main_eps_stationary} states that the quantitative stratification $S^k_\epsilon$ has uniformly bounded $k$-dimensional Hausdorff measure.  Well known examples demonstrate that this need not hold for $S^k$.  Thus,  the best one can say is that $S^k(Y)$ is a countable union of finite measure rectifiable sets,
as  stated in Theorem \ref{t:main_stratification}.
\vskip2mm

 Start with solid regular tetrahedron $Z^3_0$, centered at the origin. Attach to each face
$F^2_i$, a tetrahedron with very small altitude and base $F^2_i$. Call the resulting convex polytope $Z^3_1$.  Proceed inductively in this fashion to obtain a sequence of
convex polytopes $Z^3_2$, $Z^3_3\, \ldots$, in such a way the sequence of altitudes
goes sufficiently rapidly to zero so that the following will hold. The sequence Hausdorff converges to a convex subset
$Z^3_\infty$ and $\partial Z^3_i\stackrel{GH}{\longrightarrow}\partial Z^3_\infty$
as well.  Moreover, $\partial Z^3$ has a dense set of singular points, although for
all $\epsilon>0$, only a finite number fail to have a neighborhood which is $\epsilon$-regular. 
The polytopes $Z^3_i$ can be "sanded" to produced a sequence of smooth convex surfaces 
$M^2_i\stackrel{GH}{\longrightarrow}\partial Z^3_\infty$. Thus,
$\partial Z^3_\infty$ the $GH$-limit of a sequence of smooth manifolds with nonnegative curvature.
  Of course, higher dimensional example can be gotten e.g. by
taking isometric products with $\dR^k$.
 \vskip2mm

\section{Preliminaries}
\label{s:Prelim}

In this section, we will review the technical background material which is required for the proofs our main theorems. 
For the less standard material, particularly the results concerning entropy, we will give a detailed indication of the proofs. In all cases we will give complete references. 

\vskip2mm

  \subsection{Almost volume cones are almost metric cones}
\label{ss:prelim:volume_cone}
Let
the metric on the unique simply connect $n$-manifold with constant curvature $:=  -\kappa$
be written in geodesic polar coordinates as $dr^2+f_{\kappa}^2 g^{S^{n-1}}$, where $\tilde g^{S^{n-1}}$ denotes
the metric on the unit $(n-1)$-sphere.
Let $\Ric_{M^n_i}\geq -\kappa$ and assume
$M^n_i\stackrel{d_{GH}}{\longrightarrow} X^n$, where $X^n$ is noncollapsed.
 Then if $X^n$ is equipped with $n$-dimensional Hausdorff measure,  the convergence is actually in the measured 
Gromov-Hausdorff sense; 
 \cite{ChCoI}. If we extend the definition of  the volume ratio 
$\cV_r(x) = \cV^\kappa_r(x)$,  to points $x\in X^n$,
 then as in \eqref{e:volratio}, we have $\frac{d}{dr}\cV^\kappa_r(x)\leq 0$. 
 \vskip2mm
 
 Suppose  $\frac{d}{dr}\cV^\kappa_r(x) =0$ for some fixed $r$ and initially, and let suppose
 the metric on $X^n$ is smooth in a neighborhood of $Y:=  \partial B_r(x)$,
 with $g^Y$ the induced metric on $Y$. Then in geodesic polar coordinates
 on $B_r(x)$
 the Riemannian metric is given by
   \begin{align}
  \label{e:wp}
  dr^2+ f_{\kappa}^2 g^Y\, .
 \end{align}
The proof of this fact given in \cite{ChCoAlmost} uses the characterization of the warped product metrics
as those for which there is a potential function whose Hessian is a multiple of the metric. 
\vskip1mm

\noindent
{\bf Notation.} Below, the Hessian of  $f$ is sometimes denoted by
by $\Hess_f$ and sometimes by $\nabla^2 f$.
\vskip2mm

If  $\kappa=0$ then the warped product is a metric cone, in which case,
\begin{align}
\label{e:hess}
\Hess \, r^2-2 g=0\, .
\end{align}
This should be compared to the corresponding formula for the derivative of the time derivative of the 
entropy given in \eqref{e:Dt_W}. 
\vskip2mm

 The discussion can be extended to the case in which the smoothness assumption on 
 $\partial B_r(x):=  Y$ is dropped, provided the expression in \eqref{e:wp} is replaced by the expression 
 for the distance function $d$ on $B_r(x)$.  Let $d^Y$ denote the distance function on $Y$. Then for the case $\kappa=0$, $(B_r(x),d)$ is isometric to a ball in the metric cone with cross-section $(Y,d^Y)$
 with $\diam(Y)\leq \pi$ and $d$ is given by the {\it law of cosines} formula, 
 \begin{align}
 \label{e:mc}
 d^2((r_1,y_1),(r_2,y_2))=r_1^2+r_2^2-2r_1r_2\cdot \cos d^Y(y_1,y_2)\, .
 \end{align}
 
 By using Gromov's compactness theorem, the following "almost volume cone implies almost metric cone theorem"
 is easily seen to be equivalent to what has just been discussed. 

  \begin{theorem}[\cite{ChCoAlmost}]
  \label{t:almostmetriccone}
  	Let $(M^n,g,p)$ denote a Riemannian manifold with $\Ric_{M^n}\ge -\delta$. Given $\epsilon>0$,   if $\delta\le \delta(n,\epsilon)$ and  $\cV_2(p)\ge (1-\delta)\cV_{1}(p)$, then $B_1(p)$ is 
$(0,\epsilon)$-symmetric.
	 \end{theorem}

In  the proof of tangent cone uniqueness we also 
used the following well known result in which relation (2) below follows from volume convegence; compare Theorem \ref{t:tangentcone_uniqueness} and Theorem \ref{t:tangentcone_uniqueness}.

\begin{theorem}
\label{t:cross_tangent_cone}
	Let $(M^n_i,g_i,p_i)\to (X,d,p)$ satisfy $\Ric_{M^n_i}\ge -(n-1)$ and $\Vol(B_1(p_i))\ge \rv>0$ then the cross section space $C_x:=  \{(Y,d_Y): C(Y) \mbox{ is a tangent cone at $x$}\}$ of tangent cones at $x\in X$ satisfies 
	\begin{enumerate}
		\item $(C_x,d_{GH})$ is path connected.
		\vskip1mm
		
		\item For every $Y\in C_x$ we have $\Vol(Y)=\lim_{r\to 0} \frac{n\Vol(B_r(x))}{r^n}$.
	\end{enumerate} 
\end{theorem}

The following result was proved in \cite{Cheeger01}, \cite{ChCoI}, \cite{Co1}.
It implies in particular that at a point in the regular set, $\cR:=  X^n\setminus S$,
the tangent cone  is unique and isometric to 
$\dR^n$. In fact, since the conclusion applies to all balls $B_r(x)\subset B_3(p)$,
it is actually a kind of quantitative $\epsilon$-regularity theorem.
\begin{theorem}\label{t:subball_n_symmetric}
	Let $(M^n_i,g_i,p_i)\to (X,d,p)$ satisfy $\Ric_{M^n_i}\ge -(n-1)\delta$ and $\Vol(B_1(p_i))\ge \rv>0$. Let $\epsilon>0$,  $\delta\le \delta(n,\rv,\epsilon)$ and assume $B_4(p)$ is $(n,\delta)$-symmetric.  Then each $B_r(x)\subset B_3(p)$ is also $(n,\epsilon)$-symmetric.
\end{theorem}
\vspace{.25cm}

\subsection{Quantitative cone-splitting}
\label{ss:qcs}
  As recalled in Section \ref{s:outline_proof}, if in a metric cone has $2$ distinct vertices, then
the cone isometrically splits off a line which contains these two vertices.  If there are several such cone vertices, then this statement can be iterated to produce further splittings. A quantitative version
of cone-splitting was introduced in 
\cite{CheegerNaber_Ricci}. Prior to stating this theorem, it is convenient 
to introduce a quantitative notion of $(k+1)$ points $x_0,\ldots,x_k$ being 
$k$-independent.

In $\dR^n$ we say that points $\{x_0,\ldots,x_k\}$ are {\it $k$-independent} 
if the $\{x_i\}_0^k$ is not contained in any $(k-1)$-plane.  
Here is a quantitative version
of this notion.  

\begin{definition}[$(k,\alpha)$-independence]
\label{d:independent_points1}
In a metric space $(X,d)$ a set of points $U=\{x_0,\ldots ,x_k\}\subset B_{2r}(x)$ is $(k,\alpha)$-independent if for any subset $U'=\{x_0',\ldots, x_k'\}\subset \dR^{k-1}$ we have 
\begin{align}
\label{e:alpha}
d_{GH}(U,U')\geq \alpha\cdot r\, .
\end{align}
\end{definition}

\vskip2mm
\noindent
\begin{remark}
\label{r: kalpha_independent_Rn} 
Let $X\subset \mathbb{R}^n$, if there exists no $(k,\alpha)$-independent set in $B_r(x)\cap X$, then $B_r(x)\cap X\subset B_{4\alpha r}(\mathbb{R}^{k-1})$ for some $k-1$-plane $\mathbb{R}^{k-1}\subset \mathbb{R}^n$. To see this, if $B_r(x)\cap X$ is not a subset of $B_{3\alpha r}(\mathbb{R}^{k-1})$ for any $k-1$-plane, then one can find $(k,\alpha)$-independent set in $B_r(x)\cap X$ by induction on $k$. 
\end{remark}

The following  Quantitative Cone-Splitting Theorem was introduced in \cite{CheegerNaber_Ricci}.  

\begin{theorem}[Cone-Splitting]
\label{t:cone_splitting}
	Let $(M^n,g,p)$ satisfy $\Ric_{M^n}\ge -\delta$. Let $\epsilon,\tau>0$ and  
$\delta\le \delta(n,\epsilon,\tau)$ and assume:
	\begin{itemize}
		\item[(1)] $B_{2}(p)$ is $(k,\delta)$-symmetric with respect to $\cL^k_\delta\subseteq B_2(p)$ as in Remark \ref{rm:epssym}.
		\vskip2mm
        \item[(2)]There exists $x\in B_1(p)\setminus B_\tau \cL^k_\delta$ such that $B_2(x)$ is $(0,\delta)$-symmetric.
        \end{itemize}
\vskip2mm

Then $B_1(p)$ is $(k+1,\epsilon)$-symmetric.
\end{theorem}
\vskip2mm
\begin{remark}\label{r:cone_splitting}
We can rephrase the above as follows:  If $U=\{x_0,\cdots,x_k\}\subset B_{2r}(x)$ is $(k,\alpha)$-independent and each $x_i$ is $(0,\delta)$-symmetric, by the Cone-Splitting Theorem \ref{t:cone_splitting}, the ball $B_{2r}(x_0)$ is $(k,\epsilon)$-symmetric for $\delta\le \delta(n,\alpha,k,\epsilon)$.
\end{remark}

A second  version of quantitative cone-splitting theorem  is implicit in \cite{CheegerNaber_Ricci}. It
is a direct consequence of Theorem \ref{t:cone_splitting}.  To define it let us define the notion of the pinching set:
  \vskip2mm

\begin{definition}[Points with small volume pinching]
\label{d:smallvolpinch} 
Let $(M^n,g,p)$ satisfy $\Ric_{M^n}\ge-(n-1)\xi $ and
put 
\begin{align}
\label{e:hatV}
\bar V :=  \inf_{x\in B_1(p)}\cV_{\xi^{-1}}(x))\, .
\end{align}
The set with small volume pinching is:
\begin{align}
\label{e:cFset1}
\cF_{r,\xi}(x):= \{y\in B_{4r}(x):~\cV_{\xi r}(y)\le \bar V+\xi\}\, .
\end{align} 
\end{definition}

Note that if $\xi^{-1}$ is large, then
 the point in $\cF_{r,\xi}(x)$ is an  "almost cone vertex" for each scale between $r$ and $\xi^{-1}$.
By Theorem \ref{t:almostmetriccone}, each point $y\in \cF_{r,\xi}(x)$ is an "almost cone vertex". Thus,  with Theorem \ref{t:cone_splitting} we immediately have the following: 

\begin{theorem}[Cone-Splitting based on $k$-content]
\label{t:content_splitting} 
Let $(M^n,g,p)$ satisfy $\Vol(B_1(p))\ge \rv>0$ with $\Ric_{M^n}\ge -(n-1)\xi$. Assume: 
 $0<\delta,\epsilon\le \delta(n,\rv)$,  $\gamma\le \gamma(n,\rv,\epsilon)$, $\xi\le \xi(\delta,\epsilon,\gamma,n,\rv)$ and
\begin{align}
	\Vol(B_{\gamma}(\cF_{1,\xi}(p)))\ge \epsilon \gamma^{n-k}\, .
\end{align}
\vskip1mm

Then there exists $q\in B_4(p)$ such that $B_{\delta^{-1}}(q)$ is $(k,\delta^2)$-symmetric.
\end{theorem}

The import of Theorem \ref{t:content_splitting}, is that if  
the set of pinched points $\cF_{1,\xi}$ has a definite amount of $k$-content, then the ball must be $k$-symmetric.
The scale invariant version states that if
\begin{align}
\label{e:siv}
\Vol(B_{\gamma r}(\cF_{r,\xi}(p)))\ge \epsilon \gamma^{n-k}r^n\, ,
\end{align}
then $B_{\delta^{-1}r}(q)$ is $(k,\delta^2)$-symmetric for some $q\in B_r(p)$.

\subsection{Harmonic $\epsilon$-splitting functions}
\label{ss:delta_splitting}

The following definition, which encapsulates the technique of
 \cite{ChCoAlmost} for obtaining approximate splittings,
is essentially the one
 formalized in \cite{CheegerNaber_Codimensionfour}.
 
\begin{definition}[Harmonic $\delta$-Splitting Map]
 \label{d:splitting_function}
The map $u:B_r(p)\to \mathbb{R}^k$ is a harmonic
  $\delta$-splitting map if
	\begin{enumerate}
	\item $\Delta u = 0$.

    \item $\fint_{B_r(p)}\big|\langle\nabla u^i,\nabla u^j\rangle - \delta^{ij}\big|<\delta$.

	\item $\sup_{B_{r}(p)} |\nabla u| \leq 1+\delta$.
	\item $r^2\fint_{B_r(p)}\big|\nabla^2 u\big|^2<\delta^2$.
	\end{enumerate}
\end{definition}

For the case of limit spaces, we can define $\delta$-splitting maps as follows. 
If $B_r(p_i)\subset M_i\to B_r(p)\subset X$	and $\delta_i$-splitting maps $u_i:B_r(p_i)\to \mathbb{R}^k$ converge uniformly to $u: B_r(p)\to \mathbb{R}^k$ with $\delta_i\to \delta$, we call $u$ is $\delta$-splitting on $B_r(p)\subset X$. By the $W^{1,2}$-convergence in Proposition \ref{p:convergenc_function}, we have that the $\delta$-splitting $u$ satisfies 
(1)--(4) in the limit space.
\vskip2mm

The following is a slight extension of
the result in \cite{ChCoAlmost}. 

\begin{theorem}
\label{t:splitting_function}
	Let $(M^n,g,p)$ satisfy $\Ric_{M^n}\ge -\delta$. For any $\epsilon>0$, if $\delta\le \delta(n,\epsilon)$ then the following hold:
	\vskip2mm
	\begin{itemize}
				\item[(1)] If there exists a $\delta$-splitting function $u:B_2(p)\to \mathbb{R}^k$, then $B_{1}(p)$ is $\epsilon$-GH close to $\mathbb{R}^k\times X$\, .
				\vskip2mm
								 \item[(2)] If $B_{4}(p)$ is $\delta$-GH close to $\mathbb{R}^k\times X$, then there exists an $\epsilon$-splitting function $u:B_2(p)\to \mathbb{R}^k$.
	\end{itemize}
\end{theorem}

\subsection{A cutoff function with bounded Laplacian}
\label{ss:cutoff}
The existence of a cutoff function which satisfies the standard estimates and has a definite pointwise bound on its Laplacian is
important technical tool.
 In particular, such a cutoff function is required for the discussion of the local pointed entropy; see Subsection \ref{ss:lpe}.\footnote{The original proof of 
the existence of the required cutoff function  employed solutions of the Poisson equation,
$\Delta u = 1$ and a delicate argument 
based on the quantitative
maximum principle.  One can also give a proof by using heat flow as in \cite{MondinoNaber_Rectifiability}.}

\begin{theorem}[\cite{ChCoAlmost}]
\label{t:cutoff}
	Let $(M^n,g,p)$ be a Riemannian manifold with $\Ric_{M^n}\ge -(n-1)r^2$. Then for any 
there exists cutoff function $\phi_r: M^n\to [0,1]$ with support in $B_r(p)$ such that $\phi_r:=  1$ in $B_{r/2}(p)$. Moreover,
	\begin{align}
	\label{e:cutoff}
	r^2|\nabla \phi_r|^2+r^2|\Delta\phi_r|\le C(n)\, .
	\end{align}
\end{theorem}
\vskip2mm

\subsection{Heat kernel estimates and heat kernel convergence}
\label{ss:heat_kernel_poincare}
Let $\rho_t(x,y)$ denote the heat kernel on $M^n$. 
For each $x$ we have
$$
\int_{M^n}\rho_t(x,y)\, d\mu(y)=1\, .
$$
Define the function
$f_t(x,y)$ by
\begin{align}
\label{e:logheat}
\rho_t(x,y)=(4\pi t)^{-n/2}e^{-f_t(x,y)}\, .
\end{align}

Next we recall some classical heat kernel estimates for manifolds with lower Ricci curvature bounds, as well as the heat kernel convergence result for Gromov-Hausdorff convergence. We summarize the heat kernel estimates in the following theorem; see \cite{LiYau_heatkernel86}, \,\cite{SoZha},\, \cite{SY_Redbook},\, \cite{Hamilton_gradient},\, \cite{Kot_hamilton_gradient}.

\begin{theorem}[Heat Kernel Estimates]
\label{t:heat_kernel}
Let $(M^n,g,p)$ satisfy $\Ric_{M^n_i}\ge -(n-1)\delta^2$ and 
$\Vol(B_r(p))\ge \rv \cdot r^n>0$,  for $r\le \delta^{-1}$.  Then for any $0<t\le 10\delta^{-2}$ and $\epsilon>0$ with $x,y\in B_{10 \delta^{-1}}(p)$: 
\vskip1mm
	 
	\begin{enumerate}
     \item $-C(n,\rv,\epsilon)+\frac{d^2(x,y)}{(4+\epsilon)t}\leq f_t \leq C(n,\rv,\epsilon)+\frac{d^2(x,y)}{(4-\epsilon)t }$
     \vskip2mm
     
	\item $t|\nabla f_t|^2\leq C(n,\rv,\epsilon)+\frac{d^2(x,y)}{(4-\epsilon)t}$ .
	\vskip2mm
	
	\item $ -C(n,\rv,\epsilon)-\frac{d^2(x,y)}{(4+\epsilon)t}\le  t\Delta f_t \leq C(n,\rv,\epsilon)+\frac{d^2(x,y)}{(4-\epsilon)t }.$
	\end{enumerate}
\end{theorem}
\vskip2mm

The estimates in (1) are Li-Yau heat kernel upper and lower bound estimates. (2) follows from (1) and a local gradient estimate, see for instance \cite{SoZha}. (3) follows from the Li-Yau Harnack inequality, (2) and \cite{Hamilton_gradient},\, \cite{Kot_hamilton_gradient}.   
\vskip2mm

The following result is well known in the context of 
Ricci limit spaces and even for RCD spaces.  One direct proof is got using gradient flow convergence of Cheeger energy in \cite{Ambrosio_Calculus_Ricci}; see also \cite{AmHo,GiMoSa}; see \cite{Ambrosio_Calculus_Ricci,AmHo,AmHoTe,ZhaZhu17} for more general results in the RCD setting.   In our application the limit space $X$ is a metric cone in which case the heat kernel convergence was proved in \cite{Ding_heat}.	
\begin{proposition}[Heat kernel convergence]\label{p:heat_kernel_convergence}
Suppose $(M_i,g_i,x_i,\mu_i)\to (X,d,x_\infty,\mu)$ with 
$\Ric_{M^n_i}\ge -(n-1)$ and $\mu_i=\Vol(B_1(x_i))^{-1}\Vol(\, \cdot\,)$.   Then the heat kernel $\rho^i_t(x,y)$ converges uniformly to heat kernel $\rho_t^\infty(x,y)$ on any compact subset of $\dR_{+}\times X\times X$.
\end{proposition}

%
\begin{remark}
\label{r:W12_heatkernel}
By the heat kernel Laplacian estimate in Theorem \ref{t:heat_kernel} and $W^{1,2}$-convergence Proposition \ref{p:convergenc_function},
 it follows that for any fixed $t$, we have $\rho_t^i(x_i,\cdot)\to \rho_t^\infty(x_\infty,\cdot)$ in the $W^{1,2}$-sense as in Definition \ref{d:W12convergence}. 
 \end{remark}
\begin{remark}\label{r:heat_kernel_cone}
If the limit space is a noncollapsed metric cone $X=C(Y)$ with cone vertex $x_\infty$, then for the normalized measure $\mu_\infty=\Vol(B_1(x_\infty))^{-1}\Vol(\,\cdot \,)$ we have:
$$
\rho_t^\infty(x_\infty,\cdot)=	\frac{\Vol(S^{n-1})}{n}
\cdot \frac{e^{-d^2(x_\infty,x)/4t}}{(4\pi t)^{n/2}}\,.
$$
This follows easily by  computing the $s$-derivative of
\begin{align}
	\eta(t,s,x):= \int_{C(X)}\rho_{t-s}^\infty(x,y)\cdot \Big(\rho_s^\infty(x_\infty,y)-\frac{\Vol(S^{n-1})}{n}
	\cdot \frac{e^{-d^2(x_\infty,y)/4t}}{(4\pi t)^{n/2}}
	\Big)\, d\mu_\infty(y),
\end{align}
to conclude that $\eta(t,t,x)=\eta(t,0,x)=0$.  
\end{remark}
\vskip2mm

\subsection{The local pointed entropy, $\cW_t(x)$ and its relation to cone structure}
\label{ss:lpe}${}$
As discussed in Subsection \ref{ss:prelim:volume_cone}, ``Almost volume cones are almost metric cones'',
previously known results on quantitative cone-splitting were stated  in Theorem
 \ref{t:cone_splitting} and Theorem \ref{t:content_splitting},
 As with the definition of neck regions, 
 the hypotheses of
 these results, as well a the definition of neck regions, involve
  the volume ratio $\cV_r(x)$.
 For our purposes, it is crucial to have a sharp version 
 of quantitative cone-splitting.  As mentioned in previous sections,
 it turns out that many 
 technical details are simpler if we use 
 in place of $\cV_r(x)$,
 a less elementary  monotone quantity, the 
 {\it local pointed entropy} $\cW_t(x)$. Therefore,
 it is necessary to have a result stating that (with suitable interpretation)
   $\cW_t(x)$ and $\cV_r(x)$ have
essentially the same behavior.  This is the content of Theorem \ref{t:cW_local},
which also includes the fact that $\cW_t(x)$ is monotone in $t$.
 The  sharp cone-splitting estimate, the statement of which involves
 entropy, is given in Theorem
 \ref{t:sharp_splitting}.
 \vskip2mm
 
 In the present subsection, we derive the needed background
 results on the local pointed entropy, which is a local version of
 Perelman's $\cW$-entropy, 
generalized in \cite{Ni} to smooth manifolds. In order to emphasize the basics, we by discussing the technically
simpler concept of the {\it pointed entropy}.
\vskip1mm

If as in \eqref{e:logheat} we write $\rho_t(x,dy) = (4\pi t)^{-n/2}\cdot e^{-f_{x,t}(y)}$, then by definition the  {\it weighted Laplacian} $\Delta_f$ is the second order operator associated to the weighted Dirichlet energy 
 $$
 \int_{M^n}(4\pi t)^{-n/2}|\nabla f|^2 \, e^{-f} \, dv_g(y)= \int |\nabla f|^2 \,\rho_t(x,dy)\, .
 $$
 Then
$$
\Delta_f=\Delta-\langle \nabla f,\nabla\, \cdot \,\rangle\, .
$$
\vskip1mm

Set
\begin{align}
W_t=2t\, \Delta_f f+t|\nabla f|^2+f-n.
\end{align}

The  {\it pointed entropy}, $\cW_t(x)$, is for each $x$ a {\it global quantity} defined as follows.
\begin{definition}[Pointed entropy]
\begin{align}
\label{e:entropydef1}
\cW_t(x) := \int_{M^n} W_t\cdot \rho_t(x,dy)\, .
\end{align}
\end{definition}
\vskip3mm

\noindent
Bochner's formula states:
\begin{align}
\label{e:bochner}
\frac{1}{2}\Delta |\nabla u|^2= |\nabla^2 u|^2+\Ric(\nabla u,\nabla u)\, .
\end{align}


The following lemma is proved by direct computation. It shows in particular that $\cW_t(x)$ is monotone
decreasing if $\Ric_{M^n}\geq 0$. Moreover, if in addition $\cW_t(x)$ is constant on $[0,r]$, then 
the ball $B_r(x)$ is isometric to $B_r(0)\subset \R^n$; see (\ref{e:hess}).
\begin{lemma}
\label{l:entropy_computation}
\begin{align}
\label{e:entropy_computation}
	\partial_t\cW_t(x)=-2t\int_M\Big(|\nabla^2f-\frac{1}{2t}g|^2+\Ric(\nabla f,\nabla f)\Big)\,
	 \rho_t(x,dy)\le 0\, .
\end{align}
\end{lemma}
\begin{proof}
Equation \ref{e:entropy_computation} is easily implied by the following computation; compare \eqref{e:hess}:
\begin{align}
\label{e:Dt_W}
	\frac{d}{dt}W_t=\Delta_fW_t-\langle\nabla f,\nabla W_t\rangle -2t\Big(\,|\nabla^2f-\frac{1}{2t}g|^2
+\Ric_{M^n}(\nabla f,\nabla f)\Big)\, .
\end{align}
\end{proof}
\vskip-1mm

Next assume $(M^n,g,p)$ satisfies $\Ric_{M^n}\ge -(n-1)\delta^2$ and $\Vol(B_r(p))\ge \rv r^n>0$ for $r\le \delta^{-1}$.  In this case we will  define a local monotone quantity that  will play a role analogous to the one played by pointed entropy.
\vskip2mm

Let $\varphi: M^n\to [0,1]$ be a cutoff function as in \eqref{e:cutoff}, with support in $B_{2\delta^{-1}}(p)$, satisfying $\varphi:=  1$ in $B_{\delta^{-1}}(p)$ and $|\Delta \varphi|+|\nabla \varphi|^2\le C(n)\delta^2$. 
\vskip2mm

Set
\begin{align}
	\cW_{t,\varphi}(x):=  \int_{M^n} W_t\varphi\, \rho_t(x,dy)-\int_0^t\left(4s\int_M (n-1)\delta^2 |\nabla f|^2\varphi \rho_s(x,dy)\right) ds.
\end{align}
Then, by direct computation,
\begin{align}
\label{e:et}
	\partial_t \cW_{t,\varphi}(x)=-2t\int_M\Big(|\nabla^2f-\frac{1}{2t}g|^2+\Ric(\nabla f,\nabla f)+2(n-1)\delta^2 |\nabla f|^2\Big)\varphi \cdot \rho_t(x,dy)+\int_M W_t\Delta \varphi \, \rho_t(x,dy).
\end{align}
By using the heat kernel estimate in Theorem \ref{t:heat_kernel}, we can control the last term on the right-hand side of \eqref{e:et}.
Namely, for any $x\in B_{\delta^{-1}/2}(p)$ and $t\le \delta^{-2}$, we have
\begin{align}
\label{e:et1}
	\Big|\int_M W_t\Delta \varphi \cdot \rho_t(x,dy)\Big|&\le \int_M |W_t| ~|\Delta \varphi |\,\rho_t(x,dy)\notag\\
	&\le  C(n,\rv)\delta^2\int_{A_{\delta^{-1},2\delta^{-1}}(p)} \left(1+\frac{d^2(x,y)}{4t}\right)\,\rho_t(x,dy)\le C(n,\rv)\delta^2 \cdot e^{-1/100\delta^2 t}\, .
\end{align}
\vskip2mm

This motivates the following definition of  the local $\cW^\delta_t$ pointed entropy.
\begin{definition}[Local $\cW^\delta_t$ pointed entropy]
\label{d:local_entropy}
Let $(M^n,g,p)$ satisfy $\Ric_{M^n}\ge -(n-1)\delta^2$ and $\Vol(B_r(p))\ge \rv r^n>0$ for $r\le \delta^{-1}$. 
For any $t\le \delta^{-2}$ and $x\in B_{\delta^{-1}/2}(p)$, the local $\cW^\delta_t$ 
pointed entropy is defined by:
\begin{align}
\cW_t^\delta(x):=  \cW_{t,\varphi}^\delta(x):= \cW_{t,\varphi}(x)-C(n,\rv)\delta^2\int_0^t 
e^{-1/100\delta^2s}\, ds.	
\end{align}
\end{definition}
\vskip2mm

\begin{remark}[Scaling]
Put $\tilde{g}=r^{-2}g$. If
$\Ric_{M^n}\ge -(n-1)\delta^2$ then $\Ric_{\tilde M^n}\ge -(n-1)\delta^2 r^2$.
Let $\widetilde{\cW}_t^{\delta r}(x)$ denote 
the local $\cW$-entropy associated with $\tilde{g}$. Then
$$
\cW_{tr^2}^{\delta}(x)=\widetilde{\cW}_{t}^{\delta r}(x)\, .
$$
\end{remark}
\vskip2mm


The following theorem is the main result of this subsection.
According to relation (1),
the local $\cW^\delta_t$ pointed entropy is monotone. By relation (2), it 
  has essentially the same behavior as the volume ratio $\cV_r(x)$.

\begin{theorem}
\label{t:cW_local}
Let $(M^n,g,p)$ denote a pointed Riemannian manifold with 
$\Ric_{M^n}\ge -(n-1)\delta^2$ and $\Vol(B_r(p))\ge \rv \cdot r^n>0$, for $r\le \delta^{-1}$.  Then for all $x\in B_{\delta^{-1}/2}(p)$ and $t\le \delta^{-2}$,
the local $\cW_{t}^\delta$-entropy satisfies:
\vskip3mm

\begin{itemize}
\item[(1)]  $\partial_t\cW_{t}^{\delta}(x)
	\le -2t\int_M\Big(|\nabla^2f-\frac{1}{2t}g|^2+\Ric(\nabla f,\nabla f)+2(n-1)\delta^2|\nabla f|^2\Big)\varphi\,  \rho_t(x,dy)\le 0$.	
	\vskip2mm
	
\item[(2)] Given $\epsilon>0$,
assume that  $\delta\le \delta(n,\rv,\epsilon)$,  $0<t\le 10$,  and
\begin{align}
\label{e:ve}
|\cV_{\sqrt{t}\delta^{-1}}(x)-\cV_{\sqrt{t}\delta}(x)|\le \delta\, .
\end{align}
\end{itemize}

\noindent
Then
$$
|\cW_{t}^\delta(x)-\log \cV_{\sqrt{t}}^{\delta^2}(x)|\le \epsilon\, .
$$
	
\end{theorem} 
\begin{proof}
It suffices to prove (2).  

Assume that (2) does not hold  for some 
$\epsilon_0>0$. Then there exists $\delta_i\to 0$ and $(M^n_i,g_i,p_i)$ 
satisfying 
$\Vol(B_r(p_i))\ge \rv r^n>0$ for $r\le \delta_i^{-1}$,
$\Ric_{M^n_i}\ge -(n-1)\delta_i^2$ and such that for some 
$x_i\in B_{\delta^{-1}/2}(p_i)$, we have 
$$
|{\cV}_{\sqrt{t_i}\delta_i}(x_i)-{\cV}_{\delta^{-1}\sqrt{t_i}}(x_i)|\le \delta_i
$$ 
with $0<t_i\le 10$,
but 
$$
|\cW_{t_i}^{\delta_i}(x_i)-\log\cV_{\sqrt{t_i}}(x_i)|\ge \epsilon_0\, .
$$  
\vskip2mm

The rescaled spaces,
$({M^n}_i,\tilde{g}_i,x_i)=(M^n_i,t_i^{-1}g_i,x_i)$,
satisfy
$
\Ric_{M^n_i}\ge -(n-1)\delta_i^2$, and
\begin{align}
|\tilde{\cV}_{\delta_i}(x_i)-\tilde{\cV}_{\delta_i^{-1}}(x_i)|&\le \delta_i\, ,\notag\\
|\tilde{\cW}^{\delta_i \sqrt{t}_i}_1(x_i)-\log{\tilde{\cV}}_{1}(x_i)|
&\ge \epsilon_0\, .\notag
\end{align}
 
 Denote the  heat kernel at time $t=1$ of $({M^n}_i,x_i,\tilde{g}_i)$ by
 $$
 \tilde{\rho}_1(x_i,y)=(4\pi)^{-n/2}e^{-\tilde{f}}\, .
 $$ 
 By the heat kernel estimate in Theorem \ref{t:heat_kernel}, it follows that
 for $\delta_i$ sufficiently small, we  have 
 $$
 |\tilde{\cW}_1(x_i)-\tilde{\cW}_1^{\delta_i\sqrt{t}_i}(x_i)|<\epsilon_0/4\, ,
 $$ 
 where
\begin{align}
	\tilde{\cW}_1(x_i)=\int_{B_{\delta^{-1}_i/2}(x_i)} \Big(|\nabla \tilde{f}|^2+\tilde{f}-n\Big)\, \tilde{\rho}_1(x_i,dy)\, .
\end{align}
 Therefore, for $\delta_i$ sufficiently small,  
 \begin{align}
 \label{e:last}
 |\tilde{\cW}_{1}(x_i)-\log\tilde{\cV}_{1}(x_i)|\ge \epsilon_0/2\, ,
\end{align}
   We will deduce a contradiction to this estimate by letting $i\to\infty$. 
   \vskip2mm

   Thus by Gromov's compactness theorem, there exists a subsequence of 
   $(M^n_i,t_i^{-1}g_i,x_i)$ converging to some metric cone $(C(X^n),d,x_\infty)$. 
   By volume convergence result in \cite{Co1,Cheeger01} (see also Theorem \ref{t:cross_tangent_cone}) we have 
   $$
   \frac{\Vol(X)}{\Vol(S^{n-1})}=\lim_{i\to\infty} \tilde{\cV}_1(x_i)\, .
   $$

By using the heat kernel convergence in 
Proposition \ref{p:heat_kernel_convergence}, together with Remark \ref{r:heat_kernel_cone}, Remark \ref{r:W12_heatkernel} and  the heat kernel estimate in Theorem \ref{t:heat_kernel}, we  conclude that
\begin{align}
	\lim_{i\to\infty}\tilde{\cW}_{1}(x_i)= \int_{C(X)} \Big(|\nabla f_\infty|^2+f_\infty-n\Big)\, \rho_1(x_\infty,dy)\, ,
\end{align}
where $$
\rho_1(x_\infty,y)=(4\pi )^{-n/2}e^{-f_\infty}=\frac{\Vol(S^{n-1})}{\Vol(X)}(4\pi )^{-n/2}e^{-d^2(x_\infty,y)/4}\, .
$$
 A simple computation gives 
\begin{align}
	\int_{C(X)} \Big(|\nabla f_\infty|^2+f_\infty-n\Big)\, \rho_1(x_\infty,dy)=\log\frac{\Vol(X)}{\Vol(S^{n-1})}, .
\end{align}
Since $\tilde{\cW}_{1}(x_i)$ and 
$\log\tilde{\cV}_1(x_i)$ have the same limit, this is a contradiction. 
\end{proof}
\vskip2mm

\subsection{
$(k,\alpha,\delta)$-entropy pinching}
\label{ss:entropy_pinching}
Recall that in Definition \ref{d:independent_points1},
we introduced the notion of a collection of a $(k,\alpha)$-independent
set of points $x_0,\ldots,x_k$. We will use a refinement of this notion to define 
the pinching of the local pointed entropy $\cW_t(x)$. This will be used in the Sharp Cone-Splitting theorem, Theorem
\ref{t:sharp_splitting}.

\begin{definition}
\label{d:kalphadeltapinching}
The $(k,\alpha, \delta)$-entropy pinching, $\cE^{k,\alpha,\delta}_r(x)$ is:

\begin{align}
\label{e:outline:k_pinching}	
\cE^{k,\alpha,\delta}_r(x):= \inf_{\{x_i\}_0^k} \sum 
\big|\cW^\delta_{2r^2}(x_i)-\cW^\delta_{r^2}(x_i)\big|\, ,
\end{align}
where the infimum is taken over all $(k,\alpha)$-independent subsets and the parameter $\delta$ is corresponding to Ricci curvature lower bound.
\end{definition}
\vskip2mm

From the discussion above, it follows that if $\cE^{k,\alpha,\delta}_1(p)<\delta=\delta(\epsilon,\alpha)$, 
then there exists a $(k,\epsilon)$-splitting map $u:B_1(p)\to \dR^k$.  The sharp version
of this relationship is the content of  
Theorem \ref{t:sharp_splitting}, the Sharp Cone-Splitting Theorem. This theorem 
 states that there exists $C(n,\rv,\alpha)$ and
 a splitting map 
$u$  for which the integral of the norm
squared of the Hessian has the following sharp {\it linear} bound in terms of
 the $k$-pinching:\footnote{The proof of \eqref{e:outline:sharp_splitting} is one instance in which choosing to use the pointed entropy as our monotone quantity helps to make the argument run more smoothly than if we had chosen to work with the volume ratio $\cV_r(x)$.}
\begin{align}
\label{e:outline:sharp_splitting}
\fint_{B_1(p)} |\nabla^2 u|^2 \leq C(n,\rv,\alpha)\cdot \cE^{k,\alpha,\delta}_1(p)\, ,
\end{align}
\vskip2mm

\subsection{Poincar\'e inequalities}

We recall various Poincar\'e inequalities which hold on manifolds with Ricci lower bound; see also \cite{Bu,Cheeger_DiffLipFun,Cheeger01,ChCoIII}. We will need
the ones which follow:

\begin{theorem}[Poincar\'e Inequalities]\label{t:poincare}
	Let $(M^n,g,x)$ satisfy $\Ric_{M^n}\ge -(n-1)$. Then for any $0<r\le 10$, the following 
	Poincar\'e inequalities hold:
	\vskip3mm
	\begin{itemize}
		\item[(1)] $\fint_{B_r(x)}f^2\le C(n)\cdot r^2\fint_{B_r(x)}|\nabla f|^2$\qquad ({\rm for\,\,
		 all} $f\in C^\infty_0(B_r(x))$)\, .
		\vskip2mm
		\item[(2)] $\fint_{B_r(x)}\Big|f-\fint_{B_r(x)}f\Big|^2\le C(n)\cdot 
		r^2\fint_{B_r(x)}|\nabla f|^2$ \qquad ({\rm for\,\, all} $f\in C^\infty(B_r(x))$\, .
	\end{itemize}
\end{theorem}
\vskip1mm

	The Dirichlet Poincar\'e inequality (1) follows directly from {\it segment inequality} in \cite{ChCoAlmost} and \cite{Cheeger01}. For the Neumann Poincar\'e Inequality (2), by using segment inequality, we have a weak Poincar\'e inequality \cite{ChCoIII}. 
	By the volume doubling and a covering argument in \cite{HaKo} or \cite{Jer}, we can
	obtain the Neumann Poincar\'e inequality (2).


\subsection{$W^{1,2}$-convergence}
\label{ss:prelmiW12convergence}

 Below, the notation $(Z_i,d_i,z_i)\stackrel{d_{GH}}{\longrightarrow}(Z,d,z)$
 should always be understood as convergence  in the  measured Gromov-Hausdorff sense.
 In this subsection, we  will assume without explicit mention that the metric measure space $(Z,d,\mu)$ is separable and complete and that $\mu$ is a Borel measure $\mu$ which is finite on bounded subset of $Z$. 
\begin{definition}[Uniform convergence]
Let  $(Z_i,d_i,z_i)\stackrel{d_{GH}}{\longrightarrow}
 (Z,d,z)$. If $f_i$ are Borel functions on $Z_i$, then we say
  $f_i\to f: Z\to \mathbb{R}$ {\it uniformly} if for any compact subset $K_i\subset Z_i\to K\subset Z$ and $\epsilon_i$-GH approximation $\Psi_i: K\to K_i$ with $\epsilon_i\to 0$, the function $f_i\circ \Psi_i$ converges to $f$ uniformly on $K$. 
\end{definition}

As motivation for what follows, recall that on a fixed metric measure space, for $1<p<\infty$, weak convergence together with convergence of norms implies strong convergence.
\begin{definition}[Weak $L^p$ convergence]
	Let $(Z_i,d_i,z_i,\mu_i)\stackrel{d_{GH}}{\longrightarrow} (Z,d,z,\mu)$. If $ f_i$ are Borel function on $Z_i$, we say $f_i\to f:Z\to \mathbb{R}$ in the 
	{\it weak sense} if for any uniformly converging sequence of compactly supported Lipschitz functions $\varphi_i\to \varphi$, we have 
	\begin{align}\label{e:weak_L2}
		\lim_{i\to \infty}\int f_i\varphi_i d\mu_i= \int f \varphi d\mu.
	\end{align}
	Moreover if $f_i,f$ have uniformly bounded $L^p$ integral then we say $f_i\to f$ in weak $L^p$ sense.
\end{definition}
Any uniformly bounded $L^p$ sequence $f_i$ has a weak limit $f$.	See also \cite{GiMoSa} for a definition of the weak convergence by embedding $Z_i,Z$ to a common metric space $Y$.
\begin{definition}[$L^p$ and $W^{1,p}$ convergence]\label{d:W12convergence} 
Let  $(Z_i,d_i,z_i,\mu_i)\stackrel{d_{GH}}{\longrightarrow}
 (Z,d,z,\mu)$ 
 and let $f_i$ denote Borel functions on $Z_i$. For $p<\infty$ we say $f_i\to f:Z\to \mathbb{R}$ 
 {\it in the $L^p$ sense} if $f_i\to f$ in the weak $L^p$ sense and 
 $$\int_{Z_i}|f_i|^p\to \int_Z |f|^p\, .$$
If $f_i\to f$ in $L^p$ sense and  $$
\int_{Z_i}|\nabla f_i|^p\to \int_{Z}|\nabla f|^p\, 
$$
 we say $f_i\to f$ in the {\it $W^{1,p}$-sense}. 
 \vskip4mm
 
 The following can easily be checked. Thus, the proof will be omitted.
\end{definition}
\begin{proposition}
\label{p:fgconvergence}
${}$
\begin{enumerate}
	\item If $f_i$ converges to a constant $A$ in $L^2$ sense then $f_i^2-A$ converges in $L^1$ to zero. 
	\vskip2mm
	\item If $f_i$ and $g_i$ converge to $f$ and $g$ in $L^2$ sense respectively then $f_ig_i\to fg$ in $L^1$ sense. 
	\vskip2mm
\item Uniform convergence implies $L^p$ convergence for any $0<p<\infty$.	
\end{enumerate}
\end{proposition}
\vskip2mm

The proof of the following Proposition \ref{p:convergenc_function},
on $W^{1,2}$-convergence for functions with $L^2$ Laplacian bound, 
depends on the Mosco convergence of the Cheeger energy; see Theorem 4.4 of \cite{AmHo}.  In our application the limit $X$ is a metric cone and $u_i$ is Lipschitz in which case the proposition is simply proved by using the result in \cite{Ding_heat} without involving RCD notions;  
for related discussion in the  metric measure space contex, see\cite{AmHoTe,Cheeger_DiffLipFun,GiMoSa,MondinoNaber_Rectifiability,ZhaZhu17}.
\begin{proposition}[$W^{1,2}$-convergence]\label{p:convergenc_function}
				Let $(M^n_i,g_i,x_i,\mu_i)\to (X,d,x_\infty,\mu)$ with $\Ric_{M^n_i}\ge -(n-1)$ and $\mu_i=\Vol(B_1(x_i))^{-1}\Vol$.    Let $u_i: B_{R}(x_i)\to \mathbb{R}$ be smooth functions satisfying  for some $C$
$$
\fint_{B_R(x_i)}|u_i|^2+\fint_{B_R}|\nabla u_i|^2+\fint_{B_{R}(x_i)}|\Delta u_i|^2\le C\, . 
$$
 If $u_i$ converge in the $L^2$-sense to a $W^{1,2}$-function $u_\infty:B_R(x_\infty)\to\dR$, then 
	\begin{enumerate}
		\item $u_i\to u_\infty$ in $W^{1,2}$-sense over $B_R(x_\infty)$.
		\vskip1mm
		\item $\Delta u_i\to \Delta u_\infty$ in weak $L^2$ sense. 
		\vskip1mm
		\item If $\sup_{B_R(x_i)}|\nabla u_i|\le L$ for some uniform constant then $u_i\to u_\infty$ in $W^{1,p}$-sense for any $0<p<\infty$.
	\end{enumerate}
\end{proposition}
\begin{proof}(Outline following \cite{AmHo,Ding_heat})
\,\, We will argue
under a uniform Lipschitz assumption; the general case is similar but a bit more technical. 
\vskip2mm

In view of the uniform Lipschitz condition $\sup_{B_R(x_i)}|\nabla u_i|\le L$ it follows by an Ascoli type argument, that we have uniform converge, $u_i\to u_\infty$. Also, $f_i:=  \Delta u_i$ converges in weak $L^2$ sense to some $L^2$ function $f_\infty$. Consider the energy 
$$
E_i(u_i):=  \int_{B_{R}(x_i)} \Big(\frac{1}{2}|\nabla u_i|^2+u_if_i\Big)\, ,
$$
By the lower semicontinuity of the Cheeger energy, we have 
$$
\lim_{i\to\infty}\inf E_i(u_i)\ge E_\infty(u_\infty)\, .
$$
Moreover, using Lemma 10.7 of \cite{Cheeger_DiffLipFun} one can construct {\it some} Lipschitz sequence $v_i$ in $B_R(x_i)$ which converges uniformly to 
$u_\infty$ with $v_i=u_i$ on $\partial B_R(x_i)$ and 
$$
\lim_{i\to\infty}\sup \int_{B_{R}(x_i)} |\nabla v_i|^2\le \int_{B_{R}(x_\infty)} |\nabla u_\infty|^2\, .
$$ 
From the fact that
 $u_i$ minimizes the energy $E_i$, it then follows  that $E_i(u_i)\to E_\infty(u_\infty)$, which gives us the $W^{1,2}$-convergence.  The weak convergence of $\Delta u_i$ also follows from the energy convergence.  That is, we need to show $\Delta u_\infty=f_\infty$, i.e. for any Lipschitz $h$ with $h=u_\infty$ on $\partial B_R(x_\infty)$ that $E_\infty(h)\ge E_\infty(u_\infty)$. Assume there exists Lipschitz $h_\infty$ with $h_\infty=u_\infty$ on $\partial B_R(x_\infty)$ and $\epsilon_0>0$ such that $E_\infty(h_\infty)<E_\infty(u_\infty)-\epsilon_0$. Then we can construct by using Lemma 10.7 of \cite{Cheeger_DiffLipFun} a sequence of Lipschitz function $h_i$ in $B_R(x_i)$ with $h_i=u_i$ on $\partial B_R(x_i)$ and 
 $$
\lim_{i\to\infty}\sup \int_{B_{R}(x_i)} |\nabla h_i|^2\le \int_{B_{R}(x_\infty)} |\nabla h_\infty|^2\, .
$$ 
Since $E_i(u_i)\to E_\infty(u_\infty)$, this implies for large $i$ that 
\begin{align}
E_i(h_i)<E_i(u_i)-\epsilon_0/2,
\end{align}
which contradicts with the fact that $u_i$ minimizing $E_i(u_i)$ on all Lipschitz functions with the same boundary condition. Hence we conclude that $\Delta u_\infty=f_\infty$. This completes the (outline) proof of Proposition \ref{p:convergenc_function}.
\end{proof}

The following lemma was proved in \cite{Ding_harmonic} for metric cone limits and in \cite{AmHo} for the general case. 
\begin{lemma}\label{l:construct_converging_sequence}
	Let $(M_i,g_i,x_i,\mu_i)\to (X,d,x_\infty,\mu)$ with $\Ric_{M^n_i}\ge -(n-1)$ and 
	$\mu_i=\Vol(B_1(x_i))^{-1}\Vol(\, \cdot\, )$. Let $f,F\in L^2(X)$ have compact support,
and assume $\Delta F=f$ and $f$ is Lipschitz.  Then for any $R>0$ there exists solutions $\Delta F_i=f_i$ on $B_{R_i}(x_i)$ with $R_i\to R$ such that $F_i$ and $f_i$ converge uniformly to $F$ and $f$ in any compact subset of $B_{R}(x_\infty)$ respectively .  
\end{lemma}
\begin{proof}(Outline)
From a  generalized Bochner formula in \cite{EKS15} and standard elliptic estimate it follows that $F$ is Lipschitz. By Lemma 10.7 of \cite{Cheeger_DiffLipFun} one can construct Lipschitz functions $\hat{F}_i, f_i$ on $B_R(x_i)$ converging uniformly and in $W^{1,2}$-sense to $F$ and $f$ respectively. 
\vskip2mm

 For  $\epsilon>0$, define ${F}_{i,\epsilon}$ on $B_R(x_i)$ such that 
 $\Delta {F}_{i,\epsilon}=f_i$ on $B_{R-\epsilon}(x_i)$ with ${F}_{i,\epsilon}=\hat{F}_i$ on $\partial B_{R-\epsilon}(x_i)$, and ${F}_{i,\epsilon}=\hat{F}_i$ on $B_R\setminus B_{R-\epsilon}(x_i)$. By the definition of $F_{i,\epsilon}$ we have 
\begin{align}\label{e:Fiepsilon}
\int_{B_R(x_i)}|\nabla F_{i,\epsilon}|^2+2f_iF_{i,\epsilon}\le \int_{B_R(x_i)}|\nabla \hat{F}_i|^2+2f_i\hat{F}_i\, .
\end{align}

 Assume the limit of $F_{i,\epsilon}$ is $F_{\infty,\epsilon}$ whose existence is asserted by Proposition \ref{p:convergenc_function}. Moreover $F_{\infty,\epsilon}-F\in W^{1,2}_0(B_R)$. By applying the lower semicontinuity of Cheeger energy to $\int_{B_R(x_i)}|\nabla F_{i,\epsilon}|^2$, we have 
 \begin{align}\label{e:FiFinfty}
 \lim_{i\to\infty} \inf \int_{B_R(x_i)}|\nabla F_{i,\epsilon}|^2+2f_iF_{i,\epsilon}\ge \int_{B_R}|\nabla F_{\infty,\epsilon}|^2+2fF_{\infty, \epsilon}.
 \end{align}
 Since $F_{\infty,\epsilon}-F\in W^{1,2}_0(B_R)$ and $\Delta F=f$ on $B_R$, we have that 
 \begin{align}\label{e:Finftyepsilon}
 \int_{B_R}|\nabla F_{\infty,\epsilon}|^2+2fF_{\infty, \epsilon}\ge \int_{B_R}|\nabla F|^2+2fF.
 \end{align}

On the other hand, noting that \eqref{e:Fiepsilon} and $\int_{B_R(x_i)}|\nabla \hat{F}_i|^2\to \int_{B_R(x_\infty)}|\nabla F|^2\, $ we get
 \begin{align}\label{e:FiF}
  \lim_{i\to\infty} \sup \int_{B_R(x_i)}|\nabla F_{i,\epsilon}|^2+2f_iF_{i,\epsilon}\le \int_{B_R}|\nabla F|^2+2fF.
 \end{align}
Combining \eqref{e:FiFinfty}, \eqref{e:Finftyepsilon} and \eqref{e:FiF}, we have that 
\begin{align}
\int_{B_R}|\nabla F|^2+2fF=\int_{B_R}|\nabla F_{\infty,\epsilon}|^2+2fF_{\infty, \epsilon}.
\end{align}
Since $\Delta F=f$ on $B_R$ and $F-F_{\infty,\epsilon}\in W^{1,2}_0(B_R)$, this implies that that $F_{\infty,\epsilon}=F$.  Let us choose $\epsilon_i\to 0$ and define $F_i=F_{i,\epsilon_i}$. We have proved that $F_i\to F_{\infty,\epsilon}$ pointwisely and in $W^{1,2}$-sense. The uniform convergence in any compact subset of $B_R(x_\infty)$ follows from the standard interior gradient estimate for equation $\Delta F_{i}=f_i$ in $B_{R-\epsilon_i}$. Hence the proof of Lemma \ref{l:construct_converging_sequence} is completed.
 \end{proof}

\subsection{The Laplacian on a Metric Cone}
\label{ss: laplacian_metric_cone}

Next, we will recall the existence of the Laplacian operator on metric cones with suitable cross-sections. 
The explicit formulas, \eqref{e;har_cone}, \eqref{e:rsp}, in Theorem \ref{t:spectrum_metric_cone},
\cite{Cheeger79,Cheeger83}, were initially derived in the context of 
spaces with iterated conical singularities \cite{Cheeger79,Cheeger83}. This context is in certain ways more special 
and in other ways 
more general than at of the present subsection.
Theorem \ref{t:spectrum_metric_cone} below
 was originally proved for metric measure spaces satisfying a doubling condition and Poincar\'e inequality in 
\cite{Cheeger_DiffLipFun} and
 for Gromov-Hausdorff limits of smooth manifolds in 
\cite{ChCoIII} and \cite{Ding_heat}. 
It is also understood in the context of
 RCD spaces \cite{Ambrosio_Ricci,Ambrosio_Calculus_Ricci}. As usual,  $\cH^n$ denotes $n$-dimensional
Hausdorff measure.
\begin{theorem}
\label{t:laplacian_metric_cone}  
	Let $(M_i^n,g_i,p_i)\to (X,d_X,p):=  (C(Y),d_X,p)$ satisfy $\Ric_{M^n_i}\ge -\delta_i\to 0$ and $\Vol(B_1(p_i))\ge \rv>0$. Then: 
	\begin{enumerate}
		\item There exists nonpositive, linear, self-adjoint, Laplacian operator $\Delta_X: {\rm Dom}(X)\subset L^2(X)\to L^2(X)$ with ${\rm Dom } \sqrt{-\Delta_X}=W^{1,2}(X)$.
		\vskip2mm
		
		\item For compact supported Lipschitz functions $f$ on $X$,   
		$|\nabla f|=|\Lip f|$.
		$$
		\int_X |\nabla f|^2d\cH^n=\langle \sqrt{-\Delta_X}f,\sqrt{-\Delta_X}f\rangle \, .$$
		\vskip2mm
		
		\item There exists a nonpositive, linear, self-adjoint, Laplacian operator $\Delta_Y: {\rm Dom}(Y)\subset L^2(Y)\to L^2(Y)$  with ${\rm Dom }\sqrt{-\Delta_Y}=W^{1,2}(Y)$.
		\vskip2mm
		
		\item In geodesic polar coordinate $x=(r,y)$, the Laplace operator $\Delta_X$ and 
		$\Delta_Y$ satisfy in the $W^{1,2}(X)$ distribution sense,
\begin{align}
\label{e:lpc}	
	\Delta_X=\frac{\partial^2}{\partial r^2}+\frac{n-1}{r}\frac{\partial}{\partial r}+\frac{1}{r^2}\Delta_Y\, .
\end{align}
		\end{enumerate}
\end{theorem}
Originally, relations (1) and (2) were proved in \cite{ChCoIII} and \cite{Cheeger_DiffLipFun}. Relations, (3) and (4)  were proved in \cite{Ding_heat}.

The Sobolev space $W^{1,2}(X)$ is the closure of Lipschitz functions under a $W^{1,2}$-norm defined in \cite{Cheeger_DiffLipFun}; see the Section 2  of \cite{Cheeger_DiffLipFun} for the precise definition, which ensures that 
the $W^{1,2}$-norm behaves lower semicontinuously under $L^2$ convergence. 
 It then becomes a highly nontrivial theorem
that in actuality, $|\nabla f|=\Lip\, f$, the pointwise Lipschitz constant almost everywhere.
These results were proved in \cite{Cheeger_DiffLipFun} under the assumption that the measure 
is doubling and a Poincar\'e inequality holds.

The cross section $Y$ may itself be viewed as a space which satisfies the  lower Ricci curvature bound $\Ric_Y\ge n-2$ in a generalized sense.The consequences  were initially established directly for cross-sections of limit cones. Subsequently, it was shown that
in the precise formal sense,  $Y$ is an RCD space with $\Ric\geq n-2$; see \cite{Ket,Stu}. 


\begin{theorem}\label{t:spectrum_metric_cone}
	Let $(M_i^n,g_i,p_i)\to (X,d_X,p)=(C(Y),d_X,p)$ satisfy $\Ric_{M^n_i}\ge -\delta_i\to 0$ and $\Vol(B_1(p_i))\ge \rv>0$. Then: 
	\begin{enumerate}
	     \item $(I-\Delta_Y)^{-1}: L^2(Y)\to L^2(Y)$ is a compact operator,
	     \vskip2mm
	     
		\item Laplacian $\Delta_Y$ has discrete spectrum $0=\lambda_0<\lambda_1\le \lambda_2\le \cdots $,
		\vskip2mm
		
		\item If $\phi_i$ be an eigenfunction associated to $\lambda_i$, then $\phi_i$ is Lipschitz. 
		\vskip2mm
		
		\item The following functions are harmonic on $X$:
\begin{align}
\label{e;har_cone}		
		u(r,y)=r^{\alpha_{i}}\phi_i\, ,
\end{align}		
	where
\begin{align}
\label{e:rsp}
\alpha_i=-\frac{n-2}{2}+ \sqrt{\Big(\frac{n-2}{2}\Big)^2+\lambda_i}\, .
\end{align}				 
		\vskip2mm
\item
The harmonic functions $r^{\alpha_{i}}\phi_i$ are Lipschitz on $C(Y)$.
\vskip2mm		
		\item The first nonzero eigenvalue satisfies $\lambda_1\ge n-1$.
	\vskip2mm
		
\end{enumerate}

\end{theorem}

\begin{remark}\label{r:harmonic_metric_cone}
(1) follows from a Neumann Poincar\'e inequality on $Y$, which is induced from the Neumann Poincare inequality on $X$. See a proof of Lemma 4.3 in \cite{Ding_heat} and see also \cite{ChCoIII}.

 (2) follows from (1). See also Theorem 1.8 of \cite{ChCoIII} which only uses Neumann Poincare inequality and volume doubling.

(3) follows from the fact that the harmonic function $u(r,y)=r^{\alpha_i}\phi_i$ is locally Lipschitz, which was proved in \cite{Ding_heat}. 

(4) follows from the statement (4) of Theorem \ref{t:laplacian_metric_cone}.

(5) and (6) were proved in \cite{Ding_heat}.

\end{remark}

\subsection{$\epsilon$-regularity for $2$-sided Ricci bounds}
\label{ss:2sided}
\begin{definition}
\label{d:harmonic_radius}
	For $x\in M^n$ we define the harmonic radius $r_h(x)>0$ to be the maximum over all $r>0$ such that there exists a mapping $\psi=(\psi_1,\cdots,\psi_n):B_r(x)\to \mathbb{R}^n$ with the following properties:
	\begin{itemize}
\item[(1)] $\Delta \psi_i=0$ for $i=1,\cdots, n$.
\vskip1mm
		\item[(2)] $\psi$ is a diffeomorphism onto its image with $B_r(0^n)\subseteq \psi(B_r(x))$, and hence defines a coordinate chart.
		\vskip1mm
		
		\item[(3)] The coordinate metric $g_{ij} = \langle \nabla\psi_i,\nabla \psi_j\rangle$ on $B_{r}(x)$ satisfies $||g_{ij}-\delta_{ij}||_{C^1(B_r(x))}<10^{-n}$.
	\end{itemize}
\end{definition}

The formula for the Ricci tensor in harmonic coordinates can be viewed as a (nonlinear)
elliptic equation on the metric $g_{ij}$ in which the Ricci tensor is the inhomogeneous term.
If the Ricci curvature is bounded, $|\Ric_{M^n}|\leq n-1$, then via elliptic regularity we obtain for
for any $p<\infty$ and $0<\alpha<1$ the a priori estimates 
\begin{align}
||g_{ij}-\delta_{ij}||_{C^{1,\alpha}(B_{r/2}(x))}\, &\le C(n,\alpha)\, ,\\
||g_{ij}-\delta_{ij}||_{W^{2,p}(B_{r/2}(x))}&<C(n,\lambda,p)\, .
\end{align}

The following, $\epsilon$-regularity theorem from \cite{CheegerNaber_Codimensionfour}
can be viewed as a consequence of the proof of the codimension $4$-conjecture proved
in that paper. It  states in quantitative form 
that if a ball has a sufficient amount of symmmetry, then the ball is in the domain of harmonic
coordinate system in which the metric satisfies definite bounds.

\begin{theorem}[\cite{CheegerNaber_Codimensionfour}]\label{t:eps_reg}
There exists $\epsilon(n,\rv)>0$ such that if 
$\Vol(B_1(p))>\rv>0$, $|\Ric_{M^n}|\le n-1$ and  $B_2(p)$ is $(n-3,\epsilon)$-symmetric then $r_h(p)>1$.
\end{theorem}
\vskip5mm


\section{Outline of Proof of Neck Structure Theorem}
\label{s:ouline_proof_neckstructure}


The idea of a neck region is derived primarily from \cite{JiNa_L2} and is
motivated by ideas from \cite{NaVa_Rect_harmonicmap}.  Given the 
Neck Structure Theorem \ref{t:neck_region2},
the proof of the Neck Decomposition of Theorem \ref{t:decomposition2} follows along lines
similar to what was done in a more restricted context in \cite{JiNa_L2}.  More precisely, much of the
proof of the Neck Decomposition of Theorem \ref{t:decomposition2} involves an elaborate and highly nontrivial covering argument.  At a few places, an appeal is made to Theorem \ref{t:neck_region2} to provide sharp estimates; however none of the technology which goes into the proof of Theorem \ref{t:neck_region2} plays a role
in the proof of Theorem \ref{t:decomposition2}.
Thus, the bulk of this paper
 is focused on proving the Neck Structure Theorem \ref{t:neck_region2}. This  requires a completely 
new set of ideas and tools, quite distinct from those of the abovementioned citations.  Our purpose in this section 
is to introduce these new ideas 
in order to sketch a clean picture of the proof of Theorem \ref{t:neck_region2}.  Some of our explanations
will be repeated  subsequent  sections.
\vskip2mm

The proof of Theorem \ref{t:neck_region2} involves a nonlinear induction scheme.  
In it, we will assume a weaker version of the Ahlfor's regularity condition \eqref{e:ar} already
 holds, and use it to prove the stronger version.  Precisely, our main inductive lemma is the following:

\begin{lemma}[Inductive Lemma]
\label{l:inductive2}
Fix $\eta,B>0$ and  $\delta\leq \delta(n,\rv,\eta,B)$.
Let $\cN = B_2(p)\setminus \overline B_{r_x}(\cC)$ denote a $(k,\delta,\eta)$-neck region 
and assume for each $x\in \cC$ and $B_{2r}(x)\subset B_2(p)$ 
\begin{align}\label{e:B_induction}
B^{-1}r^{k}<\mu(B_r(x)) <B\,r^{k}\, .
\end{align}
Then:
\begin{enumerate}
	\item For each $x\in \cC$ and $B_{2r}(x)\subset B_2(p)$ we have the improved estimate $A(n)^{-1}r^{k}<\mu(B_r(x)) <A(n)r^{k}$.
\vskip1mm
	
	\item $\cC_0$ is $k$-rectifiable.
\end{enumerate}
\end{lemma}

Outlining the proof of the inductive lemma will  be the main goal of this section.
In  Section \ref{s:neck_region} we rigorously prove Theorem \ref{t:neck_region2} from the Inductive Lemma.

\subsection{Harmonic Splittings on Neck Regions}

In order to prove the Inductive Lemma \ref{l:inductive2}, 
and hence Theorem \ref{t:neck_region2}, let us first make 
the following observation.  Let 
$\cC'\subseteq B_2\subseteq \dR^k$ be a closed subset 
with $r'_x:\cC'\to \dR$ a radius function s.t. 
$\{\overline B_{\tau_n r'_x}(x)\}$ are all disjoint 
and $B_2\subseteq \cC'_0\cup \bigcup B_{r_x}(\cC'_+)$, 
where as usual $\cC'_0=\{r'_x=0\}$ and $\cC'_+=\{r'_x>0\}$. 
 Consider the packing measure 
$\mu' = \cH^k\cap \cC'_0+\omega_k\sum_{\cC'_+} r_i^k \delta_{x_i}$.
 It is a straigtforward though instructive exercise to see that $\mu'$
 automatically satisfies the Ahflors regularity condition \eqref{e:ar}. For this, one 
notes that Lebesgue measure coincides with Hausdorff measure. 
 Therefore, the strategy to prove the Inductive Lemma \ref{l:inductive2} 
will be to find a mapping $u:\cC\to \dR^k$ which is bi-H\"older onto its image 
and $(1+\epsilon)$-bi-Lipschitz on most of $\cC$.  Then, with $\cC' := u(\cC)$ 
and $r'_x:=r_x$, we can turn the covering $\{B_{r_x}(\cC)\}$ into a well 
behaved covering of $B_2(0^k)\subseteq \dR^k$, and therefore conclude the asserted Ahlfor's regularity.
For further discussion of the role of the Ahlfors regularity of the packing measure, 
see Definition \ref{r:ahlforsreg}.
\vskip2mm

\noindent
\begin{remark}[Digression]  At this point, we will digress in order to explain what 
will not work in the present context. This  will motivate the strategy used here and relate it to that in
 the the previous literature.  In \cite{NaVa_Rect_harmonicmap}
 a quite similar strategy was implemented in order to study the singular sets of nonlinear harmonic maps. 
 In that case, the map $u$ was built by hand, using a Reifenberg construction. Showing that the construction
worked required
new estimates on nonlinear harmonic maps and a new rectifiable 
Reifenberg theorem. It is natural to examine the possibility of implementing a similar approach in 
the present context, by using metric Reifenberg constructions in the spirit of \cite{ChCoI}.  
However, these ideas break down in the context of lower Ricci curvature bounds. 
 Essentially, this is because the underlying space itself is curved. This give rise
to  error terms which are quantitatively worse than those which arise in connection the bi-Lipschitz Reifenberg 
techniques of \cite{NaVa_Rect_harmonicmap}. As a result, those techniques fail in the present context.
  Therefore, of necessity, our construction of the map $u$ will be completely different from that of 
\cite{NaVa_Rect_harmonicmap}. 
Instead of relying on a Reifenberg type construction,
 our mapping $u$ will be more canonical in nature. It will solve an equation.
\end{remark}

To make the above more precise, recall from Definition \ref{d:splitting_function},
the notion a harmonic splitting function.
It follows from \cite{ChCoAlmost}, see Theorem \ref{t:splitting_function}, 
that if $B_{2r}(p)$ is $(k,\delta)$-symmetric then there exists an $\epsilon$-splitting 
map $u:B_r(p)\to \dR^k$.  In particular, splitting maps exist on neck regions.
  In general, splitting functions can degenerate on sets of infinite codimension 2 content.  
In particular, the degeneration set of $u$ may in general be much larger than the center point 
set $\cC$ of a neck region.  However, as we will see, something rather miraculous takes place.  Namely,
if we are on a neck region, then the map $u$ can degenerate in at most a weak sense
 on all of $\cC$, and on most of $\cC$,
cannot degenerate at all.
  Precisely, we will prove the following:

 \begin{theorem}[Harmonic Splittings on Neck Regions]
\label{t:harmonic_splitting_neckregion}
Let $B,\epsilon, \eta>0$ with $\delta'\leq \delta'(n,\rv,B,\epsilon,\eta)$ and 
$\delta\leq \delta(n,\rv,\eta,B,\epsilon)$.  Let $\cN = B_2\setminus \overline B_{r_x}(\cC)$ 
be a $(k,\delta,\eta)$-neck region satisfying \eqref{e:B_induction} with $u:B_4\to \dR^k$ a $\delta'$-splitting map.  
Then there exists $\cC_\epsilon \subset \cC\cap B_{15/8}(p)$ such that:
	\begin{enumerate}
\item $\mu\big(\cC_\epsilon \cap B_{15/8}(p)\big)\ge (1-\epsilon) \mu \big(\cC\cap B_{15/8}(p)\big)$.
	\vskip1mm

\item $u$ is $(1+\epsilon)$-bi-Lipschitz on $\cC_\epsilon$, i.e. $(1+\epsilon)^{-1}\cdot d(x,y)\le |u(x)-u(y)|
\le (1+\epsilon)\cdot d(x,y)$ for any $x,y\in  \cC_\epsilon$.	
\vskip1mm

\item $u$ is $(1+\epsilon)$-bi-H\"older on 
$\cC$, i.e. $(1+\epsilon)^{-1}\cdot d(x,y)^{1-\epsilon}\le |u(x)-u(y)|\le (1+\epsilon)\cdot d(x,y)$ for any $x,y\in  \cC$.	
	\end{enumerate} 
\end{theorem}

Theorem \ref{t:harmonic_splitting_neckregion} is an abreviated version of Proposition \ref{p:bilipschitz_structure}, which is the result which will be proved in the body of the paper.

The proof of Theorem \ref{t:harmonic_splitting_neckregion} relies on
 three main new points: 
The Sharp Splitting Theorem \ref{t:sharp_splitting}, the Geometric Transformation 
Theorem \ref{t:transformation},
and the Nondegeneration Theorem of \ref{t:nondegeneration}.
 The remainder of this outline will discuss these
results and explain how they lead to the proof of Theorem \ref{t:harmonic_splitting_neckregion}.
For convenience, we prestate these results below as Theorems  \ref{t:sharp_splitting2}, 
\ref{t:transformation2}, \ref{t:nondegeneration2}, respectively.
\vskip3mm

\subsection{Sharp cone-splitting}
\label{ss:outline:sharp_splitting}

It is a now classical point 
that if $B_2(p)$ is $(k,\delta)$-symmetric then there exists a harmonic $(k,\epsilon)$-splitting function 
$u:B_1(p)\to \dR^k$; see Theorem \ref{t:splitting_function} of \cite{ChCoAlmost}.  In this paper, it will be crucial  to have a quantitatively sharp understanding of 
how good a splitting exists.  

Recall that in Definition \ref{d:independent_points1} we introduced the notion of a
 collection of a $(k,\alpha)$-independent
set of points $x_0,\ldots,x_k$.
Also, in  Definition \ref{d:kalphadeltapinching} we defined the notion of 
$(k,\alpha,\delta)$-entropy pinching. The following is a slight specialization of Theorem \ref{t:sharp_splitting}.
The crucial point is the
precise linear relationship between the $k$-pinching of a ball and the squared Hessian of a splitting map. 
This is what, under appropriate circumstances,
eventually allows the result to be summed over an arbitrary number of scales, without having 
the resulting estimate blow up uncontrollably.

\begin{theorem}[Sharp Cone-Splitting]\label{t:sharp_splitting2}
Given $\epsilon,\alpha>0$ there exists $\delta(n,\rv,\alpha,\epsilon)$ and $C(n,\rv,\alpha)>0$ with the following properties. Let $(M^n,g,p)$ satisfy $\Ric_{M^n}\geq -(n-1)\delta^2$ and $\Vol(B_{\delta^{-1}}(p))>\rv \delta^{-n}>0$, and let $B_{4\delta^{-1}}(p)$ be $(k,\delta^2)$-symmetric.  Then there exists a $(k,\epsilon)$-splitting map $u:B_2(p)\to \mathbb{R}^k$ satisfying:
 \begin{align}
 \label{e:outline:sharp_splitting}
 \fint_{B_2(p)} \left(|\nabla^2 u|^2+\Ric(\nabla u,\nabla u)+2(n-1)\delta^2 |\nabla u|^2\right)\leq C\, \cE_1^{k}(p)\, .
 \end{align}
\end{theorem}
\vspace{.2cm}

\vskip3mm

\subsection{Sharp Transformation Theorem}\label{ss:outline:sharp_transformation}

The results of the last subsection tell us, in terms of the $k$-pinching, 
how good we can expect the {\it best} splitting map to be on a typical ball. 
The proof of Theorem \ref{t:harmonic_splitting_neckregion} depends on fixing a 
single splitting map on the original ball $B_2(p)$ and seeing how it behaves on smaller balls.  
\vskip2mm

To this end, let us begin by describing a simple situation.  If $u:B_2(0^n)\to \dR^k$
 is a $k$-splitting map in $\dR^n$, then as with any solution of an ellipic pde $u$ has
 {\it pointwise} bounds on the Hessian.  Among other things, this implies that if we restrict
 to some subball $B_r(x)\subseteq B_1$, then $u:B_r(x)\to \dR^k$ is still a splitting map.  
More than that, we know by the smoothness estimates that the matrix 
$T^{-1}:= \langle\nabla u_i,\nabla u_j\rangle(x)$ is close $\delta_{ij}$.  
Thus, if we look at the map $T\circ u$, so that $\langle\nabla Tu_i,\nabla Tu_j\rangle(x)=\delta_{ij}$, then 
we even know that $Tu|_{B_r}$ is becoming an increasing improved splitting map, as $Tu$
 is scale invariantly converging to an isometric linear map at a polynomial rate.  Unfortunately, on spaces with 
only lower Ricci curvature bounds, such statements are highly false. 
 For instance there could be points where $|\nabla u|=0$, so that $u|_{B_r}$ 
is not even a splitting map on small balls, much less a better one, see for instance Example \ref{example:cone_neck}.
\vskip2mm

However, it turns out  that although the restriction of $u:B_r(x)\to \dR^k$ to a sub-ball may not well behaved, 
if we are on a neck region and $x\in \cC$, then $u$ may only degenerate in a very special way.  Namely, 
though $u|_{B_r}$ may not be a splitting map, there is a $k\times k$-matrix $T$ such that 
$Tu = T^j_i u_j:B_r(x)\to \dR^k$ is a splitting map.  What is more important, and as it turns 
out a lot harder to prove, is that after transformation $Tu$ is the {\it best} splitting map on the ball, 
in that it satisfies the estimates from the Sharp Cone-SplittingTheorem \ref{t:sharp_splitting}.

\begin{remark}
\label{r:scale} Note that in comparison to the $\dR^n$ 
case above the matrix $T$ depends on the scale, and not just the point, as $T=T_{x,r}$ may blow up in norm. 
 Additionally, we of course cannot ask that $Tu$ be converging polynomially to a splitting map, as no such 
splitting map may exist at all, only that $Tu$ is the best splitting map which does exist.
\end{remark}

  Our precise result is the following, which is a slight specialization of Theorem \ref{t:transformation}.

\begin{theorem}[Geometric Transformation]
\label{t:transformation2}
Given  $\alpha,\eta,\epsilon>0$, there exists $C=C(n,\rv,\eta,\alpha)$ and
\\
 \hbox{$\gamma=\gamma(n,\rv,\eta)>0$} such that 
 if \,\,$\delta<\delta(n,\rv,\eta)$,  then the following holds:
 \vskip1mm
 
Let $(M^n,g,p)$ satisfy $\Ric_{M^n}\geq -(n-1)\delta^2$, $\Vol(B_1(p))>\rv>0$ 
 and asssume:
\begin{itemize}
\item[i)]
 $u:B_2(p)\to \dR^k$ is a $(k,\delta)$-splitting function.
 \vskip1mm

\item[ii)] For all 
$r\leq s\leq \delta^{-1}$ the ball $B_{s}(p)$ is $(k,\delta^2)$-symmetric but not
$(k+1,\eta)$-symmetric. 
\end{itemize}	
	
Then for all $s\in [r,1]$, 
there exists a $k\times k$-matrix $T=T_{p,s}$ such that:

\begin{enumerate}
\item (Weak Estimate)  $Tu:B_{s}(p)\to \mathbb{R}^k$ is a $(k,\epsilon)$-splitting map.
\vskip1mm
	
\item (Strong Estimate) For $r_j := 2^{-j}$,  
\begin{align}
s^2\fint_{B_s(p)} |\nabla^2 Tu|^2 \leq C\,\sum_{s\le r_{j}\le 1}  \left(\frac{s}{r_j}\right)^{\gamma}\cE^{k}_{r_{j}}(p) + C\delta^2\,s^\gamma\, .
\end{align}
\end{enumerate}
\end{theorem}
\vspace{.2cm}

First, note that the weak estimate above is the main ingredient in 
the proof of the bi-H\"older estimate of Theorem \ref{t:harmonic_splitting_neckregion}.3.  
To see this, obweve that since the transformation exists on every scale one can see it must change slowly.
In particular $|T^{-1}_{2r}\circ T_{r}|\leq 1+\epsilon$ and hence $|T_r|\leq r^{-\epsilon}$.  
On the other hand if one takes $x,y\in \cC$ and considers $r=d(x,y)$, then by the weak estimate we have 
\begin{align}\label{e:outline:distance_distortion}
	|d(T_r u(x),T_r u(y))- d(x,y)| <\epsilon r\, .
\end{align}
By using the norm control on $T_r$ stated above, this exactly gives the bi-H\"older estimate; for 
the details,  see Section \ref{s:nondegeneration}.
\vskip2mm

The proof of the weak estimate itself is given by a contradiction argument in the spirit of 
\cite{CheegerNaber_Codimensionfour}.  Roughly, if the result fails at some $x\in \cC$ then one 
looks for the first radius $s>r_x$ for which it fails.  Blowing up on $B_s(x)$ and passing to the limit,
 $T_{2s}u\to v:\dR^k\times C(Y)\to \dR^k$ one obtains a harmonic map $v$ which is a
 $(k,\epsilon)$-splitting map on $B_2(x)$, but for which by assumption, there is, in particular, 
no transformation so that $Tv$ is a $(k,\epsilon/2)$-splitting map on $B_1(x)$.  By using the 
transformation estimates of the previous paragraph, one gets that $\sup_{B_r(x)}|\nabla v|\leq r^{\epsilon}$ 
for all $r\geq 1$. Therefore, $v$ has slightly faster than linear growth.  Then, using that $X^n$ is not $(k+1,\eta)$-symmetric
 one can prove a Louiville type theorem stating that the map, $v$, must be exactly linear from one of the factors. 
In that case, it is  clear that after a transformation,
$v$ is precisely $(k,0)$-symmetric on $B_1(x)$.  Therefore, we get a contradiction. For the precise details,
see Section \ref{s:Transformation_Theorem}.
\vskip2mm

The proof of the strong estimate in Theorem \ref{t:transformation2} is much more involved.  
One again uses a contradiction argument, but this time to prove a more refined estimate.  
Roughly, if $\ell_r:B_r(x)\to \dR^k$ is the {\it best} $k$-splitting on $B_r$, in the sense of the 
Sharp Splitting of Theorem \ref{t:sharp_splitting2}, then one shows that $r^2\fint_{B_r}|\nabla^2(T_ru- \ell_r)|^2$ is decaying polynomially.  This involves a careful analysis and blow up argument. For the details, see Section
 \ref{s:Transformation_Theorem}.
\vskip2mm

\subsection{Nondegeneration Theorem}\label{ss:outline:nondegeneration}

As was discussed, the weak estimate of Theorem \ref{t:transformation2} is enough to prove the bi-H\"older estimate
 in Theorem \ref{t:harmonic_splitting_neckregion}.3.  Next we want to see that the strong estimate of 
Theorem \ref{t:transformation2} is enough to prove the bi-Lipschitz estimate in 
Theorem \ref{t:harmonic_splitting_neckregion}.3, but this takes a bit more work 
and a couple more points to address.  To accomplish this we want to show that at
 most points $x\in \cC$ we have for any $r_x<r<1$ that $|T_{x,r}-I|<\epsilon$.  At such points, 
$u:B_r(x)\to \dR^k$ remains a $(k,2\epsilon)$-splitting on all scales, even without transformation. 
 By using \eqref{e:outline:distance_distortion} as in the bi-H\"older estimate, we conclude that $u$ is a bi-Lipschitz map 
at such points. It is worth noting that this estimate does {\it not} hold at all points. This can be seen from 
Example \ref{example:cone_neck}. Thus sshowing that it holds at most points is the most we can hope for.
\vskip2mm

To accomplish this we introduce our Nondegeneration Theorem:

\begin{theorem}[Nondegeneration of $k$-Splittings]
\label{t:nondegeneration2}
Given $\epsilon,\eta,\alpha>0$ there exists
$ \delta(n,\rv,\eta,\alpha,\epsilon)>0$ such that the following holds. 
 Let $\delta < \delta(n,\rv,\eta,\alpha,\epsilon)$ with $(M^n,g,p)$ satisfying 
$\Ric_{M^n}\geq -(n-1)\delta^2$, $\Vol(B_1(p))>v>0$.  Let $u:B_{2}(p)\to \dR^k$ 
denote a $(k,\delta)$-splitting function.  Assume for $B_r(x)\subseteq B_1(p)$:

\begin{itemize}
\item[(1)] $B_{\delta^{-1}s}(x)$, is $(k,\delta^2)$-symmetric but $B_s(x)$ is not $(k+1,\eta)$-symmetric
 for all $r\leq s\leq 1$. 
\vskip1mm

 \item[(2)]$\sum_{r_j\geq r} \cE^{k,\alpha}_{r_j}(p) < \delta$ where $r_j = 2^{-j}$.
 \end{itemize}   
Then   $u:B_s(x)\to \dR^k$ is an $\epsilon$-splitting function  for every $r\leq s\leq 1$. 
\end{theorem}
\vspace{.2cm}

The proof of the above comes down to showing that in the context of the assumptions that $|T_{x,r}-I|<\epsilon$. 
 It turns out that the implication 
\begin{align}\label{e:outline:nondegeneration}
	\sum_{r_j>r_x} \cE^{k,\alpha}(x,r_j)<\delta \implies |T_{x,r}-I|<\epsilon \, ,
\end{align}
is fairly subtle.  It is much easier to show 
$\sum_{r_j>r_x} \sqrt{\cE^{k,\alpha}}(x,r_j)<\delta \implies |T_{x,r}-I|<\epsilon $.
 However, the square gain is crucial to our applications.  The square gain depends heavily on the fact that
 $u$ is harmonic, it does not hold for a general (nonharmonic) splitting function.  
The proof of \eqref{e:outline:nondegeneration} depends on the more local estimate:
\begin{align}
|T_{2r}\circ T_r^{-1}-I|< Cr^2\fint_{B_{2r}(x)} |\nabla^2 T_{2r} u|^2 \leq C\,\cE^k(x,2r)\, ,	
\end{align}
where as previously discussed,
the last inequality is the main result of the Transformation Theorem \ref{t:transformation2}.

The first inequality is where the square gain occurs. As above, if the right hand side 
was the $L^2$ norm instead of the {\it squared} $L^2$ norm, the inequality would be much more standard and 
would follow from a typical telescope type argument.  That one can control $T_{2r}\circ T_r^{-1}$ by the {\it squared} 
Hessian is a point very much special to harmonic functions. It is crucial to the whole paper. 

The key point is 
the  following monotonicity formula, which holds for any harmonic function:
\begin{align}
\label{e:above}
\frac{d}{dt} \int \langle\nabla u_i,\nabla u_j\rangle \rho_t(x,dy) = 2\int \langle\nabla^2 u_i,\nabla^2 u_j\rangle + \Ric(\nabla u_i,\nabla u_j) \rho_t(x,dy)\, .
\end{align}

Roughly speaking, since $\rho_t$ is a probability measure which is essentially supported on $B_{\sqrt{t}}(x)$,
 the left hand side of \ref{e:above}
 measures the rate of change of $T_{ij}(x,\sqrt{t})$. Given that we want
 to use this when $|\nabla u|\approx 1$ and $|\nabla^2 u|\approx 0$, we find
that the left hand side  behaves as a linear quantity, while the right hand side behaves as 
a quadratic quantity. This leads to a a crucial gain in the analysis.  For additional
 details, see Section \ref{s:nondegeneration}. 
\vskip2mm

\subsection{Completing the proof of Theorem \ref{t:harmonic_splitting_neckregion}}

Completing the outline proof of Theorem \ref{t:harmonic_splitting_neckregion} requires
  a brief discussion of why the assumption of \eqref{e:outline:nondegeneration} holds for most $x\in \cC$. 
 Recall in Theorem \ref{t:harmonic_splitting_neckregion} we are assuming the Ahlfors regularity of 
\eqref{e:B_induction}, so that a key point is that one has the estimate
\begin{align}\label{e:outline:k_pinching_bound}
	\cE^k(x,r) \leq C r^{-k}\int_{B_{r}(x)} |\cW_{2r}(y)-\cW_r(y)|\, d\mu\, .
\end{align}
This is because the Ahlfors regularity allows us to find $k+1$independently spaced points, $x_0,\ldots,x_k$, 
for which the quantities $|\cW_{2r}(x)-\cW_r(x_j)|$ are all roughly the same as the average drop 
$r^{-k}\int_{B_{r}(x)} |\cW_{2r}(y)-\cW_r(y)|\, d\mu$. 

 Now recall from the definition of a 
neck region that for {\it every} $x\in \cC$ we have 
\begin{align}
	|\cW_1(x)-\cW_{r_x}(x)| = \sum_{r_j=2^{-j}\geq r_x} |\cW_{2r_j}(x)-\cW_{r_j}(x)| < \delta\, . 
\end{align}
Then one has
\begin{align}
\int_{B_1(p)} \Big(\sum_{r_j=2^{-j}>r_x} r_j^{-k}\int_{B_{r_j}(x)}
|\cW_{2r_j}(y)-\cW_{r_j}(y)|\Big)\,
	d\mu(y) &\leq C\int_{B_1}\int_{B_1
}\sum_{r_j=2^{-j}>r_x} r_j^{-k}|\cW_{2r_j}(y)-\cW_{r_j}(y)\ 1_{B_{r_j}(x)}(y)\,d\mu(y)
 d\mu(x)\notag\\
&= C\int_{B_1}\int_{B_1}\sum_{r_j=2^{-j}>r_x} r_j^{-k}|\cW_{2r_j}-\cW_{r_j}|(y) 1_{B_{r_j}(y)}(x)\,d\mu[x] d\mu[y]\notag\\
&\leq C\int_{B_1}\sum_{r_j=2^{-j}>r_y} |\cW_{2r_j}(y)-\cW_{r_j}(y)|\cdot\frac{\mu(B_{r_j}(y))}{r_j^k}\,d\mu(y)\notag\\
&\leq C\, B\int_{B_1}\sum_{r_j=2^{-j}>r_y} |\cW_{2r_j}(y)-\cW_{r_j}(y)|\,d\mu(y)\notag\\
&\leq C\, B\int_{B_1} |\cW_{2}(y)-\cW_{r_y}(y)|\,d\mu(y)\notag\\
&\leq C\,B^2\,\delta\, .
\end{align}
It follows from this and \eqref{e:outline:k_pinching_bound} that {\it most} 
of $\cC$ satisfies the assumption of \eqref{e:outline:nondegeneration}, as claimed.

\vspace{1cm}

\section{Sharp cone-splitting}
\label{s:Sharp_splitting}

This is the first of three sections which constitute the third part of the paper.

In this section, we prove  Theorem \ref{t:sharp_splitting},
the {\it sharp} existence theorem for $\epsilon$-splitting functions.
The hypothesis involves the $k$-pinching of the local
pointwise entropy; see Definition \ref{d:kalphadeltapinching}. 
 The key point, the one which presents the real difficulty in the proof  is the 
{\it linear bound} of the squared $L^2$-norm of the Hessian of the splitting function in terms of the entropy pinching. 
This is what we mean by {\it sharp},.
The main argument is given in the proof of  Proposition \ref{p:sharp_splitting}.
  As explained  in Section \ref{s:ouline_proof_neckstructure}, the form of the bound is crucial  
for the proof of the Transformation Theorem, Theorem \ref{t:transformation}.

\begin{theorem}[Sharp Cone-Splitting]
\label{t:sharp_splitting}
Given $\epsilon,\alpha>0$ there exists $\delta(n,\rv,\alpha,\epsilon)$ and $C(n,\rv,\alpha)>0$ 
with the following properties. Let $(M^n,g,p)$ satisfy $\Ric_{M^n}\geq -(n-1)\delta^2 r^2$ and 
$\Vol(B_{\delta^{-1}}(p))>\rv \delta^{-n}>0$ with $r\leq 1$, and let $B_{4\delta^{-1}}(p)$ be $(k,\delta^2)$-symmetric. 
 Then there exists a $(k,\epsilon)$-splitting map $u:B_2(p)\to \mathbb{R}^k$ satisfying:
 \begin{align}
 \label{e:ssf}
 \fint_{B_2(q)}\left(|\nabla^2 u|^2+\Ric(\nabla u,\nabla u)+2(n-1)\delta^2 r^2 |\nabla u|^2\right) \leq C(n,\rv,\alpha)\cdot \cE_1^{(k,\alpha,\delta)}(q)\, .
 \end{align}
\end{theorem}
\vskip1mm

\begin{remark}[Sharpness] The example of the $2$-dimensional
cone $C(S^1_\beta)$ shows that the estimate in 
 Theorem \ref{t:sharp_splitting} is actually sharp. In 
 checking this, it is useful to employ Theorem \ref{t:cW_local}
 which states the equivalence the volume pinching and
 the entropy pinching.
 \end{remark}

\subsection{Approximation of the squared radius 
 with sharp Hessian estimates}
\label{ss:conical_map}
The first step in the proof of Theorem \ref{t:sharp_splitting} is to  construct a regularization 
$h$ of the squared distance function $d^2$.
 As in \cite{ChCoAlmost},
 the function $h$ will be
taken to satisfy the Poisson equation $\Delta h=2n$. Note that
for the case of metric cones, we have preciesely 
$h=d^2$, $\nabla^2 h=2g$.  We will obtain sharp Hessian bounds for $h$ 
in terms of the entropy drop.  The splitting map $u$ will be constructed explicitly from $h$, this will 
obtain an estimate on $\nabla^2 u$ in terms of the entropy pinching. 
Recall that the $k$-pinching
$\cE^{k,\alpha,\delta}_r(x)$
is defined to be the minimal entropy pinching over all 
\hbox{$(k,\alpha)$}-independent points; 
see  Definition \ref{d:kalphadeltapinching}.


\begin{theorem}[Sharp Poisson regularization of $d^2$]
\label{t:sharp_conical}
	Let $(M^n,g,p)$ satisfy $\Ric\geq -(n-1)\delta^2 $ with $\Vol(B_{\delta^{-1}}(p))>\rv \delta^{-n}>0$. 
 For any $\epsilon>0$ and $B_{ r}(x)\subseteq B_5(p)$ if $\delta\le \delta(n,\rv,\epsilon)$ is such that 
$B_{ r\delta^{-1}}(x)$ is $(0,\delta^2)$-symmetric, then there exists a function $h: B_{2r}(x)\to \mathbb{R}$ such that:
	\vskip2mm
\begin{itemize}
\item[(1)]
$\Delta h=2n$
\vskip3mm
\item[(2)] $\fint_{B_{2 r}(x)}\Big(|\nabla^2h-2g|^2+\Ric(\nabla h,\nabla h)+2(n-1)\delta^2 \cdot |\nabla h|^2\Big) \le C(n,\rv)\cdot 
|\cW_{ r^2}^{\delta}(x) -\cW_{2 r^2}^{\delta}(x)|$\, .
\vskip3mm	
\item[(3)] $\fint_{B_{2 r}(x)}\Big||\nabla h|^2-4h\Big|^2\le C(n,\rv) r^4\cdot 
|\cW_{ r^2}^{\delta}(x)-\cW_{2 r^2}^{\delta}(x)|$.
\vskip3mm
\item[(4)] $|\nabla h|\le C(n,\rv) r$\, .
\vskip3mm
	\item[(5)] $\sup_{B_{2 r}(x)}|h-d_x^2|\le \epsilon  r^2$\, .
\end{itemize}
\end{theorem}
\begin{proof}
Set $t= r^2$, and as usual write $\rho_t(x,y)=(4\pi t)^{-n/2}e^{-f}$ for the heat kernel. 
The Hessian estimates on $h$ will follow from the Hessian estimates on the function $4tf$, which is in turn given by the local $\cW$-entropy pinching.  
Thus, we will begin by deriving the relevant  estimates on $\nabla^2 f$.
\vskip2mm

By Theorem \ref{t:cW_local}, we have 
\begin{align}
		2\int_t^{2t}s\int_{M^n}\Big(|\nabla^2f &-\frac{1}{2s}g|^2
		+\Ric(\nabla f,\nabla f)+2(n-1)\delta^2|\nabla f|^2\Big)\,\varphi\, \rho_s(x,dy)\notag\\
		&{}\notag\\
		&\le |\cW_{t}^{\delta}(x)-\cW_{2t}^{\delta}(x)|\notag\\
		 & := \eta\, .
	\end{align}
Hence, there exists $t\le s\le 2t$ such that 
\begin{align}
		2ts\int_M\Big(|\nabla^2f-\frac{1}{2s}g|^2+\Ric(\nabla f,\nabla f)+2\delta^2(n-1)|\nabla f|^2\Big)\varphi \rho_s(x,dy)\le \eta.
\end{align}
In particular,
\begin{align}
		\int_{M^n}\Big(|\nabla^2(4sf)-2g|^2+\Ric(\nabla (4sf),\nabla (4sf))+2\delta^2(n-1)|\nabla (4sf)|^2\Big)\varphi \rho_s(x,dy)\le 8\eta.
\end{align}
	By using the heat kernel lower bound estimates in Theorem \ref{t:heat_kernel} 	and the volume noncollapsing assumption, we get 
	\begin{align}
		\fint_{B_{5\sqrt{s}}(x)}\Big(|\nabla^2(4sf)-2g|^2+\Ric(\nabla (4sf),\nabla (4sf))+2(n-1)\delta^2|\nabla (4sf)|^2\Big)\le C(n,\rv)\eta.
	\end{align}
	
Set $\tilde{f}:=  4sf$ and consider the $1$-form
	\begin{align}
		\nabla |\nabla \tilde{f}|^2-4\nabla\tilde{f}
		&=2\nabla^2\tilde{f}(\nabla\tilde{f},\, \cdot\, )-4\nabla\tilde{f}\notag\\
		&=2(\nabla^2\tilde{f}-2g)(\nabla\tilde{f},\,\cdot\,)\, .
	\end{align}
	By using the Poincar\'e inequality in Theorem \ref{t:poincare} and the gradient estimate $|\nabla\tilde{f}|^2 =16s^2|\nabla f|^2$
we get \\
\begin{align}
		\fint_{B_{4\sqrt{s}}(x)}\Big||\nabla\tilde{f}|^2-4\tilde{f}-\fint_{B_{4\sqrt{s}}(x)}(|\nabla\tilde{f}|^2-4\tilde{f})\Big|^2\le C(n)s\fint_{B_{4\sqrt{s}}(x)}|\nabla^2(4sf)-2g|^2|\nabla\tilde{f}|^2\le C(n,\rv)\eta {s}^2.
	\end{align}
\vskip1mm
	
Put $$
\hat{f}:=  \tilde{f}+\frac{1}{4}\fint_{B_{4\sqrt{s}}(x)}(|\nabla\tilde{f}|^2-4\tilde{f})\, .
$$
Then
\begin{align}\label{e:estimates_hatf}
		\fint_{B_{4\sqrt{s}}(x)}\Big||\nabla\hat{f}|^2-4\hat{f}\Big|^2\le C(n,\rv)\eta {s}^2\, ,
\end{align}
\begin{align}
\label{e:estimates_half1}
		\fint_{B_{4\sqrt{s}}(x)}\Big(|\nabla^2\hat{f}-2g|^2+\Ric(\nabla \hat{f},\nabla \hat{f})+2\delta^2(n-1)|\nabla \hat{f}|^2\Big)\le C(n,\rv)\eta.
	\end{align}

We now define
 the function $h$ to be the solution of the Poisson equation,
\begin{align}
	\Delta h=&2n\qquad ({\rm on}\,\, B_{4\sqrt{s}}(x))\, ,\notag\\
	h=&\hat{f}\qquad ({\rm on}\,\,\partial B_{4\sqrt{s}}(x))\, .
\end{align}
We will show that $h$ satisfies the desired estimates.\footnote{To be precise, here we might have to change the domain by an arbitrarily small amount such that the boundary is smooth, and in particular satisfies an exterior sphere condition. This does not effect the argument which follows.}
\vskip2mm


By integrating by parts, we have 
\begin{align}
		\fint_{B_{4\sqrt{s}}(x)}|\nabla h-\nabla\hat{f}|^2&=\fint_{B_{4\sqrt{s}}(x)}(h-\hat{f})(\Delta \hat{f}-2n)\notag\\
		&\le \left(\fint_{B_{4\sqrt{s}}(x)}|h-\hat{f}|^2\right)^{1/2}\cdot \left(\fint_{B_{4\sqrt{s}}(x)}|\Delta\hat{f}-2n|^2\right)^{1/2}\notag\\ \label{e:gradien_L2_laplacian}
		&\le C(n)\left(\fint_{B_{4\sqrt{s}}(x)}|h-\hat{f}|^2\right)^{1/2}\cdot \left(\fint_{B_{4\sqrt{s}}(x)}|\nabla^2\hat{f}-2g|^2\right)^{1/2}.
	\end{align}
	
Since $h-\hat{f}=0$ on $\partial B_{4\sqrt{s}}(x)$ we have by the Poincar\'e inequality in Theorem \ref{t:poincare} 
	\begin{align}
	\label{e:gradL2}
		\fint_{B_{4\sqrt{s}}(x)}|h-\hat{f}|^2\le C(n)s\fint_{B_{4\sqrt{s}}(x)}|\nabla h-\nabla\hat{f}|^2.
	\end{align}
By combining \eqref{e:gradL2} with \eqref{e:gradien_L2_laplacian}, we get 
	\begin{align}\label{e:gradL2hatf}
		s\fint_{B_{4\sqrt{s}}(x)}|\nabla h-\nabla\hat{f}|^2+\fint_{B_{4\sqrt{s}}(x)}|h-\hat{f}|^2\le C(n)s^2\fint_{B_{4\sqrt{s}}(x)}|\nabla^2\hat{f}-2g|^2.
	\end{align}
Choose a cutoff function $\phi$ as in \eqref{e:cutoff} with support in $B_{4\sqrt{s}}(x)$ and $\phi:=  1$ in $B_{3\sqrt{s}}(x)$ such that 
$s|\Delta\phi|+s|\nabla \phi|^2\le C(n)$. Then we have 
	\begin{align}
		\fint_{B_{4\sqrt{s}}(x)}|\Delta\phi| |\nabla h-\nabla\hat{f}|^2&\ge \fint_{B_{4\sqrt{s}}(x)}\Delta\phi |\nabla h-\nabla\hat{f}|^2\notag\\
		&\ge \fint_{B_{4\sqrt{s}}(x)}\phi \Delta |\nabla h-\nabla\hat{f}|^2\notag\\
		&\ge \fint_{B_{4\sqrt{s}}(x)}2\phi \left(|\nabla^2h-\nabla^2\hat{f}|^2+\Ric(\nabla(h-\hat{f}),\nabla(h-\hat{f}))+\langle \nabla (\Delta h-\Delta\hat{f}),\nabla (h-\hat{f})\rangle\right).
	\end{align}
	
	Therefore we have 
	
	\begin{align}
		&\fint_{B_{4\sqrt{s}}(x)}\phi\left(|\nabla^2h-\nabla^2\hat{f}|^2+\Ric(\nabla(h-\hat{f}),\nabla(h-\hat{f}))\right)\notag\\
		&\le \frac{1}{2}\fint_{B_{4\sqrt{s}}(x)}|\Delta\phi| |\nabla h-\nabla\hat{f}|^2-\fint_{B_{4\sqrt{s}}(x)}\phi \langle \nabla (\Delta h-\Delta\hat{f}),\nabla (h-\hat{f})\rangle \notag \\
		&\le  C(n)\left(\fint_{B_{4\sqrt{s}}(x)}|\Delta\phi|\cdot  |\nabla h-\nabla\hat{f}|^2+\fint_{B_{4\sqrt{s}}(x)}\phi |\Delta\hat{f}-2n|^2+\fint_{B_{4\sqrt{s}}(x)} |\Delta \hat{f}-2n|\cdot  |\nabla h-\nabla\hat{f}|\cdot |\nabla \phi|\right)\notag\\
		&\le C(n)\left(\fint_{B_{4\sqrt{s}}(x)}(|\Delta\phi|+|\nabla \phi|^2)\cdot  |\nabla h-\nabla\hat{f}|^2+\fint_{B_{4\sqrt{s}}(x)}(\phi+1) |\Delta\hat{f}-2n|^2\right)\notag\\
		&\le C(n)\fint_{B_{4\sqrt{s}}(x)}|\nabla^2\hat{f}-2g|^2\,,
	\end{align}
	where we have used \eqref{e:gradL2hatf} in the last inequality and $|\Delta \hat f-2n|^2\le n|\nabla^2\hat f-2g|^2$.
By using \eqref{e:gradL2hatf} and $s\le 10^2$,
we have 
	\begin{align}
		\fint_{B_{3\sqrt{s}}(x)}&\left(|\nabla^2h-\nabla^2\hat{f}|^2+\Ric(\nabla(h-\hat{f}),\nabla(h-\hat{f}))+2\delta^2(n-1)|\nabla h-\nabla\hat{f}|^2\right)\notag\\
		&\le \fint_{B_{4\sqrt{s}}(x)}\phi\left(|\nabla^2h-\nabla^2\hat{f}|^2+\Ric(\nabla(h-\hat{f}),\nabla(h-\hat{f}))+2\delta^2(n-1)|\nabla h-\nabla\hat{f}|^2\right)\notag\\
		&\le  C(n)\fint_{B_{4\sqrt{s}}(x)}|\nabla^2\hat{f}-2g|^2\, .
	\end{align}
	
	
	By the Schwarz inequality and \eqref{e:estimates_hatf} we get 
\begin{align}
\fint_{B_{3\sqrt{s}}(x)}|\nabla^2h-2g|^2+
		&\Ric(\nabla h,\nabla h)+2\delta^2(n-1)|\nabla h|^2\notag\\
		\le &2\fint_{B_{3\sqrt{s}}(x)}\left(|\nabla^2h-\nabla^2\hat{f}|^2+\Ric(\nabla(h-\hat{f}),\nabla(h-\hat{f}))+2\delta^2(n-1)|\nabla h-\nabla \hat{f}|^2\right)\notag\\
		&{}\,\,\,\,\,\,\,\,\,\,\,\,+2\fint_{B_{3\sqrt{s}}(x)}\left(|\nabla^2\hat{f}-2g|^2+\Ric(\nabla \hat{f},\nabla \hat{f})+2\delta^2(n-1)|\nabla \hat{f}|^2\right)\notag\\
		\le & C(n,\rv)\eta.
	\end{align}
This gives (1). 
\vskip3mm
	
 To see (3), note that $2t\ge s\ge t= r^2$ and 
	\begin{align}
		\fint_{B_{4\sqrt{s}}(x)}|\nabla h|^2\le 2\sup_{B_{4{\sqrt{s}}(x)} }|\nabla {f}|^2+ 2\fint_{B_{4\sqrt{s}}(x)}|\nabla h-\nabla \hat{f}|^2\le C(n,\rv)s\, .
	\end{align}
From this, the gradient estimate on $h$ in (3)  follows by a standard Moser iteration argument.
\vskip3mm

To prove (2),
since $2t\ge s\ge t= r^2$,
we can use estimates for $\hat{f}$ in \eqref{e:estimates_hatf} and the gradient estimates $|\nabla h|+|\nabla \hat{f}|\le C(n,\rv)\sqrt{s}$ in $B_{3\sqrt{s}}(x)$. 
By Cauchy-Schwarz
	 inequality, we have 
	\begin{align}
		\fint_{B_{3\sqrt{s}}(x)}\Big||\nabla h|^2-4h\Big|^2&\le C(n)\cdot\left(\fint_{B_{3\sqrt{s}}(x)}\Big||\nabla \hat{f}|^2-4\hat{f}\Big|^2+\fint_{B_{3\sqrt{s}}(x)}\Big||\nabla h|^2-|\nabla\hat{f}|^2\Big|^2+\fint_{B_{3\sqrt{s}}(x)}|h-\hat{f}|^2\right)\notag\\
		&\le C(n,\rv)\eta s^2+\fint_{B_{3\sqrt{s}}(x)}|\nabla h-\nabla\hat{f}|^2\cdot |\nabla h+\nabla \hat{f}|^2\notag\\
		&\le C(n,\rv)\eta s^2+C(n,\rv)s\fint_{B_{3\sqrt{s}}(x)}|\nabla h-\nabla\hat{f}|^2\notag\\
		&\le C(n,\rv)\eta s^2.
	\end{align}
This gives (3). 
\vskip2mm


To complete the proof, we need to show (4). By the gradient estimates for $h,\hat{f},d_x^2$ and \eqref{e:gradL2hatf}, it suffices to prove $|\hat{f}-d_{x}^2|\le \epsilon  r^2$.  
For this, we will use heat kernel convergence and $W^{1,2}$-convergence of functions as in Proposition \ref{p:heat_kernel_convergence} and argue by contradiction.
	
	By scaling, we can assume $ r=1$ and $\Ric\ge -(n-1)\delta^2$. Therefore assume there exists $\epsilon_0>0$ and $(M_i,g_i,x_i)$ satisfying $\Vol(B_{\delta_i^{-1}}(x_i))\ge \rv \delta_i^{-n}$, $\Ric\ge -(n-1)\delta_i^2\to 0$, the ball $B_{\delta_i^{-1}}(x_i)$ is $(0,\delta_i^2)$-symmetric, but that the function $\hat{f}_i$ defined as above satisfies 
	\begin{align}
		\sup_{B_{10}(x_i)}|\hat{f}_i-d_{x_i}^2|\ge \epsilon_0.
	\end{align}
	
Now let  $i\to \infty$.  By Gromov's compactness theorem. there exists  metric cone, $(C(Y),d,x_\infty)$, which is the Gromov-Hausdorff-limit of $(M_i,g_i,x_i)$. By the heat kernel convergence in Proposition \ref{p:heat_kernel_convergence} and Remark \ref{r:heat_kernel_cone}, the heat kernel $\rho_1(x_i,\cdot)=(4\pi)^{-n/2}e^{-f_i}$ converges to heat kernel $\rho_1(x_\infty, \cdot)=(4\pi)^{-n/2}e^{-d_{x_\infty}^2/4+A_X}$ uniformly on any compact subset, where 
$$
A_X=\log \frac{\Vol(S^{n-1})}{\Vol(X)}\, .
$$
\vskip2mm

From  the heat kernel Laplacian estimate and $W^{1,2}$-convergence in Proposition \ref{p:convergenc_function},  it follows that the sequence $f_i$ converges to $f_\infty=d_{x_\infty}^2/4-A_x$ uniformly and in the local 
$W^{1,2}$ sense. Thus, $\hat{f}_i$ converges uniformly to a limit function 
$$\tilde{f}_\infty := 4f_\infty+{4}\fint_{B_{10}(x_\infty)}(|\nabla f_\infty|^2-f_\infty)=d_{x_\infty}^2\, .
$$ 
 Since $d_{x_i}^2$ converges to $d_{x_\infty}^2$ uniformly in any compact set, while $\sup_{B_{10}(x_i)}|\hat{f}_i-d_{x_i}^2|\ge \epsilon_0$ for any $i$, this gives a contradiction.  This completes the proof of 
 Theorem \ref{t:sharp_conical}
\vskip2mm

\end{proof}

\subsection{The $k$-splitting associated to $k$ independent points}
In this subsection, we construct a $k$-splitting map from $k$-independent points which satisfy the estimates of
the splitting Theorem \ref{t:sharp_splitting}.
By rescaling and taking the infimum over all $(k,\alpha)$-independent sets of
points we will see that the proof of Theorem \ref{t:sharp_splitting} is a direct consequence of the following main result, Proposition
\ref{p:sharp_splitting}. The proof of this proposition will occupy the remainder of this section.
\begin{proposition}
\label{p:sharp_splitting}
	Let $(M^n,g,p)$ satisfy $\Ric_{M^n}\ge -(n-1)\delta^2$ with $\Vol(B_{\delta^{-1}}(p))\ge \rv \delta^{-n}>0$.  For $\epsilon,\alpha>0$ and $\delta\le \delta(n,\rv,\alpha,\epsilon)$ let $\{x_0,x_1,\cdots,x_k\}\subset B_{r}(x)\subset B_{10}(p)$ be $(k,\alpha)$-independent points with 
$$
\cE_{r}^{k,\delta}(\{x_i\}):=  \sum_{i=0}^k|\cW_{r^2/2}^\delta(x_i)-\cW_{2r^2}^\delta(x_i)<\delta\, .
$$  

Then there exists $C(n,\rv,\alpha)>0$ and a $(k,\epsilon)$-splitting map $u=(u^1,\cdots,u^k): B_{8r}(x)\to \mathbb{R}^k$ such that:
	\vskip3mm
	\begin{itemize}
	\item[(1)] $r^2\fint_{B_{8r}(x)}\Big(|\nabla^2u|^2+\Ric(\nabla u,\nabla u)+2(n-1)\delta^2|\nabla u|^2\Big) \le C\cdot \cE_{r}^k(\{x_i\}) .$
	\vskip3mm
	\item[(2)] $\fint_{B_{8r}(x)}\Big|\langle \nabla u_i,\nabla u_j\rangle -\delta_{ij}\Big|^2\le C\cdot \cE_{r}^k(\{x_i\}).$
	\vskip3mm
	\item[(3)] $|\nabla u|\le 1+\epsilon$ .
	\end{itemize}
	\end{proposition}
	
	Before giving the proof of Proposition \ref{p:sharp_splitting}, let us look at the following example to see how to
 build a splitting function from squared distance functions to distict vertices of a cone. 
	
\begin{example} (Cone-splitting; the case  $\R^2=C(S^1)$))
\label{e:R2}
Cone-splitting, and more specifically 
the relation between squared distance functions $h_\pm$ from distinct cone points and a splitting functions $u$, is perhaps most easily illustrated
 by  the case of $\R^2$.  Denote the square of the distance functions from the points $(\pm1,0)$
by
 $
 h_{\pm}(x,y)=(x\pm1)^2 +y^2\, .
 $
 Then the linear function (splitting function) $u=x$ satisfies
 \begin{align}
 \label{e:csr2}
 u=\frac{1}{4}\cdot (h_+-h_-)\, .
 \end{align} 

The expression in \eqref{e:csr2}, which builds a linear splitting function from squared distance functions, will reappear in the general quantitative context in \eqref{e:udef} of Proposition \ref{p:sharp_splitting}.
\end{example}

\begin{proof}[Proof of Proposition \ref{p:sharp_splitting} ]
It follows from Theorem \ref{t:sharp_conical} that for any $\epsilon>0$ there exists $\delta_0(n,\rv,\epsilon)$ such that for $\delta\le \delta_0$ and each point $x_i$, there is a map $h_i: B_{20r}(x_i)\to \mathbb{R}$ such that:
\begin{itemize}
		\item[(1)] $\Delta h_i=2n$.
		\vskip2mm
\item[(2)] $\fint_{B_{20r}(x_i)}\Big(|\nabla^2h_i-2g|^2+\Ric(\nabla h_i,\nabla h_i)+2(n-1)\delta^2|\nabla h_i|^2\Big) \le C(n,\rv)|\cW_{r^2/2}^\delta(x_i)
-\cW_{2r^2}^\delta|(x_i)|\, .$
	\vskip2mm

	\item[(3)] $\fint_{B_{20r}(x_i)}\Big||\nabla h_i|^2-4h_i\Big|^2\le C(n,\rv)r^4|\cW_{r^2/2}^\delta(x_i)-\cW_{2r^2}^\delta(x_i)|\, .$

	\vskip3mm
	\item[(4)] $|\nabla h_i|\le C(n,\rv)r$ on $B_{20r}(x_i)$\, .
	\vskip2mm
	\item[(5)] $\sup_{B_{20r}(x_i)}|h_i-d_{x_i}^2|\le \epsilon r^2$.
\end{itemize}
\vskip3mm

Note that  $B_{10r}(x)\subset B_{20r}(x_i)$.
 We define the $k$-splitting functions as in Example \ref{e:R2}:

\begin{align}
\label{e:udef}
\tilde{u}^i:=  \frac{h_i-h_0-d(x_0,x_i)^2}{2d(x_i,x_0)}\, .
\end{align}

Note that by (1), we have $\Delta \tilde{u}^i=0$ in $B_{10r}(x)$.  By the Cauchy-Schwartz inequality we also have:
\vskip4mm

\begin{itemize}
	\item[(a)] $r^2\fint_{B_{10r}(x)}\Big(|\nabla^2\tilde{u}^i|^2+\Ric(\nabla \tilde{u}^i,\nabla \tilde{u}^i)+2(n-1)|\nabla \tilde{u}^i|^2\Big)\le C(n,\rv,\alpha)\cdot \cE^k_{r}(\{x_i\}),$
	\vskip3mm
	\item[(b)] $\sup_{B_{10r}(x)}|\nabla \tilde{u}^i|\le C(n,\rv,\alpha)$,
	\vskip3mm
	\item[(c)] $\sup_{B_{10r}(x)}\left|\tilde{u}^i-\frac{d_{x_i}^2-d_{x_0}^2-d(x_0,x_i)^2}{2d(x_i,x_0)}\right |\le C(\alpha,n)\cdot \epsilon r$.
\end{itemize}
\vskip4mm


\begin{lemma}
\label{l:claim11} There exists $k\times k$ lower triangle matrix $A$ with $|A|\le C(n,\rv,\alpha)$ such that 
$u:=  (u^1,\cdots,u^k):= A(\tilde{u}^1,\cdots,\tilde{u}^k)$ satisfies 
$$
\fint_{B_{8r}(x)}\langle \nabla u^i,\nabla u^j\rangle=\delta^{ij}\, .
$$
\end{lemma}
Assume provisionally that the lemma holds. Then since $|A|\le C(n,\rv,\alpha)$,
by using estimate (a) and (b) and the Poincar\'e inequality,
 it follows easily  that $u$ satisfies $(1)$, $(2)$ of the proposition.
 Estimate $(3)$ follows exactly as in \cite{CheegerNaber_Codimensionfour}.
 Therefore, to complete the proof of Theorem \ref{t:sharp_splitting} it suffices to prove 
 Lemma \ref{l:claim11}.
\vskip2mm
\begin{proof}(Of Lemma \ref{l:claim11}) We will argue by contradiction.
 By rescaling $B_r(x)$ to $B_1(x)$ we can take $r=1$. Then we can assume there exists $(M^n_\beta,g_\beta,x_\beta)$ and $(k,\alpha)$-independent points $\{x_{\beta,0},x_{\beta,1},\cdots, x_{\beta,k}\}\subset B_1(x_\beta)$ with $\delta_\beta\to 0$ as $\beta\to \infty$. 
Also, for each $\beta$, we can construct regularized maps $h_{\beta,i}$ and harmonic functions $\tilde{u}_{\beta,i}$ on $B_{10}(x_\beta)$ as in \eqref{e:udef},
satisfying (a) (b) and (c) with $\epsilon_\beta\to 0$ in (c).

Now, assume  that either there exists no $k\times k$ lower triangle matrix $A_\beta$ such that 
$$
u_\beta:=A_\beta(\tilde{u}_{\beta}^{1},\cdots,\tilde{u}_{\beta}^{k})
$$
 satisfies 
 $$
 \fint_{B_{8}(x_\beta)}\langle \nabla u_{\beta}^{i},\nabla u_{\beta}^{j}\rangle =\delta^{ij}
 $$ 
 or if there exist such matrices $A_\beta$, then $|A_\beta|\to \infty$.

By the definition of independent points and the Cone-Splitting Theorem \ref{t:cone_splitting}, there is a Gromov-Hausdorff limit space of the sequence $M^n_\beta$ which is a metric  cone $\mathbb{R}^k\times C(X)$. Moreover, the set  $\{x_{\beta,i}\}$ converges to a set of $(k,\alpha)$ independent points 
$\{x_{\infty,i}\}\subset \mathbb{R}^k\times \{v\}$ where $v$ is the vertex of 
$C(X)$.


By (c) above, the $\tilde{u}_{\beta}^{i}$ converge to the linear functions
$$
\tilde{u}_{\infty}^{i}=\frac{d_{x_{\infty,i}}^2-d_{x_{\infty,0}}^2-d(x_{\infty,0},x_{\infty,i})^2}{2d(x_{\infty,i},x_{\infty,0})}\, .
$$ 
Recall that  $\{x_{\infty,i}\}\subset \mathbb{R}^k\times \{v\}$ is a collection of
  $(k,\alpha)$-independent points. Thus, the linear functions $\{\tilde{u}_{\infty}^{i},i=1,\cdots, k\}$ form a basis of linear space of $\mathbb{R}^k$, and there exists lower triangular matrix, 
  $A_\infty$ with $|A_\infty|\le C(n,\rv,\alpha)$, such that 
$$
u_\infty:= (u_{\infty}^{1},\cdots,u_{\infty}^{k}):= A_\infty(\tilde{u}_{\infty}^{1},\cdots,\tilde{u}_{\infty}^{k}) 
$$ 
satisfies 
$$
\fint_{B_{8}(x_\infty)}\langle \nabla u_{\infty}^{i},\nabla u_{\infty}^{j}\rangle=\delta^{ij}\, .
$$


 For $\beta$ large enough,  the $W^{1,2}$-convergence of harmonic functions stated in Proposition \ref{p:convergenc_function} implies for some $A_\beta$ with $|A_\beta-A_\infty|\to0$ that the set of functions  
 $$
 \hat{u}_\beta:=(\hat{u}_{\beta}^{1},\cdots,\hat{u}_{\beta}^{k})
 := A_\beta(\tilde{u}_{\beta}^{1},\cdots,\tilde{u}_{\beta}^{k})
 $$  
 is orthogonal in the integral sense over $B_{8}(x_\beta)$ as Lemma \ref{l:claim11}, which leads to a contradiction. This completes the proof of the Lemma \ref{l:claim11}.
 \end{proof}
 
 As we have seen, this also completes the
   Proposition \ref{p:sharp_splitting} and hence, of Theorem \ref{t:sharp_splitting}. 
 \end{proof}
   \vskip4mm


\section{The Geometric Transformation Theorem}
\label{s:Transformation_Theorem}
We begin with some motivation.
The results of the last section specify  how good the
{\it best} splitting will be on a sufficiently entropy pinched ball.
 However, in the eventual application to the Neck Structure Theorem \ref{t:neck_region2}, 
the proof  will depend on {\it fixing a single splitting map on the original ball $B_2(p)$} and
 showing that it behaves sufficiently well on most smaller balls.  Recall from subsection
 \ref{ss:outline:sharp_transformation} the following motivating example:
\vskip3mm

\begin{example}
\label{ex:Rn}   If $u:B_2(0^n)\to \dR^k$ is a $k$-splitting map in $\dR^n$, then as with any
 solution of an ellipic pde $u$ has {\it pointwise} bounds on the Hessian.  Among other things,
 this implies that if we restrict to some sub-ball $B_r(x)\subseteq B_1(p)$, 
then $u:B_r(x)\to \dR^k$ is still a splitting map. 
 In fact, if $T^{-1}:= \langle \nabla u^i,\nabla u^j\rangle$, then $T\circ u|_{B_r(x)}$ 
becomes an increasingly good splitting map, since $u$ converges to a  linear map at a polynomial rate.  
\end{example}

As discussed in subsection \ref{ss:outline:sharp_transformation} 
we wish to generalize, to the extent possible, the above example to 
spaces with lower Ricci curvature bounds.  In this case we cannot hope
 that $u|_{B_r(x)}$ remains a splitting map, but we will see that we can choose a matrix $T=T(x,r)$ such that $T\circ u|_{B_r(x)}$ is comparable to the best splitting maps on $B_r(x)$, in the sense of the last section.

\subsection{Statement of the Geometric Transformation Theorem}
\label{ss:transformation_statement}
The result referred to in \cite{CheegerNaber_Codimensionfour} as the
Transformation Theorem is 
a key component of the proof of the Codimension $4$ Conjecture in that paper.
 For given $\epsilon>0$, the statement of the Transformation Theorem  \ref{t:transformation} concerns 
a $(n-2,\delta(\epsilon))$-splitting functions  $u:B_1(x)\to \dR^{n-2}$.
Namely, although the restriction
of $u$ to a smaller ball
a ball $B_r(x)$ might not be an $(n-2,\epsilon)$-splitting function, the Transformation Theorem \ref{t:transformation}
gives conditions guaranteeing the existence of a 
suitable upper triangular $(n-2)\times (n-2)$ matrix $T$, 
 with positive diagonal entries, such that $Tu:B_r(x)\to \dR^{n-2}$
 is an $(n-2,\epsilon)$-splitting function, $Tu:B_r(x)\to \dR^{n-2}$.  The conditions of \cite{CheegerNaber_Codimensionfour} are special to the codimension two stratum.
\vskip1mm

In the present long and somewhat technical section, we  show that  with a {\it different hypothesis}, the conclusion of the
Transformation Theorem of \cite{CheegerNaber_Codimensionfour}
can be sharpened.  To begin, our conditions and criteria will hold for any any stratum.  We will see that beginning with a $(k,\epsilon)$-splitting function $u$, that so long as $B_r(x)$ remains $k$-symmetric that there is a transformed function $Tu$ satisfing the Hessian estimates given by Theorem \ref{t:sharp_splitting}.  More precisely, the main result of this section is the following.\footnote{As usual, $\cE_s^k(p)=\cE_s^{k,\alpha,\delta}(p)$
 denotes the $k$-pinching; see Definition \ref{d:kalphadeltapinching}.}

\begin{theorem}[Geometric Transformation]
\label{t:transformation}
Given  $\alpha,\eta,\epsilon, \delta>0$, there exists $C=C(n,\rv,\eta,\alpha)$ and
\\
 \hbox{$\gamma=\gamma(n,\rv,\eta)>0\, , \delta(n,\rv,\eta)>0$}, such that 
 if \,\,$\delta<\delta(n,\rv,\eta)$,  then the following holds:
 \vskip1mm
 
Let $(M^n,g,p)$ satisfy $\Ric_{M^n}\geq -(n-1)\delta^2$, $\Vol(B_1(p))>\rv>0$ 
 and asssume:
\begin{itemize}
\item[i)]
 $u:B_2(p)\to \dR^k$ is a $(k,\delta)$-splitting function.
 
\item[ii)] For all 
$r\leq s\leq \delta^{-1}$ the ball $B_{s}(p)$ is $(k,\delta^2)$-symmetric but not
$(k+1,\eta)$-symmetric. 
\end{itemize}	
\vskip2mm
	
 Then for all $s\in [r,1]$, 
there exists a $k\times k$-matrix $T=T_{p,s}$ such that:

\begin{itemize}
\item[(1)]  $Tu:B_{s}(p)\to \mathbb{R}^k$ is a $(k,\epsilon)$-splitting map.
		
\item[(2)] For $r_j := 2^{-j}$,  
\begin{align}
s^2\fint_{B_s(p)}\left(|\nabla^2 Tu|^2+\Ric(\nabla Tu,\nabla Tu)+2\delta^2(n-1) 
|\nabla Tu|^2\right) \leq C\cdot \sum_{s\le r_{j}\le 1}  \left(\frac{s}{r_j}\right)^{\gamma}\cE^{k}_{r_{j}}(p) + C\,\delta^2\, .
\end{align}
\end{itemize}
\end{theorem}
\vskip3mm

\subsection{Outline of the proof}
\label{ss:indication}
Essentially, the  $k\times k$ matrix $T_{p,s}$ is gotten by orthonormalization of the
gradients on the given scale. The points is to  show  that this procedure produces an $\epsilon$-splitting as in (1) and, what is more challenging, that this splitting satisfies the sharp estimate in (2). 
\vskip2mm

The statements (1) and (2) are both proved by contradiction arguments, in which the assumption
that the conclusion fails is shown to lead to a statement about metric cones which, by explicit computation, can be shown to be false. Before giving a brief description of the
arguments, we mention that there
three technical points which will have to be taken into account
 when the arguments are carried out.
 \vskip2mm
 
 The first technical point concerns our being able to 
pass the assumption that the conclusion of the theorem fails to a statement about limit cones.
For this, we use  
$W^{1,2}$ convergence result in Proposition \ref{p:convergenc_function}.
\vskip2mm

The second technical point pertains to checking that the resulting statement which concerns limit cones
is actually false. At the formal level, one can do explicit calculations which employ separation of variables.
  If we could assume that the cross
section $Y^{n-1}$ of the limit cone $C(Y^{n-1})$ were smooth, then the
 relevant computations would be straightforward exercises, using that $Y^{n-1}$ 
would be a space with $\Ric_Y\geq (n-2)$.  Making rigorous this understanding in 
our context will take a fair amount of technical work. 
\vskip2mm

The third technical point concerns the fact that
the Hessian of  the norm squared of a harmonic function
need not be well defined on a limit cone.  However,  the Laplacian {\it is} well defined  and
 it will suffice to state all of our estimates on limit cones
which correspond to Hessian estimates on manifolds in weak form
using Bochner's formula \eqref{e:bochner}. 
\vskip2mm

The proof of conclusion (1) of Theorem \ref{t:transformation}
is similar to the proof of 
 the Transformation Theorem of
 \cite{CheegerNaber_Codimensionfour}. It is a  quantitative implementation of the following fact. 
 On
a metric cone $\R^k\times C(Z)$, which is a definite amount away from splitting off $\R^{k+1}$,
a harmonic function which is assumed to grow  only slightly more 
than linearly, must in fact, be linear and have linear growth. The reason is the following.
\vskip2mm

Consider a  metric cone $C(Y)$ which is a Gromov-Hausdorff
limit with the lower bound on Ricci going to zero.  The Laplacian $\Delta_Y$ on 
the cross-section has a discrete spectrum and
an orthonormal basis of eigenfunctions, $\phi_i$,  with corresponding eigenvalues
$-\Delta_Y \phi_i=\lambda_i\phi_i$, satisfying $\lambda_0=0$, 
$\lambda_i\geq n-1$, for $i\geq 1$.  It follows from i) of Theorem \ref{t:transformation},
that in our case, 
$n-1=\lambda_1=\cdots =\lambda_k$ and from ii) that there exists $\tau(n,\rv,\eta)>0$
such that
$\lambda_{k+1}\geq (n-1)+\tau(n,\rv,\eta)$ i.e. there is a definite gap in the spectrum.
\vskip2mm

Let $u(r,x)$ denote a harmonic function on $C(Y)=\R^k\times C(Z)$
which is normalized to satisfy
 $u(0,x)=0$. Then we have an expansion in terms of homogeneous 
harmonic functions, see also Proposition \ref{p:harmonic_cone}
\begin{align}
\label{e:har_sum}
 u(r,x)=\sum_i c_i r^{a_i}\cdot \phi_i(y)\, ,
 \end{align}

 \noindent
where $a_0=0$, $a_1=\cdots=a_k= 1$ and $a_{k+1}\geq 1+\theta(n,\rv,\eta)$,
for some $\theta(n,\rv,\eta)>0$. From this, it follows that in our case, a harmonic
function on $\R^k\times C(Z)$ which grows only a bit more than linearly is actually linear. This is the fact about cones which enables us to prove (1)
via a contradiction argument.
\vskip2mm

Although the idea behind the proof of (2) is equally simple, finding the right
sharp quantitative estimate on cones is more subtle. Intuitively,
in this case we consider the behavior of an arbitrary harmonic function,
 $u(r,z)$ as in \eqref{e:har_sum}.  Note that as $r\to 0$ that the nonlinear terms in the expansion decay faster than the linear terms.  Thus $u(r,z)$ becomes increasingly linear  as $r\to 0$.  The technically precise version of this decay estimate on limit cones is given in \eqref{e:decay_psi_cone} of 
 Proposition \ref{p:decay_cone}.
The corresponding decay
estimate for manifolds is given in \eqref{e:decay_hessian_transformation} of
 Proposition \ref{p:decay_hessian_transformation}.
The latter contains a pinching term on the right-hand side which compensates for
the fact that we are not dealing with an actual metric cone.  In particular, the best we can hope for in general is that $u|_{B_r(x)}$ looks increasingly like the 'best' linear function on $B_r(x)$, in the sense of Theorem \ref{t:sharp_splitting}.
\vskip2mm


The remainder of this section can be viewed as consisting of five  parts.
\vskip2mm

In Subsection \ref{ss:eigenvalue_limit_cone},
 we derive the results on cones needed to prove (1) of Theorem \ref{t:transformation}.
 The section is essentially  technical and routine.
\vskip2mm

In Subsection \ref{ss:transformation_lemma} we give the proof of (1).
\vskip2mm

In Subsection \ref{ss:Reifenberg}, which is brief, we digress to prove a Reifenberg theorem for which
the map is canonical. The proof is an easy consequence of the arguments Subsections
\ref{ss:eigenvalue_limit_cone}, \ref{ss:transformation_lemma}.
While this result is not used elsewhere in the paper, it is of some interest
in and of itself. Moreover, it provides  motivation for the arguments in Section 
\ref{s:decomposition}
used to prove rectifiability of the strata $S^k$ for all $k$.
\vskip2mm

In Subsection \ref{ss:hessian_decay_limit_cone} we state and prove the key
the decay estimate for cones,
\eqref{e:decay_psi_cone}
 of Proposition \ref{p:decay_cone}.
\vskip2mm

In Subsection \ref{ss:hessian_decay_manifold}
we prove the corresponding decay estimate \eqref{e:decay_hessian_transformation}
of Proposition \ref{p:decay_hessian_transformation}.
 \vskip2mm

 In Subsections \ref{ss:decay_hessian_transformation}, \ref{ss:Proving_Transformation},
  we complete  the proof of (2) of 
Theorem \ref{t:transformation}.

\subsection{Harmonic functions and eigenvalue estimates on
 limit cones}
\label{ss:eigenvalue_limit_cone}
 Let 
 $$
 (M^n_i,g_i,x_i)\stackrel{d_{GH}}{\longrightarrow}(C(Y),d,x_\infty)
 =\mathbb{R}^k\times C(Z)$$ with  $\Ric_{M^n_i}\ge -\delta_i\to 0$ and $\Vol(B_1(x_i))\ge \rv>0$.  
 As in subsection \ref{ss: laplacian_metric_cone} there exist Laplacians $\Delta_{C(Y)}$, 
 $\Delta_Y$ on the cone and its cross section.  
 
 The cross-section $Y$ is an $RCD$ space with positive Ricci curvature $\Ric_{Y^{n-1}}\geq n-2$; see \cite{Ket,Stu}.  
Results on the spectrum of $Y$ which hold for smooth spaces with this lower Ricci bound
 are know to hold fo $Y^{n-1}$. Spectral results which hold for smooth spaces with this lower Ricci bound are known to hold for $Y^{n-1}$.
 In particular, the spectrum of $\Delta_Y$ is discrete, see Section \ref{ss: laplacian_metric_cone} and Theorem \ref{t:spectrum_metric_cone}. Denote the spectrum of $\Delta_Y$ by
 $0=\lambda_0<\lambda_1\le \lambda_2\le \cdots$ with an associated orthonormal basis of eigenfunctions $\phi_0=\frac{1}{\sqrt{\Vol(Y)}}, \phi_1,\phi_2,\cdots \, ,$.    
 \vskip2mm
 
 The  main results of this subsection are  Proposition \ref{p:eigenvalue_cone}
 and Proposition \ref{p:harmonic_cone}.\footnote{In the case in which the cross-section is smooth, 
the second of these results is derived from the first; see \cite{Cheeger79}. Under our assumptions,
 it will be convenient to derive the first from the second.}

\begin{proposition}[Eigenvalue estimates on limit cone]
\label{p:eigenvalue_cone}
Let $(M^n_i,g_i,x_i)\togh (X,d,x_\infty)=(C(Y),d,x_\infty)=(\mathbb{R}^k\times C(Z),d,x_\infty)$ satisfy 
$\Ric_{M^n_i}\ge-\delta_i\to 0$ and $\Vol(B_1(x_i))\ge \rv>0$.  If $B_{1}(x_i)$ is not $(k+1,\eta)$-symmetric, then 
\begin{align}
\label{e:eigenvalue_cone}
	0=\lambda_0<n-1=\lambda_1=\cdots=\lambda_k<\lambda_{k+1}\le \lambda_{k+2}\le \cdots\, .
\end{align}
Moreover, there exists $\tau(\eta,n,\rv)>0$ such that
$$
\lambda_{k+1}>\lambda_k+\tau\,  .
$$
\end{proposition}
\vspace{.2cm}

\begin{proposition}
\label{p:harmonic_cone}
	Let $(M_i^n,g_i,x_i)\togh (C(Y),d,x_\infty)$ 
satisfy $\Ric_{M^n_i}\ge-\delta_i\to 0$ and $\Vol(B_1(x_i))\ge \rv>0$.  
Then $r^{\alpha_i}\phi_i$ is harmonic where $\lambda_i=\alpha_i(n-2+\alpha_i)$
	 with $\alpha_i\ge 0$ and $-\Delta_Y\phi_i=\lambda_i\phi_i$. 
 Moreover, any harmonic function $u(r,Y): B_{1}(x_\infty)\to \mathbb{R}$ 
satisfies\footnote{We mention that on any RCD space, which includes this context, a harmonic function is automatically Lipschitz, see for instance \cite{Ambrosio_Calculus_Ricci}, \cite{Ambrosio_Ricci}} $$u=\sum_{i=0}^\infty b_ir^{\alpha_i}\phi_i\, ,
 $$
  where the convergence is in $W^{1,2}$ sense on $B_{1}(x_\infty)$.  
\end{proposition}
\begin{proof} (Of Proposition \ref{p:harmonic_cone}) By Theorem \ref{t:laplacian_metric_cone} and Remark \ref{r:harmonic_metric_cone} the function $r^{\alpha_i}\phi_i$ is harmonic.	So let us begin the proof of the second part of the proposition. Since $u$ is bounded, in particular $u\in L^2(\partial B_{1}(x_\infty))$. Then we have the expansion in $L^2(Y)$,
\begin{align}
\label{e:u}
	 u(1,y)=\sum_{i=0}^\infty b_i \phi_i\, ,
\end{align} 
where $b_i =\int_Y \phi_i(y) u(1,y)$.  Define 
$$
v_k(r,y)=\sum_{i=0}^k b_ir^{\alpha_i}\phi_i(y)
$$ in $B_{1}(x_\infty)$. 
Denote the limit
 in $L^2(B_{1}(x_\infty))$ as $k\to \infty$ of $v_k$ by $v$.
Since the operator $\Delta$ is closed, it follows that $v$ is also harmonic. We have
\begin{align}
\label{e:v}
v=\sum_{i=0}^\infty b_ir^{\alpha_i}\phi_i(y)\in L^2(B_{1})\, .
\end{align} 
	
To finish the whole proof, we require to show	the above convergence \eqref{e:v} is in $W^{1,2}$-sense and $v=u$. 
Denote the annulus 	$A_{r,1}(x_\infty)$ by
$$
A_{r,1}(x_\infty) :=B_{1}(x_\infty)\setminus \bar B_r(x_\infty)\, .
$$
The following Lemma \ref{l:limit} will suffice to complete the proof of
Proposition \ref{p:harmonic_cone}.  The argument will be given
 after the proof of the 
Lemma \ref{l:limit} is completed.
\vskip3mm

\begin{lemma}
\label{l:limit}
 With the notation above, we have
 	$v_k\to v$ in $W^{1,2}(B_{1}(x_\infty))$ and
\begin{equation}
\label{e:limit}
\lim_{r\to 1}\, \frac{1}{(1-r)^2}\int_{A_{r,1}(x_\infty)}|v-u(1,y)|^2= 0\, .
\end{equation}
\end{lemma}
\begin{proof}	
 To begin with, we will show that $v_k$ converges to $ v$ 
 in $W^{1,2}(B_{1}(x_\infty))$.
\vskip2mm

From the fact that $u$ is Lipschitz  it follows that
  $\sum_i b_i^2\lambda_i<\infty$. Namely,
\begin{align}
	\int_Y |\nabla u(1,y)-\nabla v_k(1,y)|^2&=\int_Y |\nabla u(1,y)|^2+\int_Y |\nabla v_k(1,y)|^2-2\int_Y\langle \nabla u(1,y),\nabla v_k(1,y)\rangle\notag \\
	&=\int_Y |\nabla u(1,y)|^2+\int_Y |\nabla v_k(1,y)|^2+2\int_Y u(1,y)\Delta v_k(1,y)\notag\\
	&=\int_Y |\nabla u(1,y)|^2-\sum_{i=0}^k \lambda_i b_i^2\, .
\end{align}
This  implies 
$$
\sum_{i=0}^\infty \lambda_i b_i^2\le \int_Y|\nabla u(1,y)|^2\,.
$$ Since $\alpha_i^2\le \lambda_i$, we have 
\begin{align}
	\int_{B_{1}}|\nabla v_k(r,y)|^2=\sum_{i=0}^k b_i^2\int_{B_{1}}|\nabla (r^{\alpha_i}\phi_i)|^2=\sum_{i=0}^kb_i^2 \frac{\lambda_i+\alpha_i^2}{n+2\alpha_i-2} \le C(n)\sum_{i=0}^k\lambda_i b_i^2.
\end{align}
By applying the same computation to $v_k-v_\ell$ we get
\begin{align}
	\int_{B_{1}}|\nabla (v_k-v_\ell)(r,y)|^2 \le C(n)\sum_{i=\ell}^k\lambda_i b_i^2\, .
\end{align}
Therefore,  $\{v_k\}$ is a Cauchy sequence in $W^{1,2}(B_{1}(x_\infty))$. Since the space
$W^{1,2}(B_{1}(x_\infty)$ is complete, it follows that 
$v_k\to v\in W^{1,2}(B_{1}(x_\infty))$.	This
concludes the proof of the first part of Lemma \ref{l:limit}.
\vskip4mm

To complete the proof of Lemma \ref{l:limit},  we need to prove \eqref{e:limit}. 
We will begin by  showing that $v=\sum_{i=0}^\infty b_ir^{\alpha_i}\phi_i(y)$ is also in $L^2(\partial B_r(x_\infty))$ for $0<r<1$. 
\vskip2mm

Since $\sum_i b_i^2<\infty$ 
it follows  that $\{v_k(r,y)\}$ 
is a Cauchy sequence in $L^2(\partial B_r(x_\infty))$. Denote the limit of $v_k(r,y)$ in $L^2(\partial B_r(x_\infty))$ by $\tilde{v}(r,y)$.  By Fubini's theorem we have that 
\begin{align}
\int_{B_1(x_\infty)}|v_k(r,y)-\tilde{v}(r,y)|^2d\cH^n&=\int_0^1\int_{Y}|v_k(r,y)-\tilde{v}(r,y)|^2dYdr=\int_0^1\lim_{j\to\infty}\int_Y|v_k(r,y)-v_j(r,y)|^2dYdr\\
&\le \int_0^1\sum_{\ell=k}^\infty b_\ell^2 dr\le \sum_{\ell=k}^\infty b_\ell^2\, .
\end{align}
Letting $k\to\infty$ we get
$$
\int_{B_{1}(x_\infty)}|v(r,y)-\tilde{v}(r,y)|^2d\cH^n =0\,.
$$
In particular, this implies that $v(r,y)=\sum_{i=0}^\infty b_ir^{\alpha_i}\phi_i(y)$ is in $L^2(\partial B_r)$. By Fubini's theorem we can compute 
	\begin{align}
		\frac{1}{(1-r)^2}\int_{A_{r,1}(x_\infty)}|v(s,y)-u(1,y)|^2=\frac{\int_r^{1} s^{n-1}\int_Y|v(s,y)-u(1,y)|^2}{(1-r)^2}.
	\end{align}
	Since $v(r,y)$ is the $L^2$ limit of $v_k(r,y)$ on $\partial B_r$, we have
	\begin{align} 
		\frac{1}{(1-r)^2}\int_{A_{r,1}(x_\infty)}|v(r,y)-u(1,y)|^2&=\frac{\int_r^{1} s^{n-1}\sum_{i=0}^\infty b_i^2 |s^{\alpha_i}-1|^2ds}{(1-r)^2}\notag \\
		&\le  \sum_{i=1}^\infty b_i^2\cdot \frac{\left(1-{r}^{\alpha_i}\right)^2}{1-{r}}\notag\\
		&\le C(n)\sum_{i=1}^\infty (\alpha_i+1)^2 b_i^2\left(1-{r}\right),
	\end{align}
	where we have used $\alpha_i\ge 0$ to deduce $\left(1-{r}^{\alpha_i}\right)\le (\alpha_i+1) \left(1-{r}\right)$.
	Since $\sum_{i=0}^\infty b_i^2\alpha_i^2<\infty$, this implies $\lim_{r\to 1}\frac{1}{(1-r)^2}\int_{A_{r,1}(x_\infty)}|v-u(1,y)|^2= 0$. This completes the proof of  Lemma \ref{l:limit}.
\end{proof}
\vskip4mm

Now, we can complete  the proof of the Proposition \ref{p:harmonic_cone}. For $u$, $v$,
as in \eqref{e:u}, \eqref{e:v}, 
 it suffices to prove that $u=v$.  Since $u$ is Lipschitz Lemma \ref{l:limit} implies
  $$
  \lim_{r\to 0}\frac{1}{(1-r)^2}\int_{A_{r,1}(x_\infty)}|v-u|^2\to 0\, .
  $$ 
  Choose a cutoff function $\varphi_r$ with support in $B_{1}(x_\infty)$ and $\varphi_r:=  1$ in $B_r(x_\infty)$ such that $|\nabla \varphi_r|\le C(n)/(1-r)$.  Then
\begin{align}
	\int_{B_{1}(x_\infty)}|\nabla (u-v)|^2\varphi_r^2=-2\int_{B_{1}}(u-v)\varphi_r\langle \nabla (u-v),\nabla \varphi_r\rangle \le \frac{1}{2}\int_{B_{1}}|\nabla (u-v)|^2\varphi_r^2+C(n)\int_{A_{r,1}}|u-v|^2 |\nabla \varphi_r|^2.
\end{align}
By letting $r\to 1$ we have that $\int_{B_{1}}|\nabla(u-v)|^2=0$, which implies $u-v$ is a constant. Moreover, since $\frac{1}{(1-r)^2}\int_{A_{r,1}(x_\infty)}|v-u|^2\to 0$ as $r\to 1$, we have that $u=v$. This completes the proof of Proposition \ref{p:harmonic_cone}. 
\end{proof}
\vskip3mm


Next we will prove Proposition \ref{p:eigenvalue_cone}. As explained at the beginning of this section,
the  idea is the following:
\vskip2mm

By Theorem \ref{t:spectrum_metric_cone}, we know that $\lambda_0=0$ and $\lambda_1\ge n-1$.
Consider a harmonic function $u=r^{\alpha_i}\phi_i$ on $X=C(Y)=\mathbb{R}^k\times C(Z)$, where $\phi_i$ is an eigenfunction of $Y$ with eigenvalue $\lambda_i$ and $\alpha_i\ge 0$ satisfies $\lambda_i=\alpha_i(n-2+\alpha_i)$. If $u$ is a linear function on the $\mathbb{R}^k$ component then we have $\alpha_i=1$, or equivalently $\lambda_i=n-1$.  Therefore, we have $\lambda_0=0$ and $\lambda_1=\lambda_2=\cdots \lambda_k=n-1$. To finish the proof, we will need to show that
$$
\lambda_{k+1}>n-1+\tau(n,\rv,\eta)>n-1\, .
$$
\vskip2mm

Consider the harmonic function $u=r^{\alpha_{k+1}}\phi_{k+1}$ where $-\Delta \phi_{k+1}=\lambda_{k+1}\phi_{k+1}$. We will use a contradiction argument to show that $\alpha_{k+1}>1+\alpha(n,\rv,\eta)>1$, which implies $\lambda_{k+1}>\lambda_k+\tau(n,\rv,\eta)$.  The moral is simple, we will show if $\alpha_{k+1}$ is close to $1$ then $u=r^{\alpha_{k+1}}\phi_{k+1}$ 
is close to a new linear splitting function. Then
if $\alpha_{k+1}$ is too close to $1$,  this contradicts the assumption that $B_{1}(x_i)$ is not $(k+1,\eta)$-symmetric.

\begin{proof} (Of Proposition \ref{p:eigenvalue_cone})
 Let $u=r^{\alpha_m}\phi_m$ denote a harmonic function in $X$. By scaling invariance, we have 
\begin{align}\label{e:alpham_ts}
t^{1-\alpha_m}\int |\nabla u|^2\rho_t(x_\infty,dx)
= s^{1-\alpha_m}\int |\nabla u|^2\rho_s(x_\infty,dx)\, .
\end{align}
\vskip2mm

\noindent
\begin{lemma}
\label{sl:1}
 For any $\epsilon>0$, we will show that if $|\alpha_m-1|\le \delta\le \delta(n,\rv,\eta,\epsilon)$ then there exit harmonic functions $u_i: B_1(x_i)\subset M_i\to \mathbb{R}$ converging in $W^{1,2}$-sense (see Definition \ref{d:W12convergence}) to $u$ with 
\begin{align}
\fint_{B_{1}(x_i)}|\nabla^2u_i|^2 \le \epsilon  ~~\mbox{  for $i\ge i(n,\rv,\epsilon,\eta)$.}
\end{align}
\end{lemma}
\vskip2mm

 Let us assume Lemma \ref{sl:1} and finish the proof of Proposition \ref{p:eigenvalue_cone}.
\vskip2mm

Note that $\alpha_1=\alpha_2=\cdots=\alpha_k=1<\alpha_{k+1}$, 
for any $\epsilon>0$, if $|\alpha_{k+1}-1|\le \delta\le \delta(n,\rv,\eta,\epsilon)$.
Then by Sublemma \ref{sl:1}, we have  
$k+1$ harmonic functions $u_{i}^1,u_{i}^2,\cdots, u_{i}^{k+1}: B_1(x_i)\to \mathbb{R}$ 
which converge in $W^{1,2}$-sense to $u^1=x^1,u^2=x^2,\cdots,u^k=x^k, u^{k+1}=
r^{\alpha_{k+1}}\phi_{k+1}$. Here $x^1,\cdots, x^k$ are the coordinate functions of 
$\mathbb{R}^k\subset \mathbb{R}^k\times C(Z)$, and $u^1,u^2,\cdots, u^{k+1}$ is perpendicular to 
each other with respect to the inner product 
\begin{align}
(u,v):= \fint_{B_1(x_\infty)}\langle \nabla u,\nabla v\rangle, \mbox{  for all $u,v\in W^{1,2}(B_1)$.}
\end{align}
Moreover, since $(u^\ell,u^\ell)=1$ for $\ell=1,\cdots, k$ and $|(u^{k+1},u^{k+1})-1|\le C(n)\delta$,
and  $u_i^\ell \to u^\ell$ in $W^{1,2}$-sense, we have for $i\ge i(n,\rv,\epsilon,\eta)$ that 
\begin{align}
\left|\fint_{B_1(x_i)}\langle \nabla u_i^a,\nabla u_i^b\rangle -\delta^{ab}\right|\le \epsilon, \mbox{ for all $a,b=1,\cdots, k+1$.}
\end{align}
On the other hand, Lemma \ref{sl:1} we have the Hessian estimate
\begin{align}
\fint_{B_1(x_i)}|\nabla^2u_i^a|^2\le \epsilon.
\end{align}
It follows that the map 
\begin{align}
u:= (u_i^1,\cdots, u_i^{k+1}): B_1(x_i)\to \mathbb{R}^{k+1}
\end{align}
is a $(k+1,C(n)\epsilon)$-splitting map.  If $\epsilon\le \epsilon(n,\rv,\eta)$,  this asserts that 
$\alpha_{k+1}\ge 1+\alpha(n,\rv,\eta)>1$, which
 contradicts the assumption that $B_1(x_i)$
 is not $(k+1,\eta)$-symmetric. This concludes the proof of Proposition \ref{p:eigenvalue_cone} 
under the assumption that Lemma \ref{sl:1} holds. 
\vskip2mm

\begin{proof}[of Lemma \ref{sl:1}] 
The proof of Lemma \ref{sl:1}
 requires the result on heat kernel
convergence of Proposition \ref{p:heat_kernel_convergence} 
and harmonic function convergence in Lemma \ref{l:construct_converging_sequence}.

By \eqref{e:alpham_ts}, we have for $|\alpha_m-1|\le \delta$ that 
  \begin{align}
 \label{e:rho_1_nabla_u}
\Big|\int|\nabla u|^2(x)\, \rho_1(x_\infty,x)dx-\int|\nabla u|^2(x)\, \rho_2(x_\infty,x)dx\Big|\le  
2\delta\int |\nabla u|^2(x)\, \rho_2(x_\infty,x)dx,
\end{align}
\vskip2mm

Since $u$ has polynomial growth and the heat kernel 
$\rho_t$ is exponentially decaying as in Theorem \ref{t:heat_kernel}, we can choose a 
big $R\ge R(n,\rv,\delta)$ and a cutoff function $\varphi=\psi(r^2)$ ,with support in $B_R$ 
and $\varphi:=  1$ in $B_{R/2}$, such that $|\nabla\varphi|^2+|\Delta \varphi|\le C(n)R^{-2}$ and 
 \begin{align}\notag
		\Big|\int_{B_{R}(x_\infty)}\varphi^2|\nabla u|^2(x)\rho_1(x_\infty,x)dx
-\int_{B_{R}(x_\infty)}\varphi^2|\nabla u|^2(x)\rho_2(x_\infty,x)dx\Big| 
		&\le  4\delta\int_{B_{R}(x_\infty)}\varphi^2 |\nabla u|^2(x)\rho_2(x_\infty,x)dx\,.
\end{align}
\vskip2mm

 By using Lemma \ref{l:construct_converging_sequence} and Proposition \ref{p:convergenc_function},
 we can now construct a sequence of harmonic functions, $u_i:B_{R}(x_i)\to \dR$, which converge in 
$W^{1,2}$-sense to $u:B_R(x_\infty)\to \dR$. 
 \vskip2mm

 Let  $\varphi=\psi(h_i)$ with $\Delta h_i=2n$ where $h_i$ approximates $d^2$ 
pointwise (see \cite{Cheeger01}). \footnote{Note we are not just applying Theorem \ref{t:cutoff} to produce a  cutoff function but are specifying its construction. This is to ensure $\psi(h_i)$ converge to  the cutoff function $\psi(r^2)$ in the limit space, which will be important. }
 By the heat kernel convergence in Proposition \ref{p:heat_kernel_convergence} we have for 
 $i\ge i(n,\rv,\delta)$,
 \begin{align}
  \label{e:he1}
		\Big|\int_{B_{R}(x_i)}\varphi^2|\nabla u_i|^2(x)\rho_1(x_i,x)dx
-\int_{B_{R}(x_i)}\varphi^2|\nabla u_i|^2(x)\rho_2(x_i,x)dx\Big|&\le 
 8\delta\int_{B_{R}(x_i)} \varphi^2|\nabla u_i|^2(x)\rho_2(x_i,x)dx\, .
\end{align}
\vskip2mm

   Since $\rho_t$ is the heat kernel this gives 
 \begin{align}
 	\Big|\int_1^2\int \Delta (\varphi^2|\nabla u_i|^2)\rho_t(x_i,dx)\Big|\le  
8\delta\int_{B_{R}(x_i)} \varphi^2|\nabla u_i|^2(x)\rho_2(x_i,dx).
 \end{align}
From  Bochner's formula and the Schwartz inequality, we get 
\begin{align}
	\int_1^2\int \varphi^2 |\nabla ^2u_i|^2 \rho_t(x_i,dx)dt
\le C(\delta+R^{-2}) \int_{B_R(x_i)}|\nabla u_i|^2(x)\rho_2(x_i,dx).
\end{align}
\vskip2mm

By the mean value theorem and the heat kernel lower bound estimate Theorem \ref{t:heat_kernel}, we have 
\begin{align}
	\fint_{B_1(x_i)}|\nabla^2u_i|^2\le C(n,\rv)(\delta+R^{-2})\int_{B_R(x_i)} |\nabla u_i|^2\rho_2(x_i,dx)
\le  C (\delta+R^{-2}).
\end{align}
By fixing $R=R(\epsilon,n,\rv)$  we conclude that 
$$
\fint_{B_{1}(x_i)}|\nabla^2u_i|^2\le \epsilon\, .
$$ This completes the proof of the Lemma \ref{sl:1}
\end{proof}

Hence, the proof of Proposition \ref{p:eigenvalue_cone} is complete as well. 
\end{proof}
\vskip3mm


\subsection{Part \text{(1)} of the Geometric Transformation Theorem}
\label{ss:transformation_lemma}
In this subsection, we will prove 
 estimate $(1)$ of Theorem \ref{t:transformation}.  We will see in subsequent subsections that the transformation $T$ safitifies the vastly improved estimate $(2)$.  As we explained, the proof of $(1)$ is based on a contradiction argument: 

\begin{proposition}[Transformation]
\label{p:transfor_prop}
Let $(M^n,g,x)$ satisfy $\Ric_{M^n}\ge -(n-1)\delta^2$ and $\Vol(B_1(x))\ge \rv>0$. 
Let $\epsilon >0$ and $\delta\le \delta(n,\rv,\eta,\epsilon)$.  Assume:

\begin{enumerate}
\item $B_{s}(x)$ is $(k,\delta^2)$-symmetric but not $(k+1,\eta)$ symmetric for each scale $r_0\le s\le 1$.

\item
 $u: B_2(x)\to \mathbb{R}^k$ is a $\delta$-splitting map and let $\epsilon>0$.  
\end{enumerate}

Then for each scale $r_0\le s\le 1$ there exists $k\times k$ lower triangle matrix $T_{s}$ such that:
\begin{enumerate}
\item $T_su: B_s(x)\to \mathbb{R}^k$ is a $(k,\epsilon)$-splitting map on $B_s(x)$.

\item  $\fint_{B_s}\langle \nabla (T_{s}u)^a,\nabla (T_{s}u)^b\rangle=\delta^{ab}$.

\item $|T_s\circ T_{2s}^{-1}-I|\le \epsilon$.
\end{enumerate}
\end{proposition}
\vspace{.25cm}

The proof of Proposition \ref{p:transfor_prop} will rely on the eigenvalue estimate 
\eqref{e:eigenvalue_cone} of Proposition \ref{p:eigenvalue_cone}.
The key point is that almost linear growth harmonic function on the limit cone
must be linear.  We begin with the following:

\begin{lemma}[Harmonic function with almost linear growth]
\label{l:harmonic_holder}
	Let $(M_i^n,g_i,x_i)\to (C(Y),d,x_\infty)=(\mathbb{R}^k\times C(Z),d,x_\infty)$ satisfy $\Ric\ge-\delta_i\to 0$ and $\Vol(B_1(x_i))\ge \rv>0$.  Assume $B_{10}(x_i)$ is not $(k+1,\eta)$-symmetric. \newline Then there exists $\epsilon(n,\rv,\eta)>0$ such that any harmonic function $u$ on $C(Y)$ with almost linear growth $|u|(y)\le Cd(y,x_\infty)^{1+\epsilon}+C$ is a linear function induced from an $\dR$ factor. 
\end{lemma}
\begin{proof}
To begin with it follows from Proposition \ref{p:harmonic_cone} that a harmonic function on $C(Y)$ has the form
	\begin{align}\label{e:harmonic_expansion}
		u=\sum_{i=0}^\infty b_i \cdot r^{\alpha_i}\phi_i\, ,
	\end{align}
{where the convergence is in $W^{1,2}$ on compact subsets}. 

	By using the eigenvalue estimate in Proposition \ref{p:eigenvalue_cone} and noting that 
$\alpha_0=0$, we have 
	$$
	1=\alpha_1=\cdots= \alpha_k<1+\beta(n,\eta,\Vol(Z))\le \alpha_{k+1}\, .
	$$
	
	If we put $u_0=u-\sum_{i=0}^k b_ir^{\alpha_i}\phi_i$, then we still have $|u_0|(y)\le Cd(y,x_\infty)^{1+\epsilon}+C$.

To finish the proof, it suffices to show that $\epsilon\le \beta/2$ implies $u_0=0$. For  this, we consider the $L^2$ integral of $u_0$ over $B_R(x_\infty)$. Since $r^{\alpha_i}\phi_i$ are orthogonal in $L^2(\partial B_r(x_\infty))$ for each $r$ we have 
\begin{align}
	\sum_{i=k+1}^\infty b_i^2 \Vol(Y)^{-1} \frac{n}{n+2\alpha_i} R^{2\alpha_i}
=\fint_{B_R(x_\infty)}|u_0|^2\le C+CR^{2+2\epsilon}.
\end{align}
Since $R$ is arbitrary, it follows that $b_i=0$ for all $i\ge k+1$ if $\epsilon\le \beta/2$. Indeed,
since $\alpha_i>1+\epsilon$ for $i\ge k+1$ and $b_i^2 \Vol(Y)^{-1} \frac{n}{n+2\alpha_i} R^{2\alpha_i}\le C+CR^{2+2\epsilon}$ for any $R$ we have $b_i=0$ for $i\ge k+1$. This implies $u_0=0$, which completes  the proof of Lemma \ref{l:harmonic_holder}.
	\end{proof}

\begin{proof} (Of Proposition \ref{p:transfor_prop})
We will argue by contradiction.  
Make the following assumptions:
\begin{itemize}
\item
There exists $\epsilon_0<<1$ and $(M_i,g_i,x_i)$ such that $B_{\delta_i^{-1}r}(x_i)$ is $(k,\delta^2_i)$ splitting but $B_r(x_i)$ is not $(k+1,\eta)$-splitting for all $r_i\le r\le 1$. Let $u_i: B_{2}(x_i)\to \mathbb{R}^k$ be a $(k,\delta_i)$-splitting map with $\delta_i\to 0$.   
\vskip2mm

\item
There exists $s_i>r_i\to 0$ such that for all $1\ge r\ge s_i$, there exists a lower triangle matrix $T_{x_i,r}$ such that $T_{x_i,r}u_i$ is a $(k,\epsilon_0)$ splitting on $B_r(x_i)$ with $ \fint_{B_r(x_i)}\langle \nabla (T_{x_i,r}u)^a,\nabla (T_{x_i,r}u)^b\rangle=\delta^{ab}$. 
\vskip2mm

\item
No such mapping $T_i=T_{x_i,s_i/10}$ exists on $B_{s_i/10}(x_i)$.  
(Note that since $\delta_i\to 0$ we have trivially that $s_i\to 0$.)
\end{itemize}
\vskip2mm

We will contradict  the assumption that $s_i>r_i$.
\vskip3mm
 
To complete the proof of  Proposition \ref{p:transfor_prop}, will need the  following 
lemma. It states  that  as long as they exist, the transformation matrices, 
$T_s$, change slowly.
\vskip2mm

Let $|\, \cdot\,  |$ denote the $L^\infty$-norm on matrices.

 \begin{lemma}
 \label{l:Cholesky}
 There exists $C(n)$ such that for all $1\ge r\ge s_i$,
 $$
 |T_{x_i,r}\circ T^{-1}_{x_i,2r}-I|\le C\sqrt{\epsilon_0}\, .
 $$   
\end{lemma}
\begin{proof}
By volume doubling and noting that $T_{x_i,2r}u:B_{2r}(x_i)\to \mathbb{R}^k$ is $(k,\epsilon_0)-$splitting we have 
\begin{align}
	\fint_{B_{r}(x_i)}\Big|\langle \nabla (T_{x_i,2r}u)^a,\nabla (T_{x_i,2r}u)^b\rangle-\delta^{ab}\Big|
\le C(n)\fint_{B_{2r}(x_i)}\Big|\langle \nabla (T_{x_i,2r}u)^a,\nabla (T_{x_i,2r}u)^b\rangle-\delta^{ab}\Big|\le C(n)\sqrt{\epsilon_0}.
\end{align}
Thus, there exists lower trianglular matrix $A_{2r}$ with $|A_{2r}-I|\le C(n)\sqrt{\epsilon_0}$ such that $\tilde{T}_{x_i,2r}:=  A_{2r}T_{x_i,2r}$ satisfies 
$$ 
\fint_{B_r(x_i)}\langle \nabla (\tilde{T}_{x_i,2r}u)^a,\nabla (\tilde{T}_{x_i,2r}u)^b\rangle
=\delta^{ab}\, .
$$

\noindent
By the normalization, we have 
$ \fint_{B_r(x_i)}\langle \nabla (T_{x_i,r}u)^a,\nabla (T_{x_i,r}u)^b\rangle=\delta^{ab}$. 
\vskip2mm

Define a symmetric bilinear form $B(f,h)$, on $C^\infty(B_{2r}(x_i))$ by
$$
B(f,h):=  \fint_{B_r(x_i)}\langle \nabla f,\nabla h\rangle\, .
$$
 Denote the associated positive definite symmetric $k\times k$ matrix by  $B:=   (B_{ab}):=  (B(u^a,u^b))$. Thus,
  we have 
  $$
  {T}_{x_i,r}B{T}_{x_i,r}^*=I=\tilde{T}_{x_i,2r}B\tilde{T}_{x_i,2r}^*\, .
  $$ 
  In particular,
  $$
  {T}_{x_i,r}^{-1}({T}_{x_i,r}^{-1})^*=B=\tilde{T}_{x_i,2r}^{-1}(\tilde{T}_{x_i,2r}^{-1})^*\, .
  $$
  Since ${T}_{x_i,r}$ and $\tilde{T}_{x_i,2r}$ are lower triangle matrices with positive diagonal entries, the uniqueness of Cholesky decomposition (see \cite{GV}) implies that $\tilde{T}_{x_i,2r}^{-1}={T}_{x_i,r}^{-1}$. Therefore, we have $A_{2r}T_{x_i,2r}=T_{x_i,r}$. In particular, 
  $$
  |T_{x_i,r}T_{x_i,2r}^{-1}-I|=|A_{2r}-I|\le C(n)\sqrt{\epsilon_0}\, .
  $$ 
This completes the proof of Lemma \ref{l:Cholesky}. 
\end{proof}
\vskip2mm

Now we can complete the proof of Proposition \ref{p:transfor_prop}.
For $k\times k$ matrices $A_1,A_2$ and the $L^\infty$-norm for matrices, we have by a simple triangle inequality that
\begin{align}\label{e:matrix_inequality}
|A_1A_2-I|\le |A_1-I|+|A_2-I|+k|A_1-I|\cdot |A_2-I|.
\end{align}
 By Lemma \ref{l:Cholesky} and \eqref{e:matrix_inequality} and an induction argument
we have 
\begin{align}
	|T_{x_i,r}^{-1}\circ T_{x_i,r/2^\ell}-I|\le \Big(1+(k+1)C\sqrt{\epsilon_0}\Big)^\ell-1\,. 
\end{align}

Therefore
	\begin{align}
	|T_{x_i,r}^{-1}\circ T_{x_i,r/2^\ell}|\le  \Big(1+(k+1)C\sqrt{\epsilon_0}\Big)^\ell.
		\end{align}
For simplicity we still denote $(k+1)C$ by $C$. 	Hence for all $r\ge s_i$, we have
	\begin{align}\label{e:matrix_holdergrowth}
	|T_{x_i,r}^{-1}\circ T_{x_i,s_i}|\le \left(\frac{r}{s_i}\right)^{\log(1+C\sqrt{\epsilon_0})/\log 2}\le \left(\frac{r}{s_i}\right)^{C\sqrt{\epsilon_0}}.
		\end{align}
	Define $v_i=s_i^{-1}T_{x_i,s_i}(u_i-u_i(x_i))$ on the rescaled space $(M_i,s_i^{-2}g_i,x_i)$.
	Since $\delta_i\to 0$ and $B_{\delta^{-1}_i s_i}(x_i)$ is $(k,\delta^2_i)$-symmetric, we know  that $(M_i,s_i^{-2}g_i,x_i)$ converges to a cone $C(Y)=\mathbb{R}^k\times C(Z)$. 
	By the H\"older growth estimate on $T_{x_i,r}^{-1}$ as in \eqref{e:matrix_holdergrowth} and noting that
	 $T_{x_i,r}(u_i-u_i(x_i))$ is a $(k,\epsilon_0)$ splitting map at scale $r$, we have for all $x$ with $s_i^{-1}\ge d(x,x_i)= R>1$,
$$	
|\nabla v_i(x)|\le C\cdot R^{C\sqrt{\epsilon_0}} \implies |\nabla v_i(x)|\le C\cdot d(x,x_i)^{C\sqrt{\epsilon_0}}+C\, .
$$
	Also, by Proposition \ref{p:convergenc_function},  the sequence
	$v_i$ converges in the local $W^{1,2}$ sense to a harmonic function $v$ in $C(Y)$ with H\"older growth on the gradient, i.e., $|\nabla v|(x)\le CR^{C\sqrt{\epsilon_0}}$ for $|x|\le R$. Therefore, if the $\epsilon_0$ is small as in Lemma \ref{l:harmonic_holder}, then we have $v: C(Y)\to\mathbb{R}^k$ is actually linear. Moreover, by using $W^{1,2}$ convergence in Proposition \ref{p:convergenc_function}, and noting that the energy is quadratic we have
		\begin{align}
		\fint_{B_1(x_\infty)}\langle \nabla v^a,\nabla v^b\rangle =  \delta^{ab}\, .
	\end{align}
Hence $v=(v^1,\cdots,v^k)$ forms a basis of linear functions on $\mathbb{R}^k$.  Without loss of generality we can assume $v=(x^1,\ldots, x^k)$ are the standard coordinates.  By 
$W^{1,2}$-convergence of $v_i$ as Proposition \ref{p:convergenc_function} and Proposition \ref{p:fgconvergence} we have:
\begin{align}
\lim_{i\to \infty} 4\fint_{B_1(x_i)}|\langle \nabla v_i^a,\nabla v_i^b\rangle-\delta^{ab}|
&=\lim_{i \to \infty}\fint_{B_1(x_i)}\Big||\nabla v_i^a+\nabla v_i^b|^2-|\nabla v_i^a-\nabla v_i^b|^2-4\delta^{ab}\Big|\notag\\
&=\lim_{i\to\infty} \fint_{B_1(x_\infty)}\Big||\nabla x^a+\nabla x^b|^2-
|\nabla x^a-\nabla x^b|^2-4\delta^{ab}\Big|\notag\\
&=0\,.
	\end{align}
Here we have used 
	$|\nabla x^a+\nabla x^b|^2=|{\rm Lip}(x^a+x^b)|^2=2=|{\rm Lip}(x^a-x^b)|^2=|\nabla x^a-\nabla x^b|^2$. Hence, $v_i$ satisfies 
	\begin{align}
		\lim_{i\to\infty}\fint_{B_1(x_i)}|\langle \nabla v_i^a,\nabla v_i^b\rangle -\delta^{ab}|=0\, .
	\end{align}
\vskip1mm

\noindent
Thus, by  Bochner's formula \eqref{e:bochner}, the function
$v_i$ is a $(k,\epsilon_i)$-splitting function on $B_1(x_i)$ with $\epsilon_i\to 0$. Hence for each $1/10\le r\le 1$ and sufficiently large $i$ we have a rotation $A_{r,i}$ such that $|A_{r,i}-I|\le \epsilon_i$  and 
\begin{align}
\fint_{B_r(x_i)}\langle \nabla (A_{r,i}v_i)^a,\nabla (A_{r,i}v_i)^b\rangle =\delta^{ab}.
\end{align}
 In particular, this implies for large $i$ that $A_{r,i}v_i: B_r(x_i)\to \mathbb{R}^k$ is a $(k,\epsilon_0/100)$-splitting for $1/10\leq r\leq 1$ and satisfies the orthonormal condition (2), which contradicts with the existences of a minimal $s_i>r_i$.  This finishes the proof of the existence of transformation matrices.  The matrix estimate (3) comes from the transformation estimates in 
 Lemma \ref{l:Cholesky} by choosing $\delta$ small.  
 This completes the proof of Proposition \ref{p:transfor_prop}. 
 \end{proof}

\vspace{.2cm}
\subsection{A canonical Reifenberg theorem}
\label{ss:Reifenberg}
Prior to giving the proof of (2) of Theorem \ref{t:transformation},
we will make a brief  digression and  a non-metric proof of the Reifenberg Theorem, which first proved by Cheeger-Colding in \cite{ChCoI}.  Although this result is not used elsewhere in the paper it
 seems to be of independent interest. In fact, it is  
a (much) easier instance of the sort of argument we will give 
 when we we eventually study the higher singular strata, see 
 Theorem \ref{t:approximate_neck}, and thus is a good motivator.
 
\begin{theorem}[Canonical Reifenberg Theorem]
\label{t:reifenbergcanonical}
	Let $(M^n,g,p)$ satisfy $\Ric_{M^n}\ge -\delta(n-1)$ and $d_{GH}(B_4(p),B_4(0^n))\le \delta$ with $0^n\in \mathbb{R}^n$. For any $\epsilon>0$ if $\delta\le \delta(n,\epsilon)$, then there exists a harmonic map $u: B_1(p)\to \mathbb{R}^n$ such that 
\begin{enumerate}
\item For any $x,y\in B_1(p)$ we have $(1-\epsilon) d(x,y)^{1+\epsilon}\le |u(x)-u(y)|\le (1+\epsilon) d(x,y)$,
\item For any $x\in B_1(p)$ we have that $du:T_xM\to \dR^n$ is nondegenerate.
\end{enumerate}
In particular, $u$ is a diffeomorphism that is uniformly bi-H\"older onto its image $u(B_1(p))$.
\end{theorem}
\begin{remark}
Consider limit space $(M^n_i,g_i,p_i)\to (X,d,p)$ with $d_{GH}(B_4(p),B_4(0^n))\le \delta$, and a converging  sequence of harmonic maps $u_i:B_1(p_i)\to\mathbb{R}^n$. Then by
Theorem \ref{t:reifenbergcanonical},
 we get that  $B_1(p)$ is bi-H\"older to $\mathbb{R}^n$.
	\end{remark}

\begin{proof}  (Of Theorem \ref{t:reifenbergcanonical})
Let  $\delta'>0$. By Theorem \ref{t:subball_n_symmetric}, if $\delta\le \delta(n,\delta')$, then  every sub-ball $B_r(x)\subset B_4(p)$ is $(n,\delta')$-symmetric. Moreover, there exists  $\delta'$-splitting map $u:B_{3}(p)\to \mathbb{R}^n$. 
\vskip1mm

By the Transformation Proposition \ref{p:transfor_prop},
for any $\epsilon'>0$, $x\in B_3(p)$ and $r\le 1$ if $\delta' \le \delta'(\epsilon',n)$ then there exists a $n\times n$ lower triangle matrix $T_{x,r}$, such that $T_{x,r}u: B_r(x)\to \mathbb{R}^n$ is an $\epsilon'$-splitting map. Moreover, by the transformation estimate (3),
 $|T_{x,r}|\le r^{-\epsilon'}$. We will see that these estimates imply Theorem \ref{t:reifenbergcanonical}.
 \vskip1mm
 
  First, we will prove a H\"older estimate on $u$. Let $x,y\in B_{3/2}$ with $d(x,y)=r$.  Since $T_{x,r}u: B_r(x)\to \mathbb{R}^n$ is an $\epsilon'$-splitting map, and in particular $T_{x,r}u$ is an 
  $\epsilon r$-GH map if $\epsilon'\le \epsilon'(\epsilon,n)$, we have that
\begin{align}
	|T_{x,r}u(x)-T_{x,r}u(y)|\ge (1-\epsilon) d(x,y).
\end{align}
By the matrix growth estimate $|T_{x,r}|\le r^{-\epsilon'}$
 we then have $$
 |u(x)-u(y)|\ge (1-\epsilon)d(x,y)^{1+\epsilon}\qquad {\rm for}\,\, d(x,y)=r\, .
 $$
  Since $r$ is arbitrary, by using the gradient bound $|\nabla u|\le 1+\delta'$ for splitting 
  maps $u$
   we conclude for that any $x,y\in B_{3/2}(p)$ 
\begin{align}
	(1-\epsilon)d(x,y)^{1+\epsilon}\le |u(x)-u(y)|\le (1+\epsilon)d(x,y)\, .
\end{align}
 Therefore $u$ is an injective map which, in particular, implies $u$ is  bi-H\"older to its image.  
 
 Next we show that $du:T_xM\to \dR^n$ is nondegenerate, from which it follows 
 that $u$ is a diffeomorphism.  Essentially, this is because $du(x) = T_{x,0}^{-1}$ .  In more detail, let $2r=r_h(x)$ be the harmonic radius at $x$, see Definition \ref{d:harmonic_radius}.  Then by smooth elliptic estimates the splitting map $T_{x,r}u$ satisfies the pointwise bound $|\langle \nabla T_{x,r}u^a,\nabla T_{x,r}u^b\rangle-\delta^{ab}|<\epsilon$.  In particular, this gives that $\det(du)(x)\neq 0$, as claimed.
\end{proof}
\vspace{.1cm}


\subsection{Hessian decay estimates on  limit cones}
\label{ss:hessian_decay_limit_cone}
The main result
of this subsection is Proposition \ref{p:decay_cone}, the key Hessian decay estimate for harmonic functions on limit cones. In  the next subsection, it will
be promoted to the Hessian decay estimate on manifolds, and after that, to statement
(2) of Theorem \ref{t:transformation}. Since a priori, we can't define the Hessian directly, we employ
 Bochner's formula \eqref{e:bochner}.
This will allow us to work with with a weak version.
\vskip3mm

\noindent
{\bf Notation.} Let $\varphi:\mathbb{R}\to\mathbb{R}$ denote a smooth cutoff function
such that $\varphi:=  1$ if $r\le 1$ and $\varphi:=  0$ if $r\ge 2$. 
In Proposition \ref{p:decay_cone}, we will 
consider a limit cone $(C(Y),d,x_\infty)$. We put $r=d(x,x_\infty)$ and
$\psi_s(x)=\varphi(r^2/s^2)$.
\vskip3mm

\begin{proposition}[Main decay estimate for cones]
\label{p:decay_cone}
There exists $\beta=\beta(n,\eta,\rv)>0$ with the following property.
Let $(M_i^n,g_i,x_i)\to (C(Y),d,x_\infty)=(\mathbb{R}^k\times C(Y),d,x_\infty)$ satisfy  
$\Ric_{M^n_i}\ge-\delta_i\to 0$ and $\Vol(B_1(x_i))\ge \rv>0$. Let 
$u: B_{10}(x_\infty)\subset C(Y)\to 
\mathbb{R}$ be a harmonic function and assume $B_{10}(x_\infty)$ is not 
$(k+1,\eta)$-symmetric.  Then
for all $0<s\le t\le 2$
\begin{align}
\label{e:decay_psi_cone}
s^{2-n}\int_{\mathbb{R}^k\times C(Y)}
|\nabla u|^2\Delta \psi_s \le \left(\frac{t}{s}\right)^{-\beta} 
t^{2-n}\int_{\mathbb{R}^k\times C(Y)} |\nabla u|^2\Delta \psi_t\, .
\end{align}
\end{proposition}
\vskip2mm


The proof of Proposition \ref{p:decay_cone} is given at the end of this subsection. Ultimately, it
is a consequence of the eigenvalue estimates in Subsection \ref{ss:eigenvalue_limit_cone}. 
We will begin with some preliminary computations. 
\vskip2mm

According to  Proposition \ref{p:harmonic_cone}, any harmonic function $u$ can be 
written as  $u=\sum b_ir^{\alpha_i}\phi_i$,
 where the convergence is in $W^{1,2}$-sense.
By Theorem \ref{t:laplacian_metric_cone}, we have 
\vskip-2mm

$$
|\nabla\phi_i|^2(r,y)=|\Lip \phi_i|^2(r,y)=r^{-2}|\Lip \phi_i|^2(y)=r^{-2}|\nabla\phi_i|^2_Y\, .
$$ 
\vskip-2mm
\begin{align}
|\nabla u|^2=\sum_{i,j} b_ib_j \alpha_i\alpha_j r^{\alpha_i+\alpha_j-2}\phi_i\phi_j+\sum_{i,j}b_ib_jr^{\alpha_i+\alpha_j-2}\langle \nabla\phi_i,\nabla\phi_j\rangle_Y,
\end{align}	

Let $\varphi: \mathbb{R}\to [0,1]$ be a standard cutoff function such that the function $\psi :B_{10}(x_\infty)\to [0,1]$ defined by $\psi(x)=\varphi(d^2(x,x_\infty))$ satisfying ${\rm supp } \,\psi\subset B_{10}(x_\infty)$. Then 
\begin{align}
	\Delta\psi= \varphi'(r^2)\Delta r^2+\varphi''(r^2)|\nabla r^2|^2=2n\varphi'(r^2)+4r^2\varphi''(r^2).
\end{align}	
In particular, $|\Delta\psi|\le C(n)$.

\begin{lemma}\label{l:computation_cone}
Let $(M_i^n,g_i,x_i)\to (C(Y),d,x_\infty)=(\mathbb{R}^k\times C(Z),d,x_\infty)$ satisfy 
$\Ric_{M^n_i}\ge-\delta_i\to 0$ and $\Vol(B_1(x_i))\ge \rv>0$.  Assume $u=\sum b_ir^{\alpha_i}\phi_i$ is a harmonic function on $B_{10}(x_\infty)\subset C(Y)$ where the convergence is in the $W^{1,2}$-sense.  Then
	\begin{align}
\label{e:computation_cone}
	\int_{C(Y)} |\nabla u|^2\Delta\psi =\sum_{\alpha_i>1} \Big(b_i^2\alpha_i^2+b_i^2\lambda_i\Big)(2\alpha_i-2)(n+2\alpha_i-4)\int_0^{\infty}\varphi(r^2)r^{n+2\alpha_i-5}dr.
	\end{align}
\end{lemma}
\begin{proof}
	Consider $u_\ell:=  \sum_{i=0}^\ell b_ir^{\alpha_i}\phi_i$. By Proposition \ref{p:harmonic_cone}, $u_\ell$ converges in the $W^{1,2}$-sense to $u$. Also, since $|\Delta\psi|\le C(n)$, we have 
	\begin{align}
	\int |\nabla u|^2\Delta\psi &=\lim_{\ell\to\infty} \int |\nabla u_\ell|^2\Delta\psi\, .
	\end{align}
It now suffices to compute $\int |\nabla u_\ell|^2\Delta\psi$. We have
	\begin{align}
	\int |\nabla u_\ell|^2\Delta\psi&=\int_0^{\infty}r^{n-1}\int_Y |\nabla u_\ell|^2\Delta\psi 	d\mu_Y dr\notag\\
	&=\int_0^{\infty}r^{n-1} \Big(2n\varphi'(r^2)+4r^2\varphi''(r^2)\Big)\int_Y|\nabla u_\ell|^2d\mu_Ydr\notag\\
	&=\int_0^{\infty}r^{n-1} \Big(2n\varphi'(r^2)+4r^2\varphi''(r^2)\Big)\sum_{i=0}^\ell \Big(b_i^2\alpha_i^2+b_i^2\lambda_i\Big)r^{2\alpha_i-2}dr\notag\\
	&=\int_0^{\infty} \Big(2n\varphi'(r^2)+4r^2\varphi''(r^2)\Big)\sum_{i=0}^\ell \Big(b_i^2\alpha_i^2+b_i^2\lambda_i\Big)r^{n+2\alpha_i-3}dr\, .
	\end{align}
	Since $\alpha_0=\lambda_0=0$ we can integrate by parts to get 
	\begin{align}\nonumber
	\int |\nabla u_\ell|^2\Delta\psi&=\int_0^{\infty} \sum_{i=1}^\ell\Big(b_i^2\alpha_i^2+b_i^2\lambda_i\Big)2n \varphi'(r^2)r^{n+2\alpha_i-3}dr +\int_0^{\infty}\sum_{i=1}^\ell \Big(b_i^2\alpha_i^2+b_i^2\lambda_i\Big)4\varphi''(r^2)r^{n+2\alpha_i-1}dr\notag\\ \nonumber
	&=\int_0^{\infty} \sum_{i=1}^\ell \Big(b_i^2\alpha_i^2+b_i^2\lambda_i\Big)2n \varphi'(r^2)r^{n+2\alpha_i-3}dr -\int_0^{\infty}\sum_{i=1}^\ell \Big(b_i^2\alpha_i^2+b_i^2\lambda_i\Big)2(n+2\alpha_i-2)\varphi'(r^2)r^{n+2\alpha_i-3}dr\\
	&=\sum_{i=1}^\ell \Big(b_i^2\alpha_i^2+b_i^2\lambda_i\Big)(2\alpha_i-2)\int_0^{\infty}-2\varphi'(r^2)r^{n+2\alpha_i-3}dr\notag\\
	&=\sum_{\alpha_i>1}^\ell \Big(b_i^2\alpha_i^2+b_i^2\lambda_i\Big)(2\alpha_i-2)\int_0^{\infty}-2\varphi'(r^2)r^{n+2\alpha_i-3}dr\notag\\
	&=\sum_{\alpha_i>1}^\ell \Big(b_i^2\alpha_i^2+b_i^2\lambda_i\Big)(2\alpha_i-2)(n+2\alpha_i-4)\int_0^{\infty}\varphi(r^2)r^{n+2\alpha_i-5}dr\, .
		\end{align}
In the last integration by parts, we have used the fact that $\alpha_i>1$ and $n\ge 2$ to deduce that \hbox{$\lim_{r\to 0}r^{n+2\alpha_i-4}=0$}.
\end{proof}
\vskip-3mm

Now we can complete the proof of Proposition \ref{p:decay_cone}:
\vskip-2mm

\begin{proof} 
(Of Proposition \ref{p:decay_cone})
Let
 $\varphi: \mathbb{R}\to [0,1]$ be such that $\varphi:=  1$ if $r\le 1$, $\varphi:=  0$ if $r\ge 2$, and $|\varphi'|+|\varphi''|\le 100$. For any scale $s\le 1$, define $\psi_s(x):= \varphi_s(r^2):= \varphi(r^2/s^2)$ with $r=d(x,x_\infty)$. Thus $\psi_s$ has support contained in $B_{2s}(x_\infty)$. 
 
 By Proposition \ref{p:harmonic_cone} we can write the harmonic function $u=\sum b_ir^{\alpha_i}\phi_i$, where the convergence is $W^{1,2}$. 
Applying Lemma \ref{l:computation_cone} gives
\begin{align}
	\int |\nabla u|^2\Delta\psi_s&=\sum_i \Big(b_i^2\alpha_i^2+b_i^2\lambda_i\Big)(2\alpha_i-2)(n+2\alpha_i-4)\int_0^{\infty}\varphi(r^2/s^2)r^{n+2\alpha_i-5}dr\notag\\
	&=\sum_{\alpha_i>1} \Big(b_i^2\alpha_i^2+b_i^2\lambda_i\Big)(2\alpha_i-2)(n+2\alpha_i-4)s^{n+2\alpha_i-4}\int_0^{\infty}\varphi(r^2)r^{n+2\alpha_i-5}dr
\end{align}
 Therefore, for any $0<s\le t\le 2$, we have
\begin{align}
\label{e:sum_s_nablau1}
		&s^{2-n}\int |\nabla u|^2\Delta\psi_s=\sum_{\alpha_i>1} \Big(b_i^2\alpha_i^2+b_i^2\lambda_i\Big)(2\alpha_i-2)(n+2\alpha_i-4)s^{2\alpha_i-2}\int_0^{\infty}\varphi(r^2)r^{n+2\alpha_i-5}dr\, .
\end{align}
\vskip-6mm
\begin{align}
		\label{e:sum_s_nablau2}
		&t^{2-n}\int |\nabla u|^2\Delta\psi_t=\sum_{\alpha_i>1} 	
		\Big(b_i^2\alpha_i^2+b_i^2\lambda_i\Big)(2\alpha_i-2)(n+2\alpha_i-4)t^{2\alpha_i-2}\int_0^{\infty}\varphi(r^2)r^{n+2\alpha_i-5}dr.
	\end{align}
	By the eigenvalue estimates in Proposition \ref{p:eigenvalue_cone} we have  
	$\alpha_i>1+\beta(n,\rv,\eta)>1$ for $\alpha_i\ne 1$. Hence, each of the terms in the sums \eqref{e:sum_s_nablau1} and \eqref{e:sum_s_nablau2} are nonnegative. It follows that
	 \begin{align}
	 	s^{2\alpha_i-2}\cdot \Big(b_i^2\alpha_i^2+b_i^2\lambda_i\Big)&\cdot (2\alpha_i-2)
		\cdot (n+2\alpha_i-4)\cdot \int_0^{\infty}\varphi(r^2)r^{n+2\alpha_i-5}dr\notag\\
	 	&= \left(\frac{t}{s}\right)^{2-2\alpha_i} t^{2\alpha_i-2}\Big(b_i^2\alpha_i^2+b_i^2\lambda_i\Big)(2\alpha_i-2)(n+2\alpha_i-4)\int_0^{\infty}\varphi(r^2)r^{n+2\alpha_i-5}dr\notag\\
	 	&\le \left(\frac{t}{s}\right)^{-2\beta}\cdot t^{2\alpha_i-2}\cdot \Big(b_i^2\alpha_i^2+b_i^2\lambda_i\Big)\cdot (2\alpha_i-2)\cdot (n+2\alpha_i-4)
		\cdot \int_0^{\infty}\varphi(r^2)r^{n+2\alpha_i-5}dr
	 \end{align}
This gives
\eqref{e:decay_psi_cone} i.e. the conclusion of  
Proposition \ref{p:decay_cone}:
$$
		s^{2-n}\int |\nabla u|^2\Delta \psi_s 
		\le \left(\frac{t}{s}\right)^{-\beta} t^{2-n}\int |\nabla u|^2\Delta \psi_t\, .
$$
\end{proof}
\vskip2mm


\subsection{The Hessian Decay Estimate on Manifolds}
\label{ss:hessian_decay_manifold}
In this subsection, we will prove Proposition \ref{p:decay_hessian_transformation},
which is a Hessian decay estimate for splitting maps. 
As explained at the beginning of this section,
the proof is obtained by showing that if the conclusion were to fail, then
Proposition \ref{p:decay_cone} would be contradicted.
The proof of Proposition \ref{p:decay_hessian_transformation} will be given at the end of Subsection \ref{ss:hessian_decay_manifold} which would depend on several decay estimates in Subsubsection \ref{sss:hessian_decay_general_harmonic} and \ref{sss:Hessian_decay_with_pinching}.


\begin{proposition}
\label{p:decay_hessian_transformation}
	Let $(M^n,g,x)$ satisfy $\Ric_{M^n}\ge -(n-1)\delta^2$, $\Vol(B_1(x))\ge \rv>0$. Let $\eta,\alpha>0$.  Let $u: B_{2}(x)\to \mathbb{R}^k$ be a $(k,\delta)$ splitting map. Assume:
\begin{itemize} 
	\item $B_{\delta^{-1}r}(x)$ is $(k,\delta^2)$-symmetric for all $r_0\le r\le 1$.
	\vskip2mm
	\item $B_r(x)$ is not $(k+1,\eta)$-symmetric for all $r_0\le r\le 1$. 
	\end{itemize}
	
For all $\epsilon>0$ if $\delta\le \delta(n,\rv,\epsilon,\eta,\alpha)$, then there exists $0<c(n,\rv,\eta)<1$, $C(n,\rv)>0$ and a $k\times k$ lower triangular matrix $T_r$ such that $T_ru: B_r(x)\to \mathbb{R}^k$ is a $(k,\epsilon)$-splitting map, and if $r_0\le r\le 1$ with $cs/2\le r\le cs$, then 
\begin{align}
\label{e:decay_hessian_transformation} 
		r^{2-n}\int_{B_{r}(x)}&\Big(|\nabla^2T_{r}u|^2+\Ric(\nabla T_{r}u,\nabla T_{r}u)+2\delta^2(n-1) |\nabla T_{r}u|^2\Big)\notag\\
		&\le \frac{1}{2}s^{2-n}\int_{B_s(x)}\Big(|\nabla^2T_su|^2+\Ric(\nabla T_su,\nabla T_su)+2\delta^2(n-1) |\nabla T_su|^2\Big)+C\cE_{s}^{k}(x).
	\end{align}
	
\end{proposition}
\vspace{4mm}
\begin{remark}
Recall that the constant $\alpha$ in Proposition \ref{p:decay_hessian_transformation} appears in the Definition \ref{d:kalphadeltapinching} of $\cE_s^k(x)=\cE_s^{k,\alpha,\delta}(x)$.
\end{remark}


\subsubsection{The Hessian decay for general harmonic functions}
\label{sss:hessian_decay_general_harmonic}

In this subsubsection, as an essential step in the proof of Proposition \ref{p:decay_hessian_transformation}
we will prove a decay estimate for general harmonic functions. It states that after subtracting off the linear terms, the $L^2$ Hessian has H\"older decay.  Before giving the result,
 we will need some terminology. 
\vskip2mm

\noindent
\textbf{Notation:} Let $v=(v^1,\cdots, v^k): B_{10}(x)\to\mathbb{R}^k$ be a $(k,\delta)$- splitting map which would be constructed later by  Theorem \ref{t:sharp_splitting}. For harmonic function $u: B_{10}(x)\to \mathbb{R}$ we define 
\begin{align}\label{e:u-uk}
\tilde{u}=u-\sum_{\ell=1}^k a_\ell v^\ell
\end{align}
by stipulating  that the coefficients are chosen to minimize
\begin{align}
\fint_{B_1(x)}|\nabla \tilde{u}|^2=\min_{(b_\ell)\in \mathbb{R}^k} \fint _{B_1(x)}|\nabla u-\sum_{\ell=1}^k b_\ell\nabla v^\ell|^2. 
\end{align}
\vskip2mm
After having subtracted off the 'linear' term we can prove the following decay estimate for the harmonic function $\tilde{u}$:

\begin{lemma}
\label{l:decay_tildeu}
There exists $0<c(n,\rv,\eta)<1$ such that the following holds. 
Let $\delta<\delta(n,\rv,\eta)$ and let
$(M^n,g,x)$ satisfy $\Ric_{M^n}\ge -(n-1)\delta^2$ and $\Vol(B_1(x))\ge \rv>0$.  Assume
$B_{\delta^{-1}}$ is $(k,\delta^2)$-symmetric but that $B_1(x)$ is not $(k+1,\eta)$-symmetric.  Then if $u: B_{2}(x)\to \mathbb{R}$ denotes a harmonic function with $\tilde{u}$ defined as in \eqref{e:u-uk} and
 $c/2\le r\le c$, the following holds:
	\begin{align}
		r^{2-n}\int_{B_{r}(x)}\Big(|\nabla^2\tilde{u}|^2+\Ric(\nabla \tilde{u},\nabla \tilde{u})+2\delta^2(n-1) |\nabla \tilde{u}|^2\Big)\le \frac{1}{4}\int_{B_1(x)}\Big(|\nabla^2\tilde{u}|^2+\Ric(\nabla \tilde{u},\nabla \tilde{u})+2\delta^2(n-1) |\nabla \tilde{u}|^2\Big).
	\end{align}
\end{lemma}
\begin{proof}
The constant $c(n,\rv,\eta)$ will be fixed at the end of the proof. 
\vskip1mm

The existence of $\delta(n,\rv,\eta) >0 $ will shown by arguing by contradiction.
Therefore,  assume there exists $\delta_i\to 0$ and $(M_i^n,g_i,x_i)$ 
with $\Ric_{M^n_i}\ge -(n-1)\delta^2_i$ and $\Vol(B_1(x_i))\ge \rv>0$.  Assume further that
the ball $B_{\delta_i^{-1}}(x_i)$ is $(k,\delta_i^2)$-symmetric, $B_1(x_i)$ is not $(k+1,\eta)$-symmetric, and $u_i: B_{2}(x_i)\to \mathbb{R}$ is a harmonic function with corresponding $\tilde{u}_i$ defined in \eqref{e:u-uk} such that for some $c/2\le r\le c$, 
\begin{align}\label{e:decay_converse_direction}
		r^{2-n}\int_{B_{r}(x_i)}\Big(|\nabla^2\tilde{u}_i|^2+\Ric(\nabla \tilde{u}_i,\nabla \tilde{u}_i)+2\delta_i^2(n-1) |\nabla \tilde{u}_i|^2\Big)> \frac{1}{4}\int_{B_1(x_i)}\Big(|\nabla^2\tilde{u}_i|^2+\Ric(\nabla \tilde{u}_i,\nabla \tilde{u}_i)+2\delta_i^2(n-1) |\nabla \tilde{u}_i|^2\Big) 
	\end{align}

Normalize $\tilde{u}_i$ such that $\fint_{B_1(x_i)}|\nabla \tilde{u}_i|^2=1$ and 
$\fint_{B_1(x_i)}\tilde{u}_i=0$. Then by the Poincar\'e inequality, we have 
\begin{align}\label{e:u_iL2}
\fint_{B_1(x_i)}\tilde{u}_i^2\le C(n).
\end{align}
By the definition of $\tilde{u}_i$, we have $\fint_{B_1(x_i)}\langle \nabla v_{i,\alpha}, \nabla \tilde{u}_i\rangle =0$ for any $\alpha=1,\cdots, k$ and $v_{i,\alpha}$ are the $k$ splitting maps for $B_{2}(x_i)$.   Since $B_{1}(x_i)$ is not $(k+1,\eta)$ splitting, we have 
\begin{align}
\fint_{B_1(x_i)}|\nabla^2 \tilde{u}_i|^2\ge \eta'(n,\rv,\eta).
\end{align}

Choose a cutoff function $\varphi_i$ as in Theorem \ref{t:cutoff} with $\varphi_i:=  1$ on $B_{1/4}(x_i)$ and $\varphi_i:=  0$ away from $B_{1/2}(x_i)$.  By the Bochner formula we have 
\begin{align}
\int_{B_{1/4}(x_i)}\Big(|\nabla^2\tilde{u}_i|^2+\Ric(\nabla \tilde{u}_i,\nabla \tilde{u}_i)+2\delta_i^2(n-1)|\nabla \tilde{u}_i|^2\Big)&\le \int \Big(|\nabla^2\tilde{u}_i|^2+\Ric(\nabla \tilde{u}_i,\nabla \tilde{u}_i)+2\delta_i^2(n-1) |\nabla \tilde{u}_i|^2\Big)\varphi_i\notag\\
&
=\frac{1}{2}\int \Big(\Delta |\nabla \tilde{u}_i|^2+4\delta_i^2(n-1) |\nabla \tilde{u}_i|^2\Big)\varphi_i\notag\\
&
\le 2\delta_i^2(n-1) \int_{B_1(x_i)}|\nabla\tilde{u}_i|^2+\int_{B_1(x_i)} |\nabla \tilde{u}_i|^2|\Delta \varphi_i| 
\notag\\
&
\le C(n)\int_{B_1(x_i)}|\nabla \tilde{u}_i|^2\le C(n)\, .
\end{align}
Therefore, from \eqref{e:decay_converse_direction}, we get 
\begin{align}
\label{e:W22_tildeu2}
&\fint_{B_1(x_i)}|\nabla \tilde{u}_i|^2=1\, \\
 &\fint_{B_1(x_i)}\tilde{u}_i^2\le C(n)\, ,\\
&
\int_{B_{1/4}(x_i)}\Big(|\nabla^2\tilde{u}_i|^2+\Ric(\nabla \tilde{u}_i,\nabla \tilde{u}_i)+2\delta_i^2(n-1) |\nabla \tilde{u}_i|^2\Big)\le C(n)\, ,
\\
&
\frac{\eta'(n,\rv,\eta)}{4}\le r^{2-n}\int_{B_{r}(x_i)}\Big(|\nabla^2\tilde{u}_i|^2+\Ric(\nabla \tilde{u}_i,\nabla \tilde{u}_i)+2\delta_i^2(n-1) |\nabla \tilde{u}_i|^2\Big),\qquad (\text{for some }c/2\le r\le c).
\end{align}
\vskip2mm

To complete the contradiction argument, we will show 
that one can pass to the limit and get a contradiction to the decay estimate Proposition \ref{p:decay_cone} in the limit cone.

Choose a cutoff function $\varphi: \mathbb{R}\to [0,1]$ such that $\varphi:=  1$ if $t\le 1$ and $\varphi:=  0$ if $t\ge 2$, and $|\varphi'|+|\varphi''|\le 100$. For any scale $c/2\le s\le 1/8$, define $\psi_{s,i}(x):= \varphi(h_i/s^2)$, where $\Delta h_i=2n$ such that $h$ approximates  $d(x_i,x)^2$ as in Theorem \ref{t:sharp_conical} or from \cite{ChCoAlmost}. Thus $\psi_{s,i}(x)$ has support contained in $B_{2s}(x_i)\subset B_{1/4}(x_i)$ and $\psi_{s,i}:=  1$ on $B_{s/2}(x_i)$. Moreover, by the gradient estimates for $h_i$, we have that $s^2|\Delta \psi_{s,i}|+s^2|\nabla \psi_{s,i}|^2\le C(n,\rv)$. 
\vskip1mm

Consider the quantity 
\begin{align}
	s^{2-n}\int |\nabla \tilde{u}_i|^2\Delta \psi_{s,i}&=s^{2-n}\int \Delta |\nabla \tilde{u}_i|^2\psi_{s,i}\notag\\
	 &=s^{2-n}\int 2\Big(|\nabla^2\tilde{u}_i|^2+\Ric(\nabla\tilde{u}_i,\nabla\tilde{u}_i)\Big)\psi_{s,i}.
\end{align} 
For $\delta_i$ small enough, by using \eqref{e:W22_tildeu2} 
 we can conclude that 
\begin{align}
	C(n)^{-1}\eta'\le r^{2-n}\int |\nabla \tilde{u}_i|^2\Delta \psi_{r,i},\qquad ~~\mbox{ for some $c/2\le r\le c$,}\\
	s^{2-n}\int |\nabla \tilde{u}_i|^2\Delta \psi_{s,i}\le C(n), ~~\qquad \mbox{ for all $1/16\le s\le 1/8$.}
\end{align}

By letting $i\to\infty$, we obtain a limit cone $(C(Y),d,x_\infty)=\mathbb{R}^k\times C(Z)$ and a harmonic function $u$ in $B_1(x_\infty)$. Moreover, by Proposition \ref{p:convergenc_function},
 $\tilde{u}_i\to u$ in $W^{1,2}$ sense in $B_{9/10}(x_\infty)$.  By Proposition \ref{p:convergenc_function},
 $$
 \Delta \psi_{s,i}=\varphi' \frac{2n}{s^2}+\varphi'' \frac{|\nabla h_i|^2}{s^4}\, .
 $$ 
 Also, both uniformly and in $W^{1,2}$ we  have
 $$
 h_i\to d(x,x_\infty)^2:=  d(x)^2\, .
 $$
 
On the limit cone, put $\psi_s(x)=\varphi(d(x)^2/s^2)$.  Then by Proposition \ref{p:convergenc_function}, 
 for any $c/2\le s\le 1/8$, we get
\begin{align}
	\lim_{i\to\infty} s^{2-n}\int |\nabla \tilde{u}_i|^2\Delta \psi_{s,i}&=\lim_{i\to\infty} s^{2-n}\int |\nabla \tilde{u}_i|^2\Big(\varphi'(h_i/s^2) \frac{2n}{s^2}+\varphi''(h_i/s^2) \frac{|\nabla h_i|^2}{s^4}\Big)\notag\\
	&=s^{2-n}\int |\nabla {u}|^2\Big(\varphi'(d(x)^2/s^2) \frac{2n}{s^2}+\varphi''(d(x)^2/s^2) \frac{|\nabla d(x)^2|^2}{s^4}\Big)\notag\\
	&=s^{2-n}\int |\nabla u|^2\Delta \psi_s.
\end{align}
In particular,  we have 
\begin{align}\label{e:eta_Cn}
		C(n)^{-1}\eta'(n,\rv,\eta)\le r^{2-n}\int |\nabla u|^2\Delta \psi_{r}\qquad  ~~\mbox{ for some $c/2\le r\le c$}\, .\\
	s^{2-n}\int |\nabla u|^2\Delta \psi_{s}\le C(n)\, , \qquad ~~\mbox{ for all $1/16\le s\le 1/8$}\, .
\end{align}

Now we can fix the value of $c=c(n,\rv,\eta)$ by choosing $c=c(n,\rv,\eta)= \frac{1}{10} \left(\frac{\eta'}{C(n)^2}\right)^{1/\beta}$ where $\beta$ is the constant in Proposition \ref{p:decay_cone} and $\eta',C(n)$ are in \eqref{e:eta_Cn}.
Then by the decay estimates in 
Proposition \ref{p:decay_cone}, 
we obtain a contradiction.  In fact, applying Proposition \ref{p:decay_cone} to $s=r\in [c/2,c]$ and $t=1/8$, gives
\begin{align}
	C(n)^{-1}\eta'(n,\rv,\eta)\le r^{2-n}\int |\nabla u|^2\Delta \psi_{r}\le (8r)^{\beta}8^{n-2}\int |\nabla u|^2\Delta \psi_{1/8}\le C(n)(8c)^{\beta}\, ,
\end{align}
which contradicts to $c=\frac{1}{10} \left(\frac{\eta'}{C(n)^2}\right)^{1/\beta}$.
This completes the proof of Lemma \ref{l:decay_tildeu}.
\end{proof}
\vskip2mm


\subsubsection{Hessian Decay with $k$-Pinching}
\label{sss:Hessian_decay_with_pinching}
In this subsubsection, by combining the sharp cone-splitting estimates of
Theorem \ref{t:sharp_splitting} of Section \ref{s:Sharp_splitting}
with  the Hessian decay estimate in Lemma  \ref{l:decay_tildeu},
 we will prove a decay estimate for harmonic functions which
 does not require that we subtract off the $k$-splitting map.  For this,
 we need to include an error term which is measured by $\cE_{s}^{k}(x)$. The main result is the following proposition.

\begin{proposition}
\label{p:decay_hessian}
	Let $(M^n,g,x)$ satisfy $\Ric_{M^n}\ge -(n-1)\delta^2$, $\Vol(B_1(x))\ge \rv>0$ and let $\alpha,\eta>0$.  Assume $B_{\delta^{-1}s}$ is $(k,\delta^2)$-symmetric but $B_s(x)$ is not $(k+1,\eta)$-symmetric for some fixed $s\le 1$. Let $u: B_{2s}(x)\to \mathbb{R}$ be a harmonic function with $\fint_{B_s(x)}|\nabla u|^2=1$.  If $\delta\le \delta(n,\rv,\eta,\alpha)$ then there exist constants $0<c(n,\rv,\eta)<1$ and $C(n,\rv)>0$ such that for any  $cs/2\le r\le cs$:
	\begin{align}
		r^{2-n}\int_{B_{r}(x)}&\Big(|\nabla^2u|^2+\Ric(\nabla u,\nabla u)+2\delta^2(n-1) |\nabla u|^2\Big)\\
		&\le \frac{1}{3}s^{2-n}\int_{B_s(x)}\Big(|\nabla^2u|^2+\Ric(\nabla u,\nabla u)+2\delta^2(n-1) |\nabla u|^2\Big)+C\cE_{s}^{k}(x).
	\end{align}	
\end{proposition}
\begin{proof}
By scaling, it suffices to prove the result for $s=1$.
Let $\tilde{u}=u-\sum a_i v^i:=  u-u_k$ as in \eqref{e:u-uk}. By Lemma \ref{l:decay_tildeu} for $\delta\le \delta_0(n,\rv,\eta)$ and $c(n,\rv,\eta)$ small, we have for any $c/2\le r\le c$ that 
\begin{align}\label{e:decay_u-uk}
		r^{2-n}\int_{B_{r}(x)}\Big(|\nabla^2\tilde{u}|^2+\Ric(\nabla \tilde{u},\nabla \tilde{u})+2\delta^2(n-1) |\nabla \tilde{u}|^2\Big)\le \frac{1}{4}\int_{B_1(x)}\Big(|\nabla^2\tilde{u}|^2+\Ric(\nabla \tilde{u},\nabla \tilde{u})+2\delta^2(n-1) |\nabla \tilde{u}|^2\Big).
	\end{align}

By using he Schwartz inequality on the nonnegative inner product $\Ric+(n-1)\delta^2 g$, we get 
\begin{align}\nonumber
r^{2-n}\int_{B_{r}(x)}\Big(|\nabla^2{u}|^2+\Ric(\nabla {u},\nabla {u})+2\delta^2(n-1) |\nabla {u}|^2\Big)\le &\frac{1001}{1000} r^{2-n}\int_{B_{r}(x)}\Big(|\nabla^2\tilde{u}|^2+\Ric(\nabla \tilde{u},\nabla \tilde{u})+2\delta^2(n-1) |\nabla \tilde{u}|^2\Big)\\
&+Cr^{2-n}\int_{B_{r}(x)}\Big(|\nabla^2{u_k}|^2+\Ric(\nabla {u_k},\nabla {u_k})+2\delta^2(n-1) |\nabla {u_k}|^2\Big) \notag\\
\le &\frac{1001}{1000} r^{2-n}\int_{B_{r}(x)}\Big(|\nabla^2\tilde{u}|^2+\Ric(\nabla \tilde{u},\nabla \tilde{u})+2\delta^2(n-1) |\nabla \tilde{u}|^2\Big) \notag \\
&+Cr^{2-n}\int_{B_{r}(x)}\Big(|\nabla^2{v}|^2+\Ric(\nabla {v},\nabla {v})+2\delta^2(n-1) |\nabla {u_k}|^2\Big)\, ,
\end{align}
where we have used the fact that $|a_i|\le C(n)$ from the definition of $\tilde{u}$ in \eqref{e:u-uk}. Similarly, we have
\begin{align}
\int_{B_{1}(x)}\Big(|\nabla^2\tilde{u}|^2+\Ric(\nabla \tilde{u},\nabla \tilde{u})+2\delta^2(n-1) |\nabla \tilde{u}|^2\Big)\le  &\frac{1001}{1000}\int_{B_{1}(x)}\Big(|\nabla^2{u}|^2+\Ric(\nabla {u},\nabla {u})+2\delta^2(n-1) |\nabla {u}|^2\Big)\notag\\ 
&+C\int_{B_{1}(x)}\Big(|\nabla^2{u_k}|^2+\Ric(\nabla {u_k},\nabla {u_k})+2\delta^2(n-1) |\nabla {u_k}|^2\Big)  \notag\\
\le &\frac{1001}{1000}\int_{B_{1}(x)}\Big(|\nabla^2{u}|^2+\Ric(\nabla {u},\nabla {u})+2\delta^2(n-1) |\nabla {u}|^2\Big)\notag\\ 
&+C\int_{B_{1}(x)}\Big(|\nabla^2{v}|^2+\Ric(\nabla {v},\nabla {v})+2\delta^2(n-1) |\nabla {v}|^2\Big)\, .
\end{align}

By combining the above with \eqref{e:decay_u-uk} we get
	\begin{align}\nonumber
		r^{2-n}\int_{B_{r}(x)}\Big(|\nabla^2{u}|^2+\Ric(\nabla {u},\nabla {u})+2\delta^2(n-1) |\nabla {u}|^2\Big)\le &\frac{1}{3}\int_{B_1(x)}\Big(|\nabla^2{u}|^2+\Ric(\nabla {u},\nabla {u})+2\delta^2(n-1)|\nabla {u}|^2\Big)\\ \nonumber
&+Cr^{2-n}\int_{B_{r}(x)}\Big(|\nabla^2{v}|^2+\Ric(\nabla {v},\nabla {v})+2\delta^2(n-1) |\nabla {v}|^2\Big)\\
&+C\int_{B_{1}(x)}\Big(|\nabla^2{v}|^2+\Ric(\nabla {v},\nabla {v})+2\delta^2(n-1) |\nabla {v}|^2\Big).
	\end{align}
Since $r\ge c(n,\rv,\eta)>0$, we have
	\begin{align}\nonumber
		r^{2-n}\int_{B_{r}(x)}\Big(|\nabla^2{u}|^2+\Ric(\nabla {u},\nabla {u})+2\delta^2(n-1) |\nabla {u}|^2\Big)\le &\frac{1}{3}\int_{B_1(x)}\Big(|\nabla^2{u}|^2+\Ric(\nabla {u},\nabla {u})+2\delta^2(n-1) |\nabla {u}|^2\Big)\\
&+C\int_{B_{1}(x)}\Big(|\nabla^2{v}|^2+\Ric(\nabla {v},\nabla {v})+2\delta^2(n-1) |\nabla {v}|^2\Big).
	\end{align}
On the other hand, the Sharp Cone-splitting Theorem \ref{t:sharp_splitting} gives
\begin{align}
\int_{B_{1}(x)}\Big(|\nabla^2{v}|^2+\Ric(\nabla {v},\nabla {v})+2\delta^2(n-1) |\nabla v|^2\Big)\le C(n,\rv,\alpha)\cE_{1}^{k}(x).
\end{align}
Therefore,
	\begin{align}
		r^{2-n}\int_{B_{r}(x)}\Big(|\nabla^2{u}|^2+&\Ric(\nabla {u},\nabla {u})+2\delta^2(n-1) |\nabla {u}|^2\Big)\notag\\
		&\leq 
		\frac{1}{3}\int_{B_1(x)}\Big(|\nabla^2{u}|^2
		+\Ric(\nabla {u},\nabla {u})
		+2\delta^2(n-1) |\nabla {u}|^2\Big)
		+C\cE_{1}^{k}(x)\, .
	\end{align}
This completes the proof  of Proposition \ref{p:decay_hessian}
\end{proof}

\vspace{.3cm}
\subsubsection{The proof of Proposition \ref{p:decay_hessian_transformation}}
\label{ss:decay_hessian_transformation}

Let $\epsilon>0$ small be fixed later.  By Proposition \ref{p:transfor_prop} (1), which has been proven at this stage, if $\delta\leq \delta(n,\rv,\eta,\epsilon)$ then for each $r_0\le r\le 1$ we have a $k\times k$ lower triangle matrix $T_r$ such that $T_ru$ is $(k,\epsilon)$-splitting map on $B_r(x)$ with $|T_{r/2}\circ T_r^{-1}-I|\le \epsilon$.   Applying Proposition \ref{p:decay_hessian} to $T_su$, we get for all  $cs/2\le r\le cs$ that 
	\begin{align}
		r^{2-n}\int_{B_{r}(x)}&\Big(|\nabla^2T_{s}u|^2+\Ric(\nabla T_{s}u,\nabla T_{s}u)+2\delta^2(n-1) |\nabla T_{s}u|^2\Big)\notag\\
		&\le \frac{1}{3}s^{2-n}\int_{B_s(x)}\Big(|\nabla^2T_su|^2+\Ric(\nabla T_su,\nabla T_su)+2\delta^2(n-1) |\nabla T_su|^2\Big)+C\cE_{s}^{\alpha,\delta,k}(x).
	\end{align}
Fix $\epsilon\le \epsilon(n,\rv,\eta)$ such that $|T_r\circ T_s^{-1}-I|\le 10^{-10n}$. We have 
	\begin{align}
		r^{2-n}\int_{B_{r}(x)}&\Big(|\nabla^2T_{r}u|^2+\Ric(\nabla T_{r}u,\nabla T_{r}u)+2\delta^2(n-1) |\nabla T_{r}u|^2\Big)\notag\\
		&\le \frac{3}{2}r^{2-n}\int_{B_{r}(x)}\Big(|\nabla^2T_{s}u|^2+\Ric(\nabla T_{s}u,\nabla T_{s}u)+2\delta^2(n-1) |\nabla T_{s}u|^2\Big)\notag\\
		&\le \frac{1}{2}s^{2-n}\int_{B_s(x)}\Big(|\nabla^2T_su|^2+\Ric(\nabla T_su,\nabla T_su)+2\delta^2(n-1) |\nabla T_su|^2\Big)+C\cE_{s}^{k}(x).
	\end{align}
This completes the proof of Proposition \ref{p:decay_hessian_transformation}. \qed

\vspace{.1cm}
\subsection{Proof of the Geometric Transformation Theorem}\label{ss:Proving_Transformation}
For any $0<\delta'<\epsilon$ if $\delta\le \delta(n,\rv,\eta,\delta')$, then by the Transformation Proposition \ref{p:transfor_prop}.1 we have for each scale $r\le s\le 1$ a lower triangle matrix $T_s$ such that $T_su: B_s(x)\to\mathbb{R}^k$ is a $(k,\delta')$-splitting map. In particular, $T_su: B_s(x)\to\mathbb{R}^k$ is $(k,\epsilon)$ splitting. Therefore, it suffices to estimate the hessian for $T_su$.  

First we choose $\delta'\le \delta'(n,\rv,\eta,\epsilon)<\epsilon$ small such that Proposition \ref{p:decay_hessian_transformation} holds.
Therefore, by Proposition \ref{p:decay_hessian_transformation}, for any $r\le s\le 1$ we have 
\begin{align}
		(cs)^{2-n}\int_{B_{cs}(x)}&\Big(|\nabla^2T_{cs}u|^2+\Ric(\nabla T_{cs}u,\nabla T_{cs}u)+2\delta^2(n-1) |\nabla T_{cs}u|^2\Big)\notag\\
		&\le \frac{1}{2}s^{2-n}\int_{B_s(x)}\Big(|\nabla^2T_su|^2+\Ric(\nabla T_su,\nabla T_su)+2\delta^2(n-1) |\nabla T_su|^2\Big)+C\cE_{{s}}^{k}(x),\notag\\
		&\le c^{\gamma} s^{2-n}\int_{B_s(x)}\Big(|\nabla^2T_su|^2+\Ric(\nabla T_su,\nabla T_su)+2\delta^2(n-1) |\nabla T_su|^2\Big)+C\cE_{{s}}^{k}(x),
	\end{align}

where we can take $c=2^{-i_0}$ for some integer $i_0(n,\rv,\eta)$ and $\gamma=i_0^{-1}$.  Thus, for $s_\ell=c^\ell$, we have 
	\begin{align}\label{e:hessian_Trc}
		&s_{\ell}^{2-n}\int_{B_{s_\ell}(x)}\Big(|\nabla^2T_{s_\ell}u|^2+\Ric(\nabla T_{s_\ell}u,\nabla T_{s_\ell}u)+2\delta^2(n-1) |\nabla T_{s_\ell}u|^2\Big)\notag\\
		&\le \left(\frac{s_0}{s_{\ell}}\right)^{-\gamma} s_0^{2-n}\int_{B_{s_0}(x)}\Big(|\nabla^2T_{s_0}u|^2+\Ric(\nabla T_{s_0}u,\nabla T_{s_0}u)+2\delta^2(n-1) |\nabla T_{s_0}u|^2\Big)+C\sum_{j=0}^{\ell-1} \left(\frac{s_{j+1}}{s_{\ell}}\right)^{-\gamma}\cE_{{s}_j}^{k}(x)\notag\\
		&\le C\sum_{j=0}^{\ell-1} c^{\gamma(\ell-j))}\Big(\cE_{{s}_j}^{k}(x)+{s}_j^{2}\delta^2\Big):=   \tilde{\cE}_{s_\ell}^k(x),
	\end{align}
	where in the last inequality we have used the fact that 
$$
s_0^{2-n}\int_{B_{s_0}(x)}\Big(|\nabla^2T_{s_0}u|^2+\Ric(\nabla T_{s_0}u,\nabla T_{s_0}u)+2\delta^2(n-1) |\nabla T_{s_0}u|^2\Big)\le \delta^2\,  .
$$
 For general $s>r$ with $c^{\ell+1}<s\le c^\ell$, we have 
	\begin{align}
	s^{2-n}\int_{B_{s}(x)}\Big(|\nabla^2T_{s}u|^2+&\Ric(\nabla T_{s}u,\nabla T_{s}u)+2\delta^2(n-1) |\nabla T_{s}u|^2\Big)\notag\\
	&\le Cs^{2-n}\int_{B_{s}(x)}\Big(|\nabla^2T_{s_\ell}u|^2+\Ric(\nabla T_{s_\ell}u,\nabla T_{s_\ell}u)+2\delta^2(n-1) |\nabla T_{s_\ell}u|^2\Big)\notag\\
	&\le s_{\ell}^{2-n}\int_{B_{s_\ell}(x)}\Big(|\nabla^2T_{s_\ell}u|^2+\Ric(\nabla T_{s_\ell}u,\nabla T_{s_\ell}u)+2\delta^2(n-1) |\nabla T_{s_\ell}u|^2\Big)\notag\\
	&\le C\tilde{\cE}_{s_\ell}^k(x),
	\end{align}
	where we use the estimate $|T_s\circ T_{s_\ell}^{-1}-I|\le \epsilon$ in the first inequality. 
This completes  the proof of Theorem \ref{t:transformation}, the 	
Geometric Transformation Theorem. 
	
\vspace{1cm}

\section{Nondegeneration of $k$-Splittings}
\label{s:nondegeneration}

In this section we state and prove Theorem \ref{t:nondegeneration},
which is our
 our main result for $k$-splitting maps
$u:B_2(p)\to \dR^k$.   
Theorem \ref{t:nondegeneration}
is a crucial ingredient in the proof of Theorem \ref{t:neck_region2}.
\vskip2mm

Essentially Theorem \ref{t:nondegeneration} is obtained by
combining Theorem \ref{t:sharp_splitting}, the Sharp Cone-Splitting Theorem,
 Theorem \ref{t:transformation}, the Transformation Theorem,
with Proposition \ref{p:telescope_harmonic}, 
and a telescoping estimate for harmonic functions
 based on a monotonicity property.  This telescoping estiamte is much sharper
than the corresponding more general telescoping estimate 
 for $W^{1,p}$ functions. In the proof of
Theorem \ref{t:nondegeneration}, this is essential. It
allows us to adequately control the sum over arbitrarily many scales of the Hessian estimates in Theorem \ref{t:sharp_splitting} and Theorem \ref{t:transformation}.

Recall that $\cE^{k,\alpha,\delta}$ is the entropy pinching defined in Definition \ref{d:kalphadeltapinching}.

\begin{theorem}[Nondegeneration of $k$-Splittings]
\label{t:nondegeneration}
Given $\epsilon,\eta,\alpha>0$ and $\delta < \delta(n,\rv,\eta,\alpha,\epsilon)$ we have the following.  Let $(M^n,g,p)$ satisfy $\Ric_{M^n}\geq -(n-1)\delta^2$, $\Vol(B_1(p))>v>0$, and let $u:B_{2}(p)\to \dR^k$ denote a $(k,\delta)$-splitting function.  Assume:

\begin{itemize}
\item $B_{\delta^{-1}s}(p)$, is $(k,\delta^2)$-symmetric but $B_s(p)$ is not $(k+1,\eta)$-symmetric for all $r\leq s\leq 1$. 
 
 \item $\sum_{r_j\geq r} \cE^{k,\delta,\alpha}_{r_j}(p) < \delta$ where $r_j = 2^{-j}$.
 \end{itemize}   

Then   $u:B_s(p)\to \dR^k$ is an $\epsilon$-splitting function  for every $r\leq s\leq 1$. 
\end{theorem}
\vspace{.2cm}


From the Transformation Theorem \ref{t:transformation}, we know that for some lower triangular matrix $T_r=T(p,r)$ that the compososition $T_ru: B_r(p)\to\mathbb{R}^k$ is a $\delta$-splitting function.  Our goal then is to show that under the above hypotheses that $T_r$ remains close to the identity.  Proposition \ref{p:telescope_harmonic} below provides suitable control of the difference $|T_r\circ T_{2r}^{-1}-I| $.
 From this the Nondegeneration Theorem, \ref{t:nondegeneration} will easily
 follow.  
 
 
\subsection{Hessian Estimates with respect to the heat kernel density}
The purpose of this subsection is to prove some technical results which convert ball-average estimates on $Tu$ into estimates with respect to the heat kernel measure, which is important due to our use of entropy as the monotone quantity.  
\vskip2mm

\noindent
{\bf Notation.}  Throughout this section $\varphi$ denotes a cutoff function as in \eqref{e:cutoff}, with support in $B_{1}(x)$, with
 $\varphi:=  1$ on $B_{1/2}(x)$ and such that
 $
 |\Delta \varphi|+|\nabla \varphi|^2\le C(n)\, .
 $
 \vskip2mm

 The main result of this subsection is the following technical 
 proposition. 
\begin{proposition}
\label{p:global_integral_heatkernel}
Given $\alpha,\eta>0$ and $\epsilon>0$ there exist $\delta\leq \delta(n,\rv,\eta,\alpha,\epsilon)$,
 $\gamma=\gamma(n,\rv,\eta)>0$, $C(n,\rv,\eta,\alpha)$, $C(n,\rv)$,
with the following properties.
	Let $(M^n,g,x)$ satisfy $\Ric_{M^n}\ge -(n-1)\delta^2$, $\Vol(B_1(x))\ge \rv>0$, and let $u: B_{2}(x)\to \mathbb{R}^k$ be a $(k,\delta)$- splitting map.  Assume
\begin{itemize}
\item $B_{\delta^{-1}s}(x)$ is $(k,\delta^2)$-symmetric and $B_s(x)$ is not $(k+1,\eta)$-symmetric for all $r\le s\le 1$. 
\end{itemize}
	  Then for each $r\le s_i\le 1$ there exists a $k\times k$ lower triangular matrix $T_{s_i}$ such that $T_{s_i}u: B_{s_i}(x)\to \mathbb{R}^k$ is 
a $(k,\epsilon)$-splitting map with
\begin{align}
\label{e:onheat}
		\int_{M^n}\langle \nabla ({T}_{s_i}u)^a,\nabla ({T}_{s_i}u)^b\rangle \varphi^2\rho_{s_i^2}(x,dy)&=\delta^{ab}\, ,
	\end{align}
	and the Hessian estimate of $T_{s_i}u$ satisfies 
\begin{align}\label{e:global_integral}
		s_i^{2}\int_{M}\Big(|\nabla^2T_{s_i}u|^2+\Ric(\nabla T_{s_i}u,\nabla T_{s_i}u)+2\delta^2(n-1) |\nabla T_{s_i}u|^2\Big)\varphi^2 \rho_{4s_i^2}(x,dy)\le C(n,\rv)\sum_{j=0}^i \epsilon_j 2^{j-i}\,
\end{align}
where
\begin{align}
\label{e:sumest}
\epsilon_i= C(n,\rv,\eta,\alpha)
\cdot\sum_{j=0}^i 2^{-\gamma (i-j)}\Big(\cE_{s_j}^{k}(x)+\delta s_j^2\Big)\, .
\end{align}
\end{proposition}

\begin{proof} (of Proposition \ref{p:global_integral_heatkernel})
Note that by Theorem \ref{t:transformation}, for any $\epsilon'$ if $\delta\le \delta(n,\rv,\eta,\epsilon')$ then there exists $\tilde{T}_{s_i}$ such that $\tilde{T}_{s_i}: B_{s_i}(x)\to\mathbb{R}^k$ is a $(k,\epsilon')$-splitting map and the hessian satisfies 
	\begin{align}\label{e:tildeTsiHessian}
		s_i^{2-n}\int_{B_{s_i}(x)}\Big(|\nabla^2\tilde{T}_{s_i}u|^2+\Ric(\nabla \tilde{T}_{s_i}u,\nabla \tilde{T}_{s_i}u)+2\delta^2(n-1) |\nabla \tilde{T}_{s_i}u|^2\Big)\le C(n,\rv,\eta,\alpha)\sum_{j=0}^i 2^{-\gamma (i-j)}\Big(\cE_{s_j}^{k}(x)+\delta s_j^2\Big)\, .
		\end{align}
		Denote 
		\begin{align}
\label{e:epsilon}
	\epsilon_i:=  C(n,\rv,\eta,\alpha)\sum_{j=0}^i 2^{-\gamma (i-j)}\Big(\cE_{s_j}^{k}(x)+\delta s_j^2\Big)\, .
\end{align}
In order to make sure matrix $\tilde{T}_{s_i}$ satisfies \eqref{e:onheat}, we need to do a rotation as the following Lemma \ref{l:claim1} and then we can fix $\epsilon'=\epsilon'(n,\epsilon,v)$ so that 
$T_{s_i}u: B_{s_i}(x)\to\mathbb{R}^k$ is $(k,\epsilon)$-splitting. 
	
	
\begin{lemma}
\label{l:claim1}
 For any $\epsilon>0$ if $\epsilon'\le\epsilon'(n,\epsilon,\rv)$ and $\delta\le \delta(n,\rv,\eta,\alpha,\epsilon)$, then there exists lower triangle matrix $A_{i}$ with $|A_i-I|\le C(n)\epsilon$ such that $T_{s_i}=A_i\circ \tilde{T}_{s_i}$ satisfies 	
 \begin{align}
 \label{e:claim1}
 \int_{M^n}\langle \nabla ({T}_{s_i}u)^a,\nabla ({T}_{s_i}u)^b\rangle \varphi^2\rho_{s_i^2}(x,dy)=\delta^{ab}\, ,
 \end{align}
 and $T_{s_i}u: B_{s_i}(x)\to \mathbb{R}^k$ is $(k,\epsilon)$-splitting. 
 \end{lemma}
\begin{proof}	
For any $\epsilon$, by the exponential heat kernel decay estimate 
in Theorem \ref{t:heat_kernel} and the matrix estimate in Proposition \ref{p:transfor_prop}, there exists $R(n,\rv,\epsilon)$ such that
		\begin{align}
		\int_{B_1(x)\setminus B_{Rs_i}(x)}|\langle \nabla (\tilde{T}_{s_i}u)^a,\nabla (\tilde{T}_{s_i}u)^b\rangle -\delta^{ab}|\cdot\rho_{s_i^2}(x,dy)<\epsilon/2.	
 \end{align}
		Also, by Proposition \ref{p:transfor_prop}, for any $\epsilon'>0$ if $\delta\le \delta(\epsilon',n,\rv,\eta)$ then we have the matrix growth estimate 
$$
|\tilde{T}_{s_i}\tilde{T}^{-1}_{s_j}-I|\le \left(\frac{s_j}{s_i}\right)^{\epsilon'}-1\, ,
$$ 
for any $s_i\le s_j\le 1$. Therefore, if $\delta\le \delta(\epsilon,n,\rv,\eta)$, we have  
\begin{align}
		\fint_{B_{Rs_i}(x)}|\langle \nabla (\tilde{T}_{s_i}u)^a,\nabla (\tilde{T}_{s_i}u)^b\rangle -\delta^{ab}|<\epsilon/2.
  \end{align}
These two estimates imply 
$$
\int_{M}\Big|\langle \nabla (\tilde{T}_{s_i}u)^a,\nabla (\tilde{T}_{s_i}u)^b\rangle-\delta^{ab}\Big| \varphi^2\cdot\rho_{s_i^2}(x,dy)\le \epsilon\, .
$$ 
By using the Gram-Schmidt process, there exists lower triangle matrix $A_i$ satisfying 
\eqref{e:claim1}. This completes the proof of 
Lemma \ref{l:claim1}. 
\end{proof}
	

To finish the proof of Proposition \ref{p:global_integral_heatkernel}, it suffices to prove  \eqref{e:global_integral}. 

 Since $T_{s_i}=A_i\circ \tilde{T}_{s_i}$ with bounded $A_i$, we have from \eqref{e:tildeTsiHessian} that  
		\begin{align}\label{e:scal_ri_integral}
		s_j^{2-n}\int_{B_{s_j}(x)}\Big(|\nabla^2{T}_{s_j}u|^2+\Ric(\nabla {T}_{s_j}u,\nabla {T}_{s_j}u)+2\delta^2(n-1) |\nabla {T}_{s_j}u|^2\Big)\le C(n)\epsilon_j\, .
		\end{align}
		To prove \eqref{e:global_integral}, we only need to use the estimates \eqref{e:scal_ri_integral} for each scale $s_j\ge s_i$ and the heat kernel estimates in Theorem \ref{t:heat_kernel}.  By the H\"older growth estimate for transformation matrices in Proposition \ref{p:transfor_prop}, if $\delta\le \delta(n,\rv,\eta)$ is small then we have $|T_{s_i}T_{s_j}^{-1}| \le 2^{(i-j)/100}$. Therefore,  for $s_i\le s_j\le 1$, 
we have 
\begin{align}
\int_{A_{s_{j+1},s_j}(x)}\Big(|\nabla^2T_{s_i}u|^2+\Ric(\nabla T_{s_i}u,\nabla T_{s_i}u)+2\delta^2(n-1) |\nabla T_{s_i}u|^2\Big)\le s_j^{n-2} 2^{(i-j)/10} \epsilon_j.
\end{align}
In particular, by the heat kernel estimates of Theorem \ref{t:heat_kernel}, we have 
\begin{align}
{}\!\!\!\!\!\!\!\!\!\!\!\!\!\!\!\!\!\!\!\!\!\!\!\!\!\!\!\!\!\!\!\!
\int_{A_{s_{j+1},s_j}(x)}\Big(|\nabla^2
&T_{s_i}u|^2+\Ric(\nabla T_{s_i}u,\nabla T_{s_i}u)+2\delta^2(n-1) |\nabla T_{s_i}u|^2\Big)\cdot 
\rho_{4s_i^2}(x,dy)\notag\\
&\le C(n,\rv) s_i^{-n} e^{-\frac{s_j^2}{20s_i^2}}s_j^{n-2} 2^{(i-j)/10} \epsilon_j
\,\,\le \, \, C(n,\rv) s_j^{-2} 2^{(i-j)/10} \epsilon_j\, .
\end{align}

Thus, 
\begin{align}
		s_i^2\int_{B_{1}(x)}&\Big(|\nabla^2T_{s_i}u|^2+\Ric(\nabla T_{s_i}u,\nabla T_{s_i}u)+2\delta^2(n-1) |\nabla T_{s_i}u|^2\Big)\cdot\rho_{{4s_i^2}}
	(x,dy)\notag\\
&\le s_i^2\left( \int_{B_{s_i}(x_i)}+\sum_{j=0}^{i-1}\int_{A_{s_{j+1},s_j}(x)}\right)\Big(|\nabla^2T_{s_i}u|^2+\Ric(\nabla T_{s_i}u,\nabla T_{s_i}u)+2\delta^2(n-1) |\nabla T_{s_i}u|^2\Big)\rho_{4s_i^2}(x,dy)\notag\\
&\le C(n,\rv)\epsilon_i+ C(n,\rv)\sum_{j=0}^{i-1}2^{2(j-i)} 2^{(i-j)/10}\epsilon_j
\notag\\
&\le C(n,\rv)\sum_{j=0}^i \epsilon_j 2^{j-i}.
	\end{align}
This implies \eqref{e:global_integral}, which completes the proof
of Proposition \ref{p:global_integral_heatkernel}.
\end{proof}

\subsection{A telescope estimate for harmonic functions}
In this subsection, we prove a telescoping type estimate, Proposition
\ref{p:telescope_harmonic}, for harmonic functions
in  which  {\it the squared $L^2$-norm of the Hessian linearly controls the difference of the norm of the gradient on concentric balls}; see \eqref{e:telescope_estimate1}.
  For a function
 which is not harmonic, 
 the squared $L^2$-norm would have to 
 be replaced by the $L^2$ norm itself.
 This weaker estimate would
 not suffice for our purposes.
 \vskip2mm


Let $\varphi$ be a cutoff function with support in $B_{1}(x)$ and $\varphi:=  1$ in $B_{1/2}(x)$ such that $|\Delta\varphi|+|\nabla\varphi|\le C(n)$.

\begin{proposition}
\label{p:telescope_harmonic}
			Let $(M^n,g,x)$ satisfy $\Ric_{M^n}\ge -(n-1)\delta^2$, $\Vol(B_1(x))\ge \rv>0$ and $0<s<1$. Assume $u_1,u_2: B_{2}(x)\to \mathbb{R}$ are two harmonic functions satisfying polynomial growth condition $\sup_{B_r(x)}\Big(|\nabla u_1|(y)+|\nabla u_2|(y)\Big)\le K(1+s^{-1}r)$ for all $0<r\le 2$\footnote{After rescaling $B_s(x)$ to $B_1(x)$, this condition just means that $|\nabla u|$ has linear growth in $B_{2s^{-1}}(x)$.}.  Then 
\begin{align}
\label{e:telescope_estimate1}
				&\Big|\int_{M}\langle \nabla u_1,\nabla u_2\rangle \varphi^2\rho_{s^2}(x,dy)-\int_{M}\langle \nabla u_1,\nabla u_2\rangle \varphi^2\rho_{4s^2}(x,dy)\Big|\notag\\
				&\le C(n)\sum_{i=1}^2s^2\int_M\left( |\nabla^2 u_i|^2+\Ric(\nabla u_i,\nabla u_i)+2(n-1)\delta^2|\nabla u_i|^2\right)\varphi^2\rho_{8s^2}(x,dy)+C(n,\rv,K) e^{-\frac{1}{100s^2}}\, .
\end{align}
\end{proposition}
\begin{remark}
We will apply Proposition \ref{p:telescope_harmonic} with $u_1,u_2$ different components of $T_{s}u, T_{s}u$ from Proposition \ref{p:global_integral_heatkernel}, which asserts that
\begin{align}
\sup_{B_r(x)}|\nabla T_su|\le C(n) (1+\frac{r}{s}),~~~ \mbox{ for all $0<r\le 2$.}
\end{align}
\end{remark}

\begin{proof}
	From Bochner's formula,  we get
\begin{align}
\label{e:bochnerheat}
		\Big|\partial_t \int_{M}\langle \nabla u_1,\nabla u_2\rangle \varphi^2\rho_{t}(x,dy)\Big|
		&=\Big|\int_{M}\langle \nabla u_1,\nabla u_2\rangle \varphi^2\Delta \rho_{t}(x,dy)\Big|\notag\\ 
		&=\Big|\int_M \left(\Delta \langle \nabla u_1,\nabla u_2\rangle \varphi^2+\Delta \varphi^2\langle \nabla u_1,\nabla u_2\rangle+ 2\varphi \langle \nabla \varphi,\nabla \langle \nabla u_1,\nabla u_2\rangle \rangle\right)\rho_t(x,dy)\Big|  \notag\\	
&\le C(n)\sum_{i=1}^2\int_M\left
(
|\nabla^2 u_i|^2
+\Ric(\nabla u_i,\nabla u_i)
+2(n-1)\delta^2|\nabla u_i|^2
\right
)
\varphi^2 \rho_{t}(x,dy)\notag\\
&{}\qquad\qquad\qquad+C(n)\sum_{i=1}^2\int_{A_{1/2,1}(x)} |\nabla u_i|^2 \rho_t(x,dy)\, ,
	\end{align}
where in the last inequality we used 	
\begin{align}
|\Ric(\nabla u_1,\nabla u_2)|
\le C(n)\sum_{i=1}^2\left(\Ric(\nabla u_i,\nabla u_i)+2(n-1)\delta^2|\nabla u_i|^2\right)\, .\notag
		\end{align}
By using the heat kernel estimate in Theorem \ref{t:heat_kernel}, 
we can control the last term on the right-hand side of the last line of  \eqref{e:bochnerheat}. Namely, for all $t\le 1$, we have
	\begin{align}
		\sum_{i=1}^2\int_{A_{1/2,1}(x)}|\nabla u_i|^2 \rho_t(x,dy)\le C(n,\rv,K)s^{-2} t^{-n/2}e^{-\frac{1}{20 t}}.
	\end{align}
Therefore: 
	\begin{align}
		&\Big|\int_{M^n}\langle \nabla u_1,\nabla u_2\rangle \varphi^2 \rho_{s^2}(x,dy)-\int_{M^n}\langle \nabla u_1,\nabla u_2\rangle \varphi^2 \rho_{4s^2}(x,dy)\Big|\notag\\
		&=\Big|\int_{s^2}^{4s^2}\partial_t \int_{M^n}\langle \nabla u_1,\nabla u_2\rangle \varphi^2 \rho_{t}(x,dy)dt\Big|\notag\\ 
		&\le C(n)\int_{s^2}^{4s^2}\sum_{i=1}^2\int_{M^n}\left( |\nabla^2 u_i|^2+\Ric(\nabla u_i,\nabla u_i)+2(n-1)\delta^2|\nabla u_i|^2\right)\varphi^2 \rho_{t}(x,dy)dt+\notag\\
&{}\qquad\qquad\qquad\qquad\qquad+C(n,\rv,K)\delta^2 \int_{s^2}^{4s^2}s^{-2}t^{-n/2}e^{-\frac{1}{20 t}}dt\notag
\end{align}
Hence 
\begin{align}
&\Big|\int_{M^n}\langle \nabla u_1,\nabla u_2\rangle \varphi^2 \rho_{s^2}(x,dy)-\int_{M^n}\langle \nabla u_1,\nabla u_2\rangle \varphi^2 \rho_{4s^2}(x,dy)\Big|\notag\\
&\le C(n)\int_{s^2}^{4s^2}\sum_{i=1}^2\int_M\left( |\nabla^2 u_i|^2+\Ric(\nabla u_i,\nabla u_i)+2(n-1)\delta^2|\nabla u_i|^2\right)\varphi^2 \rho_{t}(x,dy)dt+C(n,\rv,K)s^{-n}e^{-\frac{1}{80s^2}}\notag\\
&\le C(n)\sum_{i=1}^2s^2\int_{M^n}\left( |\nabla^2 u_i|^2+\Ric(\nabla u_i,\nabla u_i)+2(n-1)\delta^2|\nabla u_i|^2\right)\varphi^2 \rho_{8s^2}(x,dy)+C(n,\rv,K)e^{-\frac{1}{100s^2}},
			\end{align}
	where we have used the heat kernel estimate in Theorem \ref{t:heat_kernel} to conclude that $\rho_t(x,y)\le C(n,\rv)\rho_{8s^2}(x,y)$ for any $s^2\le t\le 4s^2$ and $y\in B_{1}(x)$. This completes the proof of Proposition \ref{p:telescope_harmonic}.
\end{proof}


\vspace{.2cm}
\subsection{Proof of Theorem \ref{t:nondegeneration}}

By Proposition \ref{p:global_integral_heatkernel} for any $\epsilon'$ if $\delta\le \delta(n,\rv,\eta,\alpha,\epsilon')$ then for each $s_i=2^{-i}$, there exists a lower triangular $k\times k$ matrix $T_{s_i}$ such that $T_{s_i}u: B_{s_i}(x)\to \mathbb{R}^k$ is $(k,\epsilon')$ splitting with 
	\begin{align}		&s_i^{2}\int_{M}\left(|\nabla^2T_{s_i}u|^2+\Ric(\nabla T_{s_i}u,\nabla T_{s_i}u) +2\delta^2(n-1) |\nabla T_{s_i} u|^2 \right)\varphi^2\rho_{4s_i^2}(x,dy)\le C(n,\rv)\sum_{j=0}^i \epsilon_j 2^{j-i}:=   \chi_i,\\ 
	\label{e:ksplitting_Ti1}
		&\epsilon_i= C(n,\rv,\eta) \sum_{j=0}^i 2^{-\gamma(i-j)} \Big(\cE_{s_j}^{k}(x)+\delta s_j^2\Big)\, ,\\ \label{e:nablaTipq}
		&\int_{M}\langle \nabla (T_{s_i}u)^a,\nabla (T_{s_i}u)^b\rangle \varphi^2 \rho_{s_i^2}(x,dy)=\delta^{ab},
	\end{align}
	where $\gamma(n,\rv,\eta)>0$ and $\varphi$ is cutoff function with support in $B_{1}(x)$ and $\varphi:=  1$ on $B_{1/2}(x)$.
	By the estimate \eqref{e:ksplitting_Ti1} for $\epsilon_i$ we get 
	\begin{align}
	&\sum_{i=0}^m \epsilon_i\le C(n,\rv,\eta)\sum_{i=0}^m\, \sum_{j=0}^i 2^{-\gamma(i-j)}\Big( \cE_{s_j}^{k}(x)+\delta s_j^2\Big)\le C(n,\rv,\eta)\sum_{j=0}^m\Big(\cE_{s_j}^{k}(x)+\delta s_j^2\Big)\le C(n,\rv,\eta)\delta.		\\
&\sum_{i=0}^m\chi_i\le C(n,\rv)\sum_{i=0}^m\,\sum_{j=0}^i\epsilon_j 2^{j-i}\le C(n,\rv)\sum_{j=0}^m\epsilon_j\le C(n,\rv,\eta)\delta.
	\end{align}
\vskip2mm

\begin{lemma}
\label{l:t}
For any $\epsilon'$, let $\delta\le \delta(n,\rv,\eta,\epsilon',\alpha)$.  Then $|T_{s_m}-I|\le \epsilon'$ for any $m\ge 1$ such that $s_m\ge r$.
\end{lemma} 
\begin{proof}
 First note that by Proposition \ref{p:transfor_prop} that $|\nabla T_{s_i}u|$ satisfies H\"older growth estimates (see also \eqref{e:matrix_holdergrowth}). Thus, we can apply Proposition \ref{p:telescope_harmonic} to obtain
\begin{align}
\Big|\int_{M}\langle \nabla ({T}_{s_i}u)^a,\nabla ({T}_{s_i}u)^b\rangle \varphi^2 \rho_{s_{i+1}^2}(x,dy)-\delta^{ab}\Big|\le C(n)\chi_i+C(n,\rv)e^{-\frac{1}{100s_i^2}}:=  \tilde{\chi}_i.
\end{align}
For any $\epsilon''$, there exists integer $N(\epsilon'',n,\rv)$ such that if $i\ge N$ and $\delta\le \delta(n,\rv,\eta,\alpha,\epsilon')$,  we have 
\begin{align}
	\sum_{j=N}^i\tilde{\chi}_j\le \epsilon''.
\end{align}
By using the Gram-Schmidt process, there exists lower triangle matrix $\tilde{A}_i$ with $|\tilde{A}_i-I|\le C(n)\tilde{\chi}_i$ such that $\hat{T}_{s_i}=\tilde{A}_i\circ {T}_{s_i}$ satisfies 
\begin{align}
\int_{M^n}\langle \nabla (\hat{T}_{s_i}u)^a,\nabla (\hat{T}_{s_i}u)^b\rangle \varphi^2 \rho_{s_{i+1}^2}(x,dy)=\delta^{ab}.
\end{align}
Since
\begin{align}
\int_{M}\langle \nabla ({T}_{s_{i+1}}u)^a,\nabla ({T}_{s_{i+1}}u)^b\rangle \varphi^2 \rho_{s_{i+1}^2}(x,dy)=\delta^{ab},
\end{align}
 the uniqueness of Cholesky decompositions (see also \cite{GV}) for positive definite symmetric matrices implies  that $\hat{T}_{s_i}={T}_{s_{i+1}}$.  In particular, we get 
 ${T}_{s_{i+1}}\circ {T}_{s_i}^{-1}=\tilde{A}_i$. Thus 
 \begin{align}
 	|{T}_{s_{i+1}}\circ {T}_{s_i}^{-1}-I|\le C(n)\tilde{\chi}_i.
 \end{align}
Recall that $T_{s_i}$ is a $k\times k$ matrix. Hence for all $i\ge N$ we have
\begin{align}\label{e:s_Ns_i1}
	|{T}_{s_{i+1}}\circ {T}_{s_N}^{-1}-I|\le \prod_{j=N}^i (1+(k+1)C(n)\tilde{\chi}_j)-1\le e^{\sum_{j=\ell}^i kC(n)\tilde{\chi}_j}-1\le C(n)\sum_{j=N}^i \tilde{\chi}_j\le C\epsilon''.
\end{align}
If $\delta\le \delta(\epsilon',\rv,n,\eta)$ and $\epsilon''\le \epsilon''(n,\rv,\epsilon')$, we have for all $i\le N$ that 
\begin{align}
	|T_{s_i}-I|\le \epsilon'/10.
\end{align}
Therefore, by \eqref{e:s_Ns_i1}, for any $i\ge N$, we have 
\begin{align}
	|{T}_{s_{i}}-I|\le \epsilon'.
\end{align}
This completes the proof of Lemma \ref{l:t}.
\end{proof}

Now we can complete the proof of Theorem \ref{t:nondegeneration} as follows.  

Since $T_{s_i}u: B_{s_i}(x)\to \mathbb{R}^k$ is $\epsilon'$-splitting when $\delta\le \delta(n,\rv,\epsilon',\eta,\alpha)$, to show $u:B_{s_i}(x)\to \mathbb{R}^k$ is $\epsilon$-splitting, it suffices to prove $T_{s_i}$ is bounded and then fix $\epsilon'=\epsilon'(n,\epsilon,\rv)$. The later has been proven in Lemma \ref{l:t}. Therefore we complete the proof of Theorem \ref{t:nondegeneration}.

\vspace{1cm}

\section{$(k,\delta,\eta)$-Neck Regions}
\label{s:neck_region}
This is the first of the two sections which constitute the fourth and last part of
the paper.  
In it we give the proof of the Neck Structure Theorem \ref{t:neck_region2} which for 
convenience, we have restated below.
Recall that neck regions are defined in Definition \ref{d:neck2}.
\vskip4mm

\noindent
{\bf Theorem} {\it (Theorem \ref{t:neck_region2} restated)
\label{t:neck_structure}
Fix $\eta>0$ and  $\delta\leq \delta(n,\rv,\eta)$.
Then if $\cN = B_2(p)\setminus \overline B_{r_x}(\cC)$ is a $(k,\delta,\eta)$-neck region, the following hold:

\begin{enumerate}
	\item For each $x\in \cC$ and $B_{2r}(x)\subset B_2(p)$ the induced packing measure $\mu$ is Ahlfors regular:
\begin{align}	
\label{e:structure:ar}
	 A(n)^{-1}r^{k}<\mu(B_r(x)) <A(n)r^{k} \, .
\end{align}
	
	\item $\cC_0$ is $k$-rectifiable.
\end{enumerate}
}
\vskip3mm


Results rectifiability of singular sets
obtained via cone-splitting 
 were first introduced in
\cite{NaVa_Rect_harmonicmap} in the context of nonlinear harmonic maps, and the notion of neck regions was first formally introduced and studied in \cite{JiNa_L2}.  As was discussed in Sections   \ref{s:outline_proof}, 
\ref{s:ouline_proof_neckstructure},
 in order to conclude the structural results we will need to build a map from the center points $\cC\to \dR^k$.  In \cite{NaVa_Rect_harmonicmap}, the relevant splitting map $u$ was built by hand using a Reifenberg construction.
This approach required new estimates on harmonic maps and a new bi-Lipschitz Reifenberg theorem.  As we have emphasized, for the case of lower Ricci curvature bounds,
the bi-Lipschitz Reifenberg ideas of \cite{NaVa_Rect_harmonicmap} {\it do not work}.
Attempting to implement them gives rise to additional error terms which are not 
summable over scales. Essentially, this is because
approximating a subset $W\subseteq X^n$ by \hbox{$k$-dimensional} subspace also involves
 approximating $X^n$ itself by a splitting. Instead, we  rely
 on the results of Sections 
 \ref{s:Sharp_splitting}--\ref{s:nondegeneration}, especially
   the Nondegeneration Theorem \ref{t:nondegeneration}.
\vskip2mm


In \cite{JiNa_L2},
results on structure and existence of
$(n-4)$-neck regions were proved under the assumption of a $2$-sided
bound on Ricci curvature.  In order to prove the final estimates in \cite{JiNa_L2} the authors introduced a new estimate which was termed a {\it superconvexity estimate}. This definitely requires a \hbox{$2$-sided} bound on the Ricci curvature.  The estimates of this paper are entirely different. As mentioned in the introduction, they give a new proof of the 
$(n-4)$-finiteness conjecture for limits with bounded Ricci curvature, first proved in \cite{JiNa_L2}.\\

We refer the reader to Section \ref{s:ouline_proof_neckstructure} for an outline of the  strategy 
for proving Theorem \ref{t:neck_structure}.
\vskip2mm

\subsection{The basic assumptions}
\label{s:basic_assumptions}
 Below, we will refer to the following standard assumptions.
 \vskip2mm
 
 Fix $\delta,\delta',\eta,B>0$.  We will assume:
\begin{align}
	\text{(S1)}\qquad& \Vol(B_1(p))>\rv>0 \text{ and }\Ric_{M^n}\ge -\delta^2(n-1).\qquad\qquad{}\notag\\
	\text{(S2)}\qquad& \cN=B_2(p)\setminus \bar B_{r_x}(\cC)\text{ is a $(k,\delta,\eta)$-neck region with associated packing measure $\mu$.}\notag\\
	\text{(S3)}\qquad& \text{\! For any $x\in \cC$ and $B_{2r}(x)\subset B_2(p)$ with $r\geq r_x$ we have}:
	\notag
\end{align}
\vskip-9mm
\begin{align}	
\label{e:B}	
		B^{-1}r^k\le \mu(B_r(x))\le Br^k\, .
\end{align}		
{}\qquad\,\,\,\,\,\,\,\,\,\,\,\text{(S4)} \,\,\,\, \,\,\,\,\,\,$u: B_4(p)\to \mathbb{R}^k$ is a $\delta'$-splitting map.
\vskip3mm

\begin{remark}
Recall from Section \ref{s:ouline_proof_neckstructure} that (S3) is connected to the strategy that we will prove the theorem inductively.  In particular, with $B>>A(n)$ we will want to eventually see that with $\delta$ sufficiently small (S3) automatically implies the stronger Ahlfors regularity estimate \eqref{e:structure:ar}. 
\end{remark}

\begin{remark}
\label{r:delta_delta'}
By the definition of neck regions and the Cone-Splitting Theorem \ref{t:splitting_function}, we can and will assume that  $ \delta(n,\rv,\eta,\delta')>0$ has been
chosen sufficiently small so that there exists a $\delta'$-splitting map $u: B_4(p)\to \mathbb{R}^k$. Then in actuality, the existence of $u$ as in (S4) is actually  a consequence of (S2).
\end{remark}
\vspace{1mm}

\subsection{Bi-Lipschitz structure of the set of centers of a
 neck region}\label{ss:BiLipschitz_neck}
This subsection is devoted to proving Proposition
 \ref{p:bilipschitz_structure}. 
Given a  $(k,\delta,\eta)$-neck region 
$\cN=B_2(p)\setminus \overline{B}_{r_x}(\cC)$,
Proposition
 \ref{p:bilipschitz_structure} implies the existence of  
 a subset set $\cC_\epsilon\subset \cC$, which is almost
 all of $\cC$, such that a splitting map $u:B_{2s}(x)\to \dR^k$ is $(1+\epsilon)$-bi-Lipschitz on $\cC_\epsilon$.  This is the key step
 which is used to improve the weak Ahlfors regularity estimate (S3) to the strong one \eqref{e:structure:ar} and to show that the singular set is rectifiable.  The results of the previous sections play a key role in the proof of
 Proposition \ref{p:bilipschitz_structure}; compare the outline in Section \ref{s:ouline_proof_neckstructure}:
 
 \begin{proposition}
\label{p:bilipschitz_structure}
For any $B,\epsilon, \eta>0$ if (S1)-(S4) hold with $\delta'\leq  \delta'(n,\rv,\eta, B,\epsilon)$ and $\delta\leq \delta(n,\rv,\eta,\delta',B,\epsilon)$, then there exists $\cC_\epsilon \subset \cC\cap B_{15/8}(p)$ such that:
	\begin{itemize}
		\item[(1)] $\mu\big(\cC_\epsilon \cap B_{15/8}(p)\big)\ge (1-\epsilon) \mu \big(\cC\cap B_{15/8}(p)\big)$.
\vskip2mm
		
\item[(2)] $u$ is $(1+\epsilon)$-bi-Lipschitz on $\cC_\epsilon$, i.e., for  any $x,y\in  \cC_\epsilon$:
\vskip-2mm

$$
(1+\epsilon)^{-1}\cdot d(x,y)\le |u(x)-u(y)|\le (1+\epsilon)\cdot d(x,y)\, .
$$		
\vskip-1mm

\item[(3)] For any $x\in \cC_\epsilon$ and $r\ge r_x$ with $B_{2r}(x)\subset B_2(p)$, the map $u: B_r(x)\to\mathbb{R}^k$ is  a $(k,\epsilon)$-splitting function.
\vskip2mm
		
\item[(4)] For any $x\in\cC_\epsilon$:
$$
\sum_{r_x\le r_i\le 2^{-5}} \fint_{B_{r_i}(x)}|\cW_{r_i^2/2}^\delta(y)-\cW_{2r_i^2}^\delta(y)|\, d\mu(y)\le \epsilon\, .
$$
\vskip3mm
		
		\item[(5)] $u: \cC\to \mathbb{R}^k$ is a bi-H\"older map onto its image i.e. for all $x,y\in B_{15/8}(p)\cap \cC$: 	
$$
(1-\epsilon)\cdot d(x,y)^{1+\epsilon}\le |u(x)-u(y)|\le (1+\epsilon)\cdot d(x,y)\, .
$$ 
	\end{itemize} 
\end{proposition}
\noindent
{\bf Note.} In (4), the integral average is taken with respect to  $\mu$.
\vspace{.25cm}

Essentially,  $\cC_\epsilon\subset \cC$  
consists of those points which satisfy (4). We will see, as in (1), that most points of $\cC$
have this property.  Then using Theorem \ref{t:nondegeneration} we will conclude (3).  The estimates (2) and (5) will follow almost verbatim as the argument from Section \ref{ss:Reifenberg}
\vskip2mm

We begin with some technical lemmas which will be used in the proof of 
Proposition \ref{p:bilipschitz_structure}.  The proof of the proposition will be
given at  the end of this subsection, after the proofs of the lemmas have been completed.

The first of these, Lemma \ref{l:fubini_integral} below, will enable us to
conclude that if $\cC_\epsilon\subset \cC$ is defined as indicated above, then
(4) holds.
\begin{lemma}
\label{l:fubini_integral}
	Let $(M^n,g,p)$ satisfy (S1)-(S4) with $\delta''>0$ fixed.  If $\delta\leq \delta(n,\rv,B,\delta'')$ and $\delta'\leq \delta'(n,\rv,B,\eta)$, then the local $\cW$-entropy satisfies:
	\begin{align}
		\fint_{B_{15/8}(p)}\Big(\sum_{r_x\le r_i\le 2^{-5}} \fint_{B_{r_i}(x)}|\cW_{r_i^2/2}^\delta(y)-\cW_{2r_i^2}^\delta (y)|\, d\mu(y)\Big)\, d\mu(x)\le \delta''.
	\end{align}
\end{lemma}
\begin{proof}
Recall that under the assumptions of Theorem \ref{t:cW_local},  including 
$\delta \leq \delta (n, \rv,\epsilon)$, we have the following relation between the volume
ratio and local pointed entropy:
$$
|\cW^\delta_{t}(x)-\log \cV^{\delta^2}_{\sqrt{t}}(x)|\leq \epsilon\, .
$$
The proof will utilize this relation together with a Fubini type argument. 
\vskip2mm

 Let $\chi_{|x-y|\le r}(x,y)$ be the characteristic function of set 
 $\{(x,y)\in M^n\times M^n: d(x,y)\le r\}$. We have for $r_i=2^{-i}$:
\begin{align}
\fint_{B_{15/8}(p)}&\Big(\sum_{r_x\le r_i\le 2^{-5}} \fint_{B_{r_i}(x)}|\cW_{r_i^2/2}^\delta(y)-\cW_{2r_i^2}^\delta(y)|\, d\mu(y)\Big)\, d\mu(x)\notag\\ \nonumber
		&\le \frac{1}{\mu(B_{15/8}(p))}\int_{B_{15/8}(p)}\Big(\sum_{r_x\le r_i\le 2^{-5}} \frac{1}{\mu(B_{r_i}(x))}\int_{B_{31/16}(p)}\chi_{\{|x-y|\le r_i\}}(x,y)\cdot |\cW_{r_i^2/2}^\delta(y)
-\cW_{2r_i^2}^\delta(y)|d\mu(y)\Big)d\mu(x)\notag\\
		&\le C(n)\cdot B\int_{B_{31/16}(p)}\sum_{r_x\le r_i\le 2^{-5}} Br_i^{-k}\int_{B_{31/16}(p)}\chi_{\{|x-y|\le r_i\}}(x,y)\cdot |\cW_{r_i^2/2}^\delta(y)-\cW_{2r_i^2}^\delta(y)|\, d\mu(y)\, 
d\mu(x)\notag\\
		&\le C(n)B^2 \sum_{r_x\le r_i\le 2^{-5}} r_i^{-k}\int_{B_{31/16}(p)}\int_{B_{31/16}(p)}\chi_{\{|x-y|\le r_i\}}(x,y)\cdot |\cW_{r_i^2/2}^\delta(y)-\cW_{2r_i^2}^\delta(y)|\, d\mu(y)\, 
		d\mu(x).
		\end{align}
Applying Fubini's Theorem gives
\begin{align}
\fint_{B_{15/8}(p)}&\Big(\sum_{r_x\le r_i\le 2^{-5}} \fint_{B_{r_i}(x)}|\cW_{r_i^2/2}^\delta(y)-\cW_{2r_i^2}^\delta(y)|\, d\mu(y)\Big)\, d\mu(x)\notag\\
&\le C(n)B^2 \sum_{r_x\le r_i\le 2^{-5}} r_i^{-k} \int_{B_{31/16}(p)} \mu(B_{r_i}(y))\cdot
|\cW_{r_i^2/2}^\delta(y)-\cW_{2r_i^2}^\delta(y)|\, d\mu(y)\notag\\
&\le C(n)B^3\sum_{r_x\le r_i\le 1}\int_{B_{31/16}(p)}|\cW_{r_i^2/2}^\delta(y)
-\cW_{2r_i^2}^\delta(y)|\, d\mu(y)\notag\\
&\le C(n)B^3\int_{B_{31/16}(p)}|\cW_{r_y^2/2}^\delta(y)-\cW_{2}^\delta(y)|\,
d\mu(y)\notag\\
&\le C(n)B^3\epsilon'\, .
\end{align}
Here we have used the monotonicity of $\cW^\delta_r$ and, by choosing a small $\delta(n,\rv,\epsilon',\eta)$ in the condition (n2) of $(k,\delta,\eta)$-neck, that we have
 the pointwise estimate $|\cW_{r_y^2/2}^\delta(y)-\cW_{2}^\delta(y)(y)|\le \epsilon'$ as in Theorem \ref{t:cW_local}.
\vskip2mm
	
By fixing $\epsilon'$ sufficiently small, so that $C(n)B^3\epsilon'\le \delta''$, the proof of Lemma \ref{l:fubini_integral} is completed.
\end{proof}
\vskip2mm

The following lemma is a direct consequence of the Nondegeneration Theorem \ref{t:nondegeneration} and the assumed Ahlfors regularity with 
constant $B$ as in (S3).
\begin{lemma}
\label{l:splitting_neck}
	Let $(M^n,g,p)$ satisfy (S1)-(S4) with $\epsilon>0$ fixed.  Assume  $\delta''\le \delta''(n,\rv,\eta,\epsilon)$, $\delta'\le \delta'(n,\rv,\eta,\epsilon)$, $\delta\leq \delta(n,\rv,B,\eta,\epsilon)$, and for some $x\in\cC\cap B_{15/8}(p)$:
\begin{align}
\sum_{s\le r_i\le 2^{-5}} \fint_{B_{r_i}(x)}|\cW_{r_i^2/2}^\delta(y)-\cW_{2r_i^2}^\delta(y)|\,      d\mu(y)\le \delta''.
	\end{align}
	
Then for any $s\le r\le 1$, the map $u: B_r(x)\to\mathbb{R}^k$, is an $\epsilon$-splitting map.
\end{lemma}
\begin{proof}
	By the Nondegeneration 
	Theorem \ref{t:nondegeneration} it suffices to find a set of
$(k,\alpha)$-independent points\\
 $\{x_0,x_1,\cdots,x_k\}\subset  B_{r_i}(x)\cap \cC$ for some $\alpha(n,B)>0$ such that for each $r_i$ we have the $k$-pinching estimate 
\begin{align}
	\cE_{r_i}^{k,\alpha,\delta}(x)\le \sum_{j=0}^k|\cW_{r_i^2/2}^\delta(x_j)-\cW_{2r_i^2}^\delta (x_j)|\le C(n,B)\fint_{B_{r_i}(x)}|\cW_{r_i^2/2}^\delta(y)-\cW_{2r_i^2}^\delta|(y)|\, 
d\mu(y)\, .
\end{align}
We will show that the existence of such points follows from the assumed Ahlfors regularity of $\mu$ in (S3).
\vskip2mm

First, note that there exists a subset $\cC_{r_i,x}\subset \cC\cap B_{r_i}(x)$  with $\mu(\cC_{r_i,x})\ge r_i^kB/2$ such that for any $y\in \cC_{r_i,x}$ we have
	\begin{align}
|\cW_{r_i^2/2}^\delta(y)-\cW_{2r_i^2}^\delta(y)|\le 
C(n)\fint_{B_{r_i}(x)}|\cW_{r_i^2/2}^\delta(z)-\cW_{2r_i^2}^\delta(z)|\, d\mu(z).
	\end{align} 
	By the Ahlfors regularity of $\mu$ from (S3) let us see that we can find $(k,\alpha)$-independent points in $\cC_{r_i,x}$ for some small $\alpha(n,B)$. First we note that for any $\epsilon'>0$ if $\delta\le \delta(n,\rv,\eta,\epsilon')$ then $\cC_{r_i,x}\subset B_{\epsilon' r_i}\Big(\iota(\mathbb{R}^{k}\times \{(y_c)\})\Big)$ where $\iota :\mathbb{R}^k\times C(X)\to B_{r_i}(x)$ is $\delta r_i$-GH map. Comparing the result in Remark \ref{r: kalpha_independent_Rn} about the $(k,\alpha)$-independent point in $\mathbb{R}^n$,  if there exists no $(k,\alpha)$-independent points  in $\cC_{r_i,x}$ as in Definition \ref{d:independent_points1}, the set $\cC_{r_i,x}$ must be contained in $B_{4\alpha r_i}\Big(\iota(\mathbb{R}^{k-1}\times \{(0,y_c)\})\Big)$ for some $\mathbb{R}^{k-1}$ plane.  Therefore, we have obtained at most $C(n)\alpha^{-k+1}$ many balls $\{B_{8\alpha r_i}(y_j)\}$, with $y_j\in \cC_{r_i,x}$ , which cover $\cC_{r_i,x}$.  Thus, by the Ahlfors regularity of $\mu$ we have
$$
\mu(\cC_{r_i,x})\le C(n)\alpha^{-k+1} B (8\alpha r_i)^k\le C(n,B)\alpha r_i^k\, .
$$
Since $\mu(\cC_{r_i,x})\ge Br_i^k/2$,  by choosing $\alpha=\alpha(n,B)$ small we get a contradiction.  Hence there exist $(k,\alpha)$-independent points in $\cC_{r_i,x}\subset B_r(x)\cap \cC$.  At this point, Lemma \ref{l:splitting_neck}  follows directly from the Nondegeneration Theorem \ref{t:nondegeneration}.
\end{proof}
\vskip2mm

The following Lemma \ref{l:GH_approximation_splitting_map} provides a 
Gromov-Hausdorff-approximation for $\epsilon$-splitting maps which will be used to prove the bi-Lipschitz estimate for $u$. The proof of 
Lemma \ref{l:GH_approximation_splitting_map}
 depends on Lemma 
\ref{l:claim20}. Thus, it will not be completed until after Lemma \ref{l:claim20} has been proved.

\begin{lemma}\label{l:GH_approximation_splitting_map}
	Let $(M^n,g,p)$ satisfy (S1)-(S4) and assume $\delta''\le \delta''(n,\rv,\eta,\epsilon)$, $\delta'\le \delta'(n,\rv,\eta,\epsilon)$ and $\delta\leq \delta(n,\rv,B,\eta,\epsilon)$.  Let $u: B_r(x)\to \mathbb{R}^k$ be a $\delta''$-splitting map for some $x\in\cC$ and all $r_x\le r\le 1$.  Then for any $y\in \cC$,
\begin{align}
\label{e:gha}
\Big|\, |u(x)-u(y)|-d(x,y)\Big|\le \epsilon d(x,y)\, .
\end{align}
\end{lemma}
\begin{proof}
	Pick $r\geq r_x$ so that $r/2\leq d(x,y)\leq r$.  By the definition of a neck region, we know that $B_{10r}(x)$ is $\delta r$-Gromov Hausdorff close to a cone $\mathbb{R}^k\times C(Y)$. Moreover, by  the splitting guaranteed by Theorem \ref{t:splitting_function} if $\delta\le \delta(n,\rv,\eta,B,\epsilon')$, then there exists a $(k,\epsilon')$-splitting map $\tilde{u}: B_{r}(x)\to \mathbb{R}^k:=   \mathbb{R}^k\times \{y_c\}\subset \mathbb{R}^k\times C(Y)$ such that $\tilde{u}\circ \iota : \mathbb{R}^k\times \{y_c\}\to \mathbb{R}^k\times \{y_c\}$ is $\epsilon'\cdot r$ close to the identity map. Here, $\iota: \mathbb{R}^k\times C(Y)\to B_{r}(x)$ is the $\delta r$-GH map in thedefinition of neck region. Since $B_r(x)$ is not $(k+1,\eta)$-symmetric we must have $\cC\cap B_r(x)\subset B_{\epsilon'r}(\iota(B_r(0^k)\times {y_c}))$. Therefore, for any $y\in \cC\cap B_{r}(x)$ we have  
\begin{align}\label{e:GH_tildeu}
\Big|\, |\tilde{u}(x)-\tilde{u}(y)|-d(x,y)\Big|\le 100\epsilon'r.	
\end{align}
\vskip2mm

In order to use \eqref{e:GH_tildeu} (which holds for $\tilde u$) to prove \eqref{e:gha}
(which pertains to $u$), the following
lemma is required.

\begin{lemma}
\label{l:claim20}
For any $\epsilon$ if $\epsilon'\le \epsilon'(\epsilon,n,\rv,\eta)$ and $\delta''\le \delta''(n,\rv,\eta,\epsilon)$ then there exist a rotation $O\in O(k)$ and a vector $Z\in \mathbb{R}^k$, such that $\sup_{B_{r}(x)}|O\tilde{u}-u-Z|\le \epsilon r/10$.\\
\end{lemma}
\vskip-12mm
\begin{proof}
The proof of Lemma \ref{l:claim20} is via a contradiction argument.
\vskip2mm

First we will show that after composing with a suitable orthogonal transformation, if necessary, the $L^2$ gradients are close.

\begin{sublemma}
\label{sl:2}
If $\epsilon'\le \epsilon'(\epsilon,n,\rv,\eta)$ and $\delta''\le \delta''(n,\rv,\eta,\epsilon)$, then there exists $O\in O(k)$ such that
\begin{align}\label{e:matrix_O}
	\fint_{B_{r}(x)}|\nabla (O\tilde{u}-u)|^2\le \epsilon^{2n}.
\end{align}
\end{sublemma}
\begin{proof}
Without loss of generality assume $\fint_{B_{r}(x)}\langle \nabla \tilde{u}^j,\nabla \tilde{u}^i\rangle=\delta^{ij}=\fint_{B_{r}(x)}\langle \nabla {u}^j,\nabla {u}^i\rangle$. Let us define a $k\times k$ matrix $A=(a_{ij})$ by 
\begin{align}\label{e:aij_define}
a_{ij}=\fint_{B_r(x)} \langle \nabla u^i,\nabla \tilde{u}^j\rangle.
\end{align}
We will see for $\epsilon'\le \epsilon'(\epsilon,n,\rv,\eta)$ and $\delta''\le \delta''(n,\rv,\eta,\epsilon)$ that
\begin{align}\label{e:a_ij_deltaij}
|\sum_{\ell=1}^k a_{i\ell}a_{j\ell}-\delta_{ij}|\le \epsilon^{7n},~~~i,j=1,\cdots, k.
\end{align}

Let us first assume \eqref{e:a_ij_deltaij} and finish the proof of the sublemma.

 By \eqref{e:a_ij_deltaij} we have 
 \begin{align}\label{e:Atildeu_u}
 \fint_{B_r(x)}|\nabla (A\tilde{u}-u)|^2\le \epsilon^{3n}.
 \end{align}
Moreover, by \eqref{e:a_ij_deltaij} we can use Gram-Schmidt process to produce a matrix $O\in O(k)$ with $|O-A|\leq C(k)\epsilon^{4n}$. Combining this with \eqref{e:Atildeu_u} would implies \eqref{e:matrix_O} i.e. the sublemma.

Now we begin the proof of \eqref{e:a_ij_deltaij}. 
Since $\fint_{B_{r}(x)}\langle \nabla \tilde{u}^j,\nabla \tilde{u}^i\rangle=\delta^{ij}
=\fint_{B_{r}(x)}\langle \nabla {u}^j,\nabla {u}^i\rangle$, it will suffice to prove 
\begin{align}\label{e:a_ii_deltaii}
|\sum_{\ell=1}^k a_{i\ell}a_{i\ell}-1|\le \epsilon^{10n},~~~i=1,\cdots, k.
\end{align}
Assume \eqref{e:a_ii_deltaii} doesn't hold for some $i=i_0\le k$ and $\epsilon=\epsilon_0>0$ 
with $\epsilon'\to 0$ and $\delta''\to 0$.  Let us consider the following $k+1$ harmonic functions
 $v^0=u^{i_0}-\sum_{j=1}^ka_{i_0 j}\tilde{u}^j$, $\tilde{u}^1$,$ \cdots$, $\tilde{u}^k$.
 From the definition of $a_{ij}$ in \eqref{e:aij_define}, we have that 
$\fint_{B_r(x)}\langle \nabla v^0,\nabla \tilde{u}^j\rangle=0$. Moreover,  By the contradiction assumption we have
\begin{align}
\fint_{B_r(x)}|\nabla v^0|^2=1-\sum_{j=1}^ka_{i_0j}^2\ge \epsilon_0^{10n}.
\end{align}
Normalize $v^0$ to $\tilde{u}^0$ such that $\tilde{u}^0$ has unit $L^2$ gradient norm. 
Therefore, for $\epsilon'$ and $\delta''$ sufficiently small, the map $(\tilde{u}^0,\tilde{u}^1,\cdots,\tilde{u}^k): B_r(x)\to \dR^{k+1}$ is a $(k+1,\eta/10)$-splitting map,  which contradicts with the fact that $B_r(x)$ is not $(k+1,\eta)$-symmetric.  
This completes the proof of \eqref{e:a_ii_deltaii} and \eqref{e:a_ij_deltaij}, hence,
the proof of Sublemma \ref{sl:2}. 
\end{proof}
 \vskip2mm
 
 Now by using the Poincar\'e inequality in Theorem \ref{t:poincare} we get 
\begin{align}\label{e:OtildeuL2}
	\fint_{B_{r}(x)}\Big|O\tilde{u}-u-\fint_{B_{r}(x)}(O\tilde{u}-u)\Big|^2
\le C(n)r^2\fint_{B_{r}(x)}|\nabla (O\tilde{u}-u)|^2\le C(n)r^2\epsilon^{2n}.
\end{align}
Denote $Z=\fint_{B_{r}(x)}(O\tilde{u}-u)\in \mathbb{R}^k$. At this point, the proof of 
Lemma \ref{l:claim20}
 follows now from 
 \eqref{e:OtildeuL2} and the gradient estimate $\sup_{B_{r}(x)}|\nabla (O\tilde{u}-u)|\le 1+\epsilon$.   
This completes the proof of Lemma \ref{l:claim20}.
\end{proof}
\vskip2mm

The proof of Lemma \ref{l:GH_approximation_splitting_map}
 can now be completed by observing
 that for any $y\in \cC\cap B_r(x)$ with $d(x,y)\ge  r/2$, we have
\begin{align}\label{e:GH_u}
	\Big|{u}(x)-{u}(y)|-d(x,y)\Big|&\le \Big||O(\tilde{u}(x)-\tilde{u}(y))|-d(x,y)\Big|
+\Big|O\tilde{u}(x)-u(x)-Z\Big|+\Big|O\tilde{u}(y)-u(y)-Z\Big| \notag\\
	&\le \Big|\tilde{u}(x)-\tilde{u}(y)|-d(x,y)\Big|+\epsilon r/5\le \epsilon r/2\le \epsilon d(x,y)\, .
\end{align}
This completes the proof of  Lemma \ref{l:GH_approximation_splitting_map}.
\end{proof}
\vskip2mm

\begin{proof} (of Proposition \ref{p:bilipschitz_structure})
Now we can finish the proof of Proposition \ref{p:bilipschitz_structure}.
For this, note that for all $\epsilon''>0$ there exists $\delta'(n,B,\rv,\eta,\epsilon'')$
 and $\delta(n,B,\rv,\eta,\epsilon'')$ such that by Lemma \ref{l:fubini_integral}: 
		\begin{align}
		\fint_{B_{15/8}(p)}\Big(\sum_{r_x\le r_i\le 2^{-5}} \fint_{B_{r_i}(x)}|\cW_{r_i^2/2}^\delta-\cW_{2r_i^2}^\delta|(y)d\mu(y)\Big)\, d\mu(x)\le \epsilon''.
	\end{align}

For all $\delta''>0$, define the set $\cC_{\delta''}\subset \cC\cap B_{15/8}(p)$ such that $x\in \cC_{\delta''}$ if 
	\begin{align}
\sum_{r_x\le r_i\le 2^{-5}} \fint_{B_{r_i}(x)}|\cW_{r_i^2/2}^\delta-\cW_{2r_i^2}^\delta|(y)d\mu(y)\le \delta''.	
\end{align}

	If $\epsilon''\le \epsilon''(n,B,\delta'')$, then by the Ahlfors regularity estimate $(S3)$ for $\mu$, we have  
$$
\mu(\cC_{\delta''})\ge (1-\delta'')\mu(\cC\cap B_{15/8}(p))\, .
$$
\vskip1mm

Given $\epsilon'>0$, if $\delta''\le \delta''(n,\epsilon')$, then  by Lemma \ref{l:splitting_neck}  for any $x\in \cC_{\delta''}$ and $r_x\le r\le 1$ we have  $u:B_{r}(x)\to \mathbb{R}^k$ is $(k,\epsilon')$-splitting map.  Thus, by fixing $\epsilon'\le \epsilon'(n,\rv,B,\eta,\epsilon)$ and putting $\cC_{\delta''}=\cC_{\epsilon}$, we obtain
 (1), (3) and (4) of Proposition \ref{p:bilipschitz_structure}.
  \vskip2mm
  
    To prove the bi-Lipschitz estimate, (2),  note that for any $x,y\in \cC_{\delta''}$ if $\epsilon'\le \epsilon'(n,\rv,\eta,B,\epsilon)$,  then
  Lemma \ref{l:GH_approximation_splitting_map} 
    gives \eqref{e:gha}, 
$$
\Big|\, |u(x)-u(y)|-d(x,y)\Big|\le \epsilon d(x,y)\, .
$$
This implies the bi-Lipschitz estimate (2) of Proposition \ref{p:bilipschitz_structure}.  
By using the Transformation Proposition \ref{p:transfor_prop}, the Transformation Proposition,
the proof of bi-H\"older estimate for $u$ can be completed in the verbatim manner
as in the proof of the 
Canonical Reifenberg Theorem, \ref{t:reifenbergcanonical}. 
This completes the proof of Proposition
\ref{p:bilipschitz_structure}.
\end{proof}

\subsection{Ahlfors regularity for the packing measure}
\label{ss:Ahlforsestimates}
In this subsection, we will show if a neck region satisfies a weak Ahlfors regularity estimate as in (S3), 
then for $\delta$ sufficiently small the neck region automatically satisfies a stronger universal Ahlfors regularity estimate. 
This is based on the bi-Lipschitz structure proved in Proposition
\ref{p:bilipschitz_structure}. It is the key to the inductive scheme:

\begin{proposition}
\label{p:Ahlfors_regular}
	Let $(M^n,g,p)$ satisfy (S1)-(S4) with $\delta\leq \delta(n,\rv,B,\eta)$ and $\delta'\leq \delta'(n,\rv,B,\eta)$.  Then there exists $A(n)$ such that
for any $x\in \cC\cap B_{2}(p)$, with $r\ge r_x$ and $B_{2r}(x)\subset B_{2}(p)$, we have 
	\begin{align}\label{e:Ahlfors_bounded}
		A(n)^{-1}r^k\le \mu(B_r(x))\le A(n)r^k.
	\end{align}
\end{proposition}
\begin{proof}
Let us assume without any loss that $x=p$ and $r=1$.  We will show that $\mu(B_1(p))$ satisfies the upper and lower bound in \eqref{e:Ahlfors_bounded}.  
\vskip2mm	
	
Consider the map $u: B_2(p)\to \dR^k$ and assume $0^k=u(p)$ and recall that $\tau=\tau_n=10^{-10n}\omega_n$.
\vskip1mm

We will begin by proving the upper bound for $\mu(B_1(p))$. 
For this, note that for any $\epsilon$ if $\delta\le \delta(n,\rv,\epsilon,B,\eta)$, then by the bi-Lipschitz estimate in Proposition \ref{p:bilipschitz_structure} the balls
$\{B_{\tau^3r_x}(u(x))\subset \dR^k\}$ are mutually disjoint for $x\in\cC_\epsilon$ and in addition,
$\cC_\epsilon\subset \cC\cap B_{15/8}(p)$ satisfies 
$$
\mu(\cC_\epsilon\cap B_{15/8})\ge (1-\epsilon)\cdot \mu(\cC\cap B_{15/8}(p))\, .
$$
By the Lipschitz bound on $u$, we have 
 $|u(x)|=|u(x)-u(p)|\le 4$.  Let $\Vol_k$ is the volume form of $\dR^k$.   Then  
\begin{align}
\mu(\cC_\epsilon\cap B_{15/8}(p))
&\le \sum_{x\in\cC_\epsilon\cap B_{15/8}}r_x^k\notag\\
&\le C(n)\cdot \sum_{x\in\cC_\epsilon\cap B_{15/8}}\Vol_k(B_{\tau^3r_x}(u(x)))\notag\\
&\le C(n)\cdot\Vol_k(B_5(0^k))\notag\\
&\le C(n)\, .
	\end{align}
 By combining the above with the estimate $\mu(\cC_\epsilon\cap B_{15/8})\ge (1-\epsilon)\mu(\cC\cap B_{15/8}(p))$, this gives the upper bound of $\mu(B_1(p))$.
\vskip2mm

The lower bound for $\mu(B_1(p))$  will follow from a covering argument. 
\vskip2mm

The Geometric Transformation Theorem \ref{t:transformation} implies that
for any $\epsilon>0$ and $\delta\le\delta(n,\rv,\epsilon,\eta)$,  there exists for $x\in\cC$ and $r_x\le s\le 1$ a $k\times k$ matrix $T_{x,s}$ such that 
the map,
$T_{x,s}u: B_s(x)\to \mathbb{R}^k$ is a $(k,\epsilon)$-splitting map.  
Since $|\nabla u|\le 1+\delta'$, we have $|T_{x,s}|\ge 1/2$. 
 The lower bound estimate in \eqref{e:Ahlfors_bounded} will follow from the next
 lemma.	

\begin{lemma}
\label{l:claim21}
Let 
$$
T_{x,r_x}^{-1}(B_{r_x}(u(x))):=  u(x)+T_{x,r_x}^{-1}\Big(B_{r_x}(0^k)\Big)\, .
$$
Then a covering of $B_{1/8}(0^k)\subset \dR^k$, is provided by 
the  collection of ellipsoids:
$$
\big\{T_{x,r_x}^{-1}(B_{r_x}(u(x)))\,|\, x\in\cC\cap B_1(p)\big\}\, .
$$ 
\end{lemma}
\begin{proof}
Assume there exists $w\in B_{1/8}(0^k)$ not in the covering. For every $x\in \cC\cap B_1(p)$ define 
\begin{align}
s_x &:=  \inf\{s\ge r_x: w\in T_{x,s}^{-1}B_s(u(x))\}\notag\\ 
\bar s &:=  s_{\bar x}:=  \min_{x\in\cC\cap B_1(p)} s_x\, .
\end{align}
  Then $\bar s> r_{\bar x}$ and 
 $$
  w\in T_{\bar x,2\bar s}^{-1} B_{2\bar s}(u(\bar x))\, .
 $$ 
 This implies  
 $$
 T_{\bar x,2\bar s}w\in B_{2\bar s}(T_{\bar x,2\bar s}u(x))\, . 
 $$
 
 On the other hand, the map is $(k,\epsilon)$-splitting:
 $$
 T_{\bar x,2\bar s}u: B_{2\bar s}(\bar x)\to \mathbb{R}^k\, .
 $$ 
 From the covering property (n5) in the neck region definition \ref{d:neck2}, 
 there exists some $y\in B_{2\bar s}(\bar x)\cap \cC$ such that 
\begin{align}
	|T_{\bar x,2\bar s}w-T_{\bar x,2\bar s}u(y)|\le 3\tau \bar s. 
\end{align}

	By the H\"older growth estimate for $T_{x,s}$, with respect to $s$ in the Transformation Proposition \ref{p:transfor_prop}, 
	we have 
$$
|T_{y,2\bar s}T_{\bar x,2\bar s}^{-1}-I|\le C(n)\epsilon\, .
$$ 
This follows since $|T_{\bar x,2\bar s}T_{\bar x,5\bar s}^{-1}-I|\le \epsilon$ and $|T_{y,2\bar s}T_{\bar x,5\bar s}^{-1}-I|\le \epsilon$ due to the fact that $T_{\bar x,5\bar s}u: B_{2\bar s}(y)\to \mathbb{R}^k$ is also a $C\epsilon$-splitting map. Therefore,
	\begin{align}
		|T_{y,2\bar s}w-T_{y,2\bar s}u(y)|\le 4\tau \bar s.
	\end{align}

	Again by the H\"older growth estimate for $T_{y,s}$ in the Transformation Proposition \ref{p:transfor_prop} we have $|T_{y,\bar s/2}w-T_{y,\bar s/2}u(y)|\le 5\tau \bar s$. Since $w\in T_{y,\bar s/2}^{-1}B_{\bar s/2}(u(y))$, this 
 contradicts  the definition of $\bar s$.
This concludes the proof of Lemma \ref{l:claim21}.
\end{proof}

From Lemma \ref{l:claim21} we obtain
	\begin{align}
		C(k)\le \sum_{x\in\cC\cap B_1(p)} \Vol_{k}(T_{x,r_x}^{-1}B_{r_x}(u(x)))\le \sum_{x\in\cC\cap B_1(p)} C_k r_x^k |T_{x,r_x}^{-1}|\le C_k\sum_{x\in\cC\cap B_1(p)}r_x^k= C_k\, \mu(B_1(p))\, 
	\end{align}
By using the estimate $|T_{x,r_x}^{-1}|\le 2$, this provides
 a lower bound for $\mu(B_1(p))$, where we use the estimate $|T_{x,r_x}^{-1}|\le 2$. This completes the proof of Proposition \ref{p:Ahlfors_regular}.
\end{proof}
\vskip2mm

\subsection{Proof of the Neck Structure Theorem for smooth
 manifolds}
\label{ss:neck_structure_smooth_manifold}
In present subsection,
we will prove the Ahlfors regularity estimate for the case of smooth riemannian manifolds. 
The Ahlfors regularity estimate in the general
case will be reduced to this one via a careful approximation argument.
\vskip2mm

 In the case of smooth Riemann manifolds,  neck regions satisfy $\cC_0=\emptyset$ and $\inf r_x>0$. 
Thus, it suffices to prove the following lemma
\begin{lemma}[The smooth case of  Theorem \ref{t:neck_region2}]
\label{l:smoothcase}
For all $\eta>0$ there exists $\delta=\delta(n,\rv,\eta)>0$ and $A(n)$ such that if $\cN\subset  B_2(p)\subset M^n$ 
is a $(k,\delta,\eta)$-neck, 
 then for all $s\ge r_x$ with $B_{2s}(x)\subset B_{2}(p)$
\begin{align}
\label{e:j}
A(n)^{-1}s^k\le \mu(B_s(x))\le A(n)s^k\, .
\end{align}
\end{lemma}
\begin{proof}${}$

\noindent
{\bf Terminology.} We will say {\it statement $(j)$ holds} if the lemma holds for all neck regions which satisfy $\inf r_x\geq 2^{-j}>0$.  The proof will be by induction on $j$.  
\vskip3mm

We begin with the base step.  Note that if $j\le 10$ and $\delta\le 10^{-10n}$ then $\#\cC\leq C(n)$.  In particular, the statement $(j)$ holds for some universal constant $A_0(n)$.
\vskip2mm

Denote the universal constant $A(n)$ in Proposition \ref{p:Ahlfors_regular} by $A_1(n)$. 
We will show that $(j)$ holds for all $j$ when  $A(n):=  A_0(n)+A_1(n)$ and 
$\delta(n,\rv,\eta)=\delta(n,\rv,\eta,B)$, where $\delta(n,\rv,\eta,B)$ is the constant 
in Proposition \ref{p:Ahlfors_regular} with $B=A(n)C(n)$, where $C(n)$ is given by 
$C(n)=C_0(n)16^k$ with $C_0(n)$ is the cardinality of maximal number of disjoint balls 
$\{B_{2^{-5}}(x_i)\, |\, x_i\in B_2(0^k)\}$ with center in $B_2(0^k)$.  Therefore, $B=A(n)C(n)$ 
is a universal constant.  
\vskip2mm

Note that if we take $\delta\le \delta(n,\rv,B,\eta,\delta')$ sufficiently small
then by the structure of neck region and cone splitting theorem \ref{t:content_splitting} 
there exists an $(k,\delta')$-splitting map $u: B_2(p)\to \mathbb{R}^k$; see also Remark \ref{r:delta_delta'}.
 Therefore, the constant $\delta'(n,\rv,B,\eta)$ in Proposition \ref{p:Ahlfors_regular} automatically holds if
 we choose $\delta\le \delta(n,\rv,B,\eta,\delta')$.  
\vskip2mm

Now let us assume statement $(j)$ holds, then we need to see that $(j+1)$ holds. 
 So let $\cN\subset B_{2}(p)$ be a $(k,\delta,\eta)$-neck region with $\min_x r_x\ge 2^{-j-1}$
 and  associated center points $\cC$. By Proposition \ref{p:Ahlfors_regular}, it suffices to obtain a 
weak Ahlfors regularity bound for 
$\mu$ with $B=A(n)C(n)$. 
\vskip2mm

Let $B_{2s}(x)\subset B_{2}(p)$. If $s\le 1/2$, then $\cN\cap B_{2s}(x)\subset B_{2s}(x)$ 
is a new $(k,\delta,\eta)$-neck which satisfies $(j)$ by rescaling $B_{2s}(x)$ to $B_2$.  
In particular we have by our inductive hypothesis that $A^{-1}(n)\le \mu(B_s(x))\le A(n)s^k$. 
If $s> 1/2$ then, in particular, we have $x\in B_{3/2}(p)\cap \cC$ and $B_s(x)\subset B_{7/4}(p)$. 
Choose a Vitali covering $\{B_{1/16}(x_j),x_j\in \cC\cap B_{7/4}(p)\}$ of $B_{7/4}(p)$ with cardinality
 at most $C_0(n)$. Since $B_{1/8}(x_j)\subset B_{2}(p)$, by using the inductive assumption again we have 
\begin{align}
	16^{-k}A^{-1}(n)\le \mu(B_{1/16}(x_j))\le A(n) 16^{-k}.
\end{align}
From this, it follows easily that 
\begin{align}
	16^{-k}A^{-1}(n)s^k\le \mu(B_s(x))\le C_0(n)16^{-k}2^{k}A(n)s^k.
\end{align}

Thus, we have proved $\mu$ satisfies weak Ahlfors regularity estimate with constant $B=C_0(n)16^{k}A(n)$. By Proposition \ref{p:Ahlfors_regular},  if $\delta\le \delta(n,\rv,\eta,B)=\delta(n,\rv,\eta)$ then in fact we have the stronger estimate $A_1(n)^{-1}{s}^k\le \mu(B_s({x}))\le A_1(n){s}^k$. In particular $A(n)^{-1}{s}^k\le \mu(B_s({x}))\le A(n){s}^k$.  This completes the proof of Lemma
\ref{l:smoothcase} i.e. Ahlfors regularity for the case of smooth manifolds.
\end{proof}

\subsection{Approximating limit neck regions by smooth neck
 regions}
\label{ss:alnrfsnr}

As mentioned in the previous subsection, to prove the neck structure theorem for neck regions for which
 $\cC_0\neq \emptyset$, we will approximate general neck regions $\cN$ by neck regions 
$\cN_j$ for which $\inf r_{x,j}>0$.  This will be carried out  in the present subsection.  In the following 
subsection, we will complete the 
proof of the Neck Structure Theorem \ref{t:neck_region2}.
\vskip2mm

Our main result in this subsection is the following:
\begin{theorem}
\label{t:approximate_neck}
	Let $(M^n_j,g_j,p_j)\togh (X^n,d,p)$ satisfy $\Vol(B_1(p_j))>v>0$ and $\Ric_i\ge-(n-1)\delta^2$ 
and let $\cN = B_2(p)\setminus \overline B_{r_x}(\cC)$ be a $(k,\delta,\eta)$-neck region.  
Then there exists $(k,\delta_j,\eta_j)$-neck region $\cN_j=B_2(p_j)\setminus \overline B_{r_{x,j}}(\cC_j)$ 
such that the following hold
	\begin{itemize}
\item[(1)] $\delta_j\to \delta$ and $\eta_j\to\eta$.
\vskip1mm
		
\item[(2)] If $\phi_j :B_2(p_j)\to B_2(p)$ are the approximating Gromov-Hausdorff maps then 
$\phi_j(\cC_j)\to \cC$ in the Hausdorff sense.
\vskip1mm
		
\item[(3)] $r_{x,j}\to r_x: \cC\to \mathbb{R}_{+}$ uniformly.
\vskip1mm

\item[(4)] If $\mu_j,\mu$ are the packing measure of $\cN_j$ and $\cN$, 
respectively, then if we consider the limit $\mu_i\to \mu_\infty$ in measure sense then we have 
the measure estimates $\mu\le C(n)\mu_\infty$.
\vskip1mm

\item[(5)] If $\cC_0\subset \cC$ is $k$-rectifiable, we have $\mu_\infty\le C(n)\mu$.
	\end{itemize}
\end{theorem}
\begin{proof}
	Consider first  the case $\inf r_x>0$. This implies that $\cC_0=\emptyset$
 and in addition, that $\cC$ is a finite set. 
\vskip2mm	
	
	Let $\psi_j: B_2(p)\to B_2(p_j)$ be the $\epsilon_j$-Gromov Hausdorff maps. 
For $j$ sufficiently large with $\epsilon_j<<\inf r_x$, let $\cC_j:=\{\psi_j(x),~x\in\cC\}$ 
and $r_{x,j}:=r_{\psi_j^{-1}(x)}$.  Then it is easy to check that 
$\cN_j:=B_2(p_j)\setminus \bar B_{r_{x,j}}(\cC_j)$ are $(k,\delta_j,\eta_j)$-neck 
regions which satisfy the criteria of the theorem. In fact, we can prove the limit $\mu_j\to \mu_\infty=\mu$ in this case.
\vskip2mm

Next, for the case in which $\inf r_x=0$, we construct a $(k,\delta,\eta)$-neck region $$
\tilde{\cN}_{s}=B_2(p)\setminus \bar B_{\tilde{r}_x}(\tilde{\cC})\, ,
$$ 
with $\inf \tilde{r}_x\ge s>0$. Given $s>0$,  we define $\tilde{r}_x$ on $\cC$ by $\tilde{r}_x:=  \max\{r_x,s\}$. 
Then $|\text{Lip }\tilde{r}_x|\le \delta$ and all
of the remaining properties of a neck region are satisfied with $\cC$ and $\tilde{r}_x$, apart from the Vitali condition (n1). 
\vskip2mm

To fix this, choose a maximal subset $\tilde{\cC}_s:=   \{x^s_i\}\subset \cC$ such that the 
balls $\{B_{\tau^2\tilde{r}_{x^s_i}}(x_i^s)\}$ are disjoint. 
It is easy to check that $\tilde{\cN}_s:=  B_2(p)	\setminus \bar B_{\tilde{r}_x}(\tilde{\cC})$ 
is a $(k,\delta,\eta)$-neck region for which $\inf \tilde{r}_x\ge s>0$. 
If we let $s\to 0$, then $\tilde{\cN}_s$ converges to $\cN$ in Hausdorff sense. 
\vskip2mm

Consider the limit packing measure $\tilde{\mu}_s\to \tilde{\mu}_\infty$. On $\cC_{+}$, 
we have $\tilde{\mu}_\infty=\mu$. If $y\in \cC_0$, then for all $s<r$, by the Vitali covering property of $\tilde{\cN}_s$, it will follow that
		\begin{align}\label{e:volume_mu_s_upperbound}
		s^{k-n}\Vol\Big(\bar B_r(y)
\cap B_s(\cC_0)\Big)\le C(n)\tilde{\mu}_s\Big(\bar B_{2s+r}(y)\cap B_{3s}(\cC_0)\Big).
	\end{align}
To see this, consider the covering $\{B_{s}(x_i^s),x_i^s\in \tilde{\cC}_s\cap B_s(\cC_0)\}$ 
of $\cC_0\cap B_r(y)$. 
Since $B_{\tau_n^2s}(x_i^s)$ are disjoint and $\tilde{\mu}_s(B_s(x_i^s))\ge s^k$, using the estimate 
of the cardinality of  $\{B_{s}(x_i^s)\, |\, x_i^s\in \tilde{\cC}_s\cap B_s(\cC_0)\}$ by 
$s^{-k}\tilde{\mu}_s\Big(B_{2s+r}(y)\cap B_{3s}(\cC_0)\Big)$ we can get \eqref{e:volume_mu_s_upperbound}. 
\vskip2mm
	
By letting $s\to 0$ in \eqref{e:volume_mu_s_upperbound}, we get the upper Minkowski $k$ content bound
\begin{align}
\label{e:upper}
\cM^k(\bar B_r(y)\cap \cC_0)\le C(n)\tilde{\mu}_\infty(\bar B_r(y)\cap \cC_0)\, .
\end{align}
In particular this implies 
\begin{align}
\mu(\bar B_r(y)\cap \cC_0)=\cH^k(\bar B_r(y)\cap \cC_0)\le C(k)\cM^k(\bar B_r(y)\cap \cC_0)
\le C(n)\tilde{\mu}_\infty(\bar B_r(y)\cap \cC_0)\, .
\end{align}
 Therefore, we get the weaker estimate $\mu\le C(n)\tilde{\mu}_\infty$. 
 \vskip2mm
 
 On the other hand, we claim that 
	\begin{align}\label{e:volume_mu_s_lowerbound}
		\tilde{\mu}_s\Big( B_{r}(y)
\cap B_{s}(\cC_0)\Big)\le C(n)s^{k-n}\Vol\Big( B_{r+2s}(y)\cap B_{3s}(\cC_0)\Big).
	\end{align}
To see this, consider the covering $\{B_{s}(x_i^s)\, |\, x_i^s\in \tilde{\cC}_s\cap B_s(\cC_0)\}$ of $\cC_0\cap B_r(y)$. Since $B_{\tau_n^2s}(x_i^s)$ are disjoint and 
$\tilde{\mu}_s(B_s(x_i^s))\le A(n)s^k$ 
the estimate \eqref{e:volume_mu_s_lowerbound} follows easily from the estimate of the cardinality of $\{B_{s}(x_i^s),x_i^s\in \tilde{\cC}_s\cap B_s(\cC_0)\}$ by $s^{-n}\Vol\Big(B_{r+2s}(y)\cap B_{3s}(\cC_0)\Big)$.
	By letting $s\to 0$, it follows that the upper Minkowski $k$ content satisfies 
	\begin{align}
	C(n)\cM^k( B_r(y)\cap \cC_0)\ge \tilde{\mu}_\infty( B_r(y)\cap \cC_0) \, . 
	\end{align}

	 To prove (5) of Theorem \ref{t:approximate_neck},
   we will initially make the  assumption that $\cC_0$ is $k$-rectifiable.
This will be  proved  in Lemma \ref{l:claim34}, the proof of which is completely independent of (5).
\vskip2mm

 By a standard geometric measure theory argument 
(see Theorem 3.2.39 of \cite{Fed}), Hausdorff measure and Minkowski content are equivalent. Thus, 
 \begin{align}
C(n)\mu( B_r(y)\cap \cC_0)\ge \tilde{\mu}_\infty( B_r(y)\cap \cC_0)\, ,
\end{align}
 and in particular,  $C(n)\mu\ge \tilde{\mu}_\infty$. 
\vskip2mm
	
Finally, for each $\tilde{\cN}_s$, we have $(k,\delta_j,\eta_j)$ neck region $\tilde{\cN}_{s,j}=B_2(p_j)\setminus \bar B_{\tilde{r}_{x,j}}(\tilde{\cC}_j)$  approximating $\tilde{\cN}_s$ with $\tilde{\mu}_{s,j}\to \tilde{\mu}_s$. By a standard diagonal argument, we finish the proof by taking a diagonal subsequence of $\tilde{\cN}_{s,j}$ to approximate $\cN$.
This completes the proof of Theorem \ref{t:approximate_neck}.
\end{proof}

\subsection{Proof of the Neck Structure Theorem
 \ref{t:neck_region2}.}
\label{ss:pns2}
We will  now complete the proof of Theorem \ref{t:neck_structure}.
\vskip2mm

\begin{proof}(of the Neck Structure Theorem \ref{t:neck_region2} )
	Given a $(k,\delta,\eta)$-neck region $\cN=B_2(p)\setminus \bar B_{r_x}(\cC)$, 
we have by the Approximation Theorem \ref{t:approximate_neck} a sequence of 
 $(k,\delta_j,\eta_j)$-neck regions $\cN_j=B_2(p_j)\setminus \bar B_{r_{x,j}}(\cC_j)\subset M_j$. 
	By the Ahlfors regularity estimates in Subsection \ref{ss:neck_structure_smooth_manifold}
 for smooth neck regions, we have for $\delta\le \delta(n,\rv,\eta)/10$ that if $B_{2r}(x_j)\subset B_2(p_j)$ 
and $x_j\in\cC_j$ then for $j$ sufficiently large
	\begin{align}
		A(n)^{-1}r^k\le \mu_j(B_r(x_j))\le A(n)r^k.
	\end{align}
Thus, by Theorem \ref{t:approximate_neck} 
we have for all $B_{2r}(x)\subset B_2(p)$ with $x\in\cC$ that the limit $\mu_j\to \mu_\infty$ satisfies 
\begin{align}\label{e:limit_packing}
	A(n)^{-1}r^k\le \mu_\infty(B_r(x))\le A(n)r^k.
\end{align}
By Theorem \ref{t:approximate_neck}, since 
$\mu\le C(n)\mu_\infty$, we directly get the 
upper bound estimates of $\mu(B_r(x))\le \tilde{A}(n)r^k$ for a universal constant $\tilde{A}(n)=A(n)C(n)$.  
\vskip2mm

In order to prove the lower measure bound, 
we will first prove $\cC_0$ is $k$-rectifiable. 
Then we can use (5) from the approximation 
Theorem \ref{t:approximate_neck} to deduce the lower bound.  
The main lemma needed for this result is the following.
\vskip2mm

\begin{lemma}
\label{l:claim32}
	For each $\epsilon>0$ if $\delta\le \delta(n,\rv,\epsilon,\eta)$,
 then for any $x\in\cC_0$ and $B_{2r}(x)\subset B_2(p)$ there exists
 a closed subset $\cR_\epsilon(B_r(x))\subset \cC_0\cap B_r(x)$ such 
that $\cR_\epsilon$ is bi-Lipschitz to a subset of $\mathbb{R}^k$ and $\mu\Big(B_r(x)\cap (\cC_0\setminus \cR_\epsilon)\Big)	<\epsilon r^k$.
\end{lemma}
\begin{proof}
For each $B_{2r}(x)\subset B_2(p)$ with $x\in\cC_0$ the set
$$
\cN_r:= B_{2r}(x)\setminus \bar B_{r_x}(\cC_r)
$$
is a $(k,\delta,\eta)$-neck region with associated $\cC_r=\cC\cap B_{2r}(x)$ 
and packing measure $\mu_r:= \mu|_{\cC_r}$. By the approximation
 Theorem \ref{t:approximate_neck}, there exists a $(k,\delta_j,\eta_j)$-neck region 
$$
\cN_{r,j}:=  B_{2r}(x_j)\setminus \bar B_{r_{x,j}}(\cC_{r,j})\subset M_j
$$
which approximate $\cN_r$. 
\vskip2mm

By Theorem \ref{t:splitting_function}, there exist $\delta_j'$-splitting maps 
$u_{r,j}: B_{2r}(x_j)\to \mathbb{R}^k$ with $\delta_j'=\delta_j'(n,\rv,\eta,\delta_j)$.  
Additionally, by the Ahlfors regularity estimate for smooth neck 
$\cN_{r,j}$ in subsection \ref{ss:neck_structure_smooth_manifold}, we have for any
 $B_{2s}(x_{r,j})\subset B_{2r}(x_j)$ and $x_{r,j}\in\cC_{r,j}$ that
\begin{align}
A(n)^{-1}s^k\le \mu(B_s(x_{r,j}))\le A(n)s^k\, .
\end{align}

By applying Proposition \ref{p:bilipschitz_structure} 
with $B=A(n)$ and $\delta\le \delta(n,\rv,\epsilon,\eta)$,
 there exists $\cC_{r,j,\epsilon}\subset \cC_{r,j}$ such that 
$u_{r,j} :\cC_{r,j,\epsilon}\to\mathbb{R}^k$ is $(1+\epsilon)$-bi-Lipschitz
 and $\mu_{r,j}(B_r(x_j)\setminus \cC_{r,j,\epsilon})\le \epsilon^2 r^{k}$. 
Let $j\to\infty$, denote the Gromov-Hausdorff limit 
by $\cC_{r,\epsilon}:=  \lim \cC_{r,\epsilon,j}$ and let $\mu_{r,\infty}$ 
denote the limit measure $\mu_{r,j}\to \mu_{r,\infty}$. 
\vskip2mm

On the other hand, since $B_r(x)\setminus \cC_{r,\epsilon}$ is an open set,
 standard measure convergence argument implies 
\begin{align}\label{e:open_measure_convergence}
		\mu_{r,\infty}(B_r(x)\setminus \cC_{r,\epsilon})
\le \lim \inf \mu_{r,j}(B_{r}(x_j)\setminus \cC_{r,j,\epsilon})\le \epsilon^2 r^k.
	\end{align}
	Indeed, for any closed set $D\subset \bar B_r(x)\subset X$ with 
$D_i\subset M_i\to D$ in GH-sense, by measure convergence we have for any $t>0$
	\begin{align}
		\lim \sup \mu_{r,j}(D_i)\le \mu_{r,\infty}(B_t(D)).
	\end{align}
	
By letting $t\to 0$, using the monotone convergence theorem for 
measures and that $D$ is a closed set, it follows that
$$
\lim \sup \mu_{r,j}(D_i)\le \mu_{r,\infty}(D)\, .
$$ 
This implies \eqref{e:open_measure_convergence}.
Hence we have $\cC_{r,\epsilon}\subset \cC_r\subset \cC$ and the following estimate 
	\begin{align}\label{e:cC_repsilon}
		\mu(B_r(x)\setminus \cC_{r,\epsilon})=\mu_r(B_r(x)
\setminus \cC_{r,\epsilon})\le C(n)\mu_{r,\infty}(B_r(x)\setminus \cC_{r,\epsilon})\le C(n)\epsilon^2 r^k\le \epsilon r^k\, .
	\end{align}
Here, we have used Theorem \ref{t:approximate_neck} in the first inequality.	
\vskip2mm

Moreover, since $u_{r,j}$ is Lipschitz, by  Ascoli's theorem,  we have a uniform limit $u_r:B_{2r}(x)\to\mathbb{R}^k$ such that $u_r:\cC_{r,\epsilon}\to \mathbb{R}^k$ is $(1+\epsilon)$-bilipschitz.  From the estimate \eqref{e:cC_repsilon}, the set $\cR_\epsilon(B_r(x)):=  \cC_0\cap \cC_{r,\epsilon}$ is our desired set. This finishes the proof of Lemma \ref{l:claim32}.
\end{proof}

Now we can prove the rectifiability of $\cC_0$.
	
\begin{lemma}
\label{l:claim34}
 $\cC_0$ is rectifiable. 
 \end{lemma}
 \begin{proof}
 Let $\{x_i\}\subset \cC_0$ be a countable dense subset of $\cC_0$ and for any $\epsilon>0$, consider the set 
	\begin{align}
		\cR:=  \bigcup _{B_{2r}(x_i): 1\ge r\in \mathbb{Q}}\cR_\epsilon(B_{r}(x_i)).
	\end{align}
	
By definition, we have $\cR\subset \cC_0$. In addition,
since $\cR$ is a countable union of rectifiable sets, it
is rectifiable. To finish the proof, we only need to choose a small $\epsilon$ and show that $\mu (\cC_0\setminus \cR)=\cH^k(\cC_0\setminus \cR)=0$. So assume $\cH^k(\cC_0\setminus \cR)>0$, then by a standard geometric measure theory argument, there exist $x\in \cC_0\setminus\cR$, $r_a\to 0$ and a dimensional constant $\epsilon_k>0$ (see Theorem 3.6 of \cite{Simon}) such that 
	\begin{align}
		\lim_{r_a\to 0} \frac{\cH^k\Big(B_{r_a}(x)\cap (\cC_0\setminus\cR)\Big)}{r_a^k}>\epsilon_k>0.
	\end{align} 
	In particular, there exists $s>0$ such that $\cH^k\Big(B_{s}(x)\cap (\cC_0\setminus\cR)\Big)\ge s^k\epsilon_k/2$. Since $\{x_i\}$ is dense, there exists some $x_i$ and $r\in\mathbb{Q}$ such that $s\le r\le 2s$ and $B_s(x)\subset B_r(x_i) $. Therefore, we have 
	$\cH^k(B_r(x_i)\cap (\cC_0\setminus \cR))\ge C(k)\epsilon_k r^k$. By choosing $\epsilon=\epsilon(n)$ small, we contradict the definition of $\cR_\epsilon$ in 
	Lemma \ref{l:claim32}.  Thus, for $\delta\le \delta(n,\rv,\eta,\epsilon)=\delta(n,\rv,\eta)$, the set $\cR\subset \cC_0$ has full measure. This completes the proof of Lemma \ref{l:claim34}.  
\end{proof}	
\vskip2mm

At this point, we can obtain the lower bound for the packing measure $\mu$
and hence, complete the proof of Theorem \ref{t:neck_region2}. 
Since $\cC_0$ is $k$-rectifiable, by Theorem \ref{t:approximate_neck} we have $\mu\ge C(n)\cdot\mu_\infty$ in \eqref{e:limit_packing}. Therefore, the Ahlfors regularity estimate for $\mu_\infty$ in \eqref{e:limit_packing} gives us the desired lower bound for $\mu$.  This completes the proof of the Neck Structure Theorem \ref{t:neck_region2}. 	
\end{proof}

\vspace{1cm}


\section{Neck Decomposition Theorem}
\label{s:decomposition}
In this section we prove the Neck Decomposition Theorem \ref{t:decomposition2}.  
Neck regions and their associated decomposition theorems were introduced in \cite{JiNa_L2},
 where the focus was on the top $(n-4)$-stratum of the singular set for limits with $2$-sided Ricci curvature bounds.  
This was an important ingredient in the proof of the a priori $L^2$ curvature bound for such spaces.  
This section follows very closely the constructions of \cite{JiNa_L2}, relying on the estimates provided by 
the Neck Structure Theorem \ref{t:neck_region2}.  The main result of this section is Theorem \ref{t:decomposition2}, 
which for convenience is recalled below.

\noindent
{\bf Theorem} {\it (Theorem \ref{t:decomposition2} restated)
\label{t:neck_decomposition}
		Let $(M^n_i,g_i,p_i)\to (X^n,d,p)$ satisfy $\Vol(B_1(p_i))>\rv>0$ 
		and $\Ric_i\geq -(n-1)$.  Then for each $\eta>0$ and $\delta\le \delta(n,\rv,\eta)$ we can write:
	\begin{align}
		&B_1(p)\subseteq \bigcup_a \big(\cN_a\cap B_{r_a}\big) \cup \bigcup_b B_{r_b}(x_b) \cup \cS^{k,\delta,\eta}\, ,\\
		&\cS^{k,\delta,\eta}\subseteq \bigcup_a \big(\cC_{0,a}\cap B_{r_a}\big)\cup \tilde \cS^{k,\delta,\eta}\, ,
	\end{align}
	such that:
	\begin{enumerate}
		\item 
		For all $a$, the set, $\cN_a=B_{2r_a}(x_a)\setminus \overline B_{r_x}(\cC)$,
		is a $(k,\delta,\eta)$-neck region.
		\vskip2mm
		
		\item The balls $B_{2r_b}(x_b)$ are $(k+1,2\eta)$-symmetric; hence
		$x_b\not\in S^k_{2\eta,r_b}$.		\vskip2mm
		
		\item $\sum_a r_a^{k} + \sum_b r_b^{k} + \cH^{k}
		\big(\cS^{k,\delta,\eta}
		\big) \leq C(n,\rv,\delta,\eta)$.		
		\vskip2mm
		
		\item $\cC_{0,a}\subseteq B_{2r_a}(x_a)$ is the $k$-singular set associated to $\cN_a$.		\vskip2mm
		
		\item $\tilde{\cS}^{k,\delta,\eta}
		$ satisfies $\cH^{k}\big(\tilde \cS^{k,\delta,\eta}
		\big)=0$.		\vskip2mm
		
		\item $\cS^{k,\delta,\eta}$
		is $k$-rectifiable.		
		\vskip2mm
		
		\item For any $\epsilon$ if $\eta\le \eta(n,\rv,\epsilon)$ and $\delta\le \delta(n,\rv,\eta,\epsilon)$ we have 
		$S_\epsilon^k
		\subset \cS^{k,\delta,\eta}
		$.
	\end{enumerate}
}

\begin{remark}
As previously mentioned, in the special case of smooth Riemannian manifolds  $M^n$
only (1)--(3) carry nontrivial information. 
\end{remark}

\subsection{Proof of Theorem \ref{t:decomposition2} modulo Proposition \ref{p:inductive_decomposition}}

The proof of Theorem \ref{t:decomposition2} proceeds via an iterative recovering argument.
In Proposition \ref{p:inductive_decomposition} of this subsection, we will 
introduce a rougher decomposition which also includes a third type of ball, 
indexed by a subscript denoted by $v$.  Then, a recovering argument will lead
to a definite decrease in the volume of the $v$-balls.
Thus, after applying  this recovering argument a definite
number of times,  the $v$-balls will no longer present. This gives  
the decomposition in Theorem \ref{t:decomposition2}.
\vskip2mm

The remaining subsections will  be devoted to 
establishing Proposition \ref{p:inductive_decomposition}, which is the primary work in the proof.
 Initially, this will involve the introduction of coverings in which additional types of balls indexed 
by $c,d,e$ will appear.  The goal will be to produce covering
s for each type of ball that will cover
 {\it most} of each ball by either $a$, $b$ or $v$ balls, which will eventually lead to 
Proposition \ref{p:inductive_decomposition} itself.  Several subsections will be required for 
this process; for additional details, see Subsection \ref{ss:lemma_implies_prop}.
\vskip2mm

To avoid confusion, we recall that in \eqref{e:vbardef} below, the subscript $1$ indicates 
radius $1$.  Set
\begin{align}
\label{e:vbardef}
\bar V:=  \inf_{y\in B_4(p)}\cV_1(y)\ge \rv>0, .
\end{align}
\vskip1mm

\begin{proposition}[Induction step decomposition]
\label{p:inductive_decomposition}
 For all $\eta>0$ and $\delta\le \delta(n,\rv,\eta)$, there exists 
 $$v^0(n,\rv,\delta,\eta)>0$$ such that if 
$(M^n_i,g_i,p_i)\stackrel{d_{GH}}{\longrightarrow} (X^n,d,p)$ satisfies 
$\Ric_{M^n_i}\ge -(n-1)$ and $\Vol(B_1(p_i))>v>0$, then 
\begin{align}
		B_1(p)\subset \bigcup_a (\cC_{0,a}\cup \cN_a\cap B_{r_a}(x_a))
\cup \bigcup_b B_{r_b}(x_b)\cup \bigcup_v B_{r_v}(x_v)\cup \tilde{\cS}^k\, ,
	\end{align}
such that the following hold: 
\vskip2mm

	\begin{itemize}
	\item[(1)] $\cN_a\subset B_{2r_a}(x_a)$ are $(k,\delta,\eta)$-neck regions with associated singular set of centers $\cC_{0,a}$.
\vskip2mm	
	
	\item[(2)] Each $b$-ball $B_{2r_b}(x_b)$ is $(k+1,2\eta)$-symmetric.
\vskip2mm

	\item[(3)] $\bar V_v\ge \bar V+v^0$ \,\, (where $\bar V_v:=\inf_{y\in B_{4r_v}(x_v)}\cV_{r_v}(y)$).
\vskip2mm

\item[(4)] $\tilde{\cS}^k\subset S(X^n)$ and $\cH^k(\tilde{\cS}^k)=0$.
\vskip2mm

\item[(5)]
$\sum_a r_a^{k}+\sum_b r_b^{k}+\sum_v r_v^{k}\le C(n,\rv,\delta,\eta)\, .$
\end{itemize}
\end{proposition}
\vskip2mm

Assuming temporarily Proposition \ref{p:inductive_decomposition}, let us complete the proof of
Theorem \ref{t:decomposition2}:

\begin{proof}[Proof of Theorem \ref{t:decomposition2}]

 Fix $\eta>0$, $\delta\le \delta(n,\rv,\eta)$ as in Theorem \ref{t:neck_region2} and
  $v^0(n,\rv,\delta,\eta)>0$ as in Proposition \ref{p:inductive_decomposition}. 
  \vskip2mm
  
  By applying Proposition \ref{p:inductive_decomposition} 
to the limit ball $B_1(p)$, 
  we get the following decomposition in which the subscript $1$ indicates the first step in the inductive argument below.
    \vskip-2mm 
    
\begin{align}
B_1(p)\subset \tilde{\cS}_1^k\cup\bigcup_{a_1} (\cC_{0,a_1}\cup\cN_{a_1}\cap 
B_{r_{a_1}}(x_{a_1}))\cup \bigcup_{b_1} B_{r_{b_1}}(x_{b_1})\cup \bigcup_{v_1} 
B_{r_{v_1}}(x_{v_1})\, ,
\end{align}
where, 
\begin{align}
\label{e:v1}
&\bar V_{v_1} := \inf_{y\in B_{4r_{v_1}}(x_{v_1}) }\cV_{r_{v_1}}(y)\ge \bar V+\rv^0\, \\
&\cH^k(\tilde{\cS}_1^k) =0\, ,	\\
&\sum_{{a_1}} \cH^k(\cC_{0,{a_1}})+\sum_{{a_1}} (r_{a_1})^k+\sum_{{b_1}} (r_{b_1})^k+\sum_{v_1} (r_{v_1})^k
	\le C(n,\rv,\eta,\delta)\, .
\end{align}
\vskip4mm

\noindent
Next, by applying Proposition \ref{p:inductive_decomposition} to each $v_1$-ball 
$B_{r_{v_1}}(x_{v_1})$ we arrive at 
\begin{align}
B_1(p)\subset \bigcup_{j =1}^2\left(\tilde{\cS}_j^k\cup\bigcup_{a_j} (\cC_{0,a_j}\cup\cN_{a_2}\cap 
B_{r_{a_j}}(x_{a_j}))\cup \bigcup_{b_2} B_{r_{b_j}}(x_{b_j})\right)\cup \bigcup_{v_2} 
B_{r_{v_2}}(x_{v_2})\, ,
\end{align}
where, 
\begin{align}
\label{e:v2}
&\bar V_{v_2}:=  \inf_{y\in B_{4r_{v_2}}(x_{v_2})}\cV_{r_{v_2}}(y)\ge \bar V+2\rv^0\, \\
&\cH^k(\tilde{\cS}_1^k) +  \cH^k(\tilde{\cS}_2^k)  =0\, ,\\	
&\sum_{j=1}^2\left(\sum_{{a_2}} \cH^k(\cC_{0,{a_2}})+\sum_{{a_2}} (r_{a_2})^k+\sum_{{b_2}} (r_{b_2})^k\right) \leq C(n,\rv,\eta,\delta)+  C(n,\rv,\eta,\delta)^2  \, .\\
	&\sum_{v_2}(r_{v_2})^k\le  C(n,\rv,\eta,\delta)^2  \, .
\end{align}
\vskip2mm
 
Note that $\bar V+2\rv^0$ in \eqref{e:v1}, has been replaced by $\bar V+2\rv^0$ 
in \eqref{e:v2}, where as in Proposition \ref{p:inductive_decomposition},  $v^0=v^0(n,\rv,\delta,\eta)$. Therefore, this process of recovering the $v$-balls can  be interated at most $i=i(n,\rv,\delta,\eta)$ times before no $v$-balls exist; otherwise, we would contradict the noncollapsing assumption
\eqref{e:lvb}. By doing  the iteration the maximal number of times,
we  obtain the following decomposition in which the $v$-balls are no longer present:
\begin{align}
B_1(p)\subset \bigcup_{j =1}^i\left(\tilde{\cS}_j^k\cup\bigcup_{a_j} (\cC_{0,a_j}\cup\cN_{a_j}\cap B_{r_{a_j}}(x_{a_j}))\cup \bigcup_{b_j}B_{r_{b_j}}(x_{b_j})\right)\, ,
\end{align}
where, $i=i(n,\rv,\delta,\eta)$ and
\begin{align}
&\cH^k(\tilde{\cS}_1^k) + \cdots  \cH^k(\tilde{\cS}_i^k)  =0\, ,\\ \label{e:vi}	
&\sum_{j=1}^i\left(\sum_{{a_j}} \cH^k(\cC_{0,{a_j}})+\sum_{{a_j}} (r_{a_j})^k+\sum_{{b_j}} (r_{b_j})^k\right) \leq C'(n,\rv,\eta,\delta)  \, .
\end{align}
Set 
\begin{align}
\tilde{\cS}^{k,\delta,\eta} := \bigcup_{j=1}^i\tilde{S}^k_j\cap B_1(p)\, ,\qquad
\cS^{k,\delta,\eta}
:=\bigcup_{j=1}^i\left(\tilde{S}^k_j \cup \bigcup_{a_j} \cC_{0,a_j}\right)\cap B_1(p)\, ,
\end{align}
\vskip2mm

\noindent
Since by the Neck Structure Theorem \ref{t:neck_region2}, 
each set $\cC_{0,a_j}$ is $k$-rectifiable, it follows that
$\cS^{k,\delta,\eta}$ is $k$-rectifiable and by \eqref{e:vi} that $\cH^k(\cS^{k,\delta,\eta})\le C(n,\rv,\eta,\delta)$. 
This gives the decomposition whose existence is asserted in Theorem
\ref{t:decomposition2}. Moreover, from our decomposition,  (1)-(6) in Theorem \ref{t:decomposition2} 
are satisfied, where the content estimate is in \eqref{e:vi} and $\cH^k(\tilde{\cS}^{k,\delta,\eta})=0$. 
\vskip2mm

Finally, we will show that if $\eta\leq \eta(n,\rv,\epsilon)$,
$\delta\leq \delta(n,\rv,\eta,\epsilon)$, then 
 $S^k_{\epsilon}\subset \cS^{k,\delta,\eta}$ which is the last statement (7) in Theorem \ref{t:decomposition2}
 \vskip2mm
 
 First, note that if $y\in \cN_a$, with  $r=d(y,\cC_a)$
 and $\delta\leq \delta(n,\rv,\eta,)$,
 then by the Cone-Splitting Theorem \ref{t:content_splitting}, the ball
 $B_{r/2}(y)$ has a $(k+1,2\eta)$-splitting.
 For any $\epsilon$, by the Almost Volume Cone implies
 Almost Metric Cone Theorem \ref{t:almostmetriccone}, it follows that
for some, $s=s(\epsilon,\rv)\cdot r$, the ball,
 $B_s(y)$ is $(0,\epsilon^3)$-symmetric.
 If in addition, $\eta\leq \eta(n,\rv,\epsilon)$,
this implies that $B_s(y)$ is $(k+1,\epsilon^2)$symmetric.
Hence, $y\not\in S^k_\epsilon$.
\vskip2mm

Similarly, suppose $y\in B_{r_b}(x_b)$ and
 $B_{2r_b}(x_b)$ is $(k+1,\eta)$-symmetric. If in addition,
 $\eta\leq \eta'(n,\rv,\epsilon)$, then it clear that
 $B_{r_b}(x_b)$ has a $(k+1,\eta')$-splitting.  Then the same argument as above shows that if 
 $\eta'\leq \eta'(n,\rv,\epsilon)$, the $y\not\in S^k_\epsilon$. Since $S^k_\epsilon$ is covered by union of $\cN_a, B_{r_b}$ and $\cS^{k,\delta,\eta}$, we see that $S_\epsilon^k\subset \cS^{k,\delta,\eta}$.
 This completes the proof of Theorem \ref{t:decomposition2},
 modulo the proof of Proposition \ref{p:inductive_decomposition}. 
 \end{proof}

The remainder of this section will now be devoted
to proving Proposition \ref{p:inductive_decomposition}. 
\vskip-2mm

\subsection{Notation:  Constants and Types of balls}
\label{ss:notation}

Throughout the remainder of this section we will consider  constants  $\xi,\delta,\gamma,\epsilon$,   which will in general satisfy
\begin{align}
	0<\xi<<\delta<\gamma<\epsilon<\epsilon(n).
\end{align}
We will assume through out that $\Ric_{M^n}\ge-(n-1)\xi $. The general case,
can be achieved by a standard covering argument and rescaling.\footnote{Given $\xi<<\delta$, choose a Vitali covering, $\{B_{\xi}(y_f)\}$, of $B_1(p)$, such that $B_{\xi/5}(y_f)$ are disjoint. By relative volume comparison,  the cardinality of such covering is less than $C(n,\rv,\xi)$. Finding the desired decomposition for $B_1(p)$ is then reduced to finding the corresponding decomposition for each $B_{\xi}(y_f)$.} 
\vskip2mm

As in Definition \ref{d:smallvolpinch}, we  define
the set of points with small volume pinching by:
\begin{align}
\label{e:hatV8}
\bar V :=\inf_{x\in B_4(p)}\cV_{\xi^{-1}}(x))\, .
\end{align}
In what follows, the set with small volume pinching is defined to be:
\begin{align}
\label{e:cFset18}
\cF_{r,\xi}(x):=\{y\in B_{4r}(x):~\cV_{\xi r}(y)\le \bar V+\xi\}\, .
\end{align} 
 \vskip2mm

 The constants, $\epsilon,\gamma>0$ will denote the constants in
 the Cone-Splitting Theorem based on $k$-content,  Theorem \ref{t:content_splitting}: 
 If $\Vol(B_{\gamma}(\cF_{1,\xi}(p)))\ge \epsilon \gamma^{n-k} $ with $0<\delta,\epsilon\le \delta(n,\rv)$,  $\gamma\le \gamma(n,\rv,\epsilon)$, $\xi\le \xi(\delta,\epsilon,\gamma,n,\rv)$,
then there exists $q\in B_4(p)$ such that $B_{\delta^{-1}}(q)$ is $(k,\delta^2)$-symmetric.\\

 Next we introduce the various ball types which appear in the proof. These are indexed by
  $a,\, b,\, c,\, d,\, e$.
Every ball $B_r(x)$ is one (or more) of these types. The balls indexed by $a,\, b$ are of the
 type as in Proposition \ref{p:inductive_decomposition}.
\vskip2mm
 
\begin{itemize}
	\item[(a)] A ball $B_{r_a}(x_a)$  is associated to a $(k,\delta,\eta)$-neck region $\cN_a\subset B_{2r_a}(x_a)$.
\vskip2mm

	\item[(b)] A ball $B_{r_b}(x_b)$  is  $(k+1,2\eta)$-symmetric.
	\vskip2mm

	\item[(c)] A ball $B_{r_c}(x_c)$  is not a $b$-ball and satisfies:\newline
$$
\Vol\Big(B_{\gamma r_c}(\cF_{r_c,\xi}(x_c))\Big)\ge \epsilon \gamma^{n-k} r_c^n\, .
$$
	\vskip2mm

	\item[(d)] A ball $B_{r_d}(x_d)$  is any ball with $\cF_{r_d,\xi}(x_d)\ne \emptyset$
	satisfying:
$$
\Vol\Big(B_{\gamma r_d}(\cF_{r_d,\xi}(x_d))\Big)< \epsilon \gamma^{n-k} r_d^n\, .
$$
	\vskip2mm

	\item[(e)] A ball $B_{r_e}(x_e)$ satisfies  $\cF_{r_e,\xi}(x_e)=\emptyset$.
	\end{itemize}
\vskip3mm

\subsection{Statements of Proposition \ref{p:d_ball_decomposition} and Proposition \ref{p:c_ball_decomposition}}
\label{ss:2propositions}

The first proposition stated in this subsection, asserts that a 
$d$-ball can be recovered using only
 $b$, $c$, and $e$-ball's.  A key point is that  in this covering, the content of the $c$-balls
 in the collection can be taken to be small.
  \vskip2mm

\begin{proposition}[d-ball decomposition]
\label{p:d_ball_decomposition}
	Let $(M_i^n,g_i,p_i)\stackrel{d_{GH}}{\longrightarrow} (X^n,d,p)$ satisfy 
$\Vol(B_1(p_i))\ge \rv>0$ and let $\eta>0$ and $\bar V\le \inf_{x\in B_4(p)}\cV_{\xi^{-1}}(x)$. 
 For any $\epsilon\le \epsilon(n,\rv)$, $\gamma\le \gamma(n,\rv,\epsilon)$, $\delta\le \delta(n,\rv,\eta)$ 
and $\xi\le \xi(n,\rv,\epsilon,\gamma,\delta,\eta)$. Then the following holds. If $\Ric_{M^n_i}\ge-(n-1)\xi$ 
and $\Vol(B_{\gamma}(\cF_{1,\xi}(p)))< \epsilon \gamma^{n-k}$, then we have the decomposition: 
	\begin{align}
	B_1(p)\subseteq \tilde {\cS}_d^k\cup \bigcup B_{r_b}(x_b)\cup \bigcup B_{r_c}(x_c)\cup \bigcup B_{r_e}(x_e)
 \end{align}
 where 
	\begin{itemize}
		\item[(b)] Each $b$-ball $B_{2r_b}(x_b)$  is $(k+1,2\eta)$-symmetric.
		\vskip2mm
		
		\item[(c)] $c$-ball $B_{2r_c}(x_c)$ is not a $b$-ball and satisfies 
		$\Vol(B_{\gamma r_c}\cF_{r_c,\xi}(x_c))\ge \epsilon \gamma^{n-k}r_c^n$.
	\vskip2mm
	
		\item[(e)] Each $e$-ball $B_{2r_e}(x_e)$ satisfies $\cF_{r_e,\xi}(x_e)=\emptyset$.		
	\vskip2mm
		
		\item[(s)] $\tilde{\cS}_d^k\subset S(X)$ and 
		$\cH^k( \tilde{\cS}_d^k)=0$.
		\vskip2mm
		
	\end{itemize}
	Furthermore, we have $k$-content estimates 
$\sum_{b}r_b^k+\sum_e r_e^k\le C(n,\gamma)$ and $\sum_c r_c^k\le C(n,\rv)\epsilon$.
\end{proposition}
\vskip1mm

\begin{remark}
In this proposition, the ball types and the pinching set $\cF_{r,\xi}$ 
are with respect to the given $\bar V\le \inf_{x\in B_4(p)}\cV_{\xi^{-1}}(x)	$ above.
\end{remark}
\vskip1mm

\begin{proposition}[$c$-ball decomposition]
\label{p:c_ball_decomposition}
	Let $(M_i,g_i,p_i)\to (X,d,p)$ satisfy $\Vol(B_1(p_i))\ge \rv>0$ and 
let $\eta>0$ and $\bar V\le \inf_{x\in B_4(p)}\cV_{\xi^{-1}}(x)$. For any 
$\epsilon\le \epsilon(n,\rv)$, $\gamma\le \gamma(n,\rv,\epsilon)$,
 $\delta\le \delta(n,\rv,\eta)$ and $\xi\le \xi(n,\rv,\epsilon,\gamma,\delta,\eta)$. 
Then the following holds. If $\Ric_{M^n_i}\ge-(n-1)\xi$, 
$\Vol(B_{\gamma}(\cF_{1,\xi}(p)))\ge \epsilon \gamma^{n-k}$ and $B_2(p)$ is not $(k+1,2\eta)$-symmetric, 
	then we have the decomposition:
	\begin{align}
 	B_1(p)\subset \Big(\cC_0\cup \cN\cap B_1(p)\Big)\cup \bigcup_b B_{r_b}(x_b)\cup 
\bigcup_c B_{r_c}(x_c)\cup \bigcup_d B_{r_d}(x_d)\cup \bigcup_e B_{r_e}(x_e)\, ,
 \end{align}
	where 
	\begin{itemize}
	    \item[(a)] $\cN=B_2(p)\setminus \Big(\cC_0\cup \bigcup_b B_{r_b}(x_b)\cup \bigcup_c B_{r_c}(x_c)\cup \bigcup_d B_{r_d}(x_d)\cup \bigcup_e B_{r_e}(x_e)\Big)$ is a $(k,\delta,\eta)$-neck region.
	    \vskip2mm
	    
		\item[(b)] Each $b$-ball $B_{2r_b}(x_b)$ is $(k+1,2\eta)$-symmetric.
		    \vskip2mm

		\item[(c)] Each $c$-ball $B_{2r_c}(x_c)$ is not 
		$(k+1,2\eta)$-symmetric and satisfies $\Vol(B_{\gamma r_c}\cF_{r_c,\xi}(x_c))\ge \epsilon \gamma^{n-k}r_c^n$. 
    \vskip2mm
	    
\item[(d)] Each $d$-ball $B_{2r_d}(x_d)$ satisfies 
$\Vol(B_{\gamma r_d}\cF_{r_d,\xi}(x_d))< \epsilon \gamma^{n-k}r_d^n$. 
    \vskip2mm
    
		\item[(e)] Each $e$-ball $B_{2r_e}(x_e)$ satisfies  $\cF_{r_e,\xi}(x_e)=\emptyset$. 
	\end{itemize}
	Furthermore, we can build this decomposition so that we have the $k$-content estimates 
	\begin{align}
		&\sum_{x_b\in B_{3/2}(p)}r_b^k+\sum_{x_d\in B_{3/2}(p)} r_d^k+\sum_{x_e\in B_{3/2}(p)} r_e^k+\cH^k(\cC_0\cap B_{3/2}(p))\le C(n,\rv)\, ,\\
		 &\sum_{x_c\in B_{3/2}(p)} r_c^k\le C(n,\rv)\epsilon.
	\end{align}
\end{proposition}
\begin{remark}
In this Proposition the ball types and the pinching set $\cF_{r,\xi}$ are defined with respect to the given $\bar V\le \inf_{x\in B_4(p)}\cV_{\xi^{-1}}(x)	$ above.
\end{remark}
\vskip1mm

\subsection{Proof of Proposition \ref{p:inductive_decomposition} modulo
Propositions \ref{p:d_ball_decomposition} and \ref{p:c_ball_decomposition}}
\label{ss:lemma_implies_prop}
In this subsection we will state and prove Lemma \ref{l:n9}.
The proof involves using iteratively the decompositions of Proposition \ref{p:c_ball_decomposition} and Proposition \ref{p:d_ball_decomposition}.
 Then by using Lemma \ref{l:n9} a definite number of times we are able to
 remove all the $c$-balls and $d$-balls, thereby proving Proposition \ref{p:inductive_decomposition}. This proves Theorem \ref{t:neck_decomposition}
 modulo the proofs of Proposition \ref{p:c_ball_decomposition} and Proposition \ref{p:d_ball_decomposition}. 
These two propositions will be proved in the
 remaining two subsections.

 \begin{lemma}
\label{l:n9}
	Let $(M_i,g_i,p_i)\togh (X^n,d,p)$ satisfy 
$\Vol(B_1(p_i))\ge \rv>0$ with $\eta>0$ and $\bar V:= \inf_{x\in B_4(p)}\cV_{\xi^{-1}}(x)$.
Then for $\delta\le \delta(n,\rv,\eta)$ and $\xi\le \xi(n,\rv,\delta,\eta)$, if 
$\Ric_{M^n_i}\ge-\xi(n-1)$, we have 
	\begin{align}
		B_1(p)\subset \bigcup_a (\cC_{0,a}\cup \cN_a\cap B_{r_a}(x_a))\cup \bigcup_b B_{r_b}(x_b)\cup \bigcup_e B_{r_e}(x_e)\cup \tilde{\cS}^k\, ,	
		\end{align}
	where 
	\begin{itemize}
	\item[(1)] $\cN_a\subset B_{2r_a}(x_a)$ are $(k,\delta,\eta)$-neck regions with associated singular set $\cC_{0,a}$.
\vskip2mm
	\item[(2)] Each $b$-ball $B_{2r_b}(x_b)$ is $(k+1,2\eta)$-symmetric.
\vskip2mm

	\item[(3)] For each $e$-ball $B_{2r_e}(x_e)$ we have $\cF_{r_e,\xi}(x_e)=\emptyset$ where $\cF_{r_e,\xi}(x_e)
:=\{y\in B_{4r_e}(x_e): \cV_{\xi r_e}(y)\le \bar V+\xi\}$.
\vskip2mm
	\item[(4)] $\tilde{\cS}^k\subset S(X)$ and $\cH^k(\tilde{\cS}^k)=0$.
\end{itemize}
Moreover, we have content estimate
$$
\sum_a r_a^{k}+\sum_b r_b^{k}+\sum_e r_e^{k}\le C(n,\rv)\, .
$$
\end{lemma}
\begin{proof}
	 Fix $\epsilon\le \epsilon(n,\rv)$, $\gamma\le \gamma(n,\rv,\epsilon)$ and $\delta\le \delta(n,\rv,\eta)$ such that Proposition \ref{p:c_ball_decomposition} and Proposition \ref{p:d_ball_decomposition} hold.
 \vskip2mm
 
We can assume $B_2(p)$ is not a $b$-ball or $e$-ball. Otherwise, there is nothing to prove. 
\vskip2mm

So assume one of the following two cases holds.

\begin{itemize}
\item[1)]
 $B_2(p)$ is a $c$-ball with $\Vol(B_{\gamma}\cF_{1,\xi}(p))\ge \epsilon \gamma^{n-k}$ and with $B_2(p)$ is not $(k+1,2\eta)$-symmetric.
 \vskip2mm
 
\item[2)]$B_2(p)$ is a $d$-ball with $\Vol(B_{\gamma}\cF_{1,\xi}(p))< \epsilon \gamma^{n-k}$. 
\end{itemize}
\vskip1mm

It will be evident that up to reversing the order of which decomposition we apply first, the argument is the same in both cases. Therefore, without essential loss of generality, we will assume that $B_2(p)$ is a $c$-ball. 
 \vskip2mm
 
 By the $c$-ball decomposition Proposition \ref{p:c_ball_decomposition}, if
  $\xi\le \xi(n,\rv,\delta,\epsilon,\eta)$, then we have:
 \vskip-2mm
  
\begin{align}
		B_1(p)\subseteq (\cC_{0}\cup \cN\cap B_1(p))\cup \bigcup_b B_{r_b}(x_b)\cup \bigcup_c B_{r_c}(x_c)\cup \bigcup_d B_{r_d}(x_d)\cup \bigcup_e B_{r_e}(x_e)\, ,
	\end{align}
and in addition  the following $k$-content estimates hold
		\begin{align}
&\sum_{b}r_b^k+\sum_d r_d^k+\sum_e r_e^k+\cH^k(\cC_{0})\le C(n)\, ,\\
&\sum_c r_c^k \le C(n,\rv)\epsilon\, .
\end{align}

By applying the $d$-ball decomposition of Proposition \ref{p:d_ball_decomposition} to each $d$-ball $B_{2r_d}(x_d)$, we arrive at 
\begin{align}
	B_1(p)\subseteq \tilde\cS^k_1\cup (\cC_{0}\cup \cN\cap B_1(p))\cup \bigcup_b B_{r_b}(x_b)\cup \bigcup_c B_{r_c^1}(x_c^1)\cup \bigcup_e B_{r_e}(x_e)\, ,
\end{align}
where $\tilde{\cS}_1^k =\bigcup_{d}\tilde{\cS}_{d}^k$ is a countable union of $k$-Hausdorff meansure zero sets, and thus $\cH^k(\tilde{\cS}_1^k)=0$. 
Moreover, we have content estimates 
\begin{align}
	&\sum_c (r_c^1)^k\le C(n,\rv)\epsilon+C(n)C(n,\rv)\epsilon\le \bar C(n,\rv)\epsilon,\\
	&\sum_{b}r_b^k+\sum_e r_e^k+\cH^k(\cC_{0})\le C(n)+C(n)C(n,\gamma)\le \bar C(n,\gamma).
\end{align}

Next, we repeat the above process verbatim, except that we  first apply the $c$-ball decomposition of Proposition \ref{p:c_ball_decomposition} to each $c$-ball above and then apply the $d$-ball decomposition of Proposition \ref{p:d_ball_decomposition} to each remaining $d$-ball.  The result is: 
\begin{align}
	B_1(p)\subseteq \tilde\cS^k_2\cup \bigcup_a (\cC_{0,a}\cup \cN_a\cap B_{r_a}(x_a))\cup \bigcup_b B_{r_b}(x_b)\cup \bigcup_c B_{r_c^2}(x_c^2)\cup \bigcup_e B_{r_e}(x_e)\, ,
\end{align}
with content estimates $\cH^{k}(\tilde{\cS}_2^k)=0$ and
\begin{align}
	&\sum_a r_a^{k}\le 1+\bar C(n,\rv)\epsilon,~~~\sum_c (r_c^2)^k\le \Big(\bar C(n,\rv)\epsilon\Big)^2,\\
	&\sum_b r_b^k+\sum_e r_e^k+\sum_a\cH^k(\cC_{0,a})\le \bar C(n,\gamma)\Big(1+\bar C(n,\rv)\epsilon\Big).
\end{align}

After repeating this process $i$ times we arrive at
\begin{align}\label{e:B1_strategy_i}
	B_1(p)\subseteq \tilde\cS^k_i\cup \bigcup_a (\cC_{0,a}\cup \cN_a\cap B_{r_a}(x_a))\cup \bigcup_b B_{r_b}(x_b)\cup \bigcup_c B_{r_c^i}(x_c^i)\cup \bigcup_e B_{r_e}(x_e)\, ,
\end{align}
with content estimates 	$\cH^{k}(\tilde{\cS}^k_i)=0$ and
\begin{align}
	&\sum_a r_a^{k}\le \sum_{j=0}^i\Big(\bar C(n,\rv)\epsilon\Big)^j,~~~\sum_c (r_c^i)^k\le \Big(\bar C(n,\rv)\epsilon\Big)^i,\\ \label{e:content_bea}
	&\sum_b r_b^k+\sum_e r_e^k+\sum_a\cH^k(\cC_{0,a})\le \bar C(n,\gamma)\sum_{j=0}^i\Big(\bar C(n,\rv)\epsilon\Big)^j.
\end{align}

Consider the discrete set $\tilde{\cS}_{c}^i:= \{x_c^i\}$. 
By the construction, we have 
\begin{align}
B_{2r_c^{i+1}}\tilde{\cS}_c^{i+1}\subset B_{2r_c^i}(\tilde{\cS}_c^i)\, ,
\end{align}
 where 
\begin{align}
B_{2r_c^i}(\tilde{\cS}_c^i):=  \cup_c B_{2r_c^i}(x_c^i)\, .
\end{align}
 Define the set limit by:
\begin{align}\label{e:cS_c_setlimit}
	\tilde{\cS}_c:=  \bigcap_{i\ge 1}\bigcup_{j\ge i}B_{2r_c^i}(\tilde{\cS}_c^i).
\end{align}
It is clear from the construction that $\tilde{\cS}_c\subset S(X^n)$. 
Set $\delta_i:=2\max_c r_c^i$. Since $\tilde{\cS}_c\subset B_{2r_c^i}(\tilde{\cS}_c^i)$, 
we have by the definition of Hausdorff measure,
\begin{align}
	H^k_{\delta_i}(\tilde{\cS}_c):=  \inf\Big\{\sum_{\alpha} r_{\alpha}^k,~~~\text{ where $r_{\alpha}\le \delta_i$ and } \tilde{\cS}_c\subset \cup_{\alpha}B_{r_{\alpha}}(y_\alpha) \Big\}\le 2^k\sum_c (r_c^i)^k\le 2^k\Big(\bar C(n,\rv)\epsilon\Big)^i,
\end{align}
which implies $\cH^k(\tilde{\cS}_c)=0$ . 
\vskip2mm

Set $	\tilde{\cS}^k:= \tilde{\cS}_c\cup \bigcup_{i\ge 1}\tilde{\cS}_i^k.$
Then $\cH^k(\tilde{\cS}^k)=0$ and $\tilde{\cS}^k\subset S(X)$. 
\vskip2mm

Fix $\epsilon=\epsilon(n,\rv)$ and $\gamma=\gamma(n,\rv)$ such that $\bar C(n,\rv)\epsilon\le 1/10$.  Then by taking the limit as $i\to\infty$, we will arrive at the decomposition 
\begin{align}
\label{e:B1_limit_covering}
	B_1(p)\subset \tilde\cS^k\cup \bigcup_a (\cC_{0,a}\cup \cN_a\cap B_{r_a}(x_a))\cup \bigcup_b B_{r_b}(x_b)\cup \bigcup_e B_{r_e}(x_e)\, .
\end{align}
To see \eqref{e:B1_limit_covering},  if $y\in B_1(p)\setminus \tilde{\cS}^k$ then by \eqref{e:cS_c_setlimit} we must have $y\notin B_{2r_c^i}(\tilde{\cS}_c^i)$ for some $i$ which in particular implies by \eqref{e:B1_strategy_i} that $y$ belongs to the set on the right-hand side of \eqref{e:B1_limit_covering}. 
\vskip2mm

By letting $i\to \infty$ we have by \eqref{e:content_bea} the following content estimates 
\begin{align}
	&\sum_a r_a^{k}\le 2\, ,\\
       &\sum_b r_b^k+\sum_e r_e^k+\sum_a\cH^k(\cC_{0,a})
	\le  C(n,\rv)\, .
\end{align}
This  completes the proof of Lemma \ref{l:n9}.
\end{proof}
\vskip2mm

Now we can prove the inductive decomposition of Proposition \ref{p:inductive_decomposition}.
\begin{proof}[Proof of Proposition \ref{p:inductive_decomposition}]
	For any $\eta$ and $\delta\le \delta(n,\rv,\eta)$, fix $\xi=\xi(n,\rv,\delta,\eta)$ as in Lemma  
\ref{l:n9}.	
	Consider a Vitali covering $\{B_{\xi^2}(x_f)\}$ of $B_2(p)$ such that $B_{\xi^2/5}(x_f)$ are disjoint. Thus, by volume comparison, the number of such balls is bounded by a constant $L(n,\rv,\xi)$. By scaling the ball $B_{\xi^2}(x_f)$ to a unit ball, we arrive at a  unit ball satisfying all the condition of Lemma \ref{l:n9} with 
$$
\bar V_f:= \inf_{y\in B_{4\xi^2}(x_f)}\, , \qquad\cV_{\xi}(y)\ge  \inf_{y\in B_4(p)}\cV_1(y):= \bar V
$$.  If apply the decomposition of Lemma \ref{l:n9} to each ball $B_{\xi^2}(x_f)$ in order, we arrive at the covering
	\begin{align}
		B_1(p)\subset \tilde{\cS}^k\cup\bigcup_a (\cC_{0,a}\cup \cN_a\cap B_{r_a}(x_a))\cup \bigcup_b B_{r_b}(x_b)\cup \bigcup_e B_{r_e}(x_e)\, ,
	\end{align}
	with $r_a,r_b,r_e\le \xi^2$ and 
$$
\sum_a r_a^k+\sum_b r_b^k+\sum_e r_e^k\le C(n,\rv)L(n,\rv,\xi)\le C(n,\rv,\delta)
$$ 
and 
$$
\cH^k(\tilde{\cS}^k)=0\, .
$$ 
\vskip2mm	
	
	To finish the proof, it suffices to recover each $e$-ball by $v$-balls.  
In fact, for each $e$-ball $B_{r_e}(x_e)\subset B_{2\xi^2}(x_f)$, consider the 
Vitali covering $\{B_{\xi r_e}(x_e^j)\}$ of $B_{r_e}(x_e)$ with $x_e^j\in B_{r_e}(x_e)$ 
such that $B_{\xi r_e/5}(x_e^j)$ are disjoint. We will show that $B_{\xi r_e}(x_e^j)$ are 
$v$-ball for $v_0=\xi$\footnote{Recall $v$-balls are defined with respect to the background parameter $v_0$}. 
\vskip2mm
	
	Since $\cF_{r_e,\xi}(x_e):=
 \{y\in B_{4r_e}(x_e): \cV_{\xi r_e}(y)\le \Bar V_f+\xi\}=\emptyset$, 
we have for all $y\in B_{4r_e}(x_e)$ that $\cV_{\xi r_e}(y)\ge \Bar V_f+\xi\ge \Bar V+\xi$. 
On the other hand, we have $B_{4\xi r_e}(x_e^j)\subset B_{2r_e}(x_e)$.
 Therefore, $\inf_{y\in B_{4\xi r_e}(x_e^j)}\cV_{\xi r_e}(y)\ge \Bar V+\xi$. 
Setting $v_0:=\xi$ we have that $B_{\xi r_e}(x_e^j)$ is a $v$-ball as in 
Proposition \ref{p:inductive_decomposition}. The content estimate for $v$-balls
follows easily from the content estimate of $e$-balls and the Vitali covering.  
This completes the proof of  Proposition \ref{p:inductive_decomposition}.
\end{proof}
\vskip2mm

\subsection{Proof of $d$-ball covering Proposition \ref{p:d_ball_decomposition}}
\label{ss:d}


\begin{proof}[Proof of Proposition \ref{p:d_ball_decomposition}]
	For any $0<\epsilon,\gamma\le 1/10$, let us first consider a Vitali covering $\{B_{\gamma}(x_f^1),~x_f^1\in B_1(p)\}$ of $B_1(p)$ such that $B_{\gamma/5}(x_f^1)$ are disjoint.
	Let us seperate $\{B_{\gamma}(x_f^1)\}$ into $b$-balls, $c$-balls, $d$-balls and $e$-ball's from subsection
 \ref{ss:notation}:
\begin{align}
	B_1(p)\subseteq \bigcup_{b=1}^{N_b^1}B_{\gamma}(x_b^1)\cup \bigcup_{c=1}^{N_c^1}B_{\gamma}(x_c^1)
\cup \bigcup_{d=1}^{N_d^1}B_{\gamma}(x_d^1)\cup \bigcup_{e=1}^{N_e^1}B_{\gamma}(x_e^1)\, ,
\end{align}
where $B_{2\gamma}(x_b^1)$ is $(k+1,2\eta)$-symmetric, $B_{2\gamma}(x_c^1)$ is not $(k+1,2\eta)$-symmetric
 and satisfies $\Vol(B_{\gamma^2}\cF_{\gamma,\xi}(x_c^1))\ge \epsilon \gamma^{n-k}\gamma^n$, and $\Vol(B_{\gamma^2}\cF_{\gamma,\xi}(x_d^1))< \epsilon \gamma^{n-k}\gamma^n$, and $\cF_{\gamma,\xi}(x_e^1)=\emptyset$ with 
$\cF_{r,\xi}(x):=  \{y\in B_{4r}(x): \cV_{r\xi}(y)\le \bar V+\xi\}$ and $\bar V:=  \inf_{y\in B_4(p)}\cV_{\xi^{-1}}(y)$.  
By volume doubling we have 
	\begin{align}
		\sum_{b=1}^{N_b^1} \gamma^{k}+\sum_{e=1}^{N_e^1}\gamma^k\le C(n,\gamma)\gamma^k
\le C(n,\gamma)\, .
	\end{align}
	
	Let us prove a slightly more refined content estimate for the $c$-balls and $d$-balls.  Since 
$B_{2\gamma}(x_c^1),B_{2\gamma}(x_d^1)\subset B_2(p)$, we have $\cF_{\gamma,\xi}(x_c),\cF_{\gamma,\xi}(x_d)\subset \cF_{1,\xi}(p)$ where we should notice that in our setting a $d$-ball is not an $e$-ball. The following content estimates for $c$-balls and
 $d$-balls depend only on the fact that $\cF_{\gamma,\xi}(x_c)$ and $\cF_{\gamma,\xi}(x_d)$ are nonempty. 
We will only discuss the content estimate for $d$-balls, since the case of $c$-balls is no different from this one.
  Indeed, for each ball $B_{\gamma}(x_d^1)$, there exists a point $y_d^1\in B_{2\gamma}(x_d^1)\cap \cF_{1,\xi}(p)$ 
which in particular implies $B_{\gamma}(y_d^1)\subset B_{\gamma}\cF_{1,\xi}(p)$. The ball $B_{\gamma}(y_d^1)$ may
 overlap with other balls $B_{\gamma}(y_{d'}^1)$. Due to the Vitali covering property and volume doubling, the balls
 overlap at most $C(n)$ times.   By standard covering argument and noting 
$\Vol(B_{\gamma}(\cF_{1,\xi}(p)))< \epsilon \gamma^{n-k}$, we can now conclude that 
	\begin{align}
		\sum_{c=1}^{N_c^1} \gamma^{k}+\sum_{d=1}^{N_d^1}\gamma^k \le C(n,\rv)\epsilon\, .
	\end{align}
	For each $d$-ball $B_{\gamma}(x_d^1)$, let us repeat this decomposition. We get 
	\begin{align}
		\bigcup_{d=1}^{N_d^1}B_{\gamma}(x_d^1)\subset \bigcup_{b=1}^{N_b^2}B_{\gamma^2}(x_b^2)\cup \bigcup_{c=1}^{N_c^2}B_{\gamma^2}(x_c^2)\cup \bigcup_{d=1}^{N_d^2}B_{\gamma^2}(x_d^2)\cup \bigcup_{e=1}^{N_e^2}B_{\gamma^2}(x_e^2)\, .
	\end{align}
	Furthermore, by the same arguments as above we have the content estimates
	\begin{align}
		\sum_{b=1}^{N_b^2} \gamma^{2k}+\sum_{e=1}^{N_e^2}\gamma^{2k}\le C(n,\gamma)\sum_{d=1}^{N_d^1}\gamma^k\le C(n,\gamma)C(n,\rv)\epsilon\\
		\sum_{d=1}^{N_d^2} \gamma^{2k}+\sum_{c=1}^{N_c^2}\gamma^{2k}\le C(n,\rv)\epsilon\sum_{d=1}^{N_d^1}\gamma^k\le \Big(C(n,\rv)\epsilon\Big)^2\, .
			\end{align}
	Therefore, we arrive at the decomposition 
	\begin{align}
		B_1(p)\subset \bigcup_{d=1}^{N_d^2}B_{\gamma^2}(x_d^2)\cup \bigcup_{j=1}^2\bigcup_{b=1}^{N_b^j}B_{\gamma^j}(x_b^j)\cup \bigcup_{j=1}^2\bigcup_{c=1}^{N_c^j}B_{\gamma^j}(x_c^j)\cup \bigcup_{j=1}^2\bigcup_{e=1}^{N_e^j}B_{\gamma^j}(x_e^j)\, ,
	\end{align}
	with content estimates
	\begin{align}
		&\sum_{d=1}^{N_d^2}\gamma^{2k}\le \Big(C(n,\rv)\epsilon\Big)^2\, ,\\
		&\sum_{j=1}^2\sum_{b=1}^{N_b^j}\gamma^{jk}+\sum_{j=1}^2\sum_{e=1}^{N_e^j}
\gamma^{jk}\le  C(n,\gamma)+C(n,\gamma)C(n,\rv)\epsilon\le C(n,\gamma)\Big(1+C(n,\rv)\epsilon\Big)\, ,\\
		&\sum_{j=1}^2\sum_{c=1}^{N_c^j}\gamma^{jk}\le C(n,\rv)\epsilon+\Big(C(n,\rv)\epsilon)\Big)^2\, .
	\end{align}
	If we repeat this $d$-ball decomposition for each $B_{\gamma^2}(x_d^2)$,
 then after $i$ iterations of the decomposition we get 
	\begin{align}
		B_1(p)\subset \bigcup_{d=1}^{N_d^i}B_{\gamma^2}(x_d^2)\cup
 \bigcup_{j=1}^i\bigcup_{b=1}^{N_b^j}B_{\gamma^j}(x_b^j)\cup
 \bigcup_{j=1}^i\bigcup_{c=1}^{N_c^j}B_{\gamma^j}(x_c^j)\cup 
\bigcup_{j=1}^i\bigcup_{e=1}^{N_e^j}B_{\gamma^j}(x_e^j)\, ,
	\end{align}
	with content estimates
		\begin{align}
		&\sum_{d=1}^{N_d^i}\gamma^{ik}\le \Big(C(n,\rv)\epsilon\Big)^i\, ,\\
		&\sum_{j=1}^i\sum_{b=1}^{N_b^j}\gamma^{jk}+\sum_{j=1}^i\sum_{e=1}^{N_e^j}\gamma^{jk}\le  C(n,\gamma)\sum_{j=0}^{i-1}\Big(C(n,\rv)\epsilon\Big)^j\, ,\\ \label{e:summ_Ceps}
		&\sum_{j=1}^i\sum_{c=1}^{N_c^j}\gamma^{jk}\le \sum_{j=1}^i\Big(C(n,\rv)\epsilon)\Big)^j\, .
	\end{align}
	Let $\epsilon\le \epsilon(n,\rv)$ and $\gamma\le \gamma(n,\rv,\epsilon)$ be such that $\gamma$ and $\epsilon$ satisfies Theorem \ref{t:content_splitting} and $C(n,\rv)\epsilon\le 1/10$. 
\vskip2mm

Consider the discrete set $\tilde{\cS}_i^k:=
 \{x_d^i\}$. By construction, we have that $\tilde{\cS}_{i+1}^k\subset B_{\gamma^i}\tilde{\cS}_i^k$. Moreover, 
	\begin{align}\label{e:content_cS_dball}
		\Vol(B_{\gamma^i}\tilde{\cS}_i^k)\le \sum_{d=1}^{N_d^i}\Vol(B_{\gamma^i}(x_d^i))\le C(n)\sum_{d=1}^{N_d^i}\gamma^{in}\le C(n)\Big(C(n,\rv)\epsilon\Big)^i\gamma^{i(n-k)}\, .
	\end{align}
	Letting $i\to \infty$, we can define the Hausdorff limit of $\tilde{\cS}_i^k$ by 
		$\tilde{\cS}_d^k:=\lim_{i\to \infty}\tilde{\cS}_i^k.$
	By \eqref{e:content_cS_dball} and $\tilde{\cS}_{i+1}^k\subset B_{\gamma^i}\tilde{\cS}_i^k$, we have for any $i\ge 1$
	\begin{align}
		\Vol(B_{\gamma^i}\tilde{\cS}_d^k)\le C(n)10^{-i}\gamma^{i(n-k)},
	\end{align}
	which implies $\cH^k(\tilde{\cS}_d^k)=0$. Moreover, we have $\tilde{\cS}_d^k\subset S(X^n)$. To see this,  assume there exists $x\in \tilde{S}_d^k\setminus S(X)$, this implies for any $\epsilon'>0$ there exists $r_{x,\epsilon'}>0$ such that $d_{GH}(B_{r_{x,\epsilon'}}(x), B_{r_{x,\epsilon'}}(0^n))\le \epsilon' r_{x,\epsilon'}$. On the other hand, since $x\in\tilde{\cS}^k_d$ we have $\cF_{r_{x,\epsilon'},\xi}(x)$ is nonempty. Hence, applying the volume convergence in \cite{Co1}  and \cite{Cheeger01} to $B_{r_{x,\epsilon'}}(x)$  would imply $\bar V+\xi \ge 1-\epsilon''$ providing $\epsilon'\le \epsilon'(n,\rv,\epsilon'')$. Therefore we arrive at $\bar V\ge 1-\xi$ which implies $B_2(p)\subset \cF_{1,\xi}(p)$. In particular $\Vol(B_\gamma(\cF_{1,\xi}(p)))\ge \Vol(B_2(p))\ge \rv>0$ which contradicts with the $d$-ball assumption if $\epsilon\le \rv$.
\vskip2mm
	
	On the other hand, since $C(n,\rv)\epsilon\le 1/10$, the content estimates \eqref{e:summ_Ceps} is finite. Therefore, we arrive at the desired decomposition. This completes the proof of Proposition \ref{p:d_ball_decomposition}.
\end{proof}
\vskip2mm

\subsection{Proof of the $c$-ball covering Proposition \ref{p:c_ball_decomposition}}
\label{ss:c}

In this subsection we prove Proposition \ref{p:c_ball_decomposition}, which
is concerned with the  decomposition of a $c$-ball.
We will construct a neck region on $B_1(p)$ which is GH-close to a ball in some some cone $\mathbb{R}^k\times C(Y)$. 
\vskip2mm

\begin{proof}[Proof of Proposition \ref{p:c_ball_decomposition}]
		Recall that in the definition of neck region
		$\tau=\tau_n=10^{-10n}\omega_n$.
Fix  $\epsilon>0$ and $\gamma\le \gamma(n,\rv,\epsilon)$ such that the cone-splitting based on $k$-content of Theorem \ref{t:content_splitting} holds. By Theorem \ref{t:content_splitting} we have that $B_{{\delta'}^{-1}}(q)$ is $(k,\delta'^2)$-symmetric for some $q\in B_4(p)$. In particular, $B_{\delta'^{-1}}(q)$ is $\delta'^2$ close to a metric cone $\mathbb{R}^k\times C(Y)$. 
\vskip2mm

Denote the $\delta'^2$-GH map $\iota_{q,1}: B_{\delta'^{-1}}(0^k,y_c)\to B_{\delta'^{-1}}(q)$ and consider the approximate singular set $\cL_{q,1}:=  \iota_{q,1}(\mathbb{R}^k\times \{y_c\})\cap B_{4}(p)$. Choose a Vitali covering $\{B_{\gamma\tau^2}(x_f^1),~~x_f^1\in \cL_{q,1}\}$ of $\cL_{q,1}$ such that $B_{\gamma\tau^3}(x_f^1)$ are disjoint. 
\vskip2mm

We denote the balls $B_{2\gamma}(x_f^1)$ by:
\begin{itemize}
\item[1)]
$\tilde{b}$-balls if $B_{2\gamma}(x_f^1)$ is $(k+1,3\eta/2)$-symmetric, 
\vskip2mm

\item[2)]$\tilde{c}$-balls if $B_{2\gamma}(x_f^1)$ is not $(k+1,3\eta/2)$-symmetric and  $\Vol(B_{\gamma\cdot \gamma}\cF_{\gamma,\xi}(x_f^1))\ge \epsilon \gamma^{n-k}\gamma^n$,
\vskip2mm

\item[3)] $\tilde{d}$-balls if $\Vol(B_{\gamma\cdot \gamma}\cF_{\gamma,\xi}(x_f^1))< \epsilon \gamma^{n-k}\gamma^n$. 
\end{itemize}
\vskip2mm

We have 
	\begin{align}
		\cL_{q,1}\subset \bigcup_b B_{\tau^2\gamma}(\tilde{x}_b^1)\cup \bigcup_c B_{\tau^2\gamma}(\tilde{x}_c^1)\cup \bigcup_d B_{\tau^2\gamma}(\tilde{x}_d^1)\, .
	\end{align}
	Therefore we arrive at an approximate neck region $\tilde{\cN}^1$:
	\begin{align}
		\tilde{\cN}^1:=  B_2(p)\setminus \Big(\bigcup_b B_{\tau^2\gamma}(\tilde{x}_b^1)\cup \bigcup_c B_{\tau^2\gamma}(\tilde{x}_c^1)\cup \bigcup_d B_{\tau^2\gamma}(\tilde{x}_d^1)\Big)\, .
	\end{align}

The approximate neck 
 $\tilde{\cN}^1$ is not yet the one we are looking for, since $c$-ball content is not small.  
Therefore, we continue to refine the construction by redecomposing the $\tilde{c}$-balls in the decomposition.  Once again, by applying content splitting of Theorem \ref{t:content_splitting} to each $\tilde{c}$-ball, we have the approximate singular set $\cL_{\tilde{x}_c^1,\gamma}:=  \iota_{\tilde{x}_c^1,\gamma}(\mathbb{R}^k\times \{y_c\})\cap B_{4\gamma}(\tilde{x}_c^1)$ associated with a $\delta'^2 \gamma$-GH map $\iota_{\tilde{x}_c^1,\gamma}: B_{\gamma \delta'^{-1}}(0^k,y_c)\to B_{\gamma \delta'^{-1}}(\tilde{x}_c^1)$. 
\vskip2mm

Consider the Vitali covering $\{B_{\tau^2\gamma^2}(x_f^2)\}$ of 
	\begin{align}\label{e:tau3_cL}
		\bigcup_c\cL_{\tilde{x}_c^1,\gamma}\setminus \Big(\bigcup_b B_{\tau^3\gamma}(\tilde{x}_b^1)\cup \bigcup_d B_{\tau^3\gamma}(\tilde{x}_d^1)\Big)\, ,
	\end{align}
 such that $B_{\tau^4\gamma^2}(x_f^2)$ are disjoint and 
 \begin{align}
 	x_f^2\in \bigcup_c\cL_{\tilde{x}_c^1,\gamma}\setminus\Big(\bigcup_b B_{\tau^3\gamma}(\tilde{x}_b^1)\cup \bigcup_d B_{\tau^3\gamma}(\tilde{x}_d^1)\Big)\, .
 \end{align}
  In particular, if $\gamma\le 10^{-10}$ the balls $B_{\tau^4\gamma^2}(x_f^2)$ are also mutually disjoint with $B_{\tau^4\gamma}(\tilde{x}_b^1)$ and $B_{\tau^4\gamma}(\tilde{x}_d^1)$. 
 \vskip2mm
  
  We denote the ball $B_{2\gamma^2}(x_f^2)$ by $B_{2\gamma^2}(\tilde{x}_b^2)$, $B_{2\gamma^2}(\tilde{x}_c^2)$ and $B_{2\gamma^2}(\tilde{x}_d^2)$ according to the same scheme as above. Thus, we have 
  \begin{align}
  	\tilde{\cN}^2:=  B_2(p)\setminus  \left(\bigcup_c \bar B_{\gamma^i}(\tilde{x}_c^i)\cup \bigcup_{1\le j\le 2}\Big( \bigcup_b \bar B_{\gamma^j}(\tilde{x}_b^j)\cup \bigcup_d \bar B_{\gamma^j}(\tilde{x}_d^j)\Big)\right)\,.
  \end{align}

  After repeating this decomposition $i$ times to each $\tilde{c}$-ball,
  we get an approximate neck region given by
	\begin{align}
		\tilde{\cN}^i:=  B_2(p)\setminus  \left(\bigcup_c \bar B_{\gamma^i}(\tilde{x}_c^i)\cup \bigcup_{1\le j\le i}\Big( \bigcup_b \bar B_{\gamma^j}(\tilde{x}_b^j)\cup \bigcup_d \bar B_{\gamma^j}(\tilde{x}_d^j)\Big)\right)\,.
	\end{align}
	
Set 
$$
\tilde{\cC}_c^i:=\{\tilde{x}_c^i\}\, .
$$
	 By construction we have 
$$
\tilde{\cC}_c^{i+1}\subset B_{\gamma^i}(\tilde{\cC}_c^i)\, .
$$
 Therefore, we can define the Hausdorff limit 
	\begin{align}
		\tilde{\cC}_0:= \lim_{i\to\infty}\tilde{\cC}_c^{i}\, .
	\end{align}
By letting $i\to\infty$, we get
	\begin{align}
		\tilde{\cN}:= B_2(p)\setminus \left(\tilde{\cC}_0\cup
 \bigcup_b \bar B_{\tilde{r}_b}(\tilde{x}_b)\cup \bigcup_d \bar B_{\tilde{r}_d}(\tilde{x}_d)\right)\, .
	\end{align}
	
Set
 $$
 \tilde{\cC}_+:= \{\tilde{x}_d,\tilde{x}_b\}\, .
 $$
 By  construction, the balls $B_{\tau^4\tilde{r}_{\tilde{x}}}(\tilde{x})$ are disjoint 
for $\tilde{x}\in \tilde{\cC}_+$ and in addtion,
 \begin{align}\label{e:tilde_x_r_y}
 	\tilde{x}\notin \bigcup_{\tilde{y}\in \tilde{\cC}_{+}, \tilde{r}_{\tilde{y}}>
\tilde{r}_{\tilde{x}}}B_{\tau^3 \tilde{r}_{\tilde{y}}}(\tilde{y}).
 	 \end{align}

 Moreover, $\tilde{\cC}:=\tilde{\cC}_+\cup \tilde{\cC}_0$
 is a closed set. 
\vskip2mm

  It is easy to check that $\tilde{\cN}$ satisfies all the condition of a 
  $(k,\delta',\eta)$-neck except (n5) i.e. $\Lip\, \tilde{r}_x\leq \delta'$.
  Therefore, our construction requires some additional refinement. 
  \vskip2mm

    In the following construction, in which we refine our covering in order
 to get the desired $(k,\delta,\eta)$-neck,
    we will use $\tilde{x}_b,\tilde{x}_d\in\tilde{\cC}_+$ to denote the center 
of a $\tilde{b}$-ball and $\tilde{d}$-ball of $\tilde{\cN}$, and the associated radius $\tilde{r}_{\tilde{x}_b}$ and 
 $\tilde{r}_{\tilde{x}_d}$, respectively. 
\vskip2mm
	

	By the construction of $B_{\tilde{r}_{\tilde{x}}}(\tilde{x})$ with $\tilde{x}\in \tilde{\cC}$, there exists some $\tilde{c}$-ball $B_{\gamma^{-1}\tilde{r}_{\tilde{x}}}(\tilde{x}_c)$ which is $(k,\delta'^2)$-symmetric with respect to $\cL_{\tilde{x}_c,\gamma^{-1}\tilde{r}_{\tilde{x}}}$ such that $B_{\tilde{r}_{\tilde{x}}}(\tilde{x})\subset B_{2\gamma^{-1}\tilde{r}_{\tilde{x}}}(\tilde{x}_c)$ and $\tilde{x}\in \cL_{\tilde{x}_c,\gamma^{-1}\tilde{r}_{\tilde{x}}}$ . It is easy to see that for any $\gamma >r\ge \tilde{r}_{\tilde{x}}$ the ball $B_{\gamma^{-1}r}(\tilde{x}_c)$ is also $(k,\delta'^2)$-symmetric with respect to a set $\cL_{\tilde{x}_c,\gamma^{-1}r}$. This follows from the volume pinching estimate 
$$
|\cV_{\xi \tilde{r}_{\tilde{x}}}(\tilde{x}_c)-\cV_{\xi^{-1}}(\tilde{x}_c)|\le \xi\, ,
$$ 
and the fact that $B_{\gamma^{-1}r}(\tilde{x}_c)$ is $(k,\delta'^3)$-splitting since this ball is contained in a $(k,\delta'^3)$-symmetric ball with comparable radius. 
\vskip2mm

\noindent
{\bf Notation.}	
	For our convenience sake we will make the following notion:
\begin{align}
\label{e:associated_x}
\text{ For any $\tilde{x}\in\tilde{\cC}$ denote $\tilde{x}_c$ to be the associated center of $\tilde{c}$-ball satisfying the above properties.}
\end{align}

To refine the approximate neck $\tilde{\cN}$, let us build a good approximate singular set $\tilde{\cS}$. 
Indeed we define $\tilde{\cS}$ to be a subset of 
$\cup_{\tilde{x}\in\tilde{\cC}} B_{ \tau^3\tilde{r}_{\tilde{x}}}(\tilde{x})$ such that $y\in \tilde{\cS}$ if and only if one of the following holds
\begin{itemize}\label{item_cS}
	\item[(1)] $y\in \cL_{\tilde{x}_c,\gamma^{-1}\tilde{r}_{\tilde{x}}}$ with $d(y,\tilde{\cC})=d(y,\tilde{x})\le \tilde{r}_{\tilde{x}}$,
	\item[(2)]$y\in \cL_{\tilde{x}_c,\gamma^{-1}r}$ with $r :=  d(y,\tilde{\cC})=d(y,\tilde{x})>\tilde{r}_{\tilde{x}}$. 
 \end{itemize}
\vskip2mm

 Now we define a radius function on $\tilde{\cS}$ such that 
 \begin{align} \nonumber
 	 &\text{$r_x:=  \delta^2\tau^4\tilde{r}_{\tilde{x}}$ if $d(x,\tilde{\cC})=d(x,\tilde{x})\le \tau^4 \tilde{r}_{\tilde{x}}$,}\\ \label{e:definition_ry}
 	 &\text{$r_x:=  \delta^2 d(x,\tilde{\cC})$ otherwise.}
 \end{align}
 It is obvious that $|\Lip r_x|\le \delta^2$ and $\tilde{\cC}\subset \tilde{\cS}$. Choose a maximal disjoint collection $\{B_{\tau^2 r_x}(x),x\in \tilde{\cS}\}$ such that the center set $\cC_{+}\subset \tilde{\cS}$ containing $\tilde{\cC}$.  This allows us to build a neck region
\begin{align}
	\cN:=  B_2(p)\setminus \Big(\tilde{\cC}_0\cup \bigcup_{x\in\cC_{+}}\bar B_{r_x}(x)\Big).
\end{align}
\vskip2mm



\noindent
{\bf Notation.} In order to make the notations consistent, we put 
$\cC_0:=  \tilde{\cC}_0$ and $\cC:=\cC_{+}\cup \cC_0$.
\vskip2mm	
	
Next, we will check that $\cN$ is a $(k,\delta,\eta)$-neck if $\delta'\le \delta'(\delta,\gamma,\eta)$ small. 
\vskip2mm

The Lipschitz condition $(n5)$ and the Vitali condition $(n1)$ in the neck region are satisfied by the construction. If $\delta'\le \delta'(\delta,\gamma,\eta)$,  let us check the volume ratio condition (n2). In fact,  for any $x\in\cC$, let $\tilde{x}\in\tilde{\cC}$ be such that $d(x,\tilde{\cC})=d(x,\tilde{x})$. Denote $\tilde{x}_c$ the associated center point of $\tilde{c}$-ball such that $B_{\tilde{r}_{\tilde{x}}}(\tilde{x})\subset B_{\gamma^{-1}\tilde{r}_{\tilde{x}}}(\tilde{x}_c)$ .  For $\delta'\le \delta'(n,\rv,\delta)$ small and $y\in B_{\delta'r_x}\cL_{\tilde{x}_c,\delta^{-3}r_x}$, we always have $|\cV_{\delta r_x}(x)-\cV_{\delta r_x}(y)|\le \delta^{100}$ since $B_{r_x}(x)$ and $B_{r_x}(y)$ are $\delta'r_x$-close to the same cone at scale $r_x$.  On the other hand, by the definition of a $\tilde{c}$-ball, there exists 
 $$
 y\in  B_{\delta'r_x}\cL_{\tilde{x}_c,\delta^{-3}r_x}\cap \cF_{\delta^{-3}r_x,\xi}(\tilde{x}_c)\, .
$$ 
 This implies $|\cV_{\xi \delta^{-3}r_x}(y)-\bar V|\le \xi$. Therefore, if $\xi\le \delta^{20}$ we conclude that $|\cV_{\delta r_x}(x)-\bar V|\le \delta^{15}$.  By \eqref{e:monotone}, the monotonicity of the volume ratio, 
 we finally get 
 $$
 |\cV_{\delta r_x}(x)-\cV_{\delta^{-1}}(x)|\le \delta^{10}\, .
 $$
 Thus, the volume ratio condition (n2) is satisfied.
\vskip2mm

	 Next, note that if $\delta'\le \delta'(\delta,\gamma,\eta)$, then by the definition of $\tilde{\cS}$  
we will have for $x\in\cC$ that $B_{r}(x)$ is not $(k+1,\eta)$-symmetric and $B_r(x)$ is $(k,\delta^2)$-symmetric 
for all $\delta^{-1}\ge r\ge r_x$. To see this, first note that $B_r(x)$ is $(0,\delta^3)$-symmetric by the volume 
pinching estimate (n2) for $\delta'\le \delta'(n,\rv,\eta,\delta)$. On the other hand, $B_r(x)$  is a subset 
of $B_{\gamma^{-1}r}(\tilde{x}_c)$ with comparable sizes, which is $(k,\delta'^2)$-symmetric with respect 
to $\cL_{\tilde{x}_c,\gamma^{-1}r}$ but not $(k+1,3\eta/2)$-symmetric.
 From this we  conclude that $B_{r}(x)$ is not $(k+1,\eta)$-symmetric and $B_r(x)$ is $(k,\delta^2)$-symmetric.  
Hence we prove the condition (n3).
\vskip2mm
	
The covering condition (n4), which says that the approximate singular set $\cL_{x,r}$ with $r\ge r_x$ is
 covered by $B_{\tau r}(\cC)$, is satisfied by the construction of $\tilde{\cN}$ and $\cN$. To see this, for
 each $x\in \cC$, denote the associated $\tilde{x}\in \tilde{\cC}$ with $d(x,\tilde{\cC})=d(\tilde{x},x)$. 
Moreover, denote $\tilde{x}_c$ the associated center point of $\tilde{c}$-ball such that 
$B_{\tilde{r}_{\tilde{x}}}(\tilde{x})\subset B_{\gamma^{-1}\tilde{r}_{\tilde{x}}}(\tilde{x}_c)$. 
Since $B_{\gamma^{-1}r}(\tilde{x}_c)$ is a $\tilde{c}$-ball, by the cone-splitting Theorem 
\ref{t:cone_splitting} and since $B_{\gamma^{-1}r}(\tilde{x}_c)$ is not $(k+1,3\eta/2)$-symmetric, 
we have $\cL_{x,r}\subset B_{\tau r/4}(\cL_{\tilde{x}_c,\gamma^{-1}r})$.

 On the other hand, by the construction of the approximate neck $\tilde{\cN}$, we have that 
$\cL_{\tilde{x}_c,\gamma^{-1}r}\subset B_{\tau r/4}(\tilde{\cC})$. Noting from the construction
 of $\cC$ that $\tilde{\cC}\subset \cC$, we arrive at 
\begin{align}
\cL_{x,r}\subset B_{\tau r/4}\cL_{\tilde{x}_c,\gamma^{-1}r}\subset B_{\tau r/2}(\tilde{\cC})\subset B_{\tau r/2}(\cC),
\end{align}
which proves the condition (n4).  Therefore, we have shown that if $\xi\le \xi(n,\rv,\gamma,\delta,\eta)$, then $\cN$ is a $(k,\delta,\eta)$-neck.
 \vskip2mm

	We now focus on the content estimates of Proposition \ref{p:c_ball_decomposition}. 
\vskip2mm

\noindent
{\bf Notation.}	
Denote the ball $B_{r_x}(x)$ with $x\in\cC$ by $B_{r_b}(x_b)$, $B_{r_c}(x_c)$, $B_{r_d}(x_d)$ and $B_{r_e}(x_e)$, where $B_{2r_b}(x_b)$ is $(k+1,2\eta)$-symmetric, $B_{2r_c}(x_c)$ is not $(k+1,2\eta)$-symmetric and $\Vol(B_{\gamma r_c}\cF_{r_c,\xi}(x_c))\ge \epsilon \gamma^{n-k}r_c^n$, $B_{2r_d}(x_d)$ satisfies $\Vol(B_{\gamma r_d}\cF_{r_d,\xi}(x_d))< \epsilon \gamma^{n-k}r_d^n$ and $B_{2r_e}(x_e)$ satisfies $\cF_{r_e,\xi}(x_e)=\emptyset$. 
\vskip2mm
	
	The content estimates rely on the neck structure of Theorem \ref{t:neck_region2}:
\begin{align}
\label{e:content9}
		\cH^k({\cC}_0\cap B_{15/8})+\sum_{x_b\in B_{15/8}}r_b^{k}+\sum_{x_c\in B_{15/8}} r_c^k+\sum_{x_d\in B_{15/8}} r_d^k+\sum_{x_e\in B_{15/8}} r_e^k\le C(n)\, .
	\end{align}
	In order to finish the proof, it suffices to show the content of $c$-balls is small. This is reasonable since the approximate neck $\tilde{\cN}$ doesn't contain any $\tilde{c}$-ball's at all.  Thus, we need to verify that our process of going from $\tilde \cN$ to $\cN$ did not create too many $c$-balls.
\vskip2mm

	Denote the center of the $c$-balls $B_{r_c}(x_c)$ by a subset $\cC_c\subset\cC\cap B_{3/2}(p)$. From the construction of the approximate neck $\tilde{\cN}$ and the definition of $\tilde{\cS}$ it follows $\tilde{\cS}\subset \cup_{\tilde{x}\in \tilde{\cC}}B_{\tau \tilde{r}_{\tilde{x}}}(\tilde{x})$. To consider $\mu(\cC_c)$ \footnote{As usual, let $\mu=\sum_{x\in\cC}r_x^k\delta_x+\cH^k|_{{\cC}_0}$ is the packing measure associated with the $(k,\delta,\eta)$-neck region $\cN$.}we will restrict to each $B_{\tilde{r}_{\tilde{x}}}(\tilde{x})$ with $\tilde{x}\in \tilde{\cC}$. Let us first consider the content of $\cC_c$ in $B_{\tilde{r}_d}(\tilde{x}_d)$. Since $\cF_{\tilde{r}_d,\xi}(\tilde{x}_d)$ has small volume  we have the following lemma:
\vskip2mm

\begin{lemma}
\label{l:claim81}	
$$\mu\Big(\cC_c\cap \bigcup_{\tilde{x}_d\in\tilde{\cC}}B_{3\tilde{r}_d/2}(\tilde{x}_d)\Big)\le C(n,\rv)\epsilon \, ,
$$ 
\end{lemma}
\begin{proof} We will see that  it suffices to prove that for each $\tilde{d}$-ball $B_{\tilde{r}_d}(\tilde{x}_d)$ 
with $\tilde{x}_d\in\tilde{\cC}$, we have
\begin{align}\label{e:cC_c_tilde_d}
	\mu\Big(\cC_c\cap B_{3\tilde{r}_d/2}(\tilde{x}_d)\Big)\le C(n,\rv)\epsilon\, 
\mu\Big(B_{\tau^4\tilde{r}_d}(\tilde{x}_d)\Big)\, .
\end{align}  In fact,  since $B_{\tau^4\tilde{r}_d}(\tilde{x}_d)$ are disjoint and $\mu$ is a doubling measure
 with $\tau=\tau_n$, we have by \eqref{e:cC_c_tilde_d} that
\begin{align}
	\mu\Big(\cC_c\cap \bigcup_{\tilde{x}_d\in\tilde{\cC}}B_{3\tilde{r}_d/2}(\tilde{x}_d)\Big)&
\le \sum_{\tilde{x}_d\in\tilde{\cC}}\mu\Big(\cC_c\cap B_{3\tilde{r}_d/2}(\tilde{x}_d)\Big)\le C(n,\rv)\epsilon
 \sum_{\tilde{x}_d\in\tilde{\cC}}\mu\Big(B_{\tau^4\tilde{r}_d}(\tilde{x}_d)\Big)\le C(n,\rv)\epsilon \mu(B_{15/8}(p))
\le C(n,\rv)\epsilon\, ,
\end{align}
where we have used the Neck Structure Theorem \ref{t:neck_region2} in the last inequality.  
Thus,  we only need to prove \eqref{e:cC_c_tilde_d}. 
\vskip2mm

By the definition of $r_y$ in \eqref{e:definition_ry}, we have for any $x_c\in\cC_c\cap 
B_{3\tilde{r}_d/2}(\tilde{x}_d)$ that $r_c\le 10\delta^2\tilde{r}_d$. Since $B_{r_c}(x_c)$ is a $c$-ball, there 
exists  $y\in B_{4r_c}(x_c)$ such that $|\cV_{\xi r_c}(y)-\bar V|\le \xi$. In particular, this implies 
$y\in \cF_{\tilde{r}_d,\xi}(\tilde{x}_d)\cap B_{5\tilde{r}_d/3}(\tilde{x}_d)$ and 
$B_{\gamma \tilde{r}_d/10}(y)\subset B_{\gamma \tilde{r}_d}\cF_{\tilde{r}_d,\xi}(\tilde{x}_d)$. 
On the other hand, since $r_c\le 10\delta^2\tilde{r}_d$ and $d(y,x_c)\le 4r_c$ we have 
$B_{\gamma \tilde{r}_d/20}(x_c)\subset B_{\gamma \tilde{r}_d/10}(y)
\subset B_{\gamma \tilde{r}_d}\cF_{\tilde{r}_d,\xi}(\tilde{x}_d)$.  
\vskip2mm

Consider a maximal disjoint collection $\{B_{\gamma \tilde{r}_d/20}(x'_i),~x'_i\in 
\cC_c\cap B_{3\tilde{r}_d/2}(\tilde{x}_d)\}$ with cardinality $N$. We have 
\begin{align}
	NC(n,\rv)\gamma^n\tilde{r}_d^n\le \sum_{x'_i} \Vol(B_{\gamma \tilde{r}_d/20}(x'_i))
\le \Vol\Big( B_{\gamma \tilde{r}_d}\cF_{\tilde{r}_d,\xi}(\tilde{x}_d)\Big)<\epsilon \gamma^{n-k}\tilde{r}_d^n\, .
\end{align}
Therefore,  we have $N\le \epsilon C(n,\rv) \gamma^{-k}$ and so,
\begin{align}
	\mu\Big(\cC_c\cap B_{3\tilde{r}_d/2}(\tilde{x}_d)\Big)\le \sum_{x_i'}\mu\Big(\cC_c\cap B_{\gamma \tilde{r}_d/5}(x_i')\Big)\le C(n)N \gamma^k\tilde{r}_d^k\le \epsilon C(n,\rv)\tilde{r}_d^k\le \epsilon C(n,\rv)\mu (B_{\tau^4\tilde{r}_d}(\tilde{x}_d))\, .
\end{align}
This  finishes the proof of \eqref{e:cC_c_tilde_d}. 
Thus,  the proof of Lemma \ref{l:claim81} is complete.
\end{proof}
\vskip2mm
	
	Having controlled the content of $\cC_c$ in $\tilde{d}$-balls in Lemma
\ref{l:claim81}, we will now consider the content of $\cC_c$ in $\tilde{b}$-balls. 
However, unlike the case of $\tilde{d}$-balls, there exists no  a priori small volume set. Thus, we will need to argue in a different way from in Lemma \ref{l:claim81}. 
\vskip2mm

\begin{remark}
\label{r:outline}
Prior to beginning the proof proper, we will give a brief
indication of the argument.
\vskip2mm

\vskip2mm
	
		Let $x_c\in B_{\tilde{r}_b}(\tilde{x}_b)$. In the definition of $\tilde{\cS}$ (see \eqref{e:definition_ry}) we
 require that $\tilde{\cS}$ is a subset of $\cup_{\tilde{x}\in\tilde{\cC}}B_{\tau^3\tilde{r}_{\tilde{x}}}(\tilde{x})$. 
Therefore, we may assume $x_c\in B_{\tau^3\tilde{r}_b}(\tilde{x}_b)$. Since each $\tilde{b}$-ball is 
$(k+1,3\eta/2)$-symmetric and each $c$-ball is not $(k+1,2\eta)$-symmetric, this will force the $c$-ball 
$B_{r_c}(x_c)\subset B_{2\tilde{r}_{b}}(\tilde{x}_b)$ to have small radius $r_c<<\tilde{r}_b$. From 
the definition of $r_x$ in \eqref{e:definition_ry} there must exist some $\tilde{x}\in\tilde{\cC}$ with 
$d(x_c,\tilde{x})=d(x_c,\tilde{\cC})$ such that $\tilde{r}_{\tilde{x}}<< \tilde{r}_b$. By \eqref{e:tilde_x_r_y}
 we have that $\tilde{x}\notin B_{\tau^3\tilde{r}_b}(\tilde{x}_b)$.  Thus one sees from the definition of $r_x$ 
in \eqref{e:definition_ry} that $x_c\in A_{\tau^3\tilde{r}_b (1-\delta),\tau^3\tilde{r}_b}(\tilde{x}_b)$; see 
\eqref{e:cCdelta} below for further details. Therefore the content estimate of $\cC_c\cap B_{\tilde{r}_b}(\tilde{x}_b)$
 would be controlled by the content estimate of  $\cC_c\cap A_{\tau^3\tilde{r}_b (1-\delta),\tau^3\tilde{r}_b}(\tilde{x}_b)$ 
which is small by a simple covering argument.  
\end{remark}
\vskip2mm

Now we begin in the proof the content estimate of $\cC_c\cap B_{\tilde{r}_b}(\tilde{x}_b)$ more carefully.  For $0<\tilde{\delta}<\epsilon^3$, we define a subset of $\cC_c$ by those points with small radius compared with a $\tilde{b}$-ball by 
	\begin{align}\label{e:cCdelta}
		\cC_{\tilde{\delta}}:=  \bigcup_{\tilde{x}_b\in\tilde{\cC}_+}\{y\in \cC_c\cap B_{\tilde{r}_b}(\tilde{x}_b):~~r_y\le \delta^4\tilde{\delta}\tau^4 \tilde{r}_b\}\,.
	\end{align}
	We will see below that $\cC_{\tilde{\delta}}$ contains all the centers of $c$-balls inside $\tilde{b}$-balls.
\vskip2mm
	
	Indeed, note that if $\xi\le \xi(\tilde{\delta},\delta,\epsilon,n,\rv,\gamma)$ and $r_y\ge \delta^4\tau^4\tilde{\delta}\tilde{r}_b$ with $y\in\cC\cap B_{\tilde{r}_b}(\tilde{x}_b)$, then the ball $B_{2r_y}(y)$ is $(k+1,5\eta/3)$-symmetric which in particular implies that $B_{2r_y}(y)$ is not a $c$-ball. Therefore, to estimate the content of $c$-ball in $\tilde{b}$-balls, it will suffice to consider the set $\cC_{\tilde{\delta}}$. 
\vskip2mm
	
From the definition of $\cC_{\tilde{\delta}}$, we can see that $\cC_{\tilde{\delta}}\cap
 B_{(1-\tilde{\delta})\tau^3\tilde{r}_b}(\tilde{x}_b)=\emptyset$. In fact, if 
$y\in \cC_{\tilde{\delta}}\cap B_{(1-\tilde{\delta})\tau^3\tilde{r}_b}(\tilde{x}_b)$, then 
by the definition of $r_y$ there must exist $\tilde{x}\in\tilde{\cC}$ such that 
$d(y,\tilde{x})=d(y,\tilde{\cC})$ with $\tilde{r}_{\tilde{x}}<\tilde{r}_b$, which
 implies $\tilde{x}\notin B_{\tau^3\tilde{r}_b}(\tilde{x}_b)$ by \eqref{e:tilde_x_r_y}. 
This is enough to deduce a contradiction with $\delta^2 d(y,\tilde{x})=r_y
\le \delta^4\tau^4\tilde{\delta}\tilde{r}_b$ since $d(y,\tilde{x})\ge \tilde{\delta}\tau^3\tilde{r}_b$. Therefore, we have showed that 
	\begin{align}\label{e:tilde_delta}
		\cC_{\tilde{\delta}}\cap B_{(1-\tilde{\delta})\tau^3\tilde{r}_b}(\tilde{x}_b)=\emptyset.
	\end{align}
\vskip2mm
	
	By removing the points in $\tilde{d}$-balls we have the following content estimate of $\cC_{\tilde{\delta}}$:

Let   $B_{\tilde{r}_d}(\tilde{x}_d)$ denote the $\tilde{d}$-balls with $\tilde{x}_d\in \tilde{\cC}$ in the approximate neck region $\tilde{\cN}$.
\vskip2mm
	
	By taking out the points in $\tilde{d}$-balls we have the following content estimate of $\cC_{\tilde{\delta}}$

\begin{lemma}
\label{l:claim82} Let $\tilde{\delta}\le \tilde{\delta}(\gamma,\epsilon)\le \epsilon^3$ and $\xi\le \xi(\tilde{\delta},\delta,n,\rv,\gamma,\epsilon,\eta)$.
\vskip2mm

Then
	\begin{align}
		\mu(\cC_{\tilde{\delta}}\setminus \Big(\bigcup_d B_{\tilde{r}_d}(\tilde{x}_d)\Big)\le \epsilon^2\, .
	\end{align}	
\end{lemma}
\begin{proof}
We will divide the proof into two steps. 
\vskip2mm

In the first step, we consider content of $\cC_{\tilde{\delta}}\cap B_{\tau^3\tilde{r}_b}(\tilde{x}_b)$ which is actually equal to the content of $\cC_{\tilde{\delta}}\cap A_{(1-\tilde{\delta})\tau^3\tilde{r}_b,\tau^3\tilde{r}_b}(\tilde{x}_b)$ 
by \eqref{e:tilde_delta}.
\vskip2mm

 In the second step,  we will consider content of the remainder $\cC_{\tilde{\delta}}\cap 
A_{\tau^3\tilde{r}_b,\tilde{r}_b}(\tilde{x}_b)$ which is zero after taking out all the points of $B_{\tilde{r}_d}(\tilde{x}_d)$. 
\vskip2mm

 The reason for  this division of cases based on the radius $\tau^3\tilde{r}_b$  is that
 the construction of the approximate neck $\tilde{\cN}$ satisfies \eqref{e:tau3_cL} and \eqref{e:tilde_x_r_y}. 
\vskip2mm	
	
	\textbf{Step 1:}  Denote $\cC_{\tilde{\delta},1}:=
\cC_{\tilde{\delta}}\cap \left(\bigcup_{\tilde{x}_b\in\tilde{\cC}} 
B_{\tau^3\tilde{r}_b}(\tilde{x}_b)\right)$. We will show that $\mu(\cC_{\tilde{\delta},1}
\setminus \Big(\bigcup_d B_{\tilde{r}_d}(\tilde{x}_d)\Big))\le C(n,\gamma)\tilde{\delta}$. 
By the Ahlfors regularity of measure $\mu$, it will suffice to prove 
$\mu(\cC_{\tilde{\delta}}\cap B_{\tau^3\tilde{r}_b}(\tilde{x}_b))
\le C(n,\gamma)\tilde{\delta}\mu(B_{\tau^4\tilde{r}_b}(\tilde{x}_b))$ 
for each $B_{\tilde{r}_b}(\tilde{x}_b)$. Since $\cC_{\tilde{\delta}}\cap 
B_{(1-\tilde{\delta})\tau^3\tilde{r}_b}(\tilde{x}_b)=\emptyset$ in \eqref{e:tilde_delta}, we will only need to prove
	\begin{align}\label{e:delta_mu_measure}
		\mu\Big(\cC_{\tilde{\delta}}\cap 
A_{(1-\tilde{\delta})\tau^3\tilde{r}_b,(1+\tilde{\delta})\tau^3\tilde{r}_b}(\tilde{x}_b)\Big)
		\le C(n,\gamma)\cdot
		\tilde{\delta}
		\cdot
		\mu(B_{\tau^4\tilde{r}_b}(\tilde{x}_b))\, .
	\end{align}
	Let us prove \eqref{e:delta_mu_measure}. In fact, by the construction of $\tilde{\cN}$, there exists a $\tilde{c}$-ball $B_{2\gamma^{-1}\tilde{r}_b}(\tilde{x}_c)$ as in \eqref{e:associated_x} which is not $(k+1,3\eta/2)$-symmetric such that $B_{\tilde{r}_b}(\tilde{x}_b)\subset B_{2\gamma^{-1}\tilde{r}_b}(\tilde{x}_c)$.  Let $\cL_{\tilde{x}_c,\gamma^{-1}\tilde{r}_b}$ be the set the $\tilde{c}$-ball $B_{\gamma^{-1}\tilde{r}_b}(\tilde{x}_c)$ is $(k,\delta')$-symmetric with respect to.  For $\xi\le \xi(\tilde{\delta},\delta,\epsilon,\rv,n,\gamma,\eta)$, we must have 
	\begin{align}
	\label{e:GH_r_b_cC}
		\cC_{\tilde{\delta}}\cap A_{(1-10\tilde{\delta})\tau^3\tilde{r}_b,(1+10\tilde{\delta})\tau^3\tilde{r}_b}(\tilde{x}_b)\subset \Big(B_{\tilde{\delta}^2\gamma^{-1}\tau^3\tilde{r}_b}\cL_{\tilde{x}_c,\gamma^{-1}\tilde{r}_b} \cap A_{(1-10\tilde{\delta})\tau^3\tilde{r}_b,(1+10\tilde{\delta})\tau^3\tilde{r}_b}(\tilde{x}_b)\Big)\, .
	\end{align}
	Otherwise, there will be another splitting factor for $B_{\gamma^{-1}\tilde{r}_b}(\tilde{x}_c)$ which would contradict with the fact that $B_{\gamma^{-1}\tilde{r}_b}(\tilde{x}_c)$ is not $(k+1,3\eta/2)$-symmetric. 
\vskip2mm

Now consider a collection of maximal disjoint balls 
	\begin{align}
		\{B_{\tilde{\delta}\tau^3\tilde{r}_b}(x_c),~~x_c\in \cC_{\tilde{\delta}}\cap A_{(1-\tilde{\delta})\tau^3\tilde{r}_b,(1+\tilde{\delta})\tau^3\tilde{r}_b}(\tilde{x}_b) \}.
	\end{align}
 Denote this set by $\{B_{\tilde{\delta}\tau^3\tilde{r}_b}(x_i), i=1,\cdots, K\}$ with cardinality $K$. 
By the covering property in \eqref{e:GH_r_b_cC}, we have  $K\le C(n,\gamma)\tilde{\delta}^{1-k}$. Therefore, we arrive at
\begin{align}
	\mu\Big(\cC_{\tilde{\delta}}\cap A_{(1-\tilde{\delta})\tau^3\tilde{r}_b,(1+\tilde{\delta})\tau^3\tilde{r}_b}(\tilde{x}_b)\Big)
	&\le \sum_{i=1}^K\mu \Big(B_{3\tilde{\delta}\tau^3\tilde{r}_b}(x_i)\Big)
	\notag\\
	&\le K\cdot A(n)\cdot \tau^{3k}\cdot \tilde{\delta}^k\cdot \tilde{r}_b^k \notag\\
	&\le C(n,\gamma)\cdot \tilde{\delta}\cdot \tilde{r}_b^k\notag\\
	&\le C(n,\gamma)\cdot \tilde{\delta}\cdot 
	\mu(B_{\tau^4\tilde{r}_b}(\tilde{x}_b))\, ,
\end{align}
where we have used the Ahlfors regularity of $\mu$ for $(k,\delta,\eta)$-neck regions in Theorem \ref{t:neck_structure}. Thus, we have proved \eqref{e:delta_mu_measure}.  Since $B_{\tau^4\tilde{r}_b}(\tilde{x}_b)$ are disjoint and $\cC_{\tilde{\delta}}\cap B_{(1-\tilde{\delta})\tau^3\tilde{r}_b}(\tilde{x}_b)=\emptyset$, and noting the definition of $\cC_{\tilde{\delta},1}$, we have
\begin{align}
\mu(\cC_{\tilde{\delta},1}\setminus \Big(\bigcup_d B_{\tilde{r}_d}(\tilde{x}_d)\Big))&\le \sum_{\tilde{x}_b\in \tilde{\cC}}\mu (B_{\tau^3\tilde{r}_b}(\tilde{x}_b)\cap \cC_{\tilde{\delta}})\le \sum_{\tilde{x}_b\in \tilde{\cC}}\mu (A_{(1-\tilde{\delta})\tau^3\tilde{r}_b,\tilde{r}_b}(\tilde{x}_b)\cap \cC_{\tilde{\delta}})\\
&\le C(n,\gamma)\tilde{\delta}\sum_{\tilde{x}_b\in \tilde{\cC}}\mu (B_{\tau^4\tilde{r}_b}(\tilde{x}_b))\le C(n,\gamma)\tilde{\delta}\mu(B_2(p))\le C(n,\gamma)\tilde{\delta}\, .
\end{align}
\vskip2mm

	\textbf{Step 2:}   Denote $\cC_{\tilde{\delta},2}:= \cC_{\tilde{\delta}}\setminus \cC_{\tilde{\delta},1}$ to be the centers of $c$-balls outside $B_{\tau^3\tilde{r}_b}(\tilde{x}_b)$. We will show that $\cC_{\tilde{\delta},2}\setminus \Big(\bigcup_d B_{\tilde{r}_d}(\tilde{x}_d)\Big)=\emptyset$ which will be good enough to conclude the estimate in Claim 2 by combining with the estimate in step 1.  To this end, one key ingredient is the construction of the approximate neck satisfying \eqref{e:tau3_cL} and \eqref{e:tilde_x_r_y}, which roughly implies there exits no approximating singular set outside $B_{\tau^3\tilde{r}_{\tilde{x}}}(\tilde{x})$ with $\tilde{x}\in \tilde{\cC}$. 
\vskip2mm
	
	For a given $\tilde{b}$-ball $B_{\tilde{r}_b}(\tilde{x}_b)$ consider
		$\cC_{\tilde{\delta},b}:=  \Big(\cC_{\tilde{\delta},2}\cap B_{\tilde{r}_b}(\tilde{x}_b)\Big)\setminus  \left(\bigcup_{x\in\tilde{\cC},\tilde{r}_x< \tilde{r}_b}B_{\tilde{r}_x/2}(\tilde{x})\right)$. Let us see that it will suffice to prove $\cC_{\tilde{\delta},b}=\emptyset$ for any $\tilde{x}_b\in\tilde{\cC}$. Indeed, assume $\cC_{\tilde{\delta},b}=\emptyset$ for any $\tilde{x}_b\in\tilde{\cC}$, then we will show that $\cC_{\tilde{\delta},2}\setminus \Big(\bigcup_d B_{\tilde{r}_d}(\tilde{x}_d)\Big)= \emptyset$.  Assume there exists $y\in \cC_{\tilde{\delta},2}\setminus \Big(\bigcup_d B_{\tilde{r}_d}(\tilde{x}_d)\Big)\ne \emptyset$, then let $B_{\tilde{r}_b}(\tilde{x}_b)$ be the minimal sized ball such that $y\in B_{\tilde{r}_b}(\tilde{x}_b)$. We must have $\tilde{r}_b>0$ since otherwise $y\in\tilde{\cC}_0$, which is not a point in $\cC_c$.  Therefore, we have $y\in \cC_{\tilde{\delta},b}$.  But this is a contradiction as $\cC_{\tilde{\delta},b}=\emptyset$. 			
\vskip2mm			
			
	We will now prove that $\cC_{\tilde{\delta},b}=\emptyset$.  For a given $\tilde{b}$-ball $B_{\tilde{r}}(\tilde{x}_b)$ if there exists $y\in \cC_{\tilde{\delta},b}$, then by the definitions of $r_y$,  $\cC_{\tilde{\delta},b}$ and $\cC_{\tilde{\delta}}$ there must exist $\tilde{y}\in\tilde{\cC}$ such that 
		\begin{align}
			\delta\tilde{\delta}\tau^4\tilde{r}_b>s :=  d(y,\tilde{y})=d(y,\tilde{\cC})\ge \tilde{r}_{\tilde{y}}/2\, .
		\end{align}
Here, the last inequality follows from the definition of $\cC_{\tilde{\delta},b}$, while the first inequality follows from the definitions of $r_y$ and 
$\tilde{\cC}_{\tilde{\delta}}$.
\vskip2mm

 Let $B_{\gamma^{-1}\tilde{r}_{\tilde{y}}}(\tilde{y}_c)$ be the associated $\tilde{c}$-ball covering $B_{\tilde{r}_{\tilde{y}}}(\tilde{y})$ as in \eqref{e:associated_x}. Then by the definition of $\tilde{\cS}$ we have $y\in \cL_{\tilde{y}_c,\gamma^{-1}s}$. By the construction of the approximating neck $\tilde{\cN}$ through \eqref{e:tau3_cL} and $s\le \delta\gamma^2\tilde{r}_b$, we have 
\begin{align}
	\cL_{\tilde{y}_c,\gamma^{-1}s}\subset \bigcup_{\tilde{x}\in\tilde{\cC}}B_{\tau^3\tilde{r}_{\tilde{x}}}(\tilde{x})\cup \bigcup_{\tilde{x}\in\tilde{\cC},\tilde{r}_{\tilde{x}}<\tilde{r}_b}B_{\tau s}(\tilde{x}).
\end{align}
Since $y\notin B_{\tau^3\tilde{r}_{\tilde{x}}}(\tilde{x})$ by the definition of $\cC_{\tilde{\delta},2}$, there exists $\tilde{x}\in \tilde{\cC}$ such that $d(y,\tilde{x})\le \tau s<s$ which contradicts that $d(y,\tilde{\cC})=s$. Therefore we finish step two.
\vskip2mm

Fix $\tilde{\delta}\le \tilde{\delta}(n,\gamma,\epsilon)$. Since 
$\cC_{\tilde{\delta}}=\cC_{\tilde{\delta},1}\cup\cC_{\tilde{\delta},2}$ 
by combining Step 1 and Step 2, we complete the proof of Lemma \ref{l:claim82}.
\end{proof}

The bound on content of $c$-balls
 follows easily from Lemma \ref{l:claim81} and Lemma \ref{l:claim82}.
This completes the proof of Proposition \ref{p:c_ball_decomposition}
and hence, of Proposition \ref{p:inductive_decomposition} and Theorem \ref{t:decomposition2}
as well.
\end{proof}

\bibliographystyle{plain}

\end{document}